\documentclass[a4paper]{amsbook}

\usepackage{leftidx}
\usepackage{amsthm}
\usepackage{epsfig}
\usepackage{alltt}
\usepackage[active]{srcltx}
\usepackage{newlfont}
\usepackage{amsmath}
\usepackage{amssymb}
\usepackage{amsfonts}
\usepackage{amscd}
\usepackage{stmaryrd}
\usepackage{enumerate}
\usepackage{verbatim}
\usepackage{changebar}
\usepackage{prettyref}
\usepackage{hyperref}
\usepackage{mathrsfs}
\usepackage[all, arc]{xy}
\SelectTips{eu}{}

\usepackage{ulem}

\makeindex

\def\11{\mathbf{1}}

\newcommand{\Bl}[1]{{\mathbb{#1}}}

\def\cE{\mathcal{E}}

\def\cG{\mathcal{G}}

\newcommand{\const}[1]{\underline{#1}}
\def\conv{\star}

\def\cT{\mathcal{T}}
\def\cX{\mathcal{X}}

\def\CC{\mathbb{C}} 

\newcommand{\DC}{\Bl{C}}
\def\DD{\mathbf{D}}

\def\DMT{\mathrm{MTDer}}

\newcommand{\DQ}{\Bl{Q}}

\newcommand{\DZ}{\Bl{Z}}

\newcommand{\et}{\mathrm{\acute{e}t}}

\def\End{\mathrm{End}}

\def\fin{\mathrm{fin}}

\def\For{\mathrm{For}}

\def\GG{\mathbb{G}}

\def\Ho{\mathrm{Hot}^\mathrm{b}}
\def\Hom{\mathrm{Hom}}

\def\id{\mathrm{id}}

\def\iHom{\underline{\Hom}}
\def\ind{\mathrm{Ind}}

\newcommand{\op}{\operatorname}

\newcommand{\pdef}{\mathrel{\mathop:}=}
\def\pt{\mathrm{pt}}

\def\PP{\mathbb{P}}

\newcommand{\qtimes}[1]{\times_{/#1}}

\def\QQ{\mathbb{Q}} 

\def\res{\mathrm{Res}}

\def\Spec{\mathrm{Spec}\,}

\newcommand{\ttimes}{\mathop{\widetilde{\boxtimes}}}

\def\ZZ{\mathbb{Z}}

\newcommand{\hra}{\hookrightarrow}

\newcommand{\mapright}[1]{\xrightarrow{#1}}

\newcommand{\ra}{\rightarrow}

\newcommand{\RA}{\Rightarrow}

\newcommand{\sira}{\mapright{\sim}}
\newcommand{\sirra}{\mapright{\approx}}
\newcommand{\sra}{\twoheadrightarrow}

\renewcommand{\thechapter}{\Roman{chapter}}
\renewcommand{\thesection}{\thechapter.\arabic{section}}

\newtheorem{theorem}{Theorem}[section]

\newtheorem{lemma}[theorem]{Lemma}
\newtheorem{proposition}[theorem]{Proposition}
\newtheorem{corollary}[theorem]{Corollary}

{
\theoremstyle{definition}
\newtheorem{definition}[theorem]{Definition}
\newtheorem{convention}[theorem]{Notational convention}
\theoremstyle{remark}
\newtheorem{remark}[theorem]{Remark}
\newtheorem{example}[theorem]{Example}

\newtheorem{Bemerkung}[theorem]{}
}

\let\emph\relax
\DeclareTextFontCommand{\emph}{\bfseries}

\begin{document}

\frontmatter

\title[Equivariant motives and representation theory]{Equivariant motives and \\ geometric representation theory\\ (with an appendix \\ by F. H\"ormann and M. Wendt)}

\date{\today}

\author{Wolfgang Soergel, Rahbar Virk and Matthias Wendt}
\address{Wolfgang Soergel, Mathematisches Institut, Albert-Ludwigs-Uni\-ver\-si\-t\"at Freiburg, Eckerstra\ss{}e 1, 79104 Freiburg im Breisgau, Germany}
\email{wolfgang.soergel@math.uni-freiburg.de}

\address{Rahbar Virk}
\email{rsvirk@gmail.com}

\address{Matthias Wendt, Institut f\"ur Mathematik, Universit\"at Osnabr\"uck, Albrechtstra\ss{}e 28a, 49076 Osnabr\"uck, Germany}
\email{m.wendt.c@gmail.com}



\date{September 2018}

\begin{abstract}
We consider categories of equivariant mixed Tate motives, where equivariant is understood in the sense of Borel. We give the two usual definitions of equivariant motives, via the simplicial Borel construction and via algebraic approximations of it. The definitions turn out to be equivalent and give rise to a full six-functor formalism. For rational \'etale motives over a finite field or the homotopical stable algebraic derivator arising from semisimplified Hodge realization, the equivariant mixed Tate motives provide a graded version of the equivariant derived category. We show that, in sufficiently nice and clean cases, these categories admit weight structures; moreover, a tilting result holds which identifies the category of equivariant mixed Tate motives with the bounded homotopy category of the heart of its weight structure. This can be seen as a formality result for equivariant derived categories. We also discuss convolution functors on equivariant mixed Tate motives, and consequences for the categorification of the Hecke algebras and modules.
\end{abstract}

\maketitle 
\setcounter{page}{4}
\setcounter{tocdepth}{1}
\tableofcontents

\mainmatter

\renewcommand{\thesection}{\arabic{section}}

\chapter*{Introduction}
\label{sec:introduction}

\setcounter{section}{1}

\begin{Bemerkung}
  An important part of geometric representation theory consists of applying Grothendieck's function-sheaf correspondence to understand various function spaces occurring in representation theory. In this way one obtains natural categorifications of the function spaces in question by derived categories of Weil sheaves. Geometric tools available for Weil sheaves, such as the decomposition theorem and the weight filtration, can then be applied to solve representation-theoretic problems. 

  Recent progress in the theory of motives allows to replace Weil sheaves by suitable motivic sheaves in the above contexts. This setting has the advantage that it is possible to construct variants of triangulated categories of motivic sheaves with the property that the Tate motivic sheaves on a point do not admit any extensions amongst themselves. Using such motivic sheaves, one  obtains direct geometric constructions of graded versions of categories of representations, which up to now were constructed, if at all, in ways which to us seem rather artificial.
\end{Bemerkung}

\begin{Bemerkung}
  In \cite{SoWe}, this philosophy was exemplified in the case of varieties with affine Whitney--Tate stratifications and applied to the construction of graded versions of category $\mathcal O$. In the present article, we discuss  how to construct equivariant versions of motivic triangulated categories in the spirit of \cite{BeLu}. We also discuss versions of equivariant mixed Tate motives and explain how these lead to a graded categorification of the Hecke algebra and more generally graded categorifications of the Hecke modules appearing in the representation theory  of real reductive groups. 

As another direct consequence of the behaviour of the weight structures on equivariant mixed Tate motives in suitably nice geometric situations, we obtain streamlined proofs of formality theorems such as \cite{SchTH}, see also the sketch in \cite[5.6]{BezFi}. In our motivic framework, this can be seen as a special case of tilting \cite{Kel}, with tilting objects of explicitly geometric origin.
\end{Bemerkung}

\begin{Bemerkung}
In the case of the Hecke algebra, a graded categorification has already been constructed as the bounded homotopy category of the category of Soergel bimodules, and this has been used by Khovanov \cite{KhoHH} to give an alternative construction of his knot homology. This approach still presents some seemingly artificial technical difficulties, e.g. in proving the braid relations among the tensor products of so-called Rouquier complexes associated to links. In our motivic version of the graded categorification, the Rouquier complexes in the homotopy category of Soergel bimodules can now be understood as  standard motivic sheaves on the double cosets of a Borel subgroup in a reductive group. From this point of view, the fact that the Rouquier complexes satisfy braid relations, which is not at all obvious in the bimodule setting, becomes an immediate consequence of suitable geometric isomorphisms.

The existence of  graded categorifications in the generality of the Hecke modules mentioned above was a conjecture in \cite{So-L}, more precisely Conjecture 4.2.2 and Conjecture 4.2.3. The other conjectures in \cite[Section 4]{So-L} concerned the existence of graded categorifications of the representation-theoretic side of the main conjecture in \cite{So-L}. We think that this can be done via a suitable category of ``monodromic motives'', but this will be discussed in a sequel. In this way we can, building on \cite{BeLuK} and unpublished work of Bernstein with one of the authors, rewrite the Langlands correspondence for the field of real numbers as a (yet conjectural) equivalence of some (non-conjectural) motivic triangulated categories. As these motivic categories have nice geometric constructions, we hope that this will eventually lead to a geometric construction of the equivalence itself.   
\end{Bemerkung}

\subsection*{Equivariant motives and formalism of six functors}

The technical foundation for the representation-theoretic applications in this paper are suitable categories of equivariant mixed Tate motives. Through the recent advances in the theory of motives (possibly with coefficients) \cite{ayoub:thesis1,ayoub:thesis2,cisinski:deglise,drew:thesis}, we now have available categories of motives over fairly general base schemes and these categories of motives are connected by a six-functor formalism. Actually, and this is relevant for our present work, it is even possible to define motives over general diagrams of schemes and develop a six-functor formalism for those. 

Using these categories of motives, in the axiomatic form of a homotopical stable algebraic derivator $\mathbb{D}$, we can proceed along the way laid out by Bernstein and Lunts \cite{BeLu} to define categories $\mathbb{D}^+_G(X)$ of $G$-equivariant motives over $X$ for a $G$-variety $X$. There are essentially two possibilities for a definition, one as cartesian motives over the simplicial Borel construction ${\op{E}}G\times_{/G}X$ and one as cartesian motives over the category $\op{Res}_G(X)$ of $G$-resolutions of $X$. Both approaches have their advantages and disadvantages, but fortunately yield equivalent categories of motives, at least when restricted to motives which are bounded below for the homotopy t-structure. 

When it comes to getting the full six-functor formalism off the ground, both approaches as well as their equivalence are necessary. On the one hand, the functors $f_\ast$ and $f^!$ do not generally preserve cartesian objects over the Borel construction, hence only the 2-functors $f^\ast$ and $f_!$ can be defined via the simplicial approach. On the other hand, all the functors can be defined using the approach via resolutions, but this involves choices and hence does not directly give rise to a 2-functor. Having both equivalent approaches means that in each of the relevant pairs $(f^\ast,f_\ast)$ and $(f_!,f^!)$ the left adjoint is part of a 2-functor, which allows to rectify the right adjoints. With the functors defined, all the usual formulas like base change, localization and Verdier duality can be deduced immediately. As a byproduct, the motivic setup even allows to formulate ``classical'' objects like equivariant derived categories of $\ell$-adic sheaves in a much cleaner way.

\subsection*{Equivariant mixed Tate motives and weight structures}

Having equivariant categories of motivic sheaves with six-functor formalism, we can ask which motives should be considered ``$G$-equivariant mixed Tate motives on $X$''. Since we know mixed Tate motives in the non-equivariant situation, the only sensible answer is to extend to the equivariant situation by requiring that equivariant mixed Tate motives should be stable under the quotient equivalence -- with a grain of salt, equivariant mixed Tate motives are equivariant motives whose restrictions to ``points'' (or better orbits) are mixed Tate. As in \cite{SoWe} (or more classical references), this notion is only well-behaved in the presence of an ``equivariant Whitney--Tate'' condition which ensures that extension and restriction functors preserve equivariant mixed Tate motives. If the equivariant Whitney--Tate condition is satisfied for a variety $X$ with $G$-action, we define in Definitions~\ref{mtderdef1} and \ref{def:mtderdef2} the subcategory $\DMT_G(X)\subset \mathbb{D}^+_G(X)$ of \emph{equivariant mixed Tate motives on $X$} by the condition that forgetting the $G$-action and (both  $\ast$- and $!$-) restricting  to any point yields a (non-equivariant) mixed Tate motive. 

At this point, we consider a suitable collection of ``Bott--Samelson motives'', defined inductively by starting from extensions of local systems on orbits and closed under induction and restriction for the various inclusions of parabolic subgroups. These motives are the equivariant analogues of the motives of Bott--Samelson resolutions of Schubert cells, cf. \cite[Section 6]{SoWe}. Under suitable assumptions which are in particular satisfied for the varieties of interest, cf. Section~\ref{sec:BS}, these motives generate all equivariant mixed Tate motives, establishing the equivariant Whitney--Tate condition. The changes necessary to go from \cite{SoWe} to the more general equivariant setting mostly follow the approach outlined in \cite{virk}.

There is an additional payoff from the study of the explicit collection of Bott--Samelson motives: the usual Springer-type argument using contracting slices implies that the Bott--Samelson motives are in fact pointwise pure. Consequently, we obtain two ways to establish the existence of a weight structure on equivariant mixed Tate motives, cf. Section~\ref{sec:weights}. On the one hand, we can use the weight structures on ordinary non-equivariant motives, use the quotient equivalence and suitable gluing of weight structures. On the other hand, the Bott--Samelson motives provide a negative generating collection, giving rise to a weight structure whose heart consists of the Bott--Samelson motives. These two constructions yield the same weight structure on equivariant mixed Tate motives.
However, we can only show the existence of these weight structures in special situations, when the approach via Bott--Samelson motives works. It remains an interesting problem to construct a weight structure on equivariant motives in general.


As in \cite{SoWe}, the structure of mixed Tate motives is greatly simplified by working in a more restrictive setting. We will usually assume that the homotopical stable algebraic derivator $\mathbb{D}$ from which we construct the equivariant motivic categories $\mathbb{D}^+_G(X)$ satisfies the following two conditions, cf. Convention~\ref{conditions:grading} and \ref{conditions:weight}: 
\begin{enumerate}
\item \index{grading condition} the  \emph{grading condition} requires the category $\DMT(k)$ of mixed Tate $\mathbb{D}$-motives over the  base field $k$ to be equivalent to the derived category of the category of $\mathbb{Z}$-graded vector spaces,
\item \index{weight condition} the \emph{weight condition} requires the existence of suitably compatible weight structures on the motivic categories.
\end{enumerate} 
These conditions are satisfied in two important cases: rational \'etale motives $\mathbf{DA}^{\et}$ or Beilinson motives over a finite field $\mathbb{F}_q$ or its algebraic closure $\overline{\mathbb{F}_q}$, and motives with coefficients in the semisimplified Hodge realization over $\mathbb{C}$ denoted by $\op{MDer}$. They imply that the category of $G$-equivariant mixed Tate motives over a point, with $G$ a connected split reductive group, can be described completely in terms of the Chow ring of ${\op{B}}G$, cf. Theorem~\ref{thm:tiltpoint}. This is relevant for the final application to the categorification of modules over the Hecke algebra.

While all the results below are formulated over algebraically closed fields, some of them actually hold over arbitrary fields. The reason is that the categories of equivariant mixed Tate motives are sufficiently combinatorial so that algebraic field extensions induce equivalences (in the situations we consider). This is also a consequence of the simplifying conditions imposed on the underlying homotopical stable algebraic derivators.

\subsection*{Tilting, formality and categorification}

Using the weight structures on equivariant mixed Tate motives described above together with the pointwise purity of Bott--Samelson motives, we can then prove tilting results in the following cases of interest, cf. Sections~\ref{sec:BS} and \ref{sec:tiltingapp}, in particular Corollary~\ref{tiFL} and Proposition~\ref{prop:symmpurity}.

\begin{theorem}[Tilting for equivariant mixed Tate motives]
\label{thm:maintilt}
Assume either that $k=\overline{\mathbb{F}_q}$ and $\mathbb{D}=\mathbf{DA}^{\et}_{\mathbb{Q}}$, or that $k=\mathbb{C}$ and
$\mathbb{D}=\op{MDer}$. 
\begin{enumerate}
\item Let $G$ be a connected reductive algebraic group and let $P,Q\subseteq G$ be parabolic subgroups. Then the tilting functor is an equivalence of categories
\[
\op{Hot}^{\op{b}}(\DMT_{P\times Q}(G)_{\op{wt}=0})\sirra \DMT_{P\times Q}(G).
\]
\item  Let $G$ be a connected reductive algebraic group with a parabolic subgroup $P\subset G$. Let $\sigma$ be an involution of $G$ and let $K=G^\sigma$ be the subgroup of its fixed points. Then the tilting functor is an equivalence of categories
\[
\op{Hot}^{\op{b}}(\DMT_{P\times K}(G)_{\op{wt}=0})\sirra \DMT_{P\times K}(G).
\]
\item Let $G$ be a connected adjoint semisimple group and let $B\subseteq G$ be a  Borel subgroup. Let $X$ be the wonderful compactification of $G$. Then the tilting functor is an equivalence of categories 
\[
\op{Hot}^{\op{b}}(\DMT_{B\times B}(X)_{\op{wt}=0})\sirra \DMT_{B\times B}(X).
\]
\end{enumerate}
\end{theorem}

The tilting result for equivariant mixed Tate motives can be seen as a stronger version of previously known formality results for equivariant derived categories, cf. e.g. \cite{SchTH}. As before, the formality essentially follows from pointwise purity of the generating objects, the Bott--Samelson motives. 

For a better understanding of equivariant mixed Tate motives, it would be much preferrable to have a combinatorial model for pure weight $0$ equivariant mixed Tate motives. In the case of $P\times Q$-equivariant motives over a reductive group $G$, combinatorial objects corresponding to weight $0$ motives are given by Soergel bimodules. The following result provides a motivic version of categorifications of the Schur algebroid via Soergel bimodules, cf.~Section~\ref{sec:parabolic}, in particular Proposition~\ref{prop:ff} and Corollary~\ref{cor:pf142}. 

\begin{theorem}[Graded categorification A]
\label{thm:motivebimod}
Assume either that $k=\overline{\mathbb{F}_q}$ and $\mathbb{D}=\mathbf{DA}^{\et}_{\mathbb{Q}}$ or $k=\mathbb{C}$ and
$\mathbb{D}=\op{MDer}$. Let $G$ be a reductive group over $k$, and let $P,Q\subseteq G$ be two parabolic subgroups. Denote by $\mathbb{H}$ the functor which computes equivariant cohomology of the realization of a motive. 
\begin{enumerate}
\item Equivariant cohomology induces an equivalence of tensor triangulated categories
\[
\mathbb{H}:\DMT_{P\times Q}(G)_{\op{wt}=0}\stackrel{\approx}{\longrightarrow}  \mathcal{A}_P\op{-SMod-}\mathcal{A}_Q
\]
where (as usual)  $\mathcal{A}_X={\op{H}}^\ast_X(\op{pt})$ denotes the cohomology ring of the classifying space of the group $X$. 
\item 
This in turn induces a zig-zag of equivalences of tensor triangulated categories
\[
\mathbb{T}:\DMT_{P\times Q}(G) \xleftarrow{\approx} \op{Hot}^{\op{b}}(\DMT_{P\times Q}(G)_{\op{wt}=0}) \xrightarrow{\approx} \op{Hot}^{\op{b}}(\mathcal{A}_P\op{-SMod-}\mathcal{A}_Q).
\]
\end{enumerate}
\end{theorem}

As we already know that the category $\mathcal{A}_P\op{-SMod-}\mathcal{A}_Q$ of Soergel bimodules provides a graded categorification of the Hecke algebra, parabolic Hecke modules or more generally the Schur algebroid, the above result shows that equivariant mixed Tate motives provide an alternative way to categorify these structures. This can be applied immediately to give a simple geometric proof of the braid relations for Rouquier complexes, which is relevant for the construction of Khovanov's knot homology via Soergel bimodules, cf.~Section~\ref{sec:knot}. 

\begin{corollary}[braid relations]
\label{cor:link}
Assume the situation of Theorem~\ref{thm:motivebimod} with $G=\op{GL}_n$ and $P=Q=B$ a Borel subgroup of $G$. Denote by $\mathcal{A}_B=\op{H}^\bullet_{B}(\op{pt})$. For a simple reflection $s\in \op{S}_n$, denote $T_s^!:=i_{s,!}\underline{BsB}$ and $T_s^\ast:=i_{s,\ast}\underline{BsB}$, where $i_s:BsB\hookrightarrow G/B$ denotes the inclusion of the $B$-orbit corresponding to $s$. These are motivic lifts of the Rouquier complexes in the sense that there are isomorphisms (in the category $\op{Hot}^{\op{b}}(\mathcal{A}_B\op{-SMod-}\mathcal{A}_B)$)
\begin{eqnarray*}
\mathbb{T}(T_s^!)&\cong& \left[\mathcal{A}_B\otimes_{\mathcal{A}_B^s}\mathcal{A}_B\twoheadrightarrow \mathcal{A}_B\right]\\
\mathbb{T}(T_s^\ast)&\cong& \left[\mathcal{A}_B\hookrightarrow\mathcal{A}_B\otimes_{\mathcal{A}_B^s}\mathcal{A}_B\right].
\end{eqnarray*}
For two simple reflections $s,t\in S_n$ with $sts=tst$, these satisfy  braid relations 
\[
T_s^!\conv_B T_t^!\conv_B T_s^!\cong T_t^!\conv_B T_s^!\conv_B T_s^!\quad\textrm{ and }\quad
T_s^\ast\conv_B T_t^\ast\conv_B T_s^\ast\cong T_t^\ast\conv_B T_s^\ast\conv_B T_t^\ast
\]
which follow immediately from the geometric isomorphisms
\[
BsB\times_B BtB\times_BBsB\cong BstsB=BtstB\cong BtB\times_B BsB\times_BBtB.
\]
Since equivariant cohomology is compatible with convolution, these braid relations imply braid relations for the Rouquier complexes.
\end{corollary}

\subsection*{Graded version of equivariant derived categories and applications}
Finally, we can also apply realization functors to compare the categories of equivariant mixed Tate motives with the usual equivariant derived categories, for $\ell$-adic sheaves or Hodge modules. The following is proved in Theorem~\ref{thm:gradedparabolic}.

\begin{theorem}
\label{thm:motivesparabolic}
\begin{enumerate}
\item
Let $k=\overline{\mathbb{F}_q}$, let $G$ be a reductive group over $k$, and let $P,Q\subseteq G$ be two parabolic subgroups. Then the $\ell$-adic realization functor
\[
\op{Real}_\ell: \DMT_{P\times Q}(G)\to \op{Der}^{\op{b}}_{P\times Q}(G;\mathbb{Q}_\ell)
\]
is fully faithful on weight $0$ objects, with the essential image of the zero-weight part consisting of shifted intersection complexes concentrated in even degrees.
\item 
Let $k=\mathbb{C}$, let $G$ be a reductive group over $k$, and let $P,Q\subseteq G$ be two parabolic subgroups. Then the Hodge realization functor 
\[
\op{Real}_{\op{H}}: \DMT_{P\times Q}(G)\to \op{Der}^{\op{b}}_{P\times Q}(G;\mathbb{C})
\]
is fully faithful on weight $0$ objects, with the essential image of the zero-weight part consisting of shifted intersection complexes concentrated in even degrees. 
\item Motivic (graded) lifts of the standard and costandard objects in the equivariant derived category are given by
\[
i_{w,!}(\underline{PwQ}[l(w)])\quad \textrm{ and }\quad
i_{w,\ast}(\underline{PwQ}[l(w)]).
\]
\item
The functors in points (1) and (2) are compatible with convolution and Verdier duality. 
\item The functors in points (1) and (2) are degrading functors in the sense of \cite{BGSo}. 
\end{enumerate}
\end{theorem}

In the case $P=Q=B$, this statement about the realization functor applied to the tilting equivalence of Theorem~\ref{thm:maintilt} recovers the formality result of Schn\"urer, cf. \cite[Theorem 1]{SchTH}.
The result also provides a graded version of the equivariant derived categories in the parabolic situations. On the level of Grothendieck groups, this result implies a second version of motivic graded categorification of the Hecke algebra and its parabolic modules, which is obtained by applying Grothendieck's function-sheaf correspondence to the $\ell$-adic sheaves obtained from realization of equivariant mixed Tate motives. The following is proved in Section~\ref{sec:parabolic}, more precisely Theorem~\ref{thm:schur}. 

\begin{theorem}[Graded categorification B]
\label{thm:functionsheaf}
Let $k=\mathbb{F}_q$, let $G$ be a connected reductive  group over $k$ let $B\subset G$ a Borel,  and let $P,Q\subseteq G$ be two parabolic subgroups. 
\begin{enumerate}
\item 
  Combining the $\ell$-adic realization functor with Grothendieck's function-sheaf correspondence induces an isomorphism from $\op{K}_0(\DMT_{B\times B}(G))$ to the Iwahori--Hecke-algebra. Applying $\op{K}_0$ to the Verdier duality functor yields the Kazhdan--Lusztig involution. 
\item The function-sheaf correspondence of Grothendieck induces an isomorphism from the split Grothendieck group of $\DMT_{B\times B}(G)_{\op{wt}=0}$ to the Iwahori--Hecke algebra, and similarly from the split Grothendieck group of $\DMT_{P\times Q}(G)_{\op{wt}=0}$ to the corresponding parabolic Hecke-bimodule. 
\end{enumerate}
\end{theorem}

Actually, Theorem~\ref{thm:schur} provides a categorification of the full Schur algebroid (whose definition is recalled in Section~\ref{sec:hecke}). The two graded categorifications we obtained in Theorem~\ref{thm:motivebimod} and Theorem~\ref{thm:functionsheaf} actually agree in the sense that we get a commutative diagram of isomorphisms of graded algebras with involution (resp. modules over them). This essentially follows from various compatibility statements proved throughout the text.

Now, finally, there is a similar statement for the case of symmetric varieties, providing a motivic graded version of the Hecke module considered by Mars and Springer in \cite{Mars-Springer}, cf. Theorems~\ref{thm:gradedsymm} and \ref{thm:comparisonMS}. 

\begin{theorem}
\label{thm:motivesymmetric}
Let $k=\overline{\mathbb{F}_q}$ be the algebraic closure of a finite field of odd characteristic, let $G$ be a connected reductive group, let $\theta:G\to G$ be a non-trivial algebraic involution, and let $T\subset B\subset G$ be $\theta$-stable maximal torus and Borel subgroup. Denote by $K$ the subgroup of $G$ fixed by the involution $\theta$. Assume all these data are defined over the finite field $\mathbb{F}_q$. 

\begin{enumerate}
\item
The $\ell$-adic realization functor 
\[
\op{Real}_\ell: \DMT_{B\times K}(G)\to \op{Der}^{\op{b}}_{B\times K}(G;\mathbb{Q}_\ell)
\]
is fully faithful on the heart of the weight structure. The essential image of the heart consists of intersection complexes concentrated in even degrees. 
\item Motivic (graded) lifts of the standard and costandard objects are given by $!$- and $\ast$-extensions of local systems on $B$-orbits of $G/K$. 
\item $\ell$-adic realization is compatible with involution and Verdier duality. 
\item $\ell$-adic realization is a degrading functor in the sense of \cite{BGSo}. 
\end{enumerate}
As a result, the $\ell$-adic realization functor induces an isomorphism of Grothendieck groups
\[
\op{K}_0(\DMT_{B\times K}(G))\xrightarrow{\cong} \op{K}_0(\mathcal{A}_{G/K}), 
\]
where $\op{K}_0(\mathcal{A}_{G/K})$ denotes the Hecke module considered in \cite{Mars-Springer}. 
\end{theorem}

There are similar results for wonderful compactifications of adjoint semisimple groups, cf. Theorems~\ref{thm:gradedwonderful} and \ref{thm:comparisonSpCompact}. There is also a version of the above result for the case where the base field is $k=\mathbb{C}$ which provides an approach to mixed geometric representation theory without passing to finite fields. 

\index{Soergel--Lunts conjecture}
The tilting results discussed above also imply strong formality results for the $P$-equivariant derived categories of symmetric varieties $G/K$. As a particular consequence, we can prove the \emph{Soergel--Lunts conjecture} which allows to identify the equivariant derived category as a derived category of modules over the geometric extension algebra. See Theorem~\ref{thm:slsymm} for a precise statement, and \ref{SLconj} for the relevant notation.

The above theorem provides graded versions of equivariant derived categories in the case of symmetric varieties. In \cite[Section 4]{So-L}, the existence of such graded versions is part of a series of conjectures formulating Langlands duality for representations of real Lie groups in the context of Koszul duality patterns. The following result establishes those conjectures which are related to the ``geometric'' side of the expected Koszul duality, cf. Theorem~\ref{thm:abvconjecture}.

\begin{corollary}[Soergel conjectures]
\index{Soergel conjecture}
The category $\DMT_G(X)$ satisfies the requirements for $\mathcal{D}_g$ in \cite[Conjecture 4.2.2 and 4.2.3]{So-L}. 
\end{corollary}

What would still be missing now for a better understanding of $\DMT_{B\times K}(G)$ are combinatorial models for the heart of the weight structure, similar to the Soergel bimodules for $\DMT_{P\times Q}(G)$. 

We also expect that the approach via suitable categories of mixed Tate motives should also provide a solution for the representation-theoretic side of Soergel's conjectures in \cite{So-L}. More precisely, there should be categories of motives with suitable monodromy conditions which, via a motivic version of the localization results of Bernstein and Lunts in \cite{BeLuK}, provide graded versions of derived categories of Harish-Chandra modules. Then we can hope for a motivic Koszul duality modelling the Langlands duality in the representation theory of real Lie groups, as envisioned \cite{So-L}. Details concerning the categories of monodromic motives will be found elsewhere. 

\subsection*{Where do we go from here}

We outline a couple of future directions for motives in geometric representation theory. 

\begin{Bemerkung}
In this work, we have concentrated on the particular cases of parabolic group actions on partial flag varieties, symmetric varieties and wonderful compactifications. There are a couple of other possible situations where results similar to ours could be achieved. 

First, we actually have ignored completely the simplest and best-studied case of group actions, namely toric varieties. It seems very likely that an application of the motivic formalism to toric varieties is possible, and we believe it could be used to recover and strengthen some of the formality results contained in \cite[Section 15]{BeLu}. 

One of the very interesting representation-theoretic situations is the action of a connected reductive group $G$ on its nilpotent cone $\mathcal{N}_G$. Formality results based on mixed geometry (based on $\ell$-adic sheaves with Frobenius action) have been established in \cite{rider:russell}. The central result which would be required to make the motivic formalism work would be ``cuspidals are clean''; this would allow to generalize and strengthen the results of \cite{rider:russell} and provide a motivic version of the Springer correspondence. We understand that Jens Eberhardt is working on this question.  

More ambitious examples of group actions to which one could try to apply motives would be the affine Grassmannian (for mixed geometry related to representations of quantum groups as in \cite{arkhipov:bezhrukavnikov:ginzburg} or \cite{mirkovic:vilonen}) or the adjoint action (for a motivic version of character sheaves). In particular for the latter, new ideas would be needed because with infinitely many $B$-orbits it clearly falls outside the scope of equivariant mixed Tate motives as discussed here.
\end{Bemerkung}

\begin{Bemerkung}
Most of our work concentrates on homotopical stable algebraic derivators $\mathbb{D}$ with rational coefficients, and in the representation-theoretic applications we mostly specialize to $\mathbb{D}=\mathbf{DA}^{\et}(-;\Lambda)$ or $\mathbb{D}=\op{MDer}$. It is natural to ask if results could be obtained with other coefficients. 

One of the natural extension would be to ask questions for modular coefficients. The work of Eberhardt and Kelly \cite{eberhardt:kelly} sets up the framework to define perverse mixed Tate motives with finite coefficients (characteristic of coefficients equal to the characteristic of the base field). Very likely our framework could be modified to work with the Milnor K-theory derivator considered in \cite{eberhardt:kelly} and obtain results for modular representation theory. 

Another sort of coefficients to which we could apply the framework of equivariant mixed Tate motives, much in the spirit of $\mathbb{D}=\op{MDer}$, would be generalized motivic cohomology theories. Given a ring spectrum $E$ in the stable $\mathbb{A}^1$-homotopy category, one can consider the homotopical stable algebraic derivators given by the categories $\mathcal{E}(X)\subset \mathcal{SH}(X)$ of $E$-modules in $\mathcal{SH}(X)$. One particular example of interest would be the Hecke algebras associated to elliptic cohomology studied e.g. in \cite{zhao:zhong} and \cite{lenart:zainoulline}. It seems likely that equivariant mixed Tate motives for the derivator given by modules over elliptic cohomology provides a categorification of the algebras and modules considered in loc.cit. Along this line, one could hope that positivity conjectures formulated in these papers could be resolved by exhibiting the relevant coefficients as dimensions of Bott--Samelson motives over elliptic cohomology.
\end{Bemerkung}

\begin{Bemerkung} 
As a slightly more drastic change of setup, we can also consider six-functor formalisms for spaces other than algebraic varieties. It seems very likely that simply plugging in the six-functor for locally compact Hausdorff spaces would recover most of the classical theory of equivariant sheaves and functors in \cite{BeLu} (of course without the theory of weights present in the motivic setting). 

A more interesting setting to look at  would be the six-functor formalism for rigid analytic varieties as set up in \cite{ayoub:rigid}. A formalism of equivariant motives over rigid analytic varieties with action of a $p$-adic Lie group should be possible, closely following the argumentation in the present work. This could be very useful to study the representation theory of $p$-adic Lie groups using motivic and geometric methods and establish results parallel to the ones for complex and real Lie groups. 
\end{Bemerkung}

\subsection*{Comparison to related work}

There have been several approaches to definitions of categories of equivariant motives and the corresponding six functor formalism. Equivariant Chow motives for smooth projective varieties with action over base fields were considered by Laterveer \cite{laterveer} and more recently by Calm{\`e}s--Neshitov--Zainoulline. 

A definition of motivic homotopy for finite groups was already given by Voevodsky. Various more general definitions of equivariant motivic homotopy (applicable to more general base schemes or actions of linear group schemes) have been given by Herrmann,  Heller--Krishna--\O{}stv\ae{}r, Carlsson--Joshua and Hoyois. 
However, only Hoyois \cite{hoyois} discussed a six functor formalism for equivariant motivic stable homotopy categories. 
It should also be noted that most of definitions of equivariant motivic homotopy so far produce Bredon-style equivariant motivic cohomology. For the representation-theoretic applications presented in the present work, we needed categories of equivariant motives which are defined over general base varieties (over some field $k$) and for arbitrary linear groups (over the field $k$) which produce a Borel-style equivariant motivic cohomology. To the best of our knowledge, there is no published work which discusses categories of equivariant motives and the corresponding six functor formalism from the perspective we need.  The Borel-style definition of equivariant motives makes sure that it ties in naturally with the definitions of equivariant higher Chow groups following Totaro \cite{totaro} and Edidin--Graham \cite{edidin:graham} as well as the classical definition of equivariant derived categories by Bernstein--Lunts \cite{BeLu}. This is why we include an extensive explanation how to set up equivariant motives and the six functors. 

In representation theory, mixed versions of the equivariant derived categories in various geometric situations have been considered by many authors, cf. e.g. \cite{rider:russell} for a very recent work. However, the construction of these categories is not quite natural; they are built to satisfy formality and the tilting results in Theorem~\ref{thm:maintilt} using the weight theory coming from eigenvalues of Frobenius. This leads to problems when trying to study functors between such categories as it usually involves complicated conjectures on semisimplicity of Frobenius eigenvalues. In this work we establish a fairly general framework to obtain such mixed versions of equivariant derived categories which has the advantages that it is more natural than the constructions that appeared in the literature so far, doesn't use Frobenius eigenvalues but weights from motives, even works intrinsically over $\mathbb{C}$ without passing to finite fields, and also encompasses the case of $K$-orbits on flag varieties which to the best of our knowledge hasn't been satisfactorily solved before.

\newpage
\subsection*{Structure of the paper} 
The paper consists of three parts. Chapter I provides some recollections on categories of motives, describes constructions of categories of equivariant motives and sets up the basics of the six-functor formalism in the equivariant situation. Chapter II provides a definition of equivariant mixed Tate motives and weight structures on these. Chapter III establishes the formalism of Bott--Samelson motives and its representation-theoretic applications to tilting results, graded versions of equivariant derived categories and categorification of the Hecke algebra. Two appendices deal with basic facts concerning motives of homogeneous spaces and a general tilting result for derivators. A short description of contents is given at the beginning of the individual chapters.

There are two main background texts making the present work possible: on the one hand, the classical setup of equivariant sheaves and functors as developed in \cite{BeLu}, and on the other hand the six-functor formalism of motivic sheaves as developed in \cite{ayoub:thesis1,ayoub:thesis2} or \cite{cisinski:deglise}. We assume throughout that the reader is familiar with these works or at least has them in reach for reference. 

\subsection*{Conventions and notation}

\begin{Bemerkung}
\label{standing}
All our schemes will be assumed to be quasi-projective separated schemes of finite type over a field. We will call such objects \emph{varieties} (although they don't necessarily need to be reduced). By  \cite[Lemme 1.3.9]{ayoub:thesis1}, morphisms between such schemes will automatically be quasi-projective. These conventions imply that quasi-projectivity assumptions in \cite{ayoub:thesis1,ayoub:thesis2} are always satisfied in our applications. If the base field $k$ is clear from the context, we will frequently denote $\op{Spec}k$ by $\op{pt}$.

We usually denote by $\op{fin}_X:X\to\op{Spec} k$ the structural morphism of a variety over $k$. 

For a variety $X$, the motive $\op{M}_X(X)$ which also is denoted by $\Lambda_X$ in the literature, will be denoted by $\underline{X}$ in most of our work (parallel to notation for the constant sheaf $\Lambda$ on $X$).
\end{Bemerkung}

\begin{Bemerkung}
  Given a graded ring $A$ we denote by $A\op{-fModfg}^\DZ$ the category of finitely generated graded free $A$-modules. 
\end{Bemerkung}

\begin{Bemerkung}
As a general rule of thumb, $\cong$ denotes isomorphisms, $\simeq$ denotes quasi-isomorphisms and $\approx$ denotes equivalences of categories or isotransformations between functors. 
\end{Bemerkung}

\begin{Bemerkung}
Most of the time, the set of morphisms $f:X\to Y$ in a category $\mathcal{C}$ will be denoted by $\mathcal{C}(X,Y)$. Exceptions will only be made in cases where the name of the category $\mathcal{C}$ is already typographically complex and cannot be meaningfully abbreviated. Inner homs in a monoidal closed category $\mathcal{C}$ will be denoted by $\iHom_{\mathcal{C}}(X,Y)$. 
\end{Bemerkung}

\newpage
\subsection*{Acknowledgements}
Wolfgang Soergel and Matthias Wendt would like to thank Fr{\'e}d{\'e}ric D{\'e}glise and J\"org Wildeshaus for discussions concerning equivariant motives during a visit to Paris 13 in September 2014. Matthias Wendt would like to thank Joseph Ayoub for helpful discussions on categories of motives and possible definitions of equivariant motives. Discussions concerning the equivariant six-functor formalism with Marc Hoyois and Simon Pepin LeHalleur at the Oberwolfach workshop 1626 were also very helpful. We thank Annette Huber, Jens Eberhardt and Brad Drew for various conversations related to this work. 

Wolfgang Soergel and Matthias Wendt thank the Warwick Mathematics Institute \index{purity!the beer} for its hospitality during the derived categories conference in March 2015. Wolfgang Soergel was supported by the DFG priority program SPP 1388. Matthias Wendt was partially supported  by the DFG SFB/TransRegio 45, EPSRC grant EP/M001113/1, DFG SFB 1085, Institut Mittag-Leffler and his parents at various stages of his ann{\'e}es de p{\`e}lerinage. Also, some of Matthias Wendt's visits to Freiburg were partially supported by DFG GK 1821.

 
\renewcommand{\thesection}{\thechapter.\arabic{section}}

\chapter{Equivariant motives and six functors}

The first chapter sets up the formalism of equivariant motives. We recall the basics concerning the construction of triangulated categories of motives and their six functor formalism in Section~\ref{sec:basicmotives}. After some preliminaries on group actions in Section~\ref{sec:grp-prelims} and the theory of acyclic resolutions in Section~\ref{sec:resolutions}, we discuss two definitions of Borel-equivariant motives in Section~\ref{sec:equivdef} via the simplicial Borel construction and the category of algebraic resolutions. In Section~\ref{sec:comparison}, we show that the two approaches yield equivalent categories of equivariant motives, and establish the quotient equivalence for these categories. The equivalence of the two approaches is very relevant for establishing the full equivariant six-functor formalism in the motivic setting, which is done in  Section~\ref{sec:sixfunctors}. Then Section~\ref{sec:further} discusses further consequences of the six-functor formalism, including refined quotient and induction equivalence, integration functors and various compatibility statements. Then Section~\ref{sec:convolution} discusses basic facts concerning convolution functors and Section~\ref{sec:realization} discusses realization functors on categories of equivariant motives.

\section{Recollection on motives and six functors} 
\label{sec:basicmotives}

In this section, we provide a short recollection of  constructions and properties of categories of motives. Following the work of Voevodsky, Ayoub, Cisinski--D{\'e}glise and others, there are now triangulated categories of motives available, and these categories are related by a full-fledged six functor formalism satisfying all the usual properties.  

On the one hand, there are axiomatic frameworks to encode all the relevant properties of a six functor formalism, and to study the interdependence of various properties of the six functors. One possible framework is given by homotopical stable algebraic derivators of \cite{ayoub:thesis1}, another framework by the motivic  triangulated categories of \cite{cisinski:deglise}. 

On the other hand, there are concrete instances of these axiomatic frameworks which provide triangulated categories of motives over fairly general base schemes. The following diagram summarizes the instances of these frameworks which will be of most interest to us: 
\begin{center}
  \begin{minipage}[c]{10cm}
    \xymatrix{
      \mathbf{DA}^{\et}(-;\Lambda) \ar[r]^{\mathsf{Real}_\ell}
      \ar[d]_{\mathsf{Real}_{\text{Hodge}}} &  \op{Der}(-;\mathbb{Q}_\ell)
        \\
      \mathbf{DH}(-) \ar[d]_{\op{Gr_W}} \\
      \op{MDer}(-;\mathbb{C}) 
      \ar[r]_{\mathsf{Real}_{\text{Betti}}} &  
      \op{Der}(-;\mathbb{C})
    }
  \end{minipage}
\end{center}
The upper part of  the diagram is the one we can use in the case where we are working over a finite base field: $\mathbf{DA}^{\et}(X;\Lambda)$ denotes \'etale motives with $\Lambda$-coefficients, $\op{Der}(-;\mathbb{Q}_\ell)$ denotes the $\ell$-adic derived categories constructed as localization of the \'etale motives, and these categories are related by the $\ell$-adic realization functor $\mathsf{Real}_\ell:\mathbf{DA}^{\et}(X;\Lambda)\to\op{Der}(X;\mathbb{Q}_\ell)$.  In the lower part of the diagram, which is the part we will use when we are working over $\mathbb{C}$, we have Drew's version $\mathbf{DH}(X)$ of mixed Hodge modules arising from the Hodge realization, the derived category  $\op{Der}(X;\mathbb{C})$ of  sheaves of $\mathbb{C}$-vector spaces arising from the Betti realization, and an intermediate category $\op{MDer}(-;\mathbb{C})$ arising from the semisimplification of the Hodge realization. The functors relating these categories are all related by the forgetful/realization functors: the Hodge realization $\mathsf{Real}_{\text{Hodge}}$ from motives to mixed Hodge modules, the associated graded for the weight filtration $\op{Gr_W}$, and finally the forgetful functor $\mathsf{Real}_{\text{Betti}}$ from graded vector spaces to vector spaces. All the categories in the diagram have a six functor formalism, and the functors in the diagram are compatible with the six functor formalism. 

The goal of this section is to provide a short recollection on constructions of the categories of motives which are of interest for our application: \'etale motives and motives with coefficients in enriched mixed Weil cohomology theories. We will make heavy use of Ayoub's framework of homotopical stable algebraic derivators  \cite{ayoub:thesis1} in the present paper. A couple of relevant definitions and properties are recalled in Appendix~\ref{sec:derivators}. Still, the reader not familiar with the definitions from \cite{ayoub:thesis1} is strongly encouraged to have a copy handy when reading the first part of the present paper. All the material can be found in the works of Ayoub \cite{ayoub:thesis1,ayoub:thesis2}, Cisinski--D{\'e}glise \cite{cisinski:deglise} and the thesis of Drew \cite{drew:thesis}.

\subsection{Axiomatics for six functors}
\label{sec:axiomatics}

We first recall some of the axiomatics and basic properties of the (non-equivariant) six functor formalism. One formulation of an axiomatic approach is the  notion of cross functors which was proposed by Voevodsky and Deligne; for the definition of cross functors, cf. \cite[Section 1.2.4]{ayoub:thesis1}. The formalism of cross functors was worked out in the thesis of Ayoub, along with two resulting frameworks for dealing with the six functors:
\begin{itemize}
\item the notion of homotopical stable 2-functor $\mathsf{H}^\ast:(\op{Sch}/S)\to\mathfrak{TR}$ from schemes over the base $S$ to triangulated categories, cf. \cite[Section 1.4.1]{ayoub:thesis1},  and  
\item 
\index{homotopical stable algebraic derivator} the notion of homotopical  stable algebraic derivator $\mathbb{D}:\mathsf{DiaSch}/S\to\mathfrak{TR}$, cf. \cite[Section 2.4.2]{ayoub:thesis1}, which encodes the six functors in the more general situation where each diagram of schemes is assigned a triangulated category.
\end{itemize}

The relevant axioms are recalled in an appendix, cf. Section~\ref{sec:derivators}. More information on derivators and various extra properties they are required to satisfy are also recalled there. In the following, we provide a list of useful properties of a homotopical stable algebraic derivator $\mathbb{D}$ satisfying the conditions of \ref{derivator:conditions}. These are all established in \cite[Section 2]{ayoub:thesis1}, but we list them for ease of reference. The list below is a variant of the dix le{\c c}ons in \cite{hebert} adapted to the present setting.

\begin{enumerate} 
\item {\bf Ordinary pullback and pushforward:} By the definition of homotopical stable algebraic derivator, for every morphism $(f,\alpha):(\mathscr{F},\mathcal{I})\to(\mathscr{G},\mathcal{J})$ in $\mathsf{DiaSch}$, we have an adjunction 
\[
(f,\alpha)^\ast:\mathbb{D}(\mathscr{G},\mathcal{J}) \leftrightarrows \mathbb{D}(\mathscr{F},\mathcal{I}): (f,\alpha)_\ast.
\]
If $(f,\alpha)$ is smooth, we have an additional adjunction
\[
(f,\alpha)_\sharp: \mathbb{D}(\mathscr{F},\mathcal{I}) \leftrightarrows \mathbb{D}(\mathscr{G},\mathcal{J}): (f,\alpha)^\ast.
\]
The above functors fit together to form 2-functors. 
\item {\bf Exceptional pullback and pushforward:} \index{exceptional functors!motives over diagrams} For any \emph{cartesian} morphism $(f,\alpha):(\mathscr{G},\mathcal{J})\to (\mathscr{F},\mathcal{I})$ in $\mathsf{DiaSch}$ there is a further pair of adjoint functors, the {\bf exceptional functors} 
\[
(f,\alpha)_!:\mathbb{D}(\mathscr{G},\mathcal{J})\leftrightarrows \mathbb{D}(\mathscr{F},\mathcal{I}):(f,\alpha)^!
\]
which fit together to form a covariant (resp. contravariant) $2$-functor  $(f,\alpha)\mapsto (f,\alpha)_!$ (resp. $(f,\alpha)\mapsto (f,\alpha)^!$). There exists a natural transformation $(f,\alpha)_!\rightarrow (f,\alpha)_\ast$ which is a morphism of $2$-functors and is an isomorphism when $(f,\alpha)$ is proper. The existence and properties for the special case of schemes are discussed in \cite[Sections 1.4--1.6]{ayoub:thesis1}, and the extension to cartesian morphisms of diagrams is discussed in \cite[p. 322--323]{ayoub:thesis1}. 
\item {\bf Monoidal structures:} If $\mathbb{D}$ is a monoidal homotopical stable algebraic derivator, there are closed symmetric monoidal structures on all the categories $\mathbb{D}(\mathscr{F},\mathcal{I})$, where we additionally require that $f^\ast$ is a strong monoidal functor and that suitable projection formulas hold, cf. Section~\ref{sec:derivators}. 

\index{projection formulas!motives over diagrams}
From the axioms follow various compatibilities between the four functors just discussed and these monoidal structures encoded in \cite[Lemma 2.4.52]{ayoub:thesis1}. We recall from \cite[Theorem 2.3.40, Propositions 2.3.51--55]{ayoub:thesis1} some more explicit projection formulas: for a morphism $f\colon Y\to X$ of schemes, we have natural isomorphisms
\begin{eqnarray*}
f_!(f^\ast M\otimes_Y N)&\xrightarrow{\cong}& M\otimes_X f_!(N)\\
\iHom_Y(M,f_\ast N)&\xrightarrow{\cong} & f_\ast \iHom_X(f^\ast M,N)\\
\iHom_X(f_\sharp M,N)&\xrightarrow{\cong} & f_\ast \iHom_Y(M,f^\ast N) \quad f \textrm{ smooth!} \\
\iHom_X(f_! M,N)&\xrightarrow{\cong} & f_\ast\iHom_Y(M,f^!N)\\
f^\ast \iHom_X(M,N) & \xrightarrow{\cong} & \iHom_Y(f^\ast M,f^\ast N) \quad f \textrm{ smooth!}\\
f^!\iHom_X(M,N) & \xrightarrow{\cong} & \iHom_Y(f^\ast M,f^!N).
\end{eqnarray*}
\item {\bf Tate twist:} There is a Tate twist functor $M\mapsto M(1)$ which is compatible with all the six functors.
\item {\bf Base change:} 
\index{base change!motives over diagrams}
For any cartesian diagram of diagrams of schemes
\[
\xymatrix{
(\mathscr{G}',\mathcal{J}') \ar[r]^{(g',\beta)} \ar[d]_{f'} & (\mathscr{G},\mathcal{J}) \ar[d]^f \\
(\mathscr{F}',\mathcal{J}') \ar[r]_{(g,\beta)} & (\mathscr{F},\mathcal{J}). 
}
\]
where $f$ is a cartesian morphism of diagrams, there are base change formulas 
\begin{eqnarray*}
f_!'\circ (g',\beta)^\ast & \xrightarrow{\approx} & (g,\beta)^\ast \circ f_!\\
f^!\circ (g,\beta)_\ast & \xrightarrow{\approx} & (g',\beta)_\ast \circ (f')^!
\end{eqnarray*}
These follow via the extension in \cite[Section 2.4]{ayoub:thesis1} from the corresponding base change formulas formulated in \cite[Scholie 1.4.2]{ayoub:thesis1}. 

In the slightly different situation where the morphism $f$ above is a pointwise closed immersion and $g$ is pointwise smooth, there is a base-change formula
\[
(g',\beta)^\ast \circ f^! \xrightarrow{\approx} (f')^! \circ (g,\beta)^\ast,
\]
cf. \cite[Lemma 2.4.26]{ayoub:thesis1}.
\item {\bf Relative purity:} \index{purity!motives over diagrams} The orientability of the derivator implies that for a smooth morphism $f:X\to Y$ of relative dimension $d$, there are canonical natural isomorphisms  
\[
\mathfrak{p}_f:f_\sharp\xrightarrow{\approx} f_!(d)[2d], \qquad
\mathfrak{p}_f^\prime:f^\ast\xrightarrow{\approx} f^!(-d)[-2d],
\]
cf. \cite[Scholie 1.4.2]{ayoub:thesis1}.
Combining this with the base change statement above recovers several other base change statements such as  \cite[Theorem 2.4.22]{ayoub:thesis1}. 
\item  {\bf Localization sequence:} \index{localization sequence!motives over diagrams} for $i:Z\rightarrow X$ a closed immersion with open complement $j:U\to X$, there are distinguished triangles of natural transformations
\[
j_!j^!\rightarrow 1\rightarrow i_\ast i^\ast \rightarrow j_!j^![1]
\]
\[
i_! i^!\to 1\to j_\ast j^\ast \to i_\ast i^! [1]
\]
where the first and second maps are the counits and units  of the
respective adjunctions, cf. \cite[Proposition 2.4.25]{ayoub:thesis1}. This holds more generally for cartesian closed immersions of diagrams of schemes and the pointwise complementary open immersion.
\item {\bf Absolute purity:} For any closed immersion $i:Z\rightarrow S$ of pure codimension $n$ between regular schemes in $\mathscr S$, the standard map $\op{M}_Z(Z)\rightarrow i^! \op{M}_S(S)(n)[2n]$ is an isomorphism, cf. \cite[Theorem 14.4.1]{cisinski:deglise}. 
\item {\bf Constructible objects:}
For a scheme $X$, define the subcategory of constructible objects $\mathbb{D}^{\op{c}}(X)\subset \mathbb{D}(X)$ to be the thick full subcategory generated by $f_\sharp \const{Y}(n)$ for $n\in\mathbb{Z}$ and $f:Y\to X$ smooth. This subcategory coincides with the full subcategory of compact objects if motives of smooth varieties are compact (which holds if we consider rational coefficients), cf. \cite[Section 2.3.10]{ayoub:thesis1}. Moreover, the six functors preserve constructible objects, cf. \cite[Section 2.2.2]{ayoub:thesis1}. 
\item {\bf Verdier duality:} \index{Verdier duality!motives over diagrams} For a smooth variety $X$ over the base $S$, the unit object $\const{X}$ is a dualizing object in $\mathbb{D}(X)$, i.e., setting $D_X(M):=\iHom(M,\const{X})$ the natural map $M\rightarrow D_X(D_X(M))$  is an isomorphism for all $M\in \mathbb{D}^{\op{c}}(X)$. More generally, for every variety $X$ over the base $S$ the motive $\op{fin}^!\const{S}$ is a dualizing object, cf. \cite[Theorem 2.3.73]{ayoub:thesis1}. 

For  all $M,  N\in \mathbb{D}^{\op{c}}(X)$, there is a canonical duality isomorphism  
$$
D_X(M\otimes_X D_X(N))\cong \iHom_X(M,N).
$$ 
Furthermore, for any morphism $f:Y\rightarrow X$ of varieties and any $M\in\mathbb{D}^{\op{c}}(X)$ and $N\in \mathbb{D}^{\op{c}}(Y)$, there are natural isomorphisms  
$$
D_Y(f^\ast(M))\cong f^!(D_X(M)),\qquad 
f^\ast(D_X(M))\cong D_Y(f^!(M))
$$
$$
D_X(f_!(N))\cong f_\ast(D_Y(M)),\qquad
f_!(D_Y(N))\cong D_X(f_\ast(N)),
$$
cf. \cite[Theorem 2.3.75]{ayoub:thesis1}.
\item {\bf Idempotent completeness:} If the triangulated categories $\mathbb{D}(X)$ or more generally $\mathbb{D}(\mathscr{F},\mathcal{I})$ have all small sums, they are necessarily idempotent complete. In this situation, the constructible subcategories $\mathbb{D}^{\op{c}}(X)$ are by definition idempotent complete. 
\end{enumerate} 

\begin{Bemerkung}
\label{def:motive}
Motives of varieties can be defined in any homotopical stable algebraic derivator, using the six functor formalism. If $\mathbb{D}$ is a monoidal homotopical stable algebraic derivator, we usually denote by $\underline{X}\in\mathbb{D}(X)$ the tensor unit of the symmetric monoidal category $\mathbb{D}$. 

If $S$ is a base variety and $f:X\to S$ is smooth, then the motive of $X/S$ is defind as  $\op{M}_S(X):=f_\sharp\underline{X}\in\mathbb{D}(S)$. In the general case of an arbitrary morphism of varieties $f:X\to S$, there are various notions of motive:
\begin{itemize}
\item \index{motive!$\op{M}(X)$} The \emph{(homological) motive} of $X/S$ is defined to be $\op{M}_S(X):= f_!f^!\const{S}$. This agrees with the above definition for $X/S$ smooth by relative purity. 
\item \index{motive!cohomological} Its Verdier dual is the \emph{cohomological motive} $X/S$, given by 
\[
D_S(\op{M}_S(X))\cong f_\ast f^\ast \const{S}. 
\]
\item \index{motive!Borel--Moore} The \emph{Borel--Moore motive} is defined as $\op{M}^{\op{BM}}_S(X):= f_!f^\ast\const{S}$. For $X/S$ smooth of relative dimension $d$, we have $\op{M}^{\op{BM}}_S(X)(d)[2d]\cong\op{M}_S(X)$ by relative purity. 
\item \index{motive!with compact support} Its Verdier dual is the (cohomological) \emph{motive with compact support} $\op{M}_S^{\op{c}}(X):=D(\op{M}^{\op{BM}}_S(X))\cong f_\ast f^!\const{S}$. 
\end{itemize}
The \emph{Tate motive} $\const{X}(1)\in\mathbb{D}(X)$ is given by 
\[
\op{cone}\left(1:\const{X}\to \op{M}_X(\mathbb{G}_{\op{m}}\times X)\right)[-1]
\]
This motive is $\otimes$-invertible and induces the Tate twist via $M\mapsto M(1)=M\otimes \const{X}(1)$.
\end{Bemerkung}

\begin{Bemerkung}
\label{def:motcohom}
Let $\mathbb{D}$ be a homotopical stable algebraic derivator satisfying the conditions of \ref{derivator:conditions}. 
\index{motivic cohomology}
For a variety $X/k$, we define the \emph{motivic cohomology} of $X$ as
\[
\op{H}^{n,i}(X,\mathbb{D}):=\mathbb{D}(\op{M}_k(X),\Lambda(i)[n]).
\]
\index{motivic cohomology!compact supports}
\emph{Motivic cohomology with compact supports} is defined by 
\[
\op{H}^{n,i}_{\op{c}}(X,\mathbb{D}):=\mathbb{D}(\op{M}^{\op{c}}(X),\Lambda(i)[n])
\]
\index{motivic cohomology!Borel--Moore}
Dually, we define \emph{Borel--Moore homology}
\[
\op{H}^{\op{BM}}_{n,i}(X,\mathbb{D}):=\mathbb{D}(\Lambda(i)[n],\op{M}^{\op{c}}(X)) \cong \mathbb{D}(\op{M}^{\op{BM}}(X),\Lambda(i)[n]).
\]

If $\Lambda$ is a field of characteristic zero and $\mathbb{D}=\mathbf{DA}^{\et}(-;\Lambda)$ the above are the (usual) motivic cohomology and motivic Borel--Moore homology with $\Lambda$-coefficients. By results of Voevodsky, cf. \cite{mazza:voevodsky:weibel}, these can be identified with higher Chow groups: 
\begin{itemize}
\item For $X$ a smooth variety over a perfect field $k$, there are natural isomorphisms $\op{H}^{n}(X,\Lambda(i))\cong \op{CH}^i(X,2i-n;\Lambda)$.
\item 
With rational coefficients (or assuming the base field $k$ admits resolution of singularities), we have for any equi-dimensional variety $X/k$ of dimension $d$ and any positive $i\leq d$ a canonical isomorphism
\[
\op{CH}^{d-i}(X,n;\Lambda)\cong\op{H}^{\op{BM}}_{2i+n}(X,\Lambda(i)).
\]
\end{itemize}
\end{Bemerkung}

\begin{Bemerkung}
\label{belu18}
\index{dualizing object!relative}
For a morphism $f:X\to Y$, we have the \emph{relative dualizing object} $D_f:=f^!(\const{Y})$. If $f$ is smooth of relative dimension $d$, the relative dualizing object $f^!(\underline{Y})$ is isomorphic to $\underline{X}(d)[2d]$ (and then in particular $\otimes$-invertible). This follows from the relative purity statements above because 
\[
f^!(\underline{Y})\cong f^\ast(\underline{Y})(d)[2d]\cong \underline{X}(d)[2d].
\]
This is stronger than the corresponding statement in \cite[1.8]{BeLu} because the motivic framework has orientability built in.
More generally, it follows from relative purity that for $f$ smooth and $M$ any motive we have $f^!(M)\cong f^\ast(M)\otimes f^!(\const{Y})$ because $f^!(\const{Y})\cong \const{X}(d)[2d]$. 
\end{Bemerkung} 

\begin{Bemerkung}
In general, Verdier duality doesn't commute with tensor products. Here is the best one can do: if $M\in\mathbb{D}(X)$ is strongly dualizable for the $\otimes$-structure and $N\in\mathbb{D}(X)$ is constructible, then we have
\[
D_X(M\otimes_X N)\cong \iHom_X(M,D_X(N))\cong M^\vee\otimes_X D_X(N).
\]
Here, $M^\vee=\iHom_X(M,\const{X})$ is the $\otimes$-dual of $M$, and this is in generally different from $D_X(M)=\iHom_X(M,\op{fin}^!\const{\pt})$ with $\op{fin}^!\const{\pt}$ the dualizing object. However, if $M$ is smooth of dimension $d$, then $\const{X}\cong\op{fin}^!\const{\pt}(-d)[-2d]$. In that case, we get 
\[
D_X(M\otimes_X N)(d)[2d]\cong D_X(M)\otimes_X D_X(N).
\]
\end{Bemerkung}

\begin{Bemerkung}
\index{exterior product}
\label{extverdier}
If we have varieties $X$ and $Y$ and motives $M\in\mathbb{D}(X)$ and $N\in\mathbb{D}(Y)$, then we define the \emph{exterior product}
\[
M\boxtimes N:=\op{pr}_1^\ast M\otimes_{X\times Y}\op{pr}_2^\ast(N)\in\mathbb{D}(X\times Y). 
\]
If $X$ and $Y$ are smooth, $M\in\mathbb{D}(X)$ is strongly dualizable for the $\otimes$-structure and $N\in\mathbb{D}(X)$ is constructible, the above compatibility of Verdier duality with the tensor product implies that Verdier duality commutes with exterior product; there is no shift or twist because the one introduced by commuting Verdier duality with $\otimes$ is cancelled by the one commuting Verdier duality and $f^\ast$ via relative purity. 
\end{Bemerkung}

We now want to motivic versions of the statements in \cite[1.4.7]{BeLu}. 

\begin{Bemerkung}
\label{belu1471}
Let $f:X\to Y$ be a smooth morphism, and let $M, N\in\mathbb{D}(Y)$ be motives which are both strongly dualizable and constructible.\footnote{Note that strong dualizability concerns the duality related to the tensor structure $\otimes$, and constructibility concerns Verdier dualizability.} We start with the natural isomorphism:
\[
\iHom_X (f^\ast M , f^\ast N) \xrightarrow{\cong} f^\ast \iHom_Y (M, N).
\]
Recall from the lesson (10) on Verdier duality that we can compute inner Homs  via $\iHom_X(A,B)\cong D_X(A\otimes_XD_X(B))$. Combining this with the previous compatibility of $f^\ast$ and $\iHom$, then we get
\begin{eqnarray*}
D_X(D_X(f^\ast N)\otimes f^\ast M)&\cong& f^\ast D_Y(D_Y(N)\otimes M)\\
D_X(f^!D_Y(N)\otimes f^\ast M)&\cong& D_Xf^!(D_Y(N)\otimes M)\\
f^! D_Y(N)\otimes f^\ast M&\cong& f^!(D_Y(N)\otimes M)\\
f^!P\otimes f^\ast M&\cong&f^!(P\otimes M)
\end{eqnarray*}
Here the second line follows by commuting pullback functors with Verdier duality, the third follows since the dualizability assumptions allow to remove the outer Verdier duality. The last line is then simply obtained by writing $D_Y(N)$ as a constructible and strongly dualizable motive $P$. 

It can be checked that the natural map in the last line corresponding to the module structure $f^\ast(-)\otimes f^!(-)\to f^!(-\otimes -)$ in \cite[Section 2.3]{ayoub:thesis1} corresponds to the exchange isomorphism for $f^\ast$ and $\iHom$ expressed via Verdier duality in the first line.

This formula requires that $P$, $M$, $f^\ast M$ and $f^!P$ are $\otimes$-invertible and constructible. This is satisfied for $P$ a constant mixed Tate motive and the derivators we consider, by arguments parallel to those in \cite[Section 4.4]{cisinski:deglise}. In particular, this implies  
\[
f^\ast D_Y\otimes f^!(\const{Y})\cong f^!(D_Y\otimes \const{Y})\cong f^!(D_Y). 
\]
As a consequence, Verdier duality commutes with $f^\ast$ up to twist by a $\otimes-$invertible constructible object.
\end{Bemerkung}

\begin{Bemerkung}
\label{belu1472}
\index{relatively smooth}
We establish a version of absolute purity slightly more general than the usual one. Recall that absolute purity means that for a closed immersion $i:Z\hookrightarrow X$ of pure codimension $d$ between regular schemes, we have $\const{Z}\cong i^!\const{X}(d)[2d]$. 

A closed immersion $i:Z\hookrightarrow X$ is called \emph{relatively smooth} if there is an \'etale neighbourhood $Y$ of $Z$ in $X$ such that $Y\cong Z\times\mathbb{A}^d$ and $i$ is the embedding of the zero section. Essentially, we require the existence of a tubular neighbourhood. Note that a closed immersion between regular schemes is relatively smooth. 

For relatively smooth closed immersions, the relative dualizing object is $\otimes$-invertible, cf. the arguments in \cite[Section 4.4]{cisinski:deglise}. 

For a relatively smooth closed immersion $i:Z\hookrightarrow X$, denote by $p:Y\cong Z\times\mathbb{A}^d\to Z$ the projection. We call a motive $M$ on $X$ smooth relative to $i$ if $M|_Y\cong p^\ast N$ for some motive $N$ on $Z$. For such motives we have $i^!M\cong i^\ast M\otimes D_f$. This holds in particular for the dualizing object, i.e., we have $D_X\cong i^!D_Y\cong i^\ast D_Y\otimes D_i$.
\end{Bemerkung}

\begin{Bemerkung}
\label{belu1473}
Recall from lesson (5) on base change formulas above that for a cartesian diagram
\[
\xymatrix{
V \ar[r]^{i'} \ar[d]_{f'} & X \ar[d]^f \\
W \ar[r]_i & Y
}
\]
with $i$ a closed immersion and $f$ a smooth morphism, there is a base-change formula $(f')^\ast\circ i^!\xrightarrow{\approx} (i')^!\circ f^\ast$. Again there is a more general statement: this holds if $f$ is \'etale-locally a product projection but not necessarily with smooth fibers. To see this, we can reduce to product projections by passing to a trivializing \'etale cover and using conservativity. By base-change along a resolution of singularities (resp. an alteration), we can also assume that the base $Y$ is smooth (but the fiber could still be singular). By assumption, the closed immersions $i$ and $i'$ would be relatively smooth. The base-change diagram can then be factored as composition of two commutative diagrams, one for the normal bundle projection, the other for the \'etale local inclusion of the normal bundle. For the diagram concerning the restriction to the normal bundle, the claimed base-change formula is satisfied by 2-functoriality because ordinary and exceptional pullback agree by relative purity. We can deal with the diagram containing the vector bundle projections by checking after pullback to an alteration (where it holds by the smooth case mentioned earlier) together with conservativity. 
\end{Bemerkung}

\begin{remark}
One would like to interpret these statements as a natural part of the six-functor formalism which is not usually included: for a morphism $f$, consider the relatively smooth motives to be those which are $\otimes$-invertible and constructible and remain so after pullback, both ordinary and exceptional. Then such motives satisfy relative and absolute purity, a base-change for $f_\ast$ and $p^\ast$, as well as an exchange for $i^!$ and $p^\ast$.
\end{remark}

\subsection{\'Etale motives}
\label{sec:etalemotives}

We recall the construction of the categories $\mathbf{DA}^{\et}(S;\Lambda)$ of \'etale motives over a diagram of schemes over some base scheme $S$. The construction of the categories is carried out in detail in \cite[Section 4.5]{ayoub:thesis2}, and overviews of the construction are given in \cite[Section 1.1]{ayoub:zucker} and \cite[Section 2.3]{ayoub:icm}. All the material in the following subsection is taken from these sources. \'Etale motives and $\ell$-adic realization functors are also discussed in \cite{cd:etale}. 

Denote by $\Lambda$ a commutative unital ring of coefficients; in the present work, we will mostly be interested in the case where $\Lambda$ is a field of characteristic zero. Let $S$ be a separated noetherian scheme of finite Krull dimension, and let $(\mathscr{F},\mathcal{I})$ be a diagram of $S$-schemes, where $\mathcal{I}$ is a small category (viz. the index category) and $\mathscr{F}:\mathcal{I}\to \op{Sch}/S$ is a functor. 

\index{smooth schemes over diagrams}
The category $\op{Sm}/(\mathscr{F},\mathcal{I})$ of \emph{smooth schemes over the diagram $(\mathscr{F},\mathcal{I})$} is the category of pairs $(U,i)$ where $i$ is an object of $\mathcal{I}$, and $U$ is a smooth $\mathscr{F}(i)$-scheme. The category $\op{Sm}/(\mathscr{F},\mathcal{I})$ is equipped with the \'etale topology (by ignoring the diagram category component). Then $\op{Sh}_{\et}(\op{Sm}/(\mathscr{F},\mathcal{I});\op{Cplx}(\Lambda))$ denotes the category of \'etale sheaves on $\op{Sm}/(\mathscr{F},\mathcal{I})$ with values in complexes of $\Lambda$-modules. For a smooth $S$-scheme $X$, denote by $\Lambda_{\et}(X)$ the \'etale sheaf represented by $X$, i.e., the \'etale sheaf associated to the presheaf 
\[
(U,i)\mapsto
\Lambda[\op{Hom}_{\mathscr{F}(i)}(U,\op{fin}_{\mathscr{F}(i)}^\ast X)],
\]
where $\op{fin}_{\mathscr{F}(i)}:\mathscr{F}(i)\to S$ is the structural morphism. 

For a morphism of diagrams of schemes $(f,\alpha):(\mathscr{G},\mathcal{J})\to (\mathscr{F},\mathcal{I})$, there is an associated functor 
\[
(f,\alpha)^\ast:
\op{Sh}(\op{Sm}/(\mathscr{F},\mathcal{I});\op{Cplx}(\Lambda)) \to
\op{Sh}(\op{Sm}/(\mathscr{G},\mathcal{J});\op{Cplx}(\Lambda))
\]
such that $\op{Sh}(\op{Sm}/(-,-);\op{Cplx}(\Lambda))$ becomes a $2$-functor. Tracing through \cite[Def. 4.5.1, Lemma 4.5.2]{ayoub:thesis2}, the functor $(f,\alpha)^\ast$ is eventually given by fiber product of schemes. For $(f,\alpha)$ a smooth morphism there is a left adjoint $(f,\alpha)_\sharp$ of $(f,\alpha)^\ast$, given by composition with $(f,\alpha)$.  

There is a model structure on $\op{Sh}_{\et}(\op{Sm}/(\mathscr{F},\mathcal{I});\op{Cplx}(\Lambda))$ whose weak equivalences are the stalkwise weak equivalences such that the associated homotopy category is the derived category of \'etale sheaves of $\Lambda$-modules on $\op{Sm}/(\mathscr{F},\mathcal{I})$. As a next step, one defines a new model structure, the \emph{$\mathbb{A}^1$-local model structure}, on $\op{Sh}_{\et}(\op{Sm}/(\mathscr{F},\mathcal{I});\op{Cplx}(\Lambda))$ by applying Bousfield localization to the class of maps 
\[
(U,i)\otimes K\to (\mathbb{A}^1\times_SU,i)\otimes K,
\]
where $i$ is an object of $\mathcal{I}$, $U$ is a smooth $\mathscr{F}(i)$-scheme, and $K$ a complex of $\Lambda$-modules. 

\begin{definition}
\index{motives!effective \'etale}
The category of \emph{effective \'etale motives} over $(\mathscr{F},\mathcal{I})$, denoted by $\mathbf{DA}^{\et}_{\op{eff}}((\mathscr{F},\mathcal{I});\Lambda)$, is the homotopy category of the $\mathbb{A}^1$-local model structure on $\op{Sh}_{\et}(\op{Sm}/(\mathscr{F},\mathcal{I});\op{Cplx}(\Lambda))$. 
\end{definition}

There is a symmetric monoidal structure on $\op{Sh}_{\et}(\op{Sm}/(\mathscr{F},\mathcal{I});\op{Cplx}(\Lambda))$ induced from the tensor product of $\Lambda$-modules. The unit section $1:S\to\mathbb{G}_{\op{m},S}$ of the multiplicative group over $S$ gives rise to a morphism $\Lambda_{\et}(S)\to\Lambda_{\et}(\mathbb{G}_{\op{m},S})$ of  representable \'etale sheaves. The Tate motive $\Lambda_{\et}(1)$ is defined to be 
\[
\Lambda_{\et}(1) := \op{Cone}\left(\Lambda_{\et}(S)\to
\Lambda_{\et}(\mathbb{G}_{\op{m},S})\right)[-1].
\]  
The next step, also called \emph{$\mathbb{P}^1$-stabilization}, is now to make the Tate motive $\Lambda_{\et}(1)$ invertible with respect to the tensor product. One very convenient way to do this is the formalism of symmetric spectra, cf. \cite{hovey} or \cite[Section 4.3]{ayoub:thesis2}. Without even touching definition or details, the result is a category of symmetric spectra denoted by
\[
\op{Spect}^{\Sigma}_{\Lambda_{\et}(1)}\left(
\op{Sh}_{\et}^{\mathbb{A}^1}(\op{Sm}/(\mathscr{F},\mathcal{I});
\op{Cplx}(\Lambda))\right) 
\]

Now we can define \'etale motives over $(\mathscr{F},\mathcal{I})$, as the result of $\mathbb{A}^1$-localization and $\mathbb{P}^1$-stabilization on  complexes of \'etale sheaves $\Lambda$-modules on $\op{Sm}/(\mathscr{F},\mathcal{I})$:

\begin{definition}
\label{def:daet}
\index{\'etale motives $\mathbf{DA}^{\et}(-;\Lambda)$} 
The category $\mathbf{DA}^{\et}((\mathscr{F},\mathcal{I});\Lambda)$ of \emph{\'etale motives}\index{motives!\'etale} (or \emph{\'etale motivic sheaves}) over $(\mathscr{F},\mathcal{I})$ is the homotopy category of the $\mathbb{A}^1$-local model structure on symmetric spectra 
\[
\op{Spect}^{\Sigma}_{\Lambda_{\et}(1)}\left(
\op{Sh}_{\et}^{\mathbb{A}^1}(\op{Sm}/(\mathscr{F},\mathcal{I});
\op{Cplx}(\Lambda))\right) 
\]
For a smooth $S$-scheme $X$, the motive $\op{M}_S(X)$ is given by the suspension spectrum of the \'etale sheaf $\Lambda_{\et}(X)$ represented by $X$.

The subcategory $\mathbf{DA}^{\et}_{\op{c}}(S;\Lambda)$ of \emph{constructible motives}\index{motives!constructible} is the smallest thick triangulated subcategory of $\mathbf{DA}^{\et}(S;\Lambda)$  generated by motives $\op{M}_S(X)$ of smooth $S$-schemes of finite presentation. 
\end{definition}

The following is essentially a homotopy version of the notion of cartesian objects in a fibered category, cf. e.g. \cite[2.4.1]{BeLu}. For the study of equivariant motives, we will mostly consider cartesian motives over the simplicial Borel construction or over the category of resolutions. 

\begin{definition}
\label{def:cartesian}
Let $\mathbb{D}$ be a homotopical stable algebraic derivator and let $(\mathscr{F},\mathcal{I})$ be a diagram. A motive $M\in\mathbb{D}(\mathscr{F},\mathcal{I})$ 
over the diagram is called \emph{cartesian}\index{motives!cartesian} if for each morphism $f:\mathscr{F}(i)\to\mathscr{F}(j)$ in the diagram, the associated morphism $f^\ast(M(\mathscr{F}(j)))\to M(\mathscr{F}(i))$ is an isomorphism. The full subcategory of cartesian objects will be denoted by $\mathbb{D}^{\op{cart}}(\mathscr{F},\mathcal{I})$. 
\end{definition}

\begin{remark}
Note that the isomorphism $f^\ast(M(\mathscr{F}(j)))\to M(\mathscr{F}(i))$ above has to be interpreted in the category $\mathbf{DA}^{\et}(\mathscr{F}(i);\Lambda)$ which is a derived category of spectra of complexes. The above definition provides a notion of ``homotopy cartesian'' rather than cartesian on the nose; it is rather similar to the condition of cohomology being locally constant in the simplicial definition of the equivariant derived category of \cite{BeLu}. We hope that this does not lead to confusion. 
\end{remark}

\begin{Bemerkung}
The functors $(f,\alpha)^\ast$, $(f,\alpha)_\ast$ and $(f,\alpha)_{\sharp}$ are compatible with the above model category constructions and provide Quillen adjunctions on the model categories of spectra of chain complexes  \cite[Theorem 4.5.23]{ayoub:thesis2}. 
\end{Bemerkung} 

\begin{Bemerkung}[Monoidal structure]
\label{symmoncat}
The tensor product of sheaves of $\Lambda$-modules induces a closed symmetric monoidal structure on $\mathbf{DA}^{\et}(-;\Lambda)$, cf. \cite[Theorem 4.5.24]{ayoub:thesis2}.
\end{Bemerkung}

\begin{Bemerkung}
By \cite[Theorem 4.5.30, Corollary 4.5.47 and Section 4.5.4]{ayoub:thesis2}, the axioms {\bf DerAlg 0-5}  are satisfied for $\mathbf{DA}^{\et}(-;\Lambda)$, and then all the assertions of \cite[Section 2.4]{ayoub:thesis1} apply. 
\end{Bemerkung}

\begin{Bemerkung}
\label{compgen}
By \cite[Theorem 4.5.67]{ayoub:thesis2} or \cite[Proposition 3.14]{ayoub:realisation}, the categories  $\mathbf{DA}^{\et}(X;\Lambda)$ are compactly generated whenever $X$ is a noetherian $S$-scheme of finite Krull dimension. By \cite[Theorem 3.9]{ayoub:realisation}, the derivator $\mathbf{DA}^{\et}(-;\Lambda)$ is separated, and hence all the usual assertions about constructible objects apply, cf. \cite[Theorem 8.10, 8.12]{ayoub:realisation}. 

Moreover, the categories $\mathbf{DA}^{\et}(-;\Lambda)$ have small sums since they arise as homotopy categories of symmetric spectra in unbounded chain complexes. Therefore, the categories $\mathbf{DA}^{\et}(-;\Lambda)$ are idempotent complete. Consequently, the subcategories $\mathbf{DA}^{\et}_{\op{c}}(-;\Lambda)$ of constructible motives are also idempotent complete. 
\end{Bemerkung}

\begin{Bemerkung}
  We finally recall a result explaining how to compute morphisms between \'etale motives, cf. \cite[Theorem 4.12]{ayoub:icm}: for $X$ a smooth variety over the base scheme $S=\Spec k$, there is a canonical isomorphism  
\[
\mathbf{DA}^{\et}_{\op{Spec}(k)}(\op{M}(X),\Lambda(p)[q])\cong
\op{H}_{\et}^{q-2p}(X;
\op{Sing}^{\mathbb{A}^1}\Lambda_{\op{tr}}(\mathbb{P}^1_k,\infty_k)^{\wedge
  p}). 
\]
The group on the right-hand side is \'etale motivic cohomology with $\Lambda$-coefficients.  For $\mathbb{Q}\subseteq \Lambda$, we can identify \'etale and Nisnevich motivic cohomology with $\Lambda$-coefficients. In particular, for $X=\op{Spec}k$, this allows to identify morphisms between mixed Tate motives over $k$ with Adams eigenspaces of algebraic K-groups: 
\[
\mathbf{DA}^{\et}_{\op{Spec}(k)}(\mathbb{Q},\mathbb{Q}(p)[q])\cong
\op{gr}_\gamma^p\op{K}_{2p-q}(k)_{\mathbb{Q}}.
\]
\end{Bemerkung}

\subsection{Enriched mixed Weil cohomology theories}

Another example of triangulated categories equipped with a six functor formalism arises from modules over (enriched) mixed Weil cohomology theories. These were discussed first in \cite{cd:weil}, and extended to coefficients in arbitrary Tannakian categories in \cite{drew:thesis}. We recall the axiomatics and the construction of the associated motivic triangulated categories resp. homotopical stable algebraic derivator. Eventually, our main interest is in the graded mixed Weil cohomology given by the graded pieces of the Hodge realization for schemes over $\mathbb{C}$. 

Recall from \cite[Definition 2.1.1]{drew:thesis} the following definition of a mixed Weil cohomology theory enriched in a Tannakian category. 

\begin{definition} 
\index{enriched mixed Weil cohomology theory}
Let $S$ be a noetherian scheme of finite Krull dimension, and let $\mathcal{T}_0$ be a Tannakian category of finite Ext-dimension, and denote $\mathcal{T}=\op{Ind-}\mathcal{T}_0$. A \emph{mixed Weil cohomology theory enriched in $\mathcal{T}$}\index{mixed Weil cohomology theory} is a presheaf $\op{E}_S$ of commutative differential graded algebras in $\mathcal{T}$ on the category of smooth affine $S$-schemes, satisfying the following axioms:
\begin{enumerate}[(W1)]
\item descent for Nisnevich hypercoverings,
\item $\mathbb{A}^1$-invariance,
\item normalization, i.e., $\op{E}_S(S)$ is contractible,
\item for $\sigma_1:S\to \mathbb{G}_{\op{m},S}$ the unit section, the object $\mathbb{Q}_{\mathcal{T}}(-1):=\ker(\op{E}_S(\sigma_1)[1])$ belongs to the heart of the natural t-structure of $\op{Der}(\mathcal{T})$ and induces an autoequivalence $\mathbb{Q}_{\mathcal{T}}(-1)\otimes^{\op{L}}_{\mathcal{T}}(-)$ of $\op{Der}(\mathcal{T})$. 
\item K\"unneth formula, i.e., for any smooth affine schemes $X,Y$ over $S$, the canonical morphism $\op{E}_S(X)\otimes^{\op{L}}_{\mathcal{T}}\op{E}_S(Y)\to\op{E}_S(X\times_S Y)$ is a weak equivalence. 
\end{enumerate}
\end{definition}

The special case where $\mathcal{T}$ is simply the category of finite-dimensional vector spaces over a field $\Lambda$ of characteristic $0$ is the definition of \emph{mixed Weil cohomology theory} from \cite{cd:weil}.  

\begin{example}[\cite{cd:weil}, Section 3]
\begin{enumerate}
\item $\ell$-adic cohomology is a mixed Weil cohomology theory whose 
  associated commutative ring spectrum is denoted by $\mathcal{E}_{\op{et},\ell}$,
  cf. \cite[Section 3.3]{cd:weil}. 
\item Algebraic de Rham cohomology is a mixed Weil cohomology whose 
  associated commutative ring spectrum is denoted  $\mathcal{E}_{\op{dR}}$,
  cf. \cite[Section 3.1]{cd:weil}. 
\end{enumerate}
In particular, the $\ell$-adic realization functors will be relevant for our discussion.
\end{example} 

The more important examples for our present work arise from the semisimplification of the Hodge realization, as discussed in \cite[Section 2.4]{SoWe}: 

\begin{example}
Associating to a smooth $\mathbb{C}$-scheme $X$  the singular cohomology of the associated complex manifold $X(\mathbb{C})$, equipped with its polarizable mixed Hodge structure, yields an enriched mixed Weil cohomology theory with coefficients in $\mathcal{MHS^{\op{pol}}}_{\mathbb{Q}}$. This cohomology theory will be denoted by $\op{E}_{\op{Hodge}}$. Composing with the functor taking the associated weight-graded pieces yields an enriched mixed Weil cohomology theory $\op{E}_{\op{GrH}}$ with coefficients in graded pure Hodge structures.
The associated commutative ring spectrum is denoted by $\mathcal{E}_{\op{GrH}}$, and there is a corresponding motivic triangulated category (alternatively, a  homotopical stable algebraic derivator) $\op{Der}(\mathcal{E}_{\op{GrH}})$  with full six-functor formalism. 
\end{example}

By \cite[Theorem 2.2.7]{drew:thesis}, for any perfect field $k$, there is an equivalence of symmetric monoidal triangulated categories
\[
\op{Der}(\op{Spec}k;\mathcal{E})\approx \op{Der}(\mathcal{T})
\]
between the $\mathcal{E}$-modules over $L$ and the derived category of the Tannakian category $\mathcal{T}$. 
\begin{enumerate}
\item In the special case of ordinary mixed Weil cohomology theories with values in finite dimensional $\mathbb K$ vector spaces, this equivalence restricts to an equivalence $\op{Der}_{\op{c}}(k;\mathcal{E})\approx \op{Der}^{\op{b}}(\mathbb{K}\textrm{-modf})$ between the compact $\mathcal{E}$-modules over $k$ and the bounded derived category of finite-dimensional $\mathbb{K}$-vector spaces. 
\item 
\index{semisimplified Hodge motives}
In the case of the associated graded of the Hodge realization, this equivalence restricts to an equivalence  \[\op{Der}_{\op{c}}(k;\mathcal{E}_{\op{GrH}})\approx \op{Der}^{\op{b}}(\mathbb{K}\op{-modf}^\DZ)
\]
between the compact $\mathcal{E}_{\op{GrH}}$-modules over $k$ and the bounded derived category of the category of graded finite-dimensional $\mathbb{K}$-vector spaces. 
\end{enumerate}

\begin{remark}
Although the results of Cisinski--D{\'e}glise and Drew are formulated using the notion of motivic triangulated categories, these constructions can be translated to the language of homotopical stable algebraic derivators. 
\end{remark}

\begin{Bemerkung}
\index{mixed derived category, $\op{MDer}$}
In most of this work, we will denote the derivator $\mathcal{E}_{\op{GrH}}$ of motives with coefficients in the semisimplification of the Hodge realization by $\op{MDer}(-;\mathbb{C})$ or simply $\op{MDer}$. We think of this as a natural mixed version of the usual derived category of $\mathbb{C}$-vector spaces on a complex variety.
\end{Bemerkung}

\begin{Bemerkung}
As before, the categories of motives with coefficients in a ring spectrum representing an enriched mixed Weil cohomology theory are compactly generated whenever the base is a noetherian scheme of finite Krull dimension. The corresponding derivator is separated, so constructible objects are well-behaved as in \ref{compgen}. The derivators also have small sums and therefore the categories $\op{MDer}(-;\mathbb{C})$ as well as the corresponding subcategories $\op{MDer}^{\op{c}}(-;\mathbb{C})$ are idempotent complete. 
\end{Bemerkung}

\subsection{Realization functors} 

Next, we recall the construction of $\ell$-adic and Hodge realization functors, as well as their compatibility with the six functor formalism. 

We first discuss the $\ell$-adic realization functors. There are several possible ways to organize the $\ell$-adic derived categories into homotopical stable algebraic derivators. 
\begin{enumerate}
\item One way is given in \cite[Corollary 4.15, Section 5, Section 9]{ayoub:realisation}. First, for $\Lambda$ a finite $\mathbb{Z}/N\mathbb{Z}$-algebra with $N$ prime to the residue characteristics of the base scheme $S$, the derived categories $\op{Der}^{\et}(-;\Lambda)$ of \'etale sheaves of $\Lambda$-modules can be identified with $\mathbf{DA}^{\et}(-;\Lambda)$, cf. \cite[Theorem 4.1]{ayoub:realisation}. For a ring $\Lambda$ and an ideal $J\subseteq \Lambda$, \cite[Section 5]{ayoub:realisation} describes a homotopical stable 2-functor $\op{Der}^{\et}(-;\Lambda/J^\ast)$ which takes values in the category of $\Lambda/J^\ast$-modules of Ekedahl, together with a corresponding $\Lambda/J^\ast$-adic realization  functor induced essentially from change of coefficients. A further change of coefficients, cf. \cite[Section 9]{ayoub:realisation}, allows, in particular, to define realization functors 
\[
\mathfrak{R}^{\et}:\mathbf{DA}^{\et}(-;\Lambda)\to
\op{Der}^{\et}(-;\mathbb{Q}_\ell)
\]
as a morphism of homotopical stable 2-functors. This can be extended to a morphism of homotopical stable algebraic derivators by extending all the constructions of \cite{ayoub:realisation} to diagrams of schemes essentially as in \cite[Section 2.4]{ayoub:thesis1}. 
\item In \cite[17.2.5]{cisinski:deglise}, realization functors on the category of Beilinson motives are defined by considering the homotopy category of $\mathcal{E}$-modules over $X$ and taking the realization functor to be  
\[
\mathbf{DA}^{\et}(X)\rightarrow \op{Der}(X;\mathcal{E}): M\mapsto
\mathcal{E}_X\otimes_X^L M.
\] 
In the above, the category $\op{Der}(X;\mathcal{E})$ is the homotopy category of a model structure on the category of $\mathcal{E}$-modules in $\mathbf{DA}^{\et}(X)$. The $\ell$-adic realization functors are obtained by plugging the $\ell$-adic cohomology spectrum $\mathcal{E}_{\et,\ell}$ into this construction.
\item Another way to define $\ell$-adic realization functors is given in \cite[Section 7.2]{cd:etale}. If $\ell$ is a prime different from the residue characteristics of the base scheme $S$, it is possible to identify the $\ell$-completion of $\mathbf{DA}^{\et}(S;\mathbb{Z})$ with the $\ell$-adic derived category $\op{Der}^{\et}(S;\mathbb{Z}_\ell)$. 
\end{enumerate}

In all of the above constructions, it turns out that the $\ell$-adic realization functors commute with the six functors, preserve constructible objects and hence also preserve the Verdier duality formalism. In the first construction, this is \cite[Theorem 9.7]{ayoub:realisation}, in the second construction, this is \cite[17.2.18]{cisinski:deglise}, and in the third construction this is \cite[Theorem 7.2.24]{cd:etale}. 
For finite coefficients $\Lambda$ away from the residue characteristics of the base scheme, the \'etale realization functor $\mathbf{DA}^{\et}(X;\Lambda)\to \op{Der}^{\et}(X;\Lambda)$ is an equivalence, and this result remains true with $\Lambda$-adic coefficients. Our notation for the $\ell$-adic realization functors will be
\[
\op{Real}_\ell:\mathbf{DA}^{\et}(X;\mathbb{Q}_\ell)\to\op{Der}^{\et}(X;\mathbb{Q}_\ell).
\]
\index{realization functor!$\ell$-adic, $\op{Real}_{\ell}$}

Next we discuss the Betti and Hodge realization functors. If $\sigma:k\hookrightarrow \mathbb{C}$ is a subfield of the complex numbers, then the Betti realization 
\[
\mathbf{DA}^{\et}(X;\Lambda)\to\op{Der}(X^{\op{an}};\Lambda)
\]
is defined in \cite[Definition 2.1]{ayoub:betti} as the composition of an analytification functor followed by an equivalence of categories from analytic motives over $X^{\op{an}}$ to complexes of sheaves of $\Lambda$-modules over $X^{\op{an}}$. By \cite[Theorem 3.19]{ayoub:betti}, Betti realization is compatible with the six functor formalism (in the sense of being a morphism of homotopical stable 2-functors resp. a fortiori a morphism of homotopical stable algebraic derivators). Moreover, by \cite[Section 4]{ayoub:betti}, Betti realization is also compatible with the vanishing cycles formalism. 

Within the framework of motivic triangulated categories of Cisinski--D{\'e}glise, a de Rham realization can be defined by $M\mapsto M\otimes\mathcal{E}_{\op{dR}}$ mapping a motive to its associated free $\mathcal{E}_{\op{dR}}$-motive, where $\mathcal{E}_{\op{dR}}$ is the spectrum representing algebraic de Rham cohomology, cf. \cite[Section 17]{cisinski:deglise}. 

These realization functors can be refined, taking into account the additional information given by mixed Hodge structures. Such realization functors have been defined by Ivorra \cite{ivorra} and Drew \cite[Theorem 3.3.9]{drew:thesis}. In Drew's work, the functor from \'etale motives over $X$ takes values in the Ind-category of the bounded derived category of holonomic quasi-coherent $\mathscr{D}$-modules over $X$. Our notation for the Hodge realization functor will be 
\[
\op{Real}_{\op{H}}:\op{MDer}(X;\mathbb{C})\to \op{Der}(X;\mathbb{C}). 
\]
\index{realization functor!Hodge, $\op{Real}_{\op{H}}$}

\section{Preliminaries on groups, actions and homogeneous spaces}
\label{sec:grp-prelims}

In this section, we recall some preliminary statements and fix notation for algebraic groups and varieties with group actions which will be used throughout the paper. We also discuss the most relevant examples of group actions for the paper. 

In this paper, we only consider linear algebraic groups defined over fields,
which we assume reduced by definition. We freely use standard definitions and results from the theory of linear algebraic groups, in particular the structure theory for reductive groups; these can be found in the classical textbooks \cite{BorAG,HumAG,SpLAG} on the subject. 

The most important sequence of groups throughout the paper will be $G\supset P\supset B\supset T$, with $G$  a split reductive group over the field $k$, $T$ a choice of maximal split torus, $B$ a Borel defined over $k$ containing the torus $T$ and $P$ a parabolic subgroup defined over $k$ containing the Borel. The Weyl group  is denoted by $W=\op{W}(G,T)=\op{N}_G(T)/\op{C}_G(T)$.

\subsection{Definitions and examples}

\begin{Bemerkung}
\label{vwa}  
By a {\bf variety with action}\index{variety with action, $(G\looparrowright X)$} $(G\looparrowright X)$ or {\bf $G$-variety} we henceforth mean a triple $(G,X,a)$ consisting of a linear algebraic group $G$ and a variety $X$ with a $G$-action $a:G\times X\ra X$. Recall that by our standing assumption \ref{standing}, our varieties are always assumed to be quasi-projective. The varieties with action form a category in which  morphisms are pairs $(\phi,f):(G\looparrowright X)\to (H\looparrowright Y)$  consisting of a homomorphism  $\phi:G\to H$ of linear algebraic groups and a  $\phi$-equivariant map $f:X\to Y$ of varieties, i.e., the following diagram commutes
\begin{center}
  \begin{minipage}[c]{10cm}
    \xymatrix{
      G\times X\ar[d]_{(\phi,f)} \ar[r]^\rho & X\ar[d]^f\\
      H\times Y\ar[r]_{\sigma} & Y.
    }
  \end{minipage}
\end{center}
\end{Bemerkung}

\begin{remark}
Our generic notation will be using actions $G\looparrowright X$ on the left, which leads to quotients being written as $G\backslash X$.
\end{remark}

\begin{example}
\index{torsor}
  Given an algebraic group $G$, a \emph{$G$-torsor} $X$ is a variety such that there is a faithfully flat $G$-equivariant map $X\to Y$, where $G$ acts trivially on $Y$, such that the canonical map $G\times X \to X\times_Y X$ is an isomorphism. In this situation, we will say that the \emph{quotient} of the $G$-action on $X$ exists.
\end{example}

\begin{example}
\label{ex:parabolic}
Let $G$ be a connected split reductive group over a field $k$, let $B\subset G$ be a Borel subgroup and let $P,Q$ be two parabolic subgroups containing $B$. Then there is an action 
\[
(P\times Q)\times G\to G:(p,q,g)\mapsto p\cdot g\cdot q^{-1}.
\]
The $(P\times Q)$-equivariant motives over $G$ (which should be viewed as motives over the double quotient $P\backslash G/Q$) will be relevant for the categorification of the parabolic Hecke module.
\end{example}

\begin{example}
\label{ex:symmetric}
Let $G$ be a connected split reductive group $G$ over a field $k$,  $\sigma:G\to G$ be a non-trivial algebraic involution,  $T$ be a $\sigma$-stable maximal torus,  $B$ be a $\sigma$-stable Borel subgroup and $P$ be a standard parabolic. Denote by $K$ the subgroup of $G$ fixed by $\sigma$. As before, left and right multiplication provide a $(P\times K)$-action on $G$. The $(P\times K)$-equivariant motives on $G$ (which should be viewed as motives over the double quotient $P\backslash G/K$) will be relevant for the representation theory of real Lie groups.
\end{example}

\begin{example}
\label{ex:wonderful}
Let $G$ be a connected adjoint semi-simple group, and let $B\subseteq G$ be a Borel subgroup. Let $X$ be the wonderful compactification of $G$, in particular there is a $G\times G$-action on $X$ and a morphism $G\to X$ which is $(G\times G)$-equivariant for the natural action of $G\times G$ on $G$ by left and right multiplication. The $(G\times G)$-equivariant motives over $X$ will be relevant for the categorification of a module for the Hecke algebra of $G\times G$, cf. \cite{SpCompact}. 
\end{example}

\begin{example}
\label{ex:horocycle}
Let $G$ be a connected split reductive group and $i:B\subseteq G$ be the inclusion of a Borel subgroup. Then there is the adjoint action $\op{ad}:G\looparrowright G$ and its restriction $\op{ad}:B\looparrowright G$. We have morphisms of varieties with action
\[
(i,\op{id}_G):(\op{ad}:B\looparrowright G)\to(\op{ad}:G\looparrowright G), \textrm{ and}
\]
\[
(\Delta,\op{id}_G):(\op{ad}:B\looparrowright G)\to(B\times B\looparrowright G),
\]
where $(B\times B\looparrowright G)$ is the action from Example~\ref{ex:parabolic}. The above maps are relevant for a motivic version of the horocycle correspondence and character sheaves. This, however, is outside the scope of the present paper.
\end{example}

\subsection{Freeness and separably defined orbits}
\begin{definition}
\label{free} 
Given a 
variety with action $H\looparrowright X$, call the action {\bf free}\index{free!variety with action} if 
the fpqc-quotient $H\backslash X$ is represented by a variety and moreover
the quotient map $X\ra H\backslash X$ is an \'{e}tale locally trivial $H$-torsor. Also note that in our conventions, 
varieties are always  assumed to be quasi-projective. 
\end{definition}

\begin{example}
\index{balanced product, $G\times_{/H}X$}
Let $G$ be a group, let $H\subseteq G$ be a subgroup and let $H\looparrowright X$ be a variety with action. We denote by $G\times_{/H}X$ the \emph{balanced product}
\[
G\times_{/H}X:=(G\times X)/H
\]
where the $H$-action on $G\times X$ is given by the right action on $G$ via the subgroup inclusion and the left action on $X$ via $H\looparrowright X$. 

  Free actions will mainly be encountered in the following situation. Given a variety with action $G\looparrowright X$, and $H=P$ a parabolic subgroup of $G$, we will be interested in the space $G\times_{/P}X$. As $G\sra G/P$ is a Zariski-locally trivial $P$-torsor, $G\times_{/P} X$ exists as a variety if we ignore our requirement of quasi-projectivity. The quasi-projectivity will be evident in all the actual situations encountered.
\end{example}

\begin{Bemerkung} 
  The assumption that $X\ra X/H$ is locally trivial in the \'etale topology will follow from  local triviality in the fpqc-topology in the situations that we encounter by a result of Seshadri \cite[Cor. 4.1.6 of Exp. XXIV]{SGA33}: the morphism $X\sra X/H$ is locally trivial in the \'etale topology if $H$ is smooth reductive and $X/H$  is locally noetherian and normal. 
\end{Bemerkung}

\begin{Bemerkung}
\label{orbitsoverk}
Fix a field $k$ and let $G\looparrowright X$ be a variety with $G$-action, where the variety $X$, the algebraic group $G$ as well as the action morphism $a:G\times X\to X$ are defined over $k$. We will make crucial use of quotient or induction equivalences in the definition of equivariant mixed Tate motives and the corresponding weight structures. For this to be well-behaved, we will usually need to assume the following:
\begin{enumerate}
\item There are only finitely many $G(\bar k)$-orbits on $X(\bar k)$.
\item Each $G(\bar{k})$-orbit on $X(\bar{k})$ contains a $k$-rational point $x\in X(k)$ such that the induced map $G\times \{x\}\to X$ is a separable morphism onto its image.
\end{enumerate}
A variety with action $G\looparrowright X$ satisfying the above conditions will be said to have \emph{finitely many orbits separably defined over $k$}.
\index{separably defined orbits!variety with action}
\end{Bemerkung}


Note that these conditions are satisfied in the examples \ref{ex:parabolic}, \ref{ex:wonderful} and \ref{ex:symmetric} which are relevant for our intended applications. However, the conditions are not satisfied in Example~\ref{ex:horocycle}, since the conjugation action of $\op{ad}:G\looparrowright G$ generally has infinitely many orbits.


\section{Acyclicity and resolutions of group actions}
\label{sec:resolutions}

In this section, we discuss descriptions of the ``homotopy quotient'' of varieties with action which will be used in the definition of equivariant motives. One very natural description is the Borel construction, viewed as a simplicial scheme. Another natural possibility exhibiting the ``algebraic structure of the homotopy quotient'' is to use approximations by finite-dimensional smooth schemes, also called $n$-acyclic resolutions of the action.

To define useful approximations to the homotopy quotient and to establish results comparing the Borel construction to its approximations, it is necessary to have a suitable notion of acyclic maps. In the definition of equivariant derived categories using resolutions, cf. \cite{BeLu}, the crucial property of $n$-acyclic maps is that they induce equivalences of truncated derived categories of complexes of length $<n$. The same now happens in the setting of motivic categories: to be able to define equivariant motives via resolutions, we need a notion of acyclic maps which has the property that they induce equivalences on categories of suitably truncated motives. Moreover, we need enough resolutions of arbitrarily high acyclicity. 

\begin{Bemerkung}
In the following, we consider a homotopical stable algebraic derivator $\mathbb{D}$ satisfying the conditions from \ref{derivator:conditions}. This in particular means that the homotopical stable algebraic derivator $\mathbb{D}$ supports a t-structure, i.e., that for each scheme $X$ over $k$ there is a t-structure on $\mathbb{D}(X)$,  compatible with the six functors in the way specified in \cite[Scholie 2.2.95]{ayoub:thesis1}. 

More specifically, some parts of the theory of acyclic maps will only be developed for the derivators $\mathbf{DA}^{\et}(-;\Lambda)$ (with $\Lambda$ is a field of characteristic $0$) and its localizations such as $\op{MDer}(-;\mathbb{C})$ which are most interesting for our applications. These are constructed via localizations of categories of complexes of sheaves, hence there is a t-structure coming from truncation of complexes. This is usually called the \emph{homotopy t-structure}\index{homotopy t-structure} (to distinguish it from the conjectural ``motivic'' t-structure whose heart would be an abelian category of mixed motives). 
\end{Bemerkung}

The key point that makes the definition of equivariant motives via resolutions work is that the usual construction of resolutions (via product with suitable representations) produces maps whose acyclicity for the homotopy t-structure is controlled by the codimension of the non-free locus. 

\subsection{$\mathbb{A}^1$-acyclic maps and truncations of motivic categories} 

\begin{definition}
\label{def:acyclicity}
Let $k$ be a field and let $\mathbb{D}$ be a homotopical stable algebraic derivator over $k$ satisfying the conditions in \ref{derivator:conditions}. 
We say that a morphism of varieties $f:X\to Y$ is \emph{$\tau$-$n$-acyclic}\index{$\tau$-$n$-acyclic map} if it satisfies the following conditions:
\begin{enumerate}
\item For any object $M\in\mathbb{D}(Y)$ in the heart of $\tau_Y$, the natural morphism $M\to \tau_{\leq n}f_\ast f^\ast M$ is an isomorphism.
\item For any morphism of varieties $g:\tilde{Y}\to Y$, the induced map $\tilde{f}:X\times_Y\tilde{Y}\to\tilde{Y}$ satisfies condition (1). 
\end{enumerate}

A morphism $f:X\to Y$ is called $\tau$-$\infty$-acyclic if it is $\tau$-$n$-acyclic for all $n$. 
\end{definition}

\begin{remark} 
\label{rem:infacyclic}
  In the homotopy t-structures on motives we consider, a morphism $M\to N$ is an isomorphism if and only if its truncations $\tau_{\leq n}M\to\tau_{\leq n}N$ are isomorphisms for all $n\in\mathbb{Z}$. In particular, a map is $\tau$-$\infty$-acyclic if and only if $M\to f_\ast f^\ast M$ is an isomorphism. However, $\tau$-$\infty$-acyclic maps of schemes 
  are very rare.
\end{remark}


Exactness properties of the functors $f^\ast$ and $f_\ast$ for the homotopy t-structure on $\mathbb{D}(X)$ are discussed in \cite[Section 2.2.3]{ayoub:thesis1}. The following is then an analogue of \cite[Proposition 1.9.2]{BeLu}. 

\begin{definition}
\index{$\mathbb{D}^I(X)$}
For a morphism $f:X\to Y$, denote by $\mathbb{D}(X\mid Y)\subseteq\mathbb{D}(X)$ the strictly full subcategory of objects of the form $f^\ast M$ with $M\in\mathbb{D}(Y)$. We will also employ the notation $\mathbb{D}^I(X)=\mathbb{D}^{\geq a}(X)\cap\mathbb{D}^{\leq b}(X)$, where  $I=[a,b]\subseteq\mathbb{Z}$ is an interval and the truncated subcategories are understood with respect to the given t-structure $\tau$.
\end{definition}

\begin{proposition}
\label{prop:belu192}
Fix an interval $I=[a,b]\subseteq \mathbb{Z}$, and let $f:X\to Y$ be a $\tau$-$n$-acyclic map with $n\geq b-a$. 
\begin{enumerate}
\item The functor $f^\ast:\mathbb{D}^I(Y)\to\mathbb{D}^I(X\mid Y)$ is an equivalence of categories. The inverse functor is given by $\tau_{\leq b}\circ f_\ast:\mathbb{D}(X)\to \mathbb{D}(Y)$. 
\item A sequence $A\to B\to C\to A[1]$ is a distinguished triangle in $\mathbb{D}^I(Y)$ if and only if $f^\ast A\to f^\ast B\to f^\ast C\to f^\ast A[1]$ is a distinguished triangle in $\mathbb{D}^I(X)$. 
\item The subcategory $\mathbb{D}^I(X\mid Y)\subseteq\mathbb{D}(X)$ is closed under extensions and taking direct summands. 
\end{enumerate}
\end{proposition}

\begin{proof}
From the adjunction morphisms we obtain morphisms of functors $\alpha:\tau_{\leq b}\to\tau_{\leq b} f_\ast f^\ast$ and $\beta:f^\ast \tau_{\leq b}f_\ast \to\tau_{\leq b}$. Here we use that $f^\ast$ is $\tau$-exact. 

The full subcategory $\mathcal{C}\subseteq \mathbb{D}(Y)$ of objects $M$ for which $\alpha$ is an isomorphism is closed under extensions. The acyclicity condition on $f$ implies that $M\to \tau_{\leq n}f_\ast f^\ast M$ is an isomorphism for any $M$ in the heart $\mathcal{H}$ of $\tau_Y$. Then, $\mathcal{C}$ contains $\mathcal{H}[-i]$  for $i\geq a$, and therefore, $\mathcal{C}$ contains $\mathbb{D}^{\geq a}(Y)$. On the subcategory $\mathbb{D}^I(Y)$, $\tau_{\leq b}f_\ast$ is then a functorial left inverse to $f^\ast$. 

From the properties of the adjunction, the composition $f^\ast\to f^\ast f_\ast f^\ast \to f^\ast$ is the identity. If $M\in \mathbb{D}^I(Y)$, applying the truncation $\tau_{\leq b}$ yields isomorphisms $\tau_{\leq b} M\to M$ and by $\tau$-exactness of $f^\ast$ also $\tau_{\leq b}f^\ast M\stackrel{\cong}{\longrightarrow} f^\ast M$. Therefore, the composition $f^\ast M\to f^\ast \tau_{\leq b} f_\ast f^\ast M\to f^\ast M$ is an isomorphism. Combining this with the above statement about $\tau_{\leq b}f_\ast$ being a left inverse of $f^\ast$ implies that an object $N\in\mathbb{D}^I(X)$ is isomorphic to $f^\ast M$ for some $M\in\mathbb{D}^I(Y)$ if and only if $\beta:f^\ast \tau_{\leq b}f_\ast N\to N$ is an isomorphism. 

The above criterion implies that $f^\ast$ and $\tau_{\leq b}f_\ast$ are inverse equivalences of the categories $\mathbb{D}^I(Y)$ and $\mathbb{D}^I(X\mid Y)$, hence proving (1). Claim (2) follows from the $\tau$-exactness of $f^\ast$ and $\tau_{\leq b}f_\ast$. Closure under extensions and direct summands in (3) follows from the criterion which is in terms of isomorphisms.
\end{proof} 

Next, we establish a statement analogous to the descent result of \cite[Lemma 1.9.3]{BeLu}. 

\begin{proposition}
\label{prop:belu193}
Let $\mathbb{D}$ be a homotopical stable algebraic derivator satisfying the conditions in \ref{derivator:conditions}. 
Let $g:\tilde{Y}\to Y$ be a morphism of varieties which  locally in the \'etale topology has sections. Let $f:X\to Y$ be a morphism of varieties and denote $\tilde{f}:\tilde{X}:=X\times_Y\tilde{Y}\to \tilde{Y}$. Then the following statements hold: 
\begin{enumerate}
\item The induced map $\tilde{f}$ is $\tau$-$n$-acyclic if and only if $f$ is $\tau$-$n$-acyclic. 
\item Assume $f$ and $\tilde{f}$ are $\tau$-$n$-acyclic, and let $M\in\mathbb{D}^I(X)$ with $I=[a,b]$ and $n\geq b-a$. Then $M\in\mathbb{D}^I(X\mid Y)$ if and only if $\tilde{M}:=g^\ast M\in\mathbb{D}^I(\tilde{X}\mid\tilde{Y})$. 
\end{enumerate}
\end{proposition}

\begin{proof}
(1) Since $\tau$-$n$-acyclicity is by definition stable under base change, $\tilde{f}$ is $\tau$-$n$-acyclic if $f$ is. 

For the converse, we first show that the $\tau$-$n$-acyclicity of $f$ can be checked \'etale locally on the base. So assume that there exists an \'etale  covering $\{U_i\to Y\}_i$ of $Y$ such that for each $i$, the morphism $f_i:X\times_YU_i\to U_i$ is $\tau$-$n$-acyclic. Then $f$ is $\tau$-$n$-acyclic, since (by assumption of separatedness of $\mathbb{D}$) the functors $\mathbb{D}(Y)\to \mathbb{D}(U_i)$ form a conservative family, hence we can check if $M\to\tau_{\leq n}f_\ast f^\ast M$ is an isomorphism by pullback to the $U_i$. The pullback of the morphism $M\to\tau_{\leq n}f_\ast f^\ast M$ to $U_i$ is $M_i\to\tau_{\leq n}(f_i)_\ast f_i^\ast M_i$ by smooth base change.

Now, by assumption, there exists an \'etale covering $\{U_i\to Y\}_i$ such that each morphism $g_i:\tilde{Y}\times_YU_i\to U_i$ has a section $s_i$. Then $f_i$ is the base change of $\tilde{f}:X\times_Y\tilde{Y}\to \tilde{Y}$ along $U_i\stackrel{s_i}{\longrightarrow} U_i\times_Y\tilde{Y}\to\tilde{Y}$.  In particular, $f_i$ is $\tau$-$n$-acyclic, and by the previous arguments, $f$ is $\tau$-$n$-acyclic. 

(2) If $M\in\mathbb{D}^I(X\mid Y)$, then $\tilde{M}\in\mathbb{D}^I(\tilde{X}\mid\tilde{Y})$ is clear from 2-functoriality of $f^\ast$. We saw in the proof of Proposition~\ref{prop:belu192} that $M\in\mathbb{D}^I(X\mid Y)$ if and only if the morphism $\beta:f^\ast\tau_{\leq b} f_\ast M\to M$ is an isomorphism. This can again be checked \'etale locally on $Y$, by the argument in Step (1). But then, locally in the \'etale topology on $Y$, $M$ is the pullback of $\tilde{M}$ along the local sections, hence $M\in\mathbb{D}^I(X\mid Y)$ if $\tilde{M}\in\mathbb{D}^I(\tilde{X}\mid\tilde{Y})$. 
\end{proof}

This result provides a criterion for $n$-acyclicity of maps similar to \cite[1.9.4]{BeLu}. Our result is much weaker in that we require local trivializations in the \'etale topology, whereas Bernstein and Lunts only require acyclicity of fibers.

\begin{definition}
\index{\'etale fiber bundle}
  A morphism  $f:X\to Y$ of varieties is called an \emph{\'etale fiber bundle} with fiber a variety $F$ if there exists an \'etale covering $\{U_i\to Y\}_i$ such that for each $i$ the morphism $ U_i\times_YX\to U_i$ is isomorphic to the projection $U_i\times F\to U_i$ as an object over $U_i$. 
\end{definition}

\begin{corollary}
\label{cor:acyclicity}
Let $f:X\to Y$ be an \'etale fiber bundle. If the fiber is $\tau$-$n$-acyclic, then $f$ is $\tau$-$n$-acyclic. 
\end{corollary}

\begin{proof}
This is a direct application of Proposition~\ref{prop:belu193}. 
\end{proof}

\begin{remark}
\'Etale motives are the natural coefficient systems of the \'etale version of Morel's stable $\mathbb{A}^1$-homology. One would naturally expect that a morphism $f:X\to Y$ is $n$-acyclic for the homotopy t-structure on $\mathbf{DA}^{\et}(-;\Lambda)$ if and only if $f$ induces isomorphisms $\op{H}^{\mathbb{A}^1}_i(X,f^\ast M)\to\op{H}^{\mathbb{A}^1}_i(Y,M)$ for all $0\leq i\leq n$ and all \'etale motives $M\in\mathbf{DA}^{\et}(Y;\Lambda)$.  We do not know if such a characterization (or even an integral version) holds in this generality. In any case, the above special case of \'etale locally trivial maps with $n$-acyclic fibers will suffice for our purposes.
\end{remark}

It remains to provide criteria for a variety to be sufficiently $n$-acyclic. The following result, based on an excision argument of Asok and Doran in \cite[Corollary 4.7]{asok:doran}, provides such a criterion. This criterion will be applied to the usual construction of resolutions for group actions. Note that the excision result makes use of an identification of \'etale and Nisnevich cohomology and is therefore restricted to $\mathbf{DA}^{\et}(-;\Lambda)$ with $\Lambda$ a field of characteristic $0$. 

\begin{proposition}
\label{prop:excision}
Let $k$ be a field, and let $Z\subseteq \mathbb{A}^n$ be a closed subvariety of codimension $i$. Then  $\mathbb{A}^n\setminus Z$ is $(i-2)$-acyclic for the homotopy t-structure on $\mathbf{DA}^{\et}(-;\Lambda)$. 
\end{proposition}

\begin{proof}
  We prove that the structure morphism $f:\mathbb{A}^n\setminus Z\to\op{Spec} k$ is $(i-2)$-acyclic. To prove this, it suffices by Proposition~\ref{prop:belu193} to pass to an \'etale cover of $\op{Spec} k$. In particular, we can assume that the base field is infinite, which is necessary to apply the techniques of \cite{morel:book} and \cite{asok:doran}. 

By definition of $n$-acyclicity, cf. Definition~\ref{def:acyclicity}, we need to show that for each $M\in\mathbf{DA}^{\et}(\op{Spec}k;\Lambda)$ in the heart of the homotopy t-structure, the induced map $M\to \tau_{\leq (i-2)}f_\ast f^\ast M$ is an isomorphism in $\mathbf{DA}^{\et}(\op{Spec}k;\Lambda)$, and that this is stable under pullbacks. By Voevodsky's results, cf. \cite{friedlander:suslin:voevodsky}, the heart of the homotopy t-structure on $\mathbf{DA}^{\et}_{\op{eff}}(\op{Spec}k;\Lambda)\cong\op{DM}_{\op{eff}}(\op{Spec}k;\Lambda)$ is given by the homotopy-invariant Nisnevich sheaves with transfers. More generally, for the non-effective motives, the heart is given by homotopy modules, i.e., sequences $(\mathbf{F}_n)_{n\in\mathbb{Z}}$ of homotopy-invariant Nisnevich sheaves with transfers which are connected by isomorphisms $\mathbf{F}_n\xrightarrow{\cong} (\mathbf{F}_{n+1})_{-1}$, where the contraction $\mathbf{A}_{-1}$ of a homotopy-invariant Nisnevich sheaf $\mathbf{A}$ is the sheaf 
\[
\mathbf{A}_{-1}(U)=\ker\left(\mathbf{A}(\mathbb{G}_{\op{m}}\times U)\to\mathbf{A}(U)\right).
\]
This has been established in the thesis of D\'eglise, cf. \cite{deglise:mh}. 
The relevant point for our proof is that the cohomology of smooth schemes with coefficients in a homotopy module can be computed using a Gersten-type complex. 

Let now $M\in\mathbf{DA}^{\et}(\op{Spec}k;\Lambda)$ be in the heart of the homotopy t-structure. 
To show that $M\to\tau_{\leq (i-2)}f_\ast f^\ast M$ is an isomorphism, it suffices to show that $\op{H}^j_{\et}(S;M)\to\op{H}^j_{\et}(S;f_\ast f^\ast M)$ is an isomorphism for all $0\leq j\leq i-2$ and all smooth $k$-schemes $S$. Now $M$ has rational coefficients, so that \'etale cohomology and Nisnevich cohomology of $M$ resp. $f_\ast f^\ast M$ are isomorphic; it thus suffices to show that $\op{H}^j_{\op{Nis}}(S;M)\to\op{H}^j_{\op{Nis}}(S;f_\ast f^\ast M)$ is an isomorphism for all $0\leq j\leq i-2$ and all smooth $k$-schemes $S$.
The projection formula yields an isomorphism $f_\ast f^\ast M\cong f_\ast\left(\underline{\mathbb{A}^n\setminus Z}\right)\otimes_{\op{Spec}k} M$, hence we are reduced to show that 
\[
\op{H}^j_{\op{Nis}}(S;M)\to
\op{H}^j_{\op{Nis}}(S;\left(\underline{\mathbb{A}^n\setminus Z}\right)\otimes_{\op{Spec}k}M)\cong
\op{H}^j_{\op{Nis}}(S\times\mathbb{A}^n\setminus Z;M) 
\]
is an isomorphism for $0\leq j\leq i-2$ and all smooth $k$-schemes $S$. 

This is the point where we can use the excision result. If $M\in\mathbf{DA}^{\et}(\op{Spec}k;\Lambda)$ is an \'etale motive in the heart of the homotopy t-structure, then it can be represented by a homotopy module as discussed above. By an argument similar to \cite[Corollary 4.7]{asok:doran}, using the Gersten resolution to compute cohomology of homotopy modules, it follows that the inclusion $S\times(\mathbb{A}^n\setminus Z)\hookrightarrow S\times \mathbb{A}^n$ induces an isomorphism on Nisnevich cohomology
\[
\op{H}^j_{\op{Nis}}(S,M)\cong \op{H}^j_{\op{Nis}}(S\times\mathbb{A}^n, M)\stackrel{\cong}{\longrightarrow} \op{H}^j_{\op{Nis}}(S\times(\mathbb{A}^n\setminus Z),M).
\]
for $0\leq j\leq i-2$ and any smooth $k$-scheme $S$. 

To prove stability under pullback, we can restrict to pullback along smooth schemes $X\to \op{Spec}k$, by cdh-descent for rational motives. In particular, we can again compute the relevant cohomology in terms of Gersten complexes, and the base-change formulas imply that this argument is stable under pullback. Alternatively, we can use the fact that isomorphisms in the Nisnevich cohomology of strictly $\mathbb{A}^1$-invariant sheaves of abelian groups can be detected over fields, reducing everything to the field case discussed above. Hence $f$ is $(i-2)$-acyclic for the homotopy t-structure on $\mathbf{DA}^{\et}(-;\Lambda)$.
\end{proof}

This allows to easily construct resolutions which are highly acyclic for the homotopy t-structure on $\mathbf{DA}^{\et}(-;\Lambda)$ as well as its localizations such as $\op{MDer}(-;\mathbb{C})$. From now on, acyclicity will always be understood to be with respect to the homotopy t-structure on one of these homotopical stable algebraic derivators.

\subsection{Applications to resolutions of $G$-spaces}

We can now apply the basic theory of $n$-acyclic maps developed above to finite-dimensional approximations of classifying spaces resp. the Borel construction. This mainly follows \cite{totaro} and \cite{edidin:graham}. 

Let $k$ be a field, and let $G$ be a linear algebraic group over $k$. We start with the definition of resolutions as in \cite[2.1.1 and 2.1.2]{BeLu}. 

\begin{proposition}
\label{prop:belu211}
Let $\nu:P\to X$ be a morphism of varieties with $G$-action, and assume that $X$ is $G$-free. Then $P\cong X\times_{G\backslash X}(G\backslash P)$ and $P$ is $G$-free. 
\end{proposition}

\begin{proof}
The argument is the same as in \cite[2.3.1]{BeLu}. 
\end{proof}

\begin{definition}
  Let $G\looparrowright X$ be a variety with action. A \emph{resolution of $X$}\index{resolution} is a $G$-equivariant morphism $p:P\to X$ where $P$ is a variety with free $G$-action. Given two resolutions of $X$, a \emph{morphism of resolutions} is a morphism of $G$-varieties  over $X$.  The category of all resolutions of $X$  will be denoted by $\op{Res}(G\looparrowright X)$.
  A resolution is said to be \emph{$n$-acyclic}\index{resolution!acyclic} if $p:P\to X$ is $n$-acyclic.
\end{definition}

\begin{remark}
We hope that the notation $\op{Res}(G\looparrowright X)$ is not too close to the restriction functors $\op{Res}_G^H$. 
\end{remark}

\begin{remark}
Note that our definition \ref{free} of free $G$-action implies that for a resolution $p:P\to X$, the quotient $G\backslash P$ exists as a quasi-projective variety and the quotient map $P\to G\backslash P$ is a $G$-torsor which is locally trivial in the \'etale topology. 
\end{remark}

\begin{remark}
\begin{enumerate}
\item As usual, one has the trivial resolution $\op{pr}_X:G\times X\to X$ with diagonal $G$-action. 
\item The category of resolutions of $G\looparrowright X$ admits finite products, which are fiber products of varieties over $X$.
\item If $f:X\to Y$ is a morphism of $G$-varieties and $p:P\to X$ is a resolution of $ X$, then $f_\circ(p):=f\circ p$ is a resolution of $ Y$. Conversely, if $q:Q\to Y$ is resolution of $ Y$, then $f^\circ (q):=\op{pr}_2:Q\times_YX\to X$ is a resolution of $ X$. These are adjoint functors between the categories of resolutions
\[
f_\circ :\op{Res}(G\looparrowright X)\leftrightarrows \op{Res}(G\looparrowright Y):f^\circ 
\]
\end{enumerate}
\end{remark}

\begin{lemma}\label{lem:resexist}
For every variety with action $(G\looparrowright X)$ and every $n\geq 0$, there is a resolution of $X$ which is $\tau$-$n$-acyclic with respect to the homotopy t-structure $\tau$ on the derivator $\mathbf{DA}^{\et}(-;\Lambda)$.
\end{lemma}

\begin{proof}
In \cite[Remark 1.4]{totaro} Totaro constructs 
for each $s\geq 0$ a representation $V$ of $G$ and a $G$-invariant closed subset $Z\subseteq V$ of  codimension $\geq s$ such that for every variety
$X$ the quotient $(X\times (V\setminus Z))/G$ of the diagonal $G$-action on $X\times(V\setminus Z)$  exists as a quasi-projective variety. We only need to show that the quotient morphism is a $G$-torsor which is locally trivial in the \'etale topology. This follows from \cite[Example 2.1.1.4 ii), Remark 2.1.1.6.i) and Remark 2.1.1.19]{schmitt}. Thus we can apply Proposition~\ref{prop:excision} to see that the projection map $X\times (V\setminus Z)\to X$ is $(s-2)$-acyclic. 
\end{proof}

Moreover, the double fibration construction, cf. \cite[proof of Theorem 1.1]{totaro}, allows to compare different such representations: 

\begin{lemma}
Let $V$ and $W$ be two representations such that $G$ acts freely outside subsets $Z_V$ and $Z_W$ of codimension $\geq s$, respectively. Then the two maps 
\[
(V\setminus Z_V)\times (W\setminus Z_W)\hookrightarrow (V\setminus Z_V)\times W\to (V\setminus Z_V) \quad\textrm{ and}
\]
\[
(V\setminus Z_V)\times (W\setminus Z_W)\hookrightarrow V\times (W\setminus Z_W)\to (W\setminus Z_W)
\]
are both $(s-2)$-acyclic.  
\end{lemma}


There are further special resolutions which will be used later for the definition of the six functors: 

\begin{definition}
\label{def:smoothres}
A resolution $p:P\to X$ of a  $G$-variety is called \emph{smooth}\index{resolution!smooth} if the morphism $p$ is smooth. The category of all smooth resolutions of $(G\looparrowright X)$ with smooth morphisms is denoted by $\op{SmRes}(G\looparrowright X)$.
\end{definition}

\begin{remark}
\begin{enumerate}
\item The resolutions provided by Lemma~\ref{lem:resexist} above are smooth.
\item For a morphism $f:(G\looparrowright X)\to (G\looparrowright Y)$, the functor $f^\circ $ preserves $n$-acyclic smooth resolutions. The functor $f_\circ $ does not generally preserve $n$-acyclic smooth resolutions unless $f$ itself is $n$-acyclic and smooth.
\end{enumerate}
\end{remark}

Finally, we recall the definition of compatibility of resolutions with respect to morphisms of varieties with actions, cf.  \cite[Definitions 6.2, 6.3]{BeLu}. 

\begin{definition}
\label{def:comp1}
Let $(\phi,f):(G\looparrowright X)\to (H\looparrowright Y)$ be a morphism of varieties with action. Let $p:P\to X$ and $q:Q\to Y$ be resolutions. Call these resolutions \emph{compatible}\index{resolution!compatible} if there exists a map $\tilde{f}:(G\looparrowright P)\to (H\looparrowright Q)$ such that the following diagram is commutative:
\begin{center}
\begin{minipage}[c]{10cm}
\xymatrix{
(G\looparrowright P) \ar[r]^{\tilde{f}}\ar[d] & 
(H\looparrowright Q)\ar[d] \\
(G\looparrowright X)\ar[r]_f & (H\looparrowright Y).
}
\end{minipage}
\end{center}
\end{definition}

Let $(\phi,f):(G\looparrowright X)\to (H\looparrowright Y)$ be a morphism of varieties with action. Then there is a bifunctor 
\[
-\times_f-:\op{Res}(G\looparrowright X)\times\op{Res}(H\looparrowright Y)\to\op{Res}(G\looparrowright X): (S,R)\mapsto S\times_X f^\circ(R)
\]
which sends pairs of smooth $n$-acyclic resolutions to smooth $n$-acyclic resolutions. 

\begin{definition}
\label{def:comp2}
Let $(\phi,f):(G\looparrowright X)\to (H\looparrowright Y)$ be a morphism of varieties with action. A resolution $P\in\op{Res}(G\looparrowright X)$ is \emph{compatible with $(\phi,f)$} if there is a resolution $Q\in\op{Res}(H\looparrowright Y)$ and a morphism $\tilde{f}$ making $P$ and $Q$ compatible resolutions in the sense of Definition~\ref{def:comp1}. A morphism of resolutions $P_1\to P_2$ in $\op{Res}(G\looparrowright X)$ is called \emph{compatible with $(\phi,f)$} if it can be completed to a commutative square of compatible resolutions in the sense of Definition~\ref{def:comp1}. 
The category of  resolutions compatible with $(\phi,f)$ (with compatible morphisms) is denoted by 
\[
\op{CRes}((\phi,f))\subset \op{Res}(G\looparrowright X).
\]
\end{definition}

\begin{remark}
Let $(\phi,f):(G\looparrowright X)\to (H\looparrowright Y)$ be a morphism of varieties with action, and let $p:P\to X$ and $q:Q\to Y$ be resolutions, compatible via the morphism $(\phi,\tilde{f}):(G\looparrowright P)\to(H\looparrowright Q)$. As before, the quotients $P/G$ and $Q/H$ exist as quasi-projective varieties and the quotient morphisms $P\to P/G$ and $Q\to Q/H$ are locally trivial in the \'etale topology. Moreover, $\tilde{f}$ induces a morphism of quotients $\bar{f}:P/G\to Q/H$.
\end{remark}

\subsection{The simplicial Borel construction}

We recall the definitions for the Borel construction, cf. also Appendix~\ref{sec:bar} for background on two-sided bar constructions and their motives.

\begin{definition}
\label{def:borel}
\index{Borel construction, $({\op{E}}G\times_{/G}X)_\bullet$}
Let $G$ be a linear algebraic group over the field $k$. Then there is the simplicial variety ${\op{E}}G$, also denoted by  $[G{\sslash}G]$ in the stacks  literature, whose scheme of $n$-simplices is ${\op{E}}G_n=G^{\times(n+1)}$. As usual, the face maps are defined using projections and the multiplication map $\mu:G\times G\to G$; the degeneracy maps are defined using partial diagonals. 

Another way to obtain the simplicial variety ${\op{E}}G$ is to consider, for each $k$-scheme $S$, the category $\cE\cG(S)$ consisting of objects given by $S$-points of $G$, and with a unique isomorphism between each pair of objects. Set
\[ 
{\op{E}}G = \mbox{nerve of $\cE\cG$}.
\]
Then ${\op{E}}G$ is represented by a simplicial variety. As the category $\cE\cG(S)$ is equivalent to the terminal category (consisting of a unique object and a unique morphism), ${\op{E}}G$ is simplicially contractible.

Now suppose $G$ acts on a variety $X$. Then we obtain a simplicial scheme ${\op{E}}G\times_{/G} X$, also denoted by $[X{\sslash}G]$ in the stacks literature or $X_{{\op{h}}G}$ in the equivariant homotopy literature, whose scheme of $n$-simplices is $({\op{E}}G\times_{/G}X)_n=G^n\times X$. The face maps use, as above, projections, multiplication of the group and additionally the action map $a:G\times X\to X$. 

Again, an alternative construction starts with the category $\cX_{{\op{h}}G}(S)$ whose objects are the $S$-points of $X$ and morphisms consist of $x\stackrel{g}{\longrightarrow} y$ for each $g\in G(S)$ with $a(g,x)=y$. Set 
\[ 
{\op{E}}G\qtimes{G}X =  \mbox{nerve of $\cX_{{\op{h}}G}$}.
\]
This is also represented by a simplicial variety.
If $X = \pt$, then we denote this simplicial variety by ${\op{B}}G$, i.e.,
\[ {\op{B}}G = \mbox{nerve of $\pt_{{\op{h}}G}$}.\]
Note that $\pt_{{\op{h}}G}$ consists of a single object whose endomorphisms are given by $G$. 

The evident functor $\cE\cG \to \pt_{{\op{h}}G}$ yields a morphism ${\op{E}}G \to {\op{B}}G$ with fiber $G$. The evident functor $\cX_{{\op{h}}G} \to \pt_{{\op{h}}G}$, yields a morphism 
\[
{\op{E}}G\qtimes{G}X \to {\op{B}}G 
\]
with fiber $X$ (viewed as a constant simplicial variety).
\end{definition}

\begin{remark}
  If we were working over the complex numbers (and using the complex analytic topology), then the geometric realizations of ${\op{E}}G$ and ${\op{B}}G$, would yield the universal $G$-bundle: the geometric realization of ${\op{E}}G$ is clearly a contractible space endowed with a free action of $G$, with quotient the geometric realization of ${\op{B}}G$.
\end{remark}

With slight abuse of notation, we can view the projection morphism ${\op{E}}G\times X\to X$ as a resolution of $X$, by a simplicial scheme. 


\begin{proposition}
The natural projection $f: {\op{E}}G\times X\to X$ is an $\infty$-acyclic simplicial resolution for $\mathbf{DA}^{\et}(-;\Lambda)$ in the sense that the unit of the adjunction is an isotransformation $\op{id}\sira f_\ast f^\ast$, cf. also Remark~\ref{rem:infacyclic}. 
\end{proposition}

\begin{proof}
We can view $X$ as a constant simplicial variety. The definition of $f_\ast$ for motives over simplicial varieties is term-wise. Then we can interpret a motive over the constant simplicial variety $X$ as a simplicial motive over the variety $X$. This way, the motive $f_\ast f^\ast M$ is going to be the colimit of the simplicial motive whose $n$-th term is $(f_n)_\ast (f_n)^\ast M$ for $f_n:G^{n+1}\times X\to X$. As mentioned above, the simplicial $X$-scheme ${\op{E}}G\times X$ has an explicit simplicial contraction making it homotopy equivalent to $X$. This contraction, by functoriality, also exists on the simplicial motive $f_\ast f^\ast M$ showing that the natural map $M\to f_\ast f^\ast M$ is an isomorphism.  
\end{proof}

\section{Definitions of categories of equivariant motives}
\label{sec:equivdef}

In this section, we will give, for a variety with action $G\looparrowright X$, several definitions of categories $\mathbb{D}^{(\op{b},+)}_G(X)$ of $G$-equivariant $\mathbb{D}$-motives over $X$. Mainly, the two alternatives are to define motives via the simplicial Borel construction or to use finite-dimensional approximations of it. Both will turn out to be equivalent (essentially because they behave as expected under the quotient equivalence). Including a discussion of both constructions is not just done for the sake of completeness; indeed, both constructions are necessary to set up a suitably functorial six-functor formalism. 

\begin{Bemerkung}
As a word of warning for all the subsequent sections: the motives we consider will only be weakly equivariant (in the sense of computing Borel-style equivariant cohomology) and not strongly equivariant (in the sense of Bredon-style equivariant cohomology). While categories of strongly equivariant motives are also under current investigation, cf. \cite{hoyois}, they are much more difficult to construct and at present do not exist for arbitrary linear groups.
\end{Bemerkung}

\begin{convention}
\label{derivator:new}
From now on, $\mathbb{D}$ will be a homotopical stable algebraic derivator satisfying the conditions of \ref{derivator:conditions} which has an appropriate theory of acyclic maps, i.e., such that all the statements in Section~\ref{sec:resolutions} apply. This is, in particular, satisfied for localizations of the derivator $\mathbf{DA}^{\et}(-;\Lambda)$ for $\Lambda$ a field of characteristic zero. 
\end{convention}

\subsection{Motives over the simplicial Borel construction}
This construction, in which equivariant motives are defined as locally constant motives over the Borel construction, was suggested by Joseph Ayoub, cf. in particular his answer to the MathOverflow question 171503 ``Equivariant motivic sheaves''. It is a motivic version of the viewpoint taken in \cite[Section 2.7, Appendix B]{BeLu}. 

\begin{definition}
\label{def:equivmotivesborel}
\index{equivariant motives, $\mathbb{D}_G^{\Delta}(X)$}
Fix a homotopical stable algebraic derivator $\mathbb{D}$ satisfying the conditions of \ref{derivator:new}, and let $(G\looparrowright X)$ be a variety with action. Denote by $\Delta ^{\op{op}}$ the simplicial category, used as index category for simplicial varieties. Following Definition~\ref{def:daet}, we have the triangulated category $\mathbb{D}({\op{E}}G\times_{/G}X,\Delta ^{\op{op}})$ of motives on the Borel construction for $(G\looparrowright X)$. For an ordinal $[m]\in\Delta ^{\op{op}}$, denote by $\iota_m$
the corresponding embedding
of the one-point-diagram of varieties $({\op{E}}G\times_{/G}X)_m$ into the diagram
${\op{E}}G\times_{/G}X$.
For $M\in \mathbb{D}({\op{E}}G\times_{/G}X,\Delta ^{\op{op}})$ and $[m]\in\Delta ^{\op{op}}$, we define the motive $M\langle m\rangle\in \mathbb{D}(({\op{E}}G\times_{/G}X)_m)$ as $M\langle m\rangle=\iota_m^\ast  M$. With this notation, an object $M\in \mathbb{D}({\op{E}}G\times_{/G}X,\Delta ^{\op{op}})$ is cartesian in the sense of Definition~\ref{def:cartesian} if for each morphism $d:[m]\to [n]$ in $\Delta ^{\op{op}}$ the induced morphism $d^\ast M\langle m\rangle\to M\langle n\rangle$ is an isomorphism. 

Finally, we define the category 
\[
\mathbb{D}^{\Delta}_G(X)\pdef \mathbb{D}^{\op{cart}}({\op{E}}G\times_{/G} X,\Delta ^{\op{op}})
\]
of $G$-equivariant motives on $X$ as the full subcategory of cartesian motives in  $\mathbb{D}({\op{E}}G\times_{/G}X,\Delta ^{\op{op}})$. 
\end{definition}

The superscript $\Delta$ should signify that this is the category of equivariant motives constructed via the simplicial approach.

\begin{remark}
\label{rem:triangulated}
Note that the category $\mathbb{D}^{\Delta}_G(X)$ is in fact a triangulated subcategory of the category $\mathbb{D}({\op{E}}G\times_{/G}X,\Delta ^{\op{op}})$: if we have a triangle $A\to B\to C\to A[1]$ in the latter category such that any two of the objects belong to $\mathbb{D}^{\Delta}_G(X)$, then the five-lemma applied to the induced commutative diagram
\[
\xymatrix{
d^\ast A\langle m\rangle \ar[r]\ar[d] & d^\ast B\langle m\rangle \ar[r]\ar[d] & d^\ast C\langle m\rangle \ar[r]\ar[d] & d^\ast A\langle m\rangle[1] \ar[d]\\
A\langle n\rangle \ar[r] & B\langle n\rangle \ar[r] & C\langle n\rangle \ar[r] & A\langle n\rangle[1] 
}
\]
shows that the third object also belongs to $\mathbb{D}^{\Delta}_G(X)$. If the derivator $\mathbb{D}$ additionally has small sums, then the categories $\mathbb{D}^{\Delta}_G(X)$ will be idempotent complete.
\end{remark}

\begin{remark}
\label{rem:simpforget}
Let $G\looparrowright X$ be a variety with action. Viewing $X$ as constant simplicial scheme, there is a natural morphism $\gamma:X\to{\op{E}}G\times_{/G} X$ of simplicial schemes. This morphism induces a natural restriction functor, the forgetful functor 
\[
\op{For}=\gamma^\ast:\mathbb{D}^\Delta_G(X)\to\mathbb{D}(X). 
\]
\end{remark}


\begin{proposition}[{\bf Quotient equivalence, simplicial}]
\label{prop:freeaction}\index{quotient equivalence} 
Let $G\looparrowright X$ be a variety with a free action. Then the pullback functor along the projection $p\colon {\op{E}}G\times_{/G}X\to G\backslash X$ yields an equivalence 
\[ 
p^\ast\colon \mathbb{D}(G\backslash X)\sirra \mathbb{D}^{\Delta}_G(X).
\]
A quasi-inverse is given by the adjoint functor $p_\ast\cong p_{\sharp}$. 
\end{proposition}

\begin{proof}
We view $G\backslash X$ as a constant simplicial scheme, then the projection ${\op{E}}G\times_{/G} X\to G\backslash X$ is given in degree $n$ by the projection $G^n\times X\to G\backslash X$. Then a smooth scheme $U\to G^n\times X$ is also a smooth $G\backslash X$-scheme via composition, and this morphism of diagrams gives rise to a pullback functor $p^\ast:\mathbb{D}(G\backslash X)\to\mathbb{D}_G^{\Delta}(X)$. As in \cite[Section 4.5]{ayoub:thesis2}, this functor has two adjoints. These functors are obtained by pointwise application of the functors $(p_n)_\sharp$ and $(p_n)_\ast$ with $p_n:G^n\times X\to G\backslash X$ and then taking the realization of the corresponding simplicial object in $\mathbb{D}(G\backslash X)$. If we can show that $p^*$ is an equivalence, both these adjoints have to be isomorphic.

So it suffices to show that the unit $\id\to p_*p^*$ and counit $p^*p_*\to \id$ of the adjunction are isomorphisms. As the derivator $\mathbb{D}$ is assumed to be separated, pullbacks along surjective maps are conservative. Consequently, by passing to a trivializing \'etale cover of $G\backslash X$, we may assume $X = G\times U$ with $G$ acting by multiplication on the first factor. In this case, $G\backslash X \simeq U$, ${{\op{E}}}G\qtimes{G} X \simeq {\op{E}}G\times U$, and the quotient map $p\colon {\op{E}}G \times U \to U$ is simply the projection map.

As ${\op{E}}G$ is simplicially contractible, this yields that the unit map $\id \to p_*p^*$ is an isomorphism. To see that the counit is also an isomorphism, note that if $M$ is a cartesian motive on ${\op{E}}G\times U$, then looking at the $0$-th degree simplicial piece of ${\op{E}}G\times U$, we see that $M$ must be the pullback of some motive on $U$. More precisely, this motive is given by the $0$-th degree piece of $M$. Thus, the contractibility of ${\op{E}}G$ once  again implies that the counit $p^*p_*\to \id$ is also an isomorphism.
\end{proof}
%
%

\subsection{Motives over individual resolutions}

Next, we show that equivariant categories of motives can alternatively be defined using finite-dimensional approximations to the Borel-construction. This is one of the approaches used in the construction of equivariant derived categories in \cite{BeLu}, and it is necessary for defining the six-functor formalism in full generality. 

\begin{definition}
\label{def:resoldiag}
\index{resolution!diagram $\mathscr{W}_P$}
Let $(G\looparrowright X)$ be a variety with action, and let $p:P\to X$ be a resolution of the $G$-action. We will denote by $\mathscr{W}_P$ the diagram 
\[
X\stackrel{p}{\longleftarrow} P\stackrel{q}{\longrightarrow} G\backslash P
\]
associated to the resolution $P$.
\end{definition}

Let $\mathbb{D}$ be a  homotopical stable algebraic derivator satisfying the conditions of \ref{derivator:new}. The following definition follows \cite[Definition 2.1.3]{BeLu}.  

\begin{definition}
\label{def:equivmotres1}
\index{equivariant motives!$\mathbb{D}_G(X,P)$}
Let $G\looparrowright X$ be a variety with action, let $p:P\to X$ be a resolution of $X$ and let $\mathscr{W}_P:X\stackrel{p}{\leftarrow} P\xrightarrow{q} G\backslash P$ be the corresponding diagram. Define the category $\mathbb{D}_G(X,P)$ as follows:  
\begin{enumerate}
\item an object $M$ of $\mathbb{D}_G(X,P)$ is a triple $M=(M_X,\overline{M},\beta)$ where $M_X\in\mathbb{D}(X)$ is a motive over $X$, $\overline{M}\in \mathbb{D}(G\backslash P)$ is a motive over the quotient $G\backslash P$, and $\beta:p^\ast M_X\cong q^\ast\overline{M}$ is an isomorphism in $\mathbb{D}(P)$. 
\item a morphism $\alpha:M\to N$ of $\mathbb{D}_G(X,P)$ is a pair $\alpha=(\alpha_X,\overline{\alpha})$ with $\alpha_X:M_X\to N_X$ and $\overline{\alpha}:\overline{M}\to \overline{N}$ such that $\beta\circ p^\ast(\alpha_X)=q^\ast(\overline{\alpha})\circ \beta$. 
\end{enumerate}
\end{definition}

\begin{Bemerkung}
  There is a natural forgetful functor
\[
\op{For}:\mathbb{D}_G(X,P)\to
\mathbb{D}(X):(M_X,\overline{M},\beta)\mapsto M_X. 
\]
For a $G$-equivariant morphism $f:X\to Y$, resolutions $p:P\to X$ and $r:R\to Y$ and a morphism $\nu:P\to R$ such that $r\circ \nu=f\circ p$, there is an inverse image functor
\[
f^\ast:\mathbb{D}_G(Y,R)\to\mathbb{D}_G(X,P):(F_Y,\overline{F},\beta)\to
(f^\ast F_Y,\overline{\nu}^\ast\overline{F},\gamma)
\]
where $\overline{\nu}:G\backslash P\to G\backslash R$ is the morphism induced on the quotients and 
\[
\gamma=\nu^\ast\beta:p^\ast
f^\ast F_Y\cong \nu^\ast r^\ast F_Y\to \nu^\ast q^\ast F\cong q^\ast
\overline{\nu}^\ast \overline{F}.
\]
\end{Bemerkung}

Actually, the definition above is only included for completeness and comparison and will not be used in a substantial way. Via the theory of $n$-acyclic maps, we can replace the above category of compatible motives by a category of cartesian motives over a resolution diagram: 

\begin{lemma}
\label{lem:resolindep}
Let $I=[a,b]\subseteq \mathbb{Z}$ be an interval and let $n\geq b-a$. Let $\mathbb{D}$ be a homotopical stable algebraic derivator satisfying the conditions of \ref{derivator:new}, let $(G\looparrowright X)$ be a variety with action and let $p:P\to X$ be an $n$-acyclic resolution. For the diagram $\mathscr{W}_P$ of Definition~\ref{def:resoldiag}, 
the obvious functor is an equivalence
\[
\mathbb{D}^{\op{cart},I}(\mathscr W_P)\sirra \mathbb{D}^I_G(X,P).
\]
\end{lemma}
\begin{proof}
  Denote by $\mathbb{D}(\mathscr{W})$ the category of motives over this diagram. There are natural restriction functors $\op{ev}_X:\mathbb{D}(\mathscr{W})\to\mathbb{D}(X)$ and $\op{ev}_{G\backslash P}: \mathbb{D}(\mathscr{W})\to\mathbb{D}(G\backslash P)$ and we can write down a natural functor 
\[
\mathbb{D}^{\op{cart},I}(\mathscr{W})\to\mathbb{D}_G^I(X,P):
M\mapsto (\op{ev}_X(M),\op{ev}_{G\backslash P}(M),\beta)
\]
on the $I$-truncated categories. Here $\beta$ is the isomorphism between $p^*\op{ev}_X(M)$ and $q^*\op{ev}_{G\backslash P}(M)$
which comes from the identification of both sides with
$\op{ev}_P(M)$ due to
 the requirement that $M$ is cartesian.
Note that the restriction functor $\mathbb{D}^I(X)\to\mathbb{D}^I(P)$ is
fully faithful by Proposition~\ref{prop:belu192} and therefore the restriction $\mathbb{D}^I_G(X,P)\to \mathbb{D}(G\backslash P)$ is also fully faithful. The composition of $\mathbb{D}^I(\mathscr{W})\to \mathbb{D}^I_G(X,P)\to \mathbb{D}^I(G\backslash P)$ of these functors is simply the restriction functor. As such, it admits a right adjoint which is given by $\tau_{\leq b}$-truncating the right adjoint to $\mathbb{D}(\mathscr{W})\to\mathbb{D}(G\backslash P)$. The latter right adjoint is given in terms of a right Kan extension, i.e., its value at any term of the diagram is given as homotopy limit over the values at the terms in the overcategory. It follows that, as in \cite[2.4.3]{BeLu}, the adjoint applied to a motive $M\in\mathbb{D}(G\backslash P)$ is given on $P$ by $q^\ast (M)$ and on $X$ by $p_\ast q^\ast (M)$. The composition
\[
\mathbb{D}^I_G(X,P)\hookrightarrow \mathbb{D}^I(G\backslash P)\to \mathbb{D}^{I}(\mathscr{W})
\]
lands in the cartesian motives by Proposition~\ref{prop:belu192}. Then Proposition~\ref{prop:belu192} also implies that the two functors relating $\mathbb{D}^{\op{cart},I}(\mathscr{W})$ and $\mathbb{D}^I_G(X,P)$ are mutually inverse equivalences of triangulated categories. 
\end{proof}

The following is an analogue of \cite[Proposition 2.2.1]{BeLu}, with a proof analogous to \cite[2.3.3]{BeLu}.  

\begin{proposition}
\label{prop:belu221}
Fix an interval $I=[a,b]$. Let $p:P\to X$ be an $n$-acyclic resolution, with $n\geq b-a$. If $X$ is a free $G$-variety, then the quotient functor is an equivalence of categories
\[
q^\ast:\mathbb{D}^I(G\backslash X)\to \mathbb{D}^{\op{cart},I}(\mathscr{W}_P).
\]
\end{proposition}

\begin{proof}
By Proposition~\ref{prop:belu211}, $P\cong X\times_{G\backslash X}(G\backslash P)$, and by assumption $X\to G\backslash X$ is locally trivial in the \'etale topology. Then Proposition~\ref{prop:belu193} implies that $\overline{p}:G\backslash P\to G\backslash X$ is $n$-acyclic and therefore the restriction functor $\mathbb{D}^I(G\backslash X)\to\mathbb{D}^I(G\backslash P)$ is fully faithful. The essential image $\mathbb{D}^I(G\backslash P\mid G\backslash X)$ can alternatively be described, using Proposition~\ref{prop:belu193}, as the full subcategory of $\mathbb{D}^I(G\backslash P)$ of motives $M$ such that the restriction $q^\ast:\mathbb{D}^I(G\backslash P)\to\mathbb{D}^I(P)$ lands inside the essential image of the restriction $p^\ast:\mathbb{D}(X)\to\mathbb{D}(P)$. The conclusion that the restriction functor induces an equivalence  $\mathbb{D}^{\op{cart},I}(\mathscr{W}_P)\to\mathbb{D}^I(G\backslash P\mid G\backslash X)$ follows as in the proof of Lemma~\ref{lem:resolindep}. 
\end{proof}


\begin{corollary}
Fix an interval $I=[a,b]$. Let $p:P\to X$ be an $n$-acyclic resolution, with  $n\geq b-a$. Let $r:R\to X$ be another resolution of $X$, and denote by $s:S=P\times_XR\to R$ the natural projection. Then the pullback 
\[
s^\ast:\mathbb{D}^{\op{cart},I}(\mathscr{W}_R)\to\mathbb{D}^{\op{cart},I}(\mathscr{W}_S)
\]
is an equivalence of categories. 
\end{corollary}

\begin{proof}
Having established Proposition~\ref{prop:belu221}, the proof is the same as \cite[2.3.4]{BeLu}.
\end{proof}

Using that the categories $\mathbb{D}^{\op{cart},I}(\mathscr{W}_R)$ are independent of the choice of $n$-acyclic resolution when $n\geq b-a$, we can take a 2-limit to define the equivariant motivic sheaves via resolutions:

\begin{definition}
Choosing a cofinal filtered category of resolutions, we can now define 
\[
\mathbb{D}^{\op{res},\op{b}}_G(X):=
\operatornamewithlimits{2-lim}_{P,I}\mathbb{D}^{\op{cart}, I}(\mathscr{W}_P).
\]

A diagram $M_1\to M_2\to M_3\to M_1[1]$ in $\mathbb{D}^{\op{res},\op{b}}_G(X)$ is a distinguished triangle if, for any sufficiently acyclic resolution $p:P\to X$, the application of the restriction functor 
\[
\mathbb{D}^{\op{res},\op{b}}_G(X)\to\mathbb{D}^{\op{cart},I}(\mathscr{W}_P)\to\mathbb{D}(G\backslash P)
\]
yields a distinguished triangle $M_1(P)\to M_2(P)\to M_3(P)\to M_1(P)[1]$ in $\mathbb{D}(G\backslash P)$. 

Define subcategories $\mathbb{D}^{\op{res},\leq a}_G(X)$ and $\mathbb{D}^{\op{res},\geq a}_G(X)$ by similarly requiring that an object $M$ lies in the subcategory if its restriction to $\mathbb{D}(G\backslash P)$ for each surjective resolution $p:P\to X$ lies in $\mathbb{D}^{\leq a}(G\backslash P)$ and $\mathbb{D}^{\geq a}(G\backslash P)$, respectively.
\end{definition}

The following is then an immediate consequence of the results established above and the standard statements on 2-limits of triangulated categories. 

\begin{proposition}
\begin{enumerate}
\item
Up to equivalence, the category is independent of the choice of a filtered subcategory of resolutions of growing acyclicity. 
\item 
With the above data, $\mathbb{D}^{\op{res},\op{b}}_G(X)$ becomes a triangulated category. 
\item
The forgetful and inverse image functors are exact for this triangulated structure. 
\end{enumerate}
\end{proposition}

Again, the viewpoint of taking fibers of a fibered triangulated category is included here for completeness and comparison with \cite{BeLu}. The better viewpoint of considering motives over diagrams of resolutions will be discussed in the next section. 

\subsection{Motives over categories of resolutions}

Now we want to give a motivic version of the definitions in \cite[Section 2.4]{BeLu}, which describes the equivariant derived category in terms of fibered categories. However, we prefer a slightly modified approach: instead of having derived categories fibered over schemes, we use Ayoub's definition of motives over a diagram. In other words, we now define a version of the equivariant derived category, denoted by $\mathbb{D}^{\op{Res}}_G(X)$, which consists of cartesian motives over the category of resolutions of the variety with action $(G\looparrowright X)$. This will have the distinct advantage that a six-functor formalism is available for the full category of
motives over  diagrams of resolutions, and all that is required for a definition of the six functors will be to check that they preserve cartesian motives under suitable conditions. 

\begin{remark}
Note that, as usual, the definition of motives over diagrams of schemes requires the diagram category to be small. The category of resolutions  $\op{Res}(G\looparrowright X)$ is not small. This is one reason for introducing an additional functor in the definition below. The other reason is that a great flexibility in the choice of diagram categories is required for setting up the six functor formalism later on.
\end{remark}

\begin{definition}
\label{def:equivmotres2}
\index{equivariant motives, $\mathbb{D}_G^{\op{Res}}(X,(\mathscr{F},\mathcal{I}))$}
Let $\mathbb{D}$ be a homotopical stable algebraic derivator satisfying the conditions in \ref{derivator:new}, and let $(G\looparrowright X)$ be a variety with action. Denote by 
\[
\mathscr{Q}_X: \op{Res}(G\looparrowright X)\to\op{Var}/k:(p:P\to X)\mapsto G\backslash P
\]
the functor from resolutions of $(G\looparrowright X)$ to $k$-varieties which associates to a resolution the quotient of the total space modulo the $G$-action. 

Let $\mathscr{F}:\mathcal{J}\to\op{Res}(G\looparrowright X)$ be a diagram of resolutions indexed by the small category $\mathcal{J}$. There is an associated diagram of varieties $\mathscr{Q}_X\circ\mathscr{F}:\mathcal{J}\to\op{Var}/k$. 
Following Definition~\ref{def:daet}, we have the triangulated category $\mathbb{D}(\mathscr{Q}_X\circ\mathscr{F},\mathcal{J})$ of motives over the given diagram. For ease of notation, we will henceforth drop the index category from the notation, writing $\mathbb{D}(\mathscr{Q}_X\circ\mathscr{F})$ or even just $\mathbb{D}(\bar{\mathscr{F}})$. There is also the triangulated subcategory $\mathbb{D}^{\op{cart}}(\mathscr{Q}_X\circ\mathscr{F})$ of cartesian motives over the diagram $(\mathscr{Q}_X\circ \mathscr{F},\mathcal{J})$.
\end{definition}

As in Remark~\ref{rem:triangulated}, the category $\mathbb{D}^{\op{cart}}(\mathscr{Q}_X\circ\mathscr{F})$ of cartesian motives is an idempotent complete triangulated subcategory of $\mathbb{D}(\mathscr{Q}_X\circ\mathscr{F},\mathcal{J})$. 

\begin{remark}
In the end, the category $\mathbb{D}^{\op{Res}}_G(X)$ of $G$-equivariant motives over $X$ will be defined as $\mathbb{D}^{\op{cart}}(\mathscr{Q}_X\circ \mathscr{F})$ where $\mathscr{F}:\mathcal{J}\to \op{Res}(G\looparrowright X)$ is a choice of skeleton of the category of resolutions, cf. Definition~\ref{def:equivmotres}. However, we need to establish a couple of comparison statements before we can prove that this is well-defined.
\end{remark}

The following general derivator statement will imply the appropriate version of the quotient equivalence, analogous to \cite[Proposition 2.2.5]{BeLu}. 

\begin{proposition}
\label{prop:derifinal}
Let $k$ be a field, and let $\mathbb{D}$ be a homotopical stable algebraic derivator on $k$-varieties satisfying the conditions of \ref{derivator:new}. Let $\mathcal{I}$ be a small category with a terminal object $i\in\mathcal{I}$ and let $\mathscr{F}:\mathcal{I}\to \op{Var}/k$ be a diagram. Then the restriction functor
\[
\iota^\ast:\mathbb{D}^{\op{cart}}(\mathscr{F})\to \mathbb{D}(\mathscr{F}(i))
\]
is an equivalence of triangulated categories. 
\end{proposition}

\begin{proof}
This restriction functor has two adjoints, given by the left and right Kan extensions, respectively. However, the colimits resp. limits defining the Kan extension have the category with one object as index category. Therefore, the value of the right adjoint $i_\ast$ over any object of $\mathcal{I}$ is obtained by pull-back from the value over the terminal object. This implies the claim. 
\end{proof}

\begin{corollary}[{\bf Quotient equivalence, resolutions}]
\label{cor:quoteqres}
\index{quotient equivalence}
Let $\mathbb{D}$ be a homotopical stable algebraic derivator satisfying the conditions of \ref{derivator:new}, and let $(G\looparrowright X)$ be a variety with free action. Let $\mathscr{F}:\mathcal{J}\to\op{Res}(G\looparrowright X)$ be a choice of skeleton for the category of resolutions of $X$ containing the trivial resolution
$\op{id}:X\sira X$.
Consider the following two morphisms of diagrams:
\begin{enumerate}
\item the morphism $q:(\mathscr{Q}_X\circ\mathscr{F},\mathcal{J})\to G\backslash X$ associating to a resolution $p:P\to X$ the morphism $\bar{p}:G\backslash P\to G\backslash X$, and
\item the morphism $i:\{G\backslash X\}\to(\mathscr{Q}_X\circ \mathscr{F},\mathcal{J})$ given by inclusion of the object $X$ in the diagram  $\mathscr{Q}_X\circ\mathscr{F}$.
\end{enumerate}
Then the corresponding restriction functors 
\[
q^\ast:\mathbb{D}(G\backslash X)\to\mathbb{D}^{\op{cart}}(\mathscr{Q}_X\circ\mathscr{F}) \textrm{ and } i^\ast:\mathbb{D}^{\op{cart}}(\mathscr{Q}_X\circ \mathscr{F})\to\mathbb{D}(G\backslash X)
\]
are inverse equivalences of triangulated categories. 
\end{corollary}

\begin{proof}
By assumption, $\id:X\to X$ is a resolution of $X$. In particular, the category $\op{Res}(G\looparrowright X)$ has $X$ as a final object.  The claim then follows from Proposition~\ref{prop:derifinal}. 
\end{proof}

\subsection{Extension to stable derivators}
\begin{Bemerkung}
We can extend the definitions of the equivariant derived categories as follows. Fix a variety with action $G\looparrowright X$ and consider the diagram $({\op{E}}G\times_{/G}X,\Delta^{\op{op}})$. For an other small index category $\mathcal{J}$ we get another diagram 
\[
({\op{E}}G\times_{/G}X\circ\op{pr}_1,\Delta^{\op{op}}\times\mathcal{J}).
\]
Then we can define the category 
\[
\mathbb{D}_G^{\Delta,+}(X,\mathcal{J}) := \mathbb{D}^{+,\op{cart}_1}({\op{E}}G\times_{/G}X\circ\op{pr}_1, \Delta^{\op{op}}\times\mathcal{J})
\]
where the superscript $\op{cart}_1$ means that we only enforce the cartesian isomorphism requirement for morphisms in the first index category. A similar definition can be done for the resolution definition of equivariant motives. 
\end{Bemerkung}

\begin{proposition}
\label{prop:equivderivator}
Let $G\looparrowright X$ be a variety with action. Then the 2-functors 
\[
\mathcal{J}\mapsto \mathbb{D}_G^{\Delta,+}(X,\mathcal{J}) \textrm{ and } \mathcal{J}\mapsto \mathbb{D}_G^{\op{Res},+}(X,\mathcal{J})
\]
are stable derivators in the sense of Definition~\ref{def:derivator}. 
\end{proposition}

\begin{proof}
This is basically a consequence of \cite[Remark 2.4.14]{ayoub:thesis1} which allows to replace axiom DerAlg 4 by the requirement that for any diagram of varieties $(\mathscr{F},\mathcal{I})$ the 2-functor $\mathbb{D}_{(\mathscr{F},\mathcal{I})}(-):\mathcal{J}\mapsto \mathbb{D}(\mathscr{F}\circ\op{pr}_1,\mathcal{I}\times\mathcal{J})$ is a triangulated derivator. 
\end{proof}

\section{Comparison results}
\label{sec:comparison}

In this section, we will compare the various versions of categories of equivariant motives constructed in the previous section. We first construct the comparison functors and then prove they are equivalences. In the subsequent sections we will use all these versions to build up the equivariant six functor formalism.

\subsection{Comparison functors} 


\begin{Bemerkung}[\textbf{From resolutions to Borel construction}]
\label{resolution2Borel}
Let $(G\looparrowright X)$ be a variety with action. The simplicial Borel resolution, considered as a functor ${\op{E}}G\times X:\Delta ^{\op{op}}\to \op{Sch}/X$, factors through a functor ${\op{E}}G\times X:\Delta ^{\op{op}}\to\op{Res}(G\looparrowright X)$.
Composition with the functor 
\[
\mathscr{Q}_X:\op{Res}(G\looparrowright X)\to\op{Var}/k:(P\ra X)\mapsto G\backslash P
\]
of Definition~\ref{def:equivmotres2} yields the simplicial Borel construction ${\op{E}}G\times_{/G} X$. For any choice of diagram $\mathscr{F}:\mathcal{I}\to\op{Res}(G\looparrowright X)$ such that there is a lift $L:\Delta ^{\op{op}}\to\mathcal{I}$ of the simplicial Borel resolution, we get a morphism of diagrams of varieties $\delta:({\op{E}}G\times_{/G} X,\Delta^{\op{op}})\to (\mathscr{Q}_X\circ\mathscr{F},\mathcal{I})$.
Associated to this morphism of diagrams is a natural restriction functor
\[
\delta^\ast:\mathbb{D}^{\op{cart}}(\mathscr{Q}_X\circ\mathscr{F})\to \mathbb{D}^{\Delta}_G(X).
\]
\end{Bemerkung}


\begin{Bemerkung}
\label{res2individual}
  Let $(G\looparrowright X)$ be a variety with action. For any resolution $p:P\to X$, there is a diagram of resolutions
\[
G\times X\stackrel{\id\times p}{\longleftarrow} G\times P\stackrel{\op{pr}_2}{\longrightarrow} P
\]
Assume that $\mathscr{F}:\mathcal{I}\to \op{Res}(G\looparrowright X)$ is a diagram of resolutions containing a lift $\tilde{\mathscr{W}_P}$ of this diagram. Then composition with the functor 
\[
\mathscr{Q}_X:\op{Res}(G\looparrowright X)\to\op{Var}/k:(Q\ra X)\mapsto G\backslash Q
\]
of Definition~\ref{def:equivmotres2} maps this lift to the diagram $\mathscr{W}_P:X\leftarrow P\to G\backslash P$ of Definition~\ref{def:resoldiag}. The corresponding morphism of diagrams of varieties has an associated natural restriction functor
\[
\rho^\ast:\mathbb{D}^{\op{cart}}(\mathscr{Q}_X\circ\mathscr{F})\to \mathbb{D}^{\op{cart}}(\mathscr{W}_P).
\]
\end{Bemerkung}

If we assume that $\mathscr{F}:\mathcal{I}\to\op{Res}(G\looparrowright X)$ is a diagram which contains the trivial resolution $\op{pr}_2:G\times X\to X$ and is stable under direct products with $G$, then the above procedure can be applied to an arbitrary resolution, because $\mathscr{F}$ will admit lifts for all diagrams $G\times X\leftarrow G\times P\rightarrow P$. 

Finally, we can also compare the motives over an individual resolution with the motives over the simplicial Borel construction. 

\begin{Bemerkung}
Let $G\looparrowright X$ be a variety with action, and let $p:P\to X$ be a resolution of $X$. The resolution provides a morphism of simplicial schemes $p:{\op{E}}G\times_{/G} P\to {\op{E}}G\times_{/G} X$, which induces a natural restriction functor
\[
p^\ast:\mathbb{D}_G^\Delta(X)\to\mathbb{D}_G^\Delta(P).
\]
On the other hand, the projection $\pi:P\to G\backslash P$ provides an augmentation morphism of simplicial schemes $\pi:{\op{E}}G\times_{/G} P\to G\backslash P$. This induces a natural restriction functor 
\[
\pi^\ast:\mathbb{D}^{\op{cart}}(\mathscr{W}_P)\to \mathbb{D}(G\backslash P)\to  \mathbb{D}_G^{\Delta}(P).
\]
\end{Bemerkung}

\begin{remark}
Note that the functors considered above have left and right adjoints. The functors in Proposition~\ref{prop:restrictcompat} can be restricted to truncated categories $\mathbb{D}^I$. This will be relevant in establishing comparison equivalences. 
\end{remark}

\begin{remark}
\label{rem:compat}
Since the different comparison functors are given by restriction functors associated to morphisms of diagrams of schemes, compatibility with other functors follows from commutativity of the corresponding morphisms of diagrams and base-change or exchange formulas. This will be relevant for the compatibility of the comparison functors with the six-functor formalism later on.
\end{remark}


\begin{proposition}[\textbf{Compatibility of comparison functors}]
\label{prop:restrictcompat}
Let $G\looparrowright X$ be a variety with action, and let  $\mathscr{F}:\mathcal{I}\to\op{Res}(G\looparrowright X)$ be a diagram of resolutions.
\begin{enumerate}
\item  For a resolution $p:P\to X$, assume that we are
  given lifts of the diagram $G\times X\leftarrow G\times P\rightarrow P$ associated to the resolution $p:P\to X$ and of the Borel construction ${\op{E}}G\times P$ along $\mathscr F$. Then the following diagram is commutative up to an isotransformation:
\[
\xymatrix{
\mathbb{D}^{\op{cart}}(\mathscr{W}_P) \ar[d] & \mathbb{D}^{\op{cart}}(\mathscr{Q}_X\circ\mathscr{F}) \ar[l]_{\rho^\ast} \ar[d]^{\delta_P^\ast} \\
\mathbb{D}(G\backslash P)\ar[r] & \mathbb{D}^{\Delta}_G(P) .
}
\]
Here the top horizontal morphism is the restriction of \ref{res2individual} associated to the lift of the diagram $G\times X\leftarrow G\times P\rightarrow P$, and the left vertical morphism is the further restriction to $G\backslash P$. The right vertical morphism is the restriction functor from \ref{resolution2Borel}, and the bottom horizontal map is the  simplicial quotient equivalence from Proposition~\ref{prop:freeaction}.
\item For a resolution $p:P\to X$, assume that there is a lift of the morphism ${\op{E}}G\times P\to{\op{E}}G\times X$ of Borel constructions associated to $p$ along $\mathscr F$. Then the diagram 
\[
\xymatrix{
\mathbb{D}^{\op{cart}}(\mathscr{Q}_X\circ\mathscr{F}) \ar[d]_{\delta_P^\ast} \ar[rd]^{\delta_X^\ast} \\
\mathbb{D}^\Delta_G(P) & \mathbb{D}^\Delta_G(X) \ar[l]^{p^\ast}
}
\]
is commutative up to isotransformation. 
\end{enumerate}
\end{proposition}

\begin{proof}
(1) The functors exist because the conditions for \ref{resolution2Borel} and \ref{res2individual} are satisfied by assumption. The composition of top horizontal and left vertical is then simply the evaluation of a motive in $\mathbb{D}^{\op{cart}}(\mathscr{Q}_X\circ\mathscr{F})$ on $G\backslash P$. 
It suffices to show that for a motive $M\in\mathbb{D}^{\op{cart}}(\mathscr{Q}_X\circ\mathscr{F})$, the pullback of the evaluation of $M$ at $G\backslash P$ to $\mathbb{D}^\Delta_G(P)$ is isomorphic to the restriction  $\delta_P^\ast M$ to $\mathbb{D}^\Delta_G(P)$. The fact that $M$ is cartesian provides an explicit such isomorphism.  

(2) The commutativity up to isotransformation follows again since we start with a cartesian motive in $\mathbb{D}^{\op{cart}}(\mathscr{Q}_X\circ\mathscr{F})$. This implies that there is a specific isomorphism between $p^\ast\circ\delta_X^\ast M$ and $\delta_P^\ast M$. 
\end{proof}

\subsection{Comparison statements, equivalence proofs}

Now we want to establish comparison results. The main result is that the comparison functor 
\[
\delta_X^\ast:\mathbb{D}^{\op{cart}}(\mathscr{Q}_X\circ\mathscr{F})\to\mathbb{D}^{\Delta}_G(X)
\] 
is an equivalence under fairly weak restrictions on the diagram $(\mathscr{F},\mathcal{I})$. This will, in particular, imply that $\mathbb{D}^{\op{cart}}(\mathscr{Q}_X\circ\mathscr{F})$ is, up to equivalence, independent of the choice of diagram $(\mathscr{F},\mathcal{I})$ of resolutions in favourable circumstances. These equivalences will be established by showing that unit and counit of the comparison adjunction are equivalences. The latter can be checked on $I$-truncated categories of motives where it follows from the acyclicity results we proved earlier. This, however, does not provide a comparison for the full derived categories of motives, but only for the bounded-below part $\mathbb{D}^+$. This will be sufficient for our purposes.


\begin{proposition}
\label{prop:resolindep}
Let $\mathbb{D}$ be a homotopical stable algebraic derivator satisfying the conditions of \ref{derivator:new}, and let $(G\looparrowright X)$ be a variety with action. Let $\mathscr{F}:\mathcal{I}\to\op{Res}(G\looparrowright X)$ be a diagram containing a lift of the diagram
$G\times X\leftarrow G\times P\rightarrow P$ for an $n$-acyclic resolution $p:P\to X$. Then for every interval $I=[a,b]$ with $n\geq b-a$, the natural restriction functor induces an equivalence:
\[
\mathbb{D}^{\op{cart},I}(\mathscr{Q}_X\circ\mathscr{F})\to\mathbb{D}^{\op{cart}, I}(\mathscr{W}_P). 
\]
\end{proposition}

\begin{proof}
Recall that the restriction functor $\mathbb{D}^{\op{cart},I}(\mathscr{Q}_X\circ \mathscr{F})\to\mathbb{D}^{\op{cart},I}(\mathscr{W}_P)$ is given by restriction along inclusion of the diagram $G\times X\leftarrow G\times P\to P$, which by our assumptions is in the image of $\mathscr{F}$. The restriction functor $\mathbb{D}^{\op{cart},I}(\mathscr{Q}_X\circ \mathscr{F})\to\mathbb{D}^{\op{cart},I}(\mathscr{W}_P)$ has an adjoint functor. As in \cite[2.4.3]{BeLu}, this adjoint functor can be described explicitly by describing its restriction to arbitrary resolutions $r:R\to X$ as follows. Consider the product $s:S=P\times_X R\to X$. The restriction of the adjoint to $r:R\to X$ is given by $M\mapsto (\op{pr}_R^\ast)^{-1}\circ \op{pr}_P^\ast M$. This follows from the properties of $n$-acyclic maps in Proposition~\ref{prop:belu221}. From these descriptions it follows that the unit and counit of the adjunction are isomorphisms, establishing the claim.
\end{proof}

\begin{corollary}
\label{cor:resRes}
If $\mathscr{F}:\mathcal{J}\to\op{Res}(G\looparrowright X)$ is a skeleton of the category of resolutions, there is an equivalence 
\[
\mathbb{D}^{\op{cart},\op{b}}(\mathscr{Q}_X\circ \mathscr{F}) \stackrel{\simeq}{\longrightarrow}
\mathbb{D}^{\op{res},\op{b}}_G(X).
\]
\end{corollary}

\begin{proof} 
There exist $n$-acyclic resolutions for arbitrary $n>0$, cf. Lemma~\ref{lem:resexist}.
\end{proof}

Now we will compare the definitions of equivariant motives via the simplicial Borel construction and categories of resolutions. This comparison will be essentially based on acyclicity statements in the following zig-zag of morphisms, where $p:P\to X$ is an $n$-acyclic resolution:
\[
G\backslash P\leftarrow {\op{E}}G\times_{/G} P\rightarrow{\op{E}}G\times_{/G} X
\]

\begin{proposition}
\label{prop:simplacyclic}
Let $(G\looparrowright X)$ be a $G$-variety, and let $p:P\to X$ be an $n$-acyclic resolution. Then the map ${\op{E}}G\times_{/G} p:{\op{E}}G\times_{/G} P\to{\op{E}}G\times_{/G} X$ is $n$-acyclic, i.e., for any interval $I=[a,b]$ with $n\geq b-a$, the restriction functor 
\[
({\op{E}}G\times_{/G} p)^\ast:\mathbb{D}^{\Delta,I}_G(X)\to\mathbb{D}^{\Delta,I}_G(P)
\]
is fully faithful. 
\end{proposition}

\begin{proof}
  Consider the  functor $\op{For}$ forgetting the $G$-action
  from Remark~\ref{rem:simpforget}. By definition, we have the following commutative diagram of functors
\[
\xymatrix{
\mathbb{D}^{\Delta}_G(X)\ar[r]^{p^\ast}\ar[d]_{\op{For}_X} & \mathbb{D}^{\Delta}_G(P)\ar[d]^{\op{For}_P} \\
\mathbb{D}(X)\ar[r]_{p^\ast} & \mathbb{D}(P).
}
\]
Restricting to $I$-truncated categories for $I=[a,b]$ with $n\geq b-a$, the acyclicity assumption implies that $p^\ast:\mathbb{D}^I(X)\to\mathbb{D}^I(P)$ is fully faithful, 
cf. Proposition~\ref{prop:belu192}. Now we can consider the right adjoint $\tau_{\leq b}p_*$ of $p^*:\mathbb{D}^{\Delta,I}_G(X)\rightarrow\mathbb{D}^{\Delta,I}_G(P)$. To prove the claim, we want to show that the unit of the adjunction is an
isotransformation $u:\op{id}\xrightarrow{\approx} \tau_{\leq b}p_*p^*$. 
Now using \cite[Lemma 2.4.17]{ayoub:thesis1}, the restriction functor 
\[
i^\ast:\mathbb{D}({\op{E}}G\times_{/G}X)\to\prod_{[n]\in \op{Ob}\Delta }\mathbb{D}({\op{E}}G_n\times_{/G}X)\approx
\prod_{n\in\mathbb{N}}\mathbb{D}(G^n\times X)
\] 
is conservative. The objects of $\mathbb{D}^{\Delta}_G(X)\cong\mathbb{D}^{\op{cart}}({\op{E}}G\times_{/G}X)$ are cartesian, hence $i^\ast M$ vanishes if and only if $i_0^\ast M=\op{For}(M)$ vanishes. Hence the forgetful functor is conservative. Our unit $u:\op{id}\ra \tau_{\leq b}p_*p^*$
gets mapped by the forgetful functor to
the corresponding transformation in the non-equivariant setting, and this is an isomorphism by definition.
\end{proof}



\begin{theorem}
\label{thm:inftyquotient}
Let $\mathbb{D}$ be a homotopical stable algebraic derivator satisfying the conditions of \ref{derivator:new}, and let $(G\looparrowright X)$ be a variety with action. Assume that  $\mathscr{F}:\mathcal{I}\to\op{Res}(G\looparrowright X)$ is a diagram such that
\begin{enumerate}
\item $\mathcal{I}$ has nonempty finite direct products and $\mathscr{F}$ preserves them, 
\item there is a lift ${\op{E}}G\times X:\Delta ^{\op{op}}\to\mathcal{I}$ of the simplicial Borel resolution, and 
\item for each $n>0$, the image of $\mathscr{F}:\mathcal{I}\to\op{Res}(G\looparrowright X)$ contains an $n$-acyclic resolution $p_n:P_n\to X$.
\end{enumerate} 
Then the natural restriction functor is an equivalence 
\[
\delta^\ast:\mathbb{D}^{\op{cart},+}(\mathscr{Q}_X\circ\mathscr{F})\sirra \mathbb{D}^{\Delta,+}_G(X)
\]

\end{theorem}

\begin{proof}
First, we note that the restriction functor $\delta^\ast$ has both left and right adjoint, cf. \cite[Proposition 4.5.4]{ayoub:thesis2}. Considering the adjunction $\delta^\ast\dashv\delta_\ast$, it suffices to show that the unit $\id\to \delta_\ast\delta^\ast$ and counit $\delta^\ast\delta_\ast\to\id$ are isomorphisms.
Now we consider the homotopy t-structure $\tau$ on the derivator $\mathbb{D}$. For any interval $I=[a,b]$, the functor $\delta^\ast$ restricts to $\delta^\ast:\mathbb{D}^{\op{cart},I}(\mathscr{Q}_X\circ\mathscr{F})\to\mathbb{D}^{\Delta,I}_G(X)$. Conversely, truncating with $\tau_{\leq b}$ provides the right adjoint functor $\tau_{\leq b}\circ\delta_\ast:\mathbb{D}^{\Delta,I}_G(X)\to\mathbb{D}^{\op{cart},I}(\mathscr{Q}_X\circ \mathscr{F})$. 

The unit and counit of the adjunction $\delta^\ast\dashv\delta_\ast$ are isomorphisms if for each interval $I$ the truncated unit $\id\to\tau_{\leq b}\delta_\ast\delta^\ast$ and counit $\delta^\ast\tau_{\leq b}\delta_\ast\to\id$ are isomorphisms: for the unit, we can apply $\tau_{\leq b}$ to the unit to get $\tau_{\leq b}M\to\tau_{\leq b}\delta_\ast\delta^\ast M$. Then $\tau_{\leq b}\delta_\ast\delta^\ast\tau_{\leq b}M\cong \tau_{\leq b}\delta_\ast\tau_{\leq b}\delta^\ast M\to \tau_{\leq b}\delta_\ast\delta^\ast M$ is an isomorphism. 
Therefore, it suffices to show that the restricted functor
\[
\delta^\ast:\mathbb{D}^{\op{cart},I}(\mathscr{Q}_X\circ \mathscr{F})\to \mathbb{D}^{\Delta,I}_G(X)
\]
is an equivalence for each interval $I=[a,b]$. 

Using assumption (3), it will be sufficient to show the following claim: 
Let $p:P\to X$ be any $n$-acyclic resolution contained in the image of the diagram $\mathscr{F}:\mathcal{I}\to\op{Res}(G\looparrowright X)$. Replacing $\mathbb{D}^{\Delta,I}_G(P)$ by the essential image $\mathbb{D}^{\Delta,I}_G(P\mid X)$ of the restriction $p^\ast:\mathbb{D}^{\Delta,I}_G(X)\to\mathbb{D}^{\Delta,I}_G(P)$, all the comparison functors in the diagram of Proposition~\ref{prop:restrictcompat} restrict to equivalences on the $I$-truncated categories, for any interval $I=[a,b]$ with $n\geq b-a$. 
This claim in turn follows from the following three statements.

The functor $p^\ast$ is an equivalence onto its essential image, by Proposition~\ref{prop:simplacyclic}.

The functor $\rho^\ast$ is an equivalence, by Proposition~\ref{prop:resolindep}. 

We show that $\pi^\ast$ is an equivalence. First, by the simplicial quotient equivalence, cf. Proposition~\ref{prop:freeaction}, we can identify $\mathbb{D}(G\backslash P)\simeq \mathbb{D}^{\Delta}_G(P)$. Moreover, the quotient equivalence induces an equivalence between $\mathbb{D}^{\Delta,I}_G(P\mid X)$ and the subcategory $\mathbb{D}^I(G\backslash P\mid p)$ of $\mathbb{D}^I(G\backslash P)$ consisting of those objects $M$ such that $q^\ast M\in\mathbb{D}^I(P)$ is in the image of $p^\ast:\mathbb{D}^I(X)\to\mathbb{D}^I(P)$. As in the last paragraph of the proof of Proposition~\ref{prop:belu221}, the restriction functor $\mathbb{D}^{\op{cart},I}(\mathscr{W}_P)\to\mathbb{D}^I(G\backslash P)$ induces an equivalence $\mathbb{D}^{\op{cart},I}(\mathscr{W}_P)\stackrel{\simeq}{\longrightarrow} \mathbb{D}^I(G\backslash P\mid p)$. Combining these, we get that $\pi^\ast$ induces an equivalence $\mathbb{D}^{\op{cart},I}(\mathscr{W}_P)\stackrel{\simeq}{\longrightarrow}\mathbb{D}^{\Delta,I}_G(P\mid X)$, as required.
\end{proof}

In a similar way as above,  we can now also show that the equivariant categories $\mathbb{D}^{\op{cart}}(\mathscr{Q}_X\circ \mathscr{F})$ are independent of the choice of diagram under suitable assumptions. These are the motivic versions of \cite[2.4.3 and 2.4.4]{BeLu}. 

\begin{proposition}
\label{prop:richenough}
Let $\mathbb{D}$ be a homotopical stable algebraic derivator satisfying the conditions of \ref{derivator:new}, and let $(G\looparrowright X)$ be a variety with action. Let $\mathscr{G}:\mathcal{I}\to\mathcal{J}$ be a functor of small categories, and let $\mathscr{F}:\mathcal{J}\to\op{Res}(G\looparrowright X)$ be a diagram of resolutions.  Assume that
\begin{enumerate}
\item both categories $\mathcal{I}$ and $\mathcal{J}$ have nonempty finite direct products and the functors $\mathscr{G}$ and $\mathscr{F}$ preserve them, 
\item the image $\mathscr{F}\circ\mathscr{G}(\mathcal{I})$ contains the trivial resolution, 
\item For every $n>0$, $\mathscr{F}\circ \mathscr{G}(\mathcal{I})$ contains an $n$-acyclic resolution. 
\end{enumerate}
Then the natural restriction functor
\[
\mathscr{G}^{\ast}:\mathbb{D}^{\op{cart},+}(\mathscr{Q}_X\circ \mathscr{F})\to
\mathbb{D}^{\op{cart},+}(\mathscr{Q}_X\circ \mathscr{F}\circ \mathscr{G})
\]
is an equivalence. 
\end{proposition}

\begin{proof}
As before, the restriction functor $\mathscr{G}^\ast$ has both left and right adjoint, cf. \cite[Proposition 4.5.4]{ayoub:thesis2}, hence it suffices to show that unit and counit of the adjunction $\mathscr{G}^\ast\dashv \mathscr{G}_\ast$ are isomorphisms. The truncations for the homotopy t-structure on $\mathbb{D}^+$ provide for any interval $I=[a,b]$ an adjunction between $I$-truncated categories
\[
\mathscr{G}^{\ast}:\mathbb{D}^{\op{cart},I}(\mathscr{Q}_X\circ\mathscr{F})
\rightleftarrows
\mathbb{D}^{\op{cart},I}(\mathscr{Q}_X\circ \mathscr{F}\circ \mathscr{G}):\mathscr{G}_\ast.
\]
It suffices to show that the unit and counit maps for this adjunction are isomorphisms for any interval $I$.

Now let $p:P\to X$ be a resolution of $X$, and consider the restrictions from cartesian motives over the big diagrams to $\mathbb{D}^{\op{cart},I}(\mathscr{W}_P)$. Since all restriction functors are induced from morphisms of diagrams, we get a commutative diagram
\[
\xymatrix{
\mathbb{D}^{\op{cart},I}(\mathscr{Q}_X\circ\mathscr{F}) \ar[rr]^{\mathscr{G}^\ast}  \ar[dr]_{\rho^\ast} && 
\mathbb{D}^{\op{cart},I}(\mathscr{Q}_X\circ\mathscr{F}\circ\mathscr{G}) \ar[dl]^{\rho^\ast} \\
&\mathbb{D}^{\op{cart},I}(\mathscr{W}_P).
}
\]
If $P$ is an $n$-acyclic resolution and $I=[a,b]$ satisfies $n\geq b-a$, then the restriction functors $\rho^\ast$ are equivalences by Proposition~\ref{prop:resolindep}. By commutativity of the diagram, the restriction functor $\mathscr{G}^\ast$ is then also an equivalence. By assumption (3), the restriction functors $\mathscr{G}^\ast$ are equivalences for any interval $I$, proving the claim. 
\end{proof}

\begin{remark}
If we actually have lifts of the Borel construction, we can also use Theorem~\ref{thm:inftyquotient} to prove the above. However, we also want to apply Proposition~\ref{prop:richenough} to categories of smooth resolutions, and in these cases lifts of the Borel construction do not exist. 
\end{remark}

In particular, if $\mathscr{F}:\mathcal{I}\to\op{Res}(G\looparrowright X)$ is a skeleton of the category of resolutions with $\mathscr{F}$ the inclusion functor, the category $\mathbb{D}^{\op{cart},+}(\mathscr{Q}_X\circ\mathscr{F})$ does not depend on the choice of skeleton, up to equivalence. 

\begin{definition}
\index{equivariant motives, $\mathbb{D}_G(X)$}
\label{def:equivmotres}
Let $\mathbb{D}$ be a homotopical stable algebraic derivator satisfying the conditions of \ref{derivator:new} and let $(G\looparrowright X)$ be a variety with action. 

For $\mathscr{F}:\mathcal{I}\to\op{Res}(G\looparrowright X)$ a choice of skeleton of the category of resolutions, we denote 
\[
\mathbb{D}^{\op{Res},+}_G(X):=\mathbb{D}^{\op{cart},+}(\mathscr{Q}_X\circ\mathscr{F}). 
\]

The category of \emph{$G$-equivariant motives over $X$} is defined to be 
\[
\mathbb{D}_G(X):=\mathbb{D}_G^\Delta(X)
\]
\end{definition}

\begin{remark}
It should be noted that the equivalence in the definition of $\mathbb{D}^{\op{Res},+}_G(X)$ is not canonical: for two different choices of skeleta of $\op{Res}(G\looparrowright X)$ and a choice of skeleton containing them, we get a zig-zag of equivalences. This implies in particular difficulties for setting up the six-functor formalism in this situation. While the functors can be defined without much problems, proving 2-functoriality and various exchange or base-change formulas relating these functors is difficult precisely because of the non-canonical choices involved. This is rectified by using Theorem~\ref{thm:inftyquotient} to compare to the simplicial Borel construction approach where the six functors are 2-functorial.
\end{remark}

\section{Equivariant six-functor formalism}
\label{sec:sixfunctors}

In this section, we will now give the definition of the equivariant six functors and derive their basic properties. Here, the definition of the six functors is more problematic than deducing their properties from the non-equivariant setting. Once the functors are defined the various properties like base-change formulas and Verdier duality will follow rather immediately from the corresponding facts in the non-equivariant situation. 

For the definition,  we need the definition of equivariant categories of motives via both the Borel and the resolution approach. In the approach via Borel constructions, all of the functors can be defined for categories of motives over diagrams. The difficulty arises when one tries to show that the functors preserve cartesian objects. On the other hand, for the definition via motives over diagrams of resolutions, all the functors can be defined, but the definition involves choices of appropriate diagrams of resolutions. This leads to difficulties when one wants to prove 2-functoriality or various exchange properties.

By combining both approaches, the construction of the six functors can then be explained in the following diagram (for the exceptional functors):
\[
\xymatrix{
  \mathbb D_G^{\Delta,+}( X)\ar_{f_!}[d]\ar@{=>}^-\sim[rrd]  && \mathbb D^{\op{cart},+}(\mathscr{Q}_X\circ\mathscr{F}_f) \ar_{f_!}[d]\ar_{\approx}[ll] \\
 \mathbb D_G^{\Delta,+}( Y)  && \mathbb D^{\op{cart},+}(\mathscr{Q}_Y\circ\mathscr{G}_f) 
 \ar_{\approx}[ll]\ar@/_1pc/[u]_-{f^!}}
\]
In the simplicial approach, the left adjoints can be defined and shown to preserve cartesian objects, thus providing the left vertical arrow. In the resolution approach, the left adjoints can also be defined and shown to preserve cartesian objects, thus providing the right vertical arrow pointing downwards. Then we need to trace through the comparison isomorphisms established previously to show that the square commutes up to isotransformation. On the resolution side, we can then show that the right adjoints exists, giving the right-hand vertical arrow pointing upward. Uniqueness of adjoints, cf. \cite[Section 1.1]{ayoub:thesis1}, then implies that in the diagram above, the functors on the simplicial side have right adjoints. Since the left adjoints in the simplicial definition are 2-functorial, the right adjoints will also be. This is not obvious on the right-hand side of the diagram, because the definition of the categories on the right involves choices of appropriate diagrams, possibly depending on the morphism $f$ in question.\footnote{It would appear that issues of 2-functoriality in the classical equivariant derived categories have not been addressed in the literature, cf. in particular the Math-Overflow question 233708 ``2-functoriality of equivariant derived categories''.}

\subsection{Functors using simplicial Borel construction}

We first discuss the simplicial description of some of the six functors, namely the left adjoints $f^\ast$ and $f_!$. Unfortunately, not all of the six functors can be immediately defined using the simplicial definition. The problems are the same as the ones appearing already in \cite{BeLu}: the property of being locally constant over the Borel construction does not generally behave well enough under all the six functors. Nevertheless, the functors that can be defined have the expected properties,  as follows by rewriting the respective portions of \cite[Section 2.6, 3, 6 and 7]{BeLu} using \cite{ayoub:thesis1,ayoub:thesis2}.

\begin{Bemerkung}
A morphism $(\phi,f)\colon (H\looparrowright Y)\to(G\looparrowright X)$ of  varieties with action induces a morphism of simplicial schemes $(\phi,f)\colon {\op{E}}H\times_{/H}Y\to{\op{E}}G\times_{/G}X$. The  functors we will consider below arise from applying \cite[Section 2.4.4]{ayoub:thesis1} to the above morphisms of Borel constructions. Note that the properties of being quasi-projective, closed immersions, smooth maps etc. are defined termwise for simplicial schemes; therefore, any property which is stable under products and true for $\phi:H\to G$ and $f:Y\to X$ will hold for the morphism of Borel constructions ${\op{E}}H\times_{/H}Y\to {\op{E}}G\times_{/G}X$.

In the case where $H=G$, i.e., the group does not change, we take $\mathscr{S}=({\op{B}}G,\Delta ^{\op{op}})$, and consider the Borel constructions for varieties with $G$-action as $\mathscr{S}$-schemes via the projection ${\op{E}}G\times_{/G}X\to {\op{E}}G\times_{/G}\op{pt}={\op{B}}G$.  Note that morphisms of $\mathscr{S}$-schemes are automatically cartesian by \cite[Lemma 2.4.33]{ayoub:thesis1}. This means that for  $\phi=\id$ the diagrams 
\[
\xymatrix{
  G^{n+1}\times X\ar_{\sigma\times\id}[d] \ar^{\id^{n+1}\times f}[rr] & &
  G^{n+1}\times Y \ar^{\sigma'\times\id}[d] \\
  G^n\times X\ar_{\id^n\times f}[rr] && G^n\times Y
}
\]
are cartesian, in which $\sigma$ and $\sigma'$ are face maps for the simplicial objects ${\op{E}}G\times_{/G}X$ and ${\op{E}}G\times_{/G}Y$, respectively. The same holds for the degeneracy maps. 

It is, however, important to note that the six functors from \cite[Section 2.4.4]{ayoub:thesis1} require that the morphism of simplicial schemes is cartesian, which restricts us to morphisms of varieties with $G$-action where \emph{$G$ is fixed}. The diagrams above fail to be cartesian in the more general setting when the morphism $\phi:H\to G$ is not the identity, and consequently most of the functors below will not exist for general morphisms $(\phi,f)$ of varieties with actions.
\end{Bemerkung}

\begin{Bemerkung}
\label{def:ordinary}
To the morphism $(\phi,f)\colon (H\looparrowright Y)\to(G\looparrowright X)$  of varieties with action resp. the associated morphism of simplicial schemes $(\phi,f)\colon {\op{E}}H\times_{/H}Y\to{\op{E}}G\times_{/G}X$, the derivator $\mathbb{D}$ associates a functor  
\[
(\phi,f)^\ast\colon
\mathbb{D}({\op{E}}G\times_{/G}X,\Delta ^{\op{op}})\to 
\mathbb{D}({\op{E}}H\times_{/H}Y,\Delta ^{\op{op}}).
\]

The coherence isomorphisms $(g\circ h)^\ast\stackrel{\approx}{\rightarrow} h^\ast\circ g^\ast$ show that the functor $(\phi,f)^\ast$ preserves locally constant motives. In particular, we get, for any morphism $(\phi,f)$ of varieties with action, an associated functor of equivariant motivic categories
\[
(\phi,f)^\ast: \mathbb{D}_G(X)\to \mathbb{D}_H(Y).
\]

Since $\mathsf{H}^\ast$ is a $2$-functor, the above assignments assemble into 
a $2$-functor from the category of varieties with action to the category of  triangulated categories. This encodes all the properties like $(g\circ f)^\ast\cong f^\ast\circ g^\ast$, etc. 
\end{Bemerkung}

Fix a linear group $G$, and let $f=(\id,f):(G\looparrowright Y)\to (G\looparrowright X)$ be a $G$-equivariant morphism of varieties with action. 
Axiom DerAlg 2d  in \cite[Section 2.4.2]{ayoub:thesis1} implies the existence of a right adjoint to $(\id,f)^\ast$: 
\[
(\id,f)_\ast\colon
\mathbb{D}({\op{E}}G\times_{/G}Y,\Delta ^{\op{op}})\to 
\mathbb{D}({\op{E}}G\times_{/G}X,\Delta ^{\op{op}}).
\]
Assuming that the morphism $f\colon Y\to X$ of $G$-varieties is smooth, Axiom  DerAlg 2g in \cite[Section 2.4.2]{ayoub:thesis1} implies that the functor $(\id,f)^\ast$ also admits a left adjoint 
\[
(\id,f)_{\sharp}\colon
\mathbb{D}({\op{E}}G\times_{/G}Y,\Delta ^{\op{op}})\to 
\mathbb{D}({\op{E}}G\times_{/G}X,\Delta ^{\op{op}}).
\]
Furthermore, for a cartesian square of varieties
\[
\xymatrix{
  X'\ar[r]^{f'} \ar[d]_{g'} &
  X \ar[d]^{g}\\
  Y'\ar[r]_{f} &
  Y
}
\]
we have that 
\begin{itemize}
\item the exchange morphism $g'_{\sharp}\circ f'^\ast\longrightarrow f^\ast\circ g_{\sharp}$ is an isomorphism  if $g$ is smooth.
\item the exchange morphism $f^\ast \circ g_\ast\longrightarrow g'_\ast\circ f'^\ast$ is an isomorphism whenever $g$ is projective or $f$ is smooth.
\end{itemize}

The base-change formulas and the 2-functorial statements from \cite{ayoub:thesis1} now imply the following:

\begin{proposition}
\label{prop:simpcart}
\begin{enumerate}
\item 
Assume that $f:Y\to X$ is a smooth morphism of $G$-varieties. Then the functor $(\id,f)_{\sharp}$ preserves locally constant objects, and we get an adjoint pair  of functors 
\[
(\id,f)_{\sharp}\colon\mathbb{D}_G(Y)\rightleftarrows\mathbb{D}_G(X)\colon(\id,f)^\ast.
\]
These assemble into an adjunction of 2-functors on the category of $G$-varieties with smooth morphisms. 
\item
Assume that $f:Y\to X$ is a projective morphism of $G$-varieties. Then the functor $(\id,f)_\ast$ preserves locally constant objects, and we get an adjoint pair of functors 
\[
(\id,f)^\ast\colon\mathbb{D}_G(X)\rightleftarrows\mathbb{D}_G(Y)\colon
(\id,f)_\ast. 
\]
These assemble into an adjunction of 2-functors on the category of $G$-varieties with projective morphisms. 
\end{enumerate}
\end{proposition}

\begin{example}\label{akF}
\begin{enumerate}
\item In the special case where $H=G=1$ the above constructions reduce to $f^\ast$ and $f_\ast$  for a morphism  $f:X\to Y$ of varieties. 
\item 
The functor $(\phi,\op{id})^\ast:\mathbb{D}_G(X)\to\mathbb{D}_H(X)$ could also be called restriction functor $\op{Res}_G^H$ along $\phi:H\rightarrow G$. An even more special case is $\phi:H=1\to G$, in which case we obtain a forgetful functor 
\[
\op{Res}_G^1:\mathbb{D}_G(X)\to\mathbb{D}(X).
\]
This functor is by construction the same as the one considered after Definition~\ref{def:equivmotres1}.
\index{restriction functor $\op{Res}_G^H$}
\item Let $G\looparrowright X$ be a variety with action. Assume everything is defined over the field $k$ and denote by $\pt:=\op{Spec} k$. For each motive $M\in\mathbb{D}(\pt)$, there is an associated constant equivariant motive $\const{M}$ over $X$, given by $\op{fin}^\ast M$ with $\op{fin}:{\op{E}}G\times_{/G}X\to \pt$ the projection map. 
\end{enumerate}
\end{example}

\begin{Bemerkung}[{\bf Forgetting the action}]
\label{forget}
We will occasionally denote the restriction functor $\op{Res}_G^1:\mathbb{D}_G(X)\to\mathbb{D}(X)$ as forgetful functor $\op{For}:\mathbb{D}_G(X)\to \mathbb{D}(X)$;\index{For@$\op{For}$ forgetful functor} it assigns to a $G$-equivariant motive $M$ its underlying non-equivariant motive $\op{For}(M)$.

The forgetful functor is conservative, i.e., an equivariant motive $M\in\mathbb{D}_G(X)$ is trivial if and only if its underlying motive $\op{Res}_G^1(M)$ is trivial. This follows from \cite[Lemma 2.4.17]{ayoub:thesis1}, which states that the functor 
\[
i^\ast:\mathbb{D}({\op{E}}G\times_{/G}X)\to\prod_{[n]\in \op{Ob}\Delta }\mathbb{D}({\op{E}}G_n\times_{/G}X)\approx
\prod_{n\in\mathbb{N}}\mathbb{D}(G^n\times X)
\] 
is conservative. An equivariant motive  $M\in\mathbb{D}_G(X)$ is additionally locally constant, and this implies that $i^\ast M$ is trivial if and only if  $i_0^\ast M=\op{Res}_G^1(M)$ is trivial. In particular, triviality of equivariant motives can be checked after forgetting the equivariance.

For a group homomorphism $\phi:G\to H$, we have $\op{Res}_H^1\simeq \op{Res}_G^1\circ \op{Res}_H^G$. From this composition and the above observation that $\op{Res}_H^1$ is conservative, it follows that restriction along any group homomorphism is conservative.
\end{Bemerkung}

Next, we need to discuss the definition and basic properties of the exceptional inverse and direct image functors. 

\begin{definition}
\label{def:exceptional}
\index{exceptional functors!equivariant motives}
Let $G$ be a linear algebraic group, and let $f=(\id,f): (G\looparrowright Y)\ra (G\looparrowright X)$ be a morphism of varieties with $G$-action.  By \cite[Scholie 1.4.2 and p.323]{ayoub:thesis1}, the corresponding morphism of simplicial varieties ${\op{E}}G\times_{/G}Y\to {\op{E}}G\times_{/G}X$ induces an adjoint pair of exceptional direct and inverse image functors 
\[
(\id,f)_!:
\mathbb{D}({\op{E}}G\times_{/G}Y,\Delta ^{\op{op}})\leftrightarrows 
\mathbb{D}({\op{E}}G\times_{/G}X,\Delta ^{\op{op}}):(\id,f)^!.
\]
The base-change formula for the $2$-functor $\mathsf{H}_!$  implies that the condition of being locally constant is preserved under the  functor $(\id,f)_!$. Therefore,  $\mathsf{H}_!$ induces a $2$-functor from the category of varieties with $G$-action to the category of triangulated categories. 
\end{definition}

\begin{Bemerkung}
\label{rem:adjoints}
The right adjoints can be obtained by applying Neeman's existence theorem. By \ref{compgen}, the categories $\mathbf{DA}^{\et}(X,\Lambda)$ are compactly generated whenever $X$ is a noetherian scheme of finite Krull dimension. The same holds for all the relevant localizations of $\mathbf{DA}^{\et}$ we are interested in. 

Then the categories $\mathbb{D}_G^\Delta(X)$ are also compactly generated. Since it is induced from a morphism of diagrams of schemes, the forgetful functor $\op{For}:\mathbb{D}_G^\Delta(X)\to\mathbb{D}(X)$ has a left adjoint. Applying this left adjoint to a set of compact generators of $\mathbb{D}(X)$ provides a set of compact generators of $\mathbb{D}_G^\Delta(X)$: the generating property follows from the fact that $\op{For}$ is conservative and the compactness follows since the left adjoint of the forgetful functor preserves arbitrary small sums.

For $f:Y\to X$ a morphism of varieties, the functors
\begin{eqnarray*}
f^\ast:\mathbb{D}^{\op{cart}}({\op{E}}G\times_{/G}X)&\to&\mathbb{D}^{\op{cart}}({\op{E}}G\times_{/G} Y), \\ f_!:\mathbb{D}^{\op{cart}}({\op{E}}G\times_{/G}Y)&\to&\mathbb{D}^{\op{cart}}({\op{E}}G\times_{/G} X)
\end{eqnarray*}
preserve arbitrary small sums  because they are left adjoints on the full categories of motives and sums of cartesian objects agree with sums in the full category of motives over the diagrams. By \cite[Theorem 4.1]{neeman}, the right adjoints exist. However, we want to have more explicit descriptions of the functors which is why we also provide constructions of these functors via motives over categories of resolutions, see the next subsection. 

The same procedure can also be used to obtain general direct image functors $(\phi,f)_\ast$ for a morphism $(\phi,f):(G\looparrowright X)\to (H\looparrowright Y)$ of varieties with action. This generalizes the functors $Q_{f_\ast}$ of \cite{BeLu}. 
\end{Bemerkung}

\begin{Bemerkung}
\label{rem:cartesify}
Another way, using Neeman's theorem, to get a slightly improved description of the right adjoints would be to cartesify objects. As in the preceding remark, the categories $\mathbb{D}_G^\Delta(X)$ are compactly generated whenever $X$ is noetherian of finite Krull dimension. The inclusion
\[
\mathbb{D}^\Delta_G(X)\subset \mathbb{D}({\op{E}}G\times_{/G} X)
\]
preserves arbitrary small sums; the sum of cartesian objects is still cartesian because this property is tested via $f^\ast$-functors which have right adjoints. By \cite[Theorem 4.1]{neeman}, the inclusion has a right adjoint which ``cartesifies'' motives over the simplicial variety ${\op{E}}G\times_{/G} X$. The right adjoints $f_\ast$ and $f^!$ for equivariant motives can then alternatively be obtained by composing the relevant functors $f_\ast$ and $f^!$ (for the full categories of motives over the simplicial varieties, whose existence follows from \cite{ayoub:thesis1}) with the cartesification.
\end{Bemerkung}

\begin{remark}
Brad Drew pointed out that the ordinary pullback $f^\ast$ provides a functor from schemes to presentable $\infty$-categories with left-adjoint functors as morphisms. This induces a functor from simplicial schemes to simplicial presentable $\infty$-categories, and taking the limit over the simplicial index category provides exactly the corresponding category of cartesian objects appearing in the simplicial definition of equivariant motives. There is a formal identification of the opposite category of $\infty$-categories with left adjoints and the $\infty$-categories with right adjoints. This identification provides another way to get 2-functorial right adjoint functors $f_\ast$.  
\end{remark}

\subsection{Functors using resolutions}

Next, we will define the six functors using resolutions, basically following \cite[Section 3]{BeLu}. For a fixed linear algebraic group $G$ and a morphism $f:(G\looparrowright Y)\to(G\looparrowright X)$ of varieties with action, we will define the adjoint pairs $f^\ast\dashv f_\ast$ and $f_!\dashv f^!$. 

Let $(G\looparrowright X)$ be a $G$-variety, and let $\mathscr{F}:\mathcal{I}\to\op{SmRes}(G\looparrowright X)$ be a choice of skeleton for the category of \emph{smooth} resolutions of $X$. Then the conditions of Proposition~\ref{prop:richenough} are satisfied, i.e., there is an induced equivalence 
\[
\mathbb{D}^{\op{Res},+}_G(X)\stackrel{\approx}{\longrightarrow} \mathbb{D}^{\op{cart},+}(\mathscr{Q}_X\circ\mathscr{F}). 
\]
This follows since direct products of smooth resolutions are smooth, the trivial resolution is smooth and Lemma~\ref{lem:resexist} guarantees the existence of smooth resolutions of arbitrarily high acyclicity. 

\begin{definition}
Let $(G\looparrowright X)$ be a variety with action and let $\mathbb{D}$ be a homotopical stable algebraic derivator satisfying the conditions of \ref{derivator:new}. This means in particular, that the categories $\mathbb{D}$ of motives over a diagram are equipped with a closed symmetric monoidal structure. It also means that the functors $(f,\alpha)^\ast$ are symmetric monoidal. By \cite[Section 4.5.2]{ayoub:thesis2}, these conditions are satisfied for the homotopical stable algebraic derivators of interest to us.

Fix a choice of skeleton $\mathscr{F}:\mathcal{I}\to\op{SmRes}(G\looparrowright X)$. 
For smooth morphisms, the functors $(f,\alpha)^\ast$ strictly preserve the closed monoidal structure. Therefore, the closed symmetric monoidal structure on motives over the diagram $(\mathscr{Q}_X\circ \mathscr{F},\mathcal{I})$ restricts to the category of equivariant motives $\mathbb{D}^{\op{cart},+}(\mathscr{Q}_X\circ \mathscr{F})$. Via Proposition~\ref{prop:richenough}, this provides a closed symmetric monoidal structure on $\mathbb{D}^{\op{Res},+}_G(X)$.
\end{definition}

Let $f:Y\to X$ be a morphism of varieties with $G$-action. Then there is a functor 
\begin{eqnarray*}
f^\circ:\op{SmRes}(G\looparrowright X)&\to&\op{SmRes}(G\looparrowright Y): \\
(p:P\to X)&\mapsto& (f^\circ(p):f^\circ(P):=P\times_XY\to Y)
\end{eqnarray*}
which preserves $n$-acyclicity of resolutions. We remarked above that 
for a choice $\mathscr{F}_X:\mathcal{I}_X\to\op{SmRes}(G\looparrowright X)$ of skeleton of the category of smooth resolutions, the natural restriction functor
\[
\mathbb{D}^{\op{Res},+}_G(X)\to \mathbb{D}^{\op{cart},+}(\mathscr{Q}_X \circ\mathscr{F}_X)
\]
is an equivalence. On the other hand, if we consider the inclusion of the subcategory 
\[
\op{SmRes}(G\looparrowright f)\subset \op{Res}(G\looparrowright Y)
\]
consisting of the image $f^\circ(\op{SmRes}(G\looparrowright X))$ and choose a skeleton $\mathscr{F}_Y:\mathcal{I}_Y\to\op{SmRes}(G\looparrowright f)$, then the composed functor $\mathscr{Q}_Y\circ\mathscr{F}_Y:\mathcal{I}_Y\to\op{Res}(G\looparrowright Y)$ also satisfies the conditions of Proposition~\ref{prop:richenough}. Therefore, the restriction functor 
\[
\mathbb{D}^{\op{Res},+}_G(Y)\to \mathbb{D}^{\op{cart},+}(\mathscr{Q}_Y\circ \mathscr{F}_Y)
\]
is also an equivalence.

\begin{definition}
\index{exceptional functors!equivariant motives}
Let $f:Y\to X$ be a morphism of varieties with $G$-action. The morphism of diagrams $\phi_f:\op{SmRes}(G\looparrowright f)\to\op{SmRes}(G\looparrowright X)$ which maps $f^\circ(P)=P\times_X Y\to Y$ to the resolution $P\to X$ has induced adjoint pairs $\phi_f^\ast\dashv (\phi_f)_\ast$ and $(\phi_f)_!\dashv(\phi_f)^!$. Smooth base-change shows that these functors preserve cartesian motives over the respective categories of smooth resolutions. Combining these functors with the above equivalences, we get adjoint pairs
\[
f^\ast:\mathbb{D}^{\op{Res},+}_G(X)\leftrightarrows\mathbb{D}^{\op{Res},+}_G(Y):f_\ast,\qquad 
f_!:\mathbb{D}^{\op{Res},+}_G(Y)\leftrightarrows\mathbb{D}^{\op{Res},+}_G(X):f^!
\]
\end{definition}

\begin{remark}
There is an issue with the 2-functoriality at this point. Although we have definitions of the six functors, these definitions depend on the choice of a suitable small category of resolutions. This means that there is no global comparison transformation for compositions of these functors. To get 2-functoriality, we need to apply some rectification procedure. For this and other reasons, we need to compare the simplicial and resolution definitions of the left adjoints $f^\ast$ and $f_!$. 
\end{remark}

There is also a general inverse image, analogous to the $Q_f^\ast$ functors defined in \cite[Section 6]{BeLu}, that can be defined using the categories of compatible resolutions, cf. Definitions~\ref{def:comp1} and \ref{def:comp2}. 

\begin{definition}
For a morphism $(\phi,f):(H\looparrowright Y)\to (G\looparrowright X)$ of varieties with action, we have the category $\op{CRes}((\phi,f))$ of resolutions of $(H\looparrowright Y)$ which are compatible with the morphism $(\phi,f)$. For any choice of skeleton $\mathscr{F}:\mathcal{I}\to\op{CRes}((\phi,f))$, there is by definition an induced diagram 
\[
\mathscr{P}\circ\mathscr{F}:\mathcal{I}\to\op{CRes}((\phi,f))\to \op{Res}(G\looparrowright X). 
\]
The diagram $(\mathscr{P}\circ \mathscr{F},\mathcal{I})$ satisfies the conditions of Proposition~\ref{prop:richenough}, hence the induced restriction functor 
\[
\mathbb{D}^{\op{Res},+}_H(Y)\to\mathbb{D}^{\op{cart},+}(\mathscr{P}\circ \mathscr{F})
\]
is an equivalence. Combining the restriction functor along $\mathscr{P}$ with this equivalence, we get an induced functor 
\[
(\phi,f)^\ast_{\op{Res}}:\mathbb{D}^{\op{Res},+}_G(X)\to \mathbb{D}^{\op{Res},+}(\mathscr{P}\circ \mathscr{F}) \stackrel{\approx}{\longrightarrow} \mathbb{D}^{\op{Res},+}_H(Y). 
\]
\end{definition}

The functor $(\phi,f)^\ast_{\op{Res}}$ can be used to define restriction functors along group homomorphisms $\phi:H\to G$, as in \ref{akF}.

\begin{Bemerkung}
The descriptions of the change-of-group functors $\op{Res}_G^H$ associated to a subgroup inclusion $\phi:H\subset G$ can  be  obtained as in \cite[2.6.1]{BeLu}. Let $(G\looparrowright X)$ be a variety with action. For any resolution $p:P\to X$ of the $H$-action on $X$, one can consider the induced resolution of the $G$-action given by 
\[
p':G\times_{/H} P\to X: (g,l)\mapsto g(p(l)).
\]
The projection induces a canonical isomorphism of quotients $G\backslash (G\times_{/H}P)\cong H\backslash P$. In the resolution approach, we view a $G$-equivariant motive over $X$ as a motive  $M\in \mathbb{D}^{\op{cart},+}(\mathscr{Q}_X\circ \mathscr{F})$ where $\mathscr{F}:\mathcal{I}\to\op{Res}(G\looparrowright X)$ is a skeleton of the category of resolutions of the $G$-action on $X$. Denoting by $\mathscr{G}:\mathcal{J}\to\op{Res}(H\looparrowright X)$ a skeleton of the resolutions of the $H$-action, the functor $\phi^\ast$ is given by restriction for the morphism of diagrams
\[
\mathcal{J}\to\op{Res}(H\looparrowright X)\to\op{Res}(G\looparrowright X):
j\mapsto G\times_{/H}\mathscr{G}(j). 
\]
In particular, the value of $\phi^\ast M$ on a resolution $p:P\to X$ is the value of $M$ on the induced resolution $p':G\times_{/H}P\to X$. 

As a very special case, the forgetful functor $\op{Res}_G^1$ is given by evaluation on the quotient $X$ of the trivial resolution $G\times X$. 
\end{Bemerkung}

\subsection{Comparison and 2-functoriality}

Now we have two approaches of defining categories of equivariant motives, via resolutions and simplicial Borel constructions, and we have seen that both are equivalent. It remains to show that  the two versions of the left adjoint functors $f^\ast$ and $f_!$ also agree via the comparison equivalences. 

\begin{proposition}
Let $(\phi,f):(H\looparrowright Y)\to(G\looparrowright X)$ be a morphism of varieties with action. Fix a choice $\mathscr{F}:\mathcal{I}\to\op{Res}(G\looparrowright X)$ of skeleton for $\op{Res}(G\looparrowright X)$, and choose a skeleton  $\mathscr{F}':\mathcal{I}'\to\op{CRes}((\phi,f))$ compatible with that. The following diagram of functors is commutative:
\[
\xymatrix{
  \mathbb{D}^{\op{Res},+}_G(X) \ar[d]_\approx &&   \mathbb{D}^{\op{Res},+}_H(Y) \ar[d]^\approx \\
  \mathbb{D}^{\op{cart},+}(\mathscr{Q}_X\circ \mathscr{F}) \ar[rr]^{(\phi,f)^\ast_{\op{Res}}} \ar[d]_{\delta^\ast} && \mathbb{D}^{\op{cart},+}(\mathscr{Q}_Y\circ\mathscr{F}')  \ar[d]^{\delta^\ast} \\ 
  \mathbb{D}^{\Delta,+}_G(X) \ar[rr]_{(\phi,f)_{\Delta}^\ast}  && \mathbb{D}^{\Delta,+}_H(Y).
}
\]
\end{proposition}

\begin{proof}
The diagram is induced by a commutative diagram of morphisms of diagrams. The main point here is that the simplicial Borel construction is compatible with the morphism $(\phi,f)$, hence the category of compatible resolutions contains a lift of the Borel construction for $(H\looparrowright Y)$. 
\end{proof}

Similarly, we get the following compatibility between the two definitions of $f^\ast$ via compatible and smooth resolutions. 

\begin{proposition}
Let $f:(G\looparrowright Y)\to (G\looparrowright X)$ be a morphism of $G$-varieties with action. Fix a choice $\mathscr{F}:\mathcal{I}\to\op{Res}(G\looparrowright X)$ of skeleton for $\op{Res}(G\looparrowright X)$, and choose a skeleton of $\op{CRes}(G\looparrowright Y)$ compatible with that. This implies choices for skeleta of $\op{SmRes}(G\looparrowright X)$ and $\op{SmRes}(G\looparrowright f)$. Then the following diagram of functors is commutative:
\[
\xymatrix{
  \mathbb{D}^{\op{cart},+}(\op{Res}(G\looparrowright X)) \ar[r]^{f^\ast_{\op{Res}}} \ar[d] & \mathbb{D}^{\op{cart},+}(\op{CRes}(G\looparrowright Y))  \ar[d] \\ 
  \mathbb{D}^{\op{cart},+}(\op{SmRes}(G\looparrowright X)) \ar[r]_{f^\ast_{\op{SmRes}}} & \mathbb{D}^{\op{cart},+}(\op{SmRes}(G\looparrowright f)) 
}
\]
Here we write the full categories of resolutions instead of the diagram functors, to emphasize where we are working.
\end{proposition}

A combination of the above results together with the base-change formula shows the compatibility of the two definitions of the exceptional direct image. 

\begin{proposition}
Let $f:Y\to X$ be a morphism of varieties with $G$-action. The following diagram of functors is commutative:
\[
\xymatrix{
  \mathbb{D}^{\op{cart},+}(\op{SmRes}(G\looparrowright f)) \ar[r]^{f^{\op{Res}}_!} &   \mathbb{D}^{\op{cart},+}(\op{SmRes}(G\looparrowright X))  \\
  \mathbb{D}^{\op{cart},+}(\op{Res}(G\looparrowright f)) \ar[r]^{f^{\op{Res}}_!} \ar[u]_\approx \ar[d]_{\delta^\ast} & \mathbb{D}^{\op{cart},+}(\op{Res}(G\looparrowright X)) \ar[u]^\approx \ar[d]^{\delta^\ast} \\ 
  \mathbb{D}^{\Delta,+}_G(Y) \ar[r]_{f_!^\Delta}  & \mathbb{D}^{\Delta,+}_G(X).
}
\]
Here we write the full categories of resolutions instead of the diagram functors, to emphasize where we are working.
\end{proposition}

\begin{proof}
The only thing to note here is that the category $\op{Res}(G\looparrowright f)$ is defined in the same way as $\op{SmRes}(G\looparrowright f)$, by pullbacks of resolutions. This implies, in particular, that the morphism of diagrams $\op{Res}(G\looparrowright f)\to \op{Res}(G\looparrowright X)$ is cartesian, hence we have a functor $f_!$. By the base-change formula, this functor preserves cartesian objects, which explains the middle horizontal arrow. Then both squares arise from commutative diagrams of morphisms of diagrams of schemes, one by restriction to smooth resolutions, the other one because the Borel construction is compatible with pullback. The commutativity of the squares follows from the base-change formula.
\end{proof}

\begin{Bemerkung}
These comparison statements are relevant for the 2-functoriality issues in the six-functor formalism. They are also useful for computations later on.

Note that the definitions of the right adjoints $f_\ast$ and $f^!$ in the categories $\mathbb{D}^{\op{Res},+}_G(-)$ made use of changes of diagrams. This introduces dependencies on choices and obstructs 2-functoriality. On the other hand, the above results identify the left adjoints for the categories $\mathbb{D}^{\op{Res},+}_G(-)$ with those defined for the categories $\mathbb{D}^{\Delta,+}_G(-)$. The latter arise from Ayoub's formalism, so we know 2-functoriality for those. By \cite[Corollary 1.1.7, Proposition 1.1.17]{ayoub:thesis1}, the right adjoints defined via resolutions can then be rectified to satisfy 2-functoriality as well. 
\end{Bemerkung}

\subsection{Remarks on general direct image functors}
\label{sec:qfstar}

Given a morphism of varieties with action  $(\phi,f):(G\looparrowright X)\to(H\looparrowright Y)$, we would also like to define a direct image functor $(\phi,f)_\ast$ associated to it, similar to the definition of the functor $Q_{f\ast}$ in \cite[Section 6]{BeLu}. 

The first obvious difficulty is that the definition of $Q_{f\ast}$ in \cite{BeLu} uses $\infty$-acyclic resolutions which are not available in the algebraic setting. This is a serious obstacle for the definition of general direct image functors. The different approach to the construction would be the general method of Neeman's adjoint functor theorem, cf. Remarks~\ref{rem:adjoints} and \ref{rem:cartesify}. We can also describe more explicitly what the general direct images functors $(\phi,f)_\ast$ look like in situations that are relevant for us. 

To describe $(\phi,f)_\ast M$, it is sufficient to describe the evaluation of the motive $(\phi,f)_\ast M$ in $\mathbb{D}_H^+(Y)$. For this, let $r:R\to Y$ be a smooth resolution of $(H\looparrowright Y)$. Using the simplicial Borel construction for $(G\looparrowright X)$, we get a compatible pair of resolutions $\rho:\left({\op{E}}G\times X\right)\times_Y R\to R$ which induces a morphism of quotient simplicial varieties
\[
\bar{\rho}:G\backslash\left(\left({\op{E}}G\times X\right)\times_Y R\right)\to H\backslash R.
\]
We can apply the general pushforward functor from \cite{ayoub:thesis1} 
\[
\bar{\rho}_\ast:\mathbb{D}^+(G\backslash\left(\left({\op{E}}G\times X\right)\times_Y R\right))\to\mathbb{D}^+(R,\Delta^{\op{op}}),
\]
where $R$ is viewed as constant simplicial variety. The resulting motive over the constant simplicial variety can then be viewed as a simplicial motive in $\mathbb{D}^+(R)$ where we now view $R$ just as a variety.

Then we can use \cite[Theorem 2.4.22]{ayoub:thesis1} (replacing the use of the good base-change lemma in \cite{BeLu}) to show that this construction maps a cartesian motive over a diagram of resolutions for $(G\looparrowright X)$ to a cartesian motive over a diagram of resolutions for $(H\looparrowright Y)$. This provides a description of the value of the right adjoint functor $(\phi,f)_\ast$. This will only be used in the case of pushforward along $(G\looparrowright X)\to (1\looparrowright Y)$ in which case it is clear that the pushforward along the structure map from the simplicial Borel construction is the right thing. Note also that we already have definitions of the functors $(\phi,f)_\ast$ in the situation of a morphism $(G\looparrowright X)\to (H\looparrowright Y)$ of varieties with action where $G\hookrightarrow H$ is a closed subgroup: one can factor the morphism as $(G\looparrowright X)\to (G\looparrowright Y)\to (H\looparrowright Y)$ and apply the ordinary pushforward for the first morphism and then the induction functor $\op{Ind}_G^H$, cf.~\ref{PBPF}. Uniqueness of right adjoints implies that the general direct image functor $(\phi,f)_\ast$ must agree with this composition of ordinary pushforward and induction in this situation.


\subsection{Motives for varieties with actions}

We shortly explain how to associate to a variety with action an equivariant motive. 

\begin{definition}
\label{def:equivmotive}
\index{equivariant motive for $G\looparrowright X$}
Let $G\looparrowright X$ be a variety with action, considered as the ``base scheme''; we'll mostly consider this in the case $G\looparrowright \pt$. Given a morphism $f:(G\looparrowright Y)\to (G\looparrowright X)$ of $G$-varieties, we define the \emph{equivariant motive of $G\looparrowright Y$} in $\mathbb{D}^+_G(X)$ as follows:
\[
\op{M}_X(G\looparrowright Y):= (\op{id},f)_! (\op{id},f)^! \const{X}_G,
\]
its Verdier dual is the cohomological equivariant motive $(\op{id},f)_\ast (\op{id},f)^\ast \const{X}_G$. 

Similarly, we have the \emph{equivariant Borel--Moore motive of $G\looparrowright Y$} given by
\[
\op{M}^{\op{BM}}_X(G\looparrowright Y):= (\op{id},f)_! (\op{id},f)^\ast \const{X}_G\cong (\op{id},f)_!\const{Y}_G
\]
and its dual, the \emph{equivariant motive with compact supports of $G\looparrowright Y$} is defined to be 
\[
\op{M}^{\op{c}}_X(G\looparrowright Y):= (\op{id},f)_\ast(\op{id},f)^!\const{X}_G.
\]
\end{definition}

\begin{Bemerkung}
This fits very well with the usual definitions of motive and motives with compact support in Definition~\ref{def:motive}. As a direct consequence of the definition (and the fact that the forgetful functor commutes with all the other six functors) we find that 
\[
\op{Res}_G^1(\op{M}_X(G\looparrowright Y))\cong \op{M}_X(Y),
\]
i.e., the underlying motive of the equivariant motive of $G\looparrowright Y$ is the ordinary, non-equivariant motive of $Y$. Similarly, we have for the motives with compact support
\[
\op{Res}_G^1(\op{M}_X^{\op{c}}(G\looparrowright Y))\cong \op{M}_X^{\op{c}}(Y).
\]
\end{Bemerkung}

\subsection{Properties: base-change and Verdier duality}

Now that we have definitions for the six functors connecting the various categories of equivariant motives, the basic properties of the six-functor formalism follow directly from the properties in the non-equivariant case.

\begin{Bemerkung}
\label{belu71}
The following is a motivic version of \cite[Theorem 7.1]{BeLu}. Let $\phi:H\to G$ be a homomorphism of linear algebraic groups. In a pullback diagram
\[
\xymatrix{
(H\looparrowright \tilde{X})\ar[r]^{(\id,g')} \ar[d]_{(\phi,f')} & (H\looparrowright X) \ar[d]^{(\phi,f)} \\
(G\looparrowright \tilde{Y}) \ar[r]_{(\id,g)} & (G\looparrowright Y)
}
\]
with $g$ smooth, we have
\begin{enumerate}
\item all the functors $\otimes$, $\op{Hom}$, $f^\ast\dashv f_\ast$, $f_!\dashv f^!$ and $(\phi,f)^\ast\dashv (\phi,f)_\ast$ between equivariant categories of motives commute with the smooth base-change $g^\ast$. 
\item The functor $g^\ast$ commutes with the Verdier duality up to twist by the dualizing object $D_{g',H}\in\mathbb{D}_H^+(\tilde{X})$. 
\end{enumerate}
The first claim follows directly from the non-equivariant smooth base change theorem. The second claim follows from \ref{belu1471}. 
\end{Bemerkung}

\begin{Bemerkung}
  From the homotopy and stability properties for homotopical stable algebraic derivators, we get the following, cf. \cite[Proposition 2.4.27]{ayoub:thesis1}: for a diagram $(\mathscr{F},\mathcal{I})$ of varieties, with $p:(\mathscr{F},\mathcal{I})\times_S\mathbb{A}^1_S\to    (\mathscr{F},\mathcal{I})$ and $s:(\mathscr{F},\mathcal{I})\to (\mathscr{F},\mathcal{I})\times_S\mathbb{A}^1_S$ the projection and  zero-section, respectively, we have  
\begin{itemize}
\item the morphism $\op{id}\to p_\ast p^\ast$ is invertible, 
\item the functor $p_{\sharp}s_\ast$ is an equivalence of categories.
\end{itemize}
In particular, the equivariant motives will satisfy an $\mathbb{A}^1$-invariance property.
\end{Bemerkung}

\begin{Bemerkung}[{\bf General base change}]
\label{BC} 
\index{base-change!equivariant}
Scholium 1.4.2 and its extension to diagrams on p.~323 of \cite{ayoub:thesis1} does not only imply the existence of adjoint functors, but implies that  $(\mathsf{H}^\ast,\mathsf{H}_\ast,\mathsf{H}_!,\mathsf{H}^!)$ is a cross functor for cartesian squares. The exchange properties encoded in the notion of cross functor imply in particular that for any square 
$$    \xymatrix{
      (H\looparrowright Y')\ar[r]^{(\phi,f)} \ar[d]_{(\op{id},q)} &
      (G\looparrowright X') \ar[d]^{(\op{id},p)}\\
      (H\looparrowright Y)\ar[r]_{(\phi,g)} &
      (G\looparrowright X)
    }
$$
which is cartesian for the underlying spaces, we have exchange isomorphisms 
$$
(\phi,g)^\ast (\op{id},p)_!\stackrel{\sim}{\RA}
(\op{id},q)_!(\phi,f)^\ast \quad 
\text{ and }\quad
(\op{id},p)^!(\phi,g)_\ast\stackrel{\sim}{\RA}
(\phi,f)_\ast(\op{id},q)^!. 
$$
\end{Bemerkung}

\begin{Bemerkung}[{\bf Proper base change}]
\label{BCp} 
For every morphism $(\op{id},f):(G\looparrowright X) \to (G\looparrowright Y)$, there is a morphism of functors $(\op{id},f)_!\RA (\op{id},f)_\ast$. It is invertible whenever $f$ is proper. Under our standing assumption \ref{standing}, this is equivalent to $f$ being projective. Using this, we  get variants of proper base change from \ref{BC}. 
\end{Bemerkung}

\begin{Bemerkung}[{\bf Smooth base change}]
\label{BCs} 
The purity statement of Scholium 1.4.2 and its extension to diagrams on    p. 323 \cite{ayoub:thesis1} implies in particular that for every morphism $(\op{id},f):(G\looparrowright X) \to (G\looparrowright Y)$ with $f$ smooth and  $d\pdef \op{dim}X-\op{dim}Y$, there is an isomorphism of functors 
$$(\op{id},f)^!\stackrel{\sim}{\RA} (\op{id},f)^\ast(d)[2d].$$ 
Using  this, we get  variants of smooth base change from \ref{BC}. 
\end{Bemerkung}


\begin{Bemerkung}
\label{CompRES} 
Let $H\ra G$ be a homomorphism of linear algebraic groups. Then the six functors defined above commute with the restriction functors $\op{Res}_G^H:\mathbb{D}_G^+(X)\to\mathbb{D}_H^+(X)$. This can be proved as \cite[Proposition 7.2]{BeLu}; essentially, it is a consequence of smooth base change. 

\label{prop:6fmorderivator}
Note that this implies in particular that the restriction functors induce morphisms of derivators for the derivators of Proposition~\ref{prop:equivderivator}. Since the restriction functors are left adjoints, it also implies that $\op{Res}_G^H$ is compatible with homotopy left Kan extensions. 
\end{Bemerkung}

\begin{Bemerkung}[{\bf Localization triangles}]
Let $G\looparrowright X$ be a variety with action, and let $Z\subseteq X$ be a $G$-stable closed subvariety with open complement $U=X\setminus Z$. Corollary 2.4.19 and Proposition 2.4.25 of \cite{ayoub:thesis1} imply that the units and counits of the adjunctions lead to canonical distinguished triangles in $\mathbb{D}^+_G(X)$:
\[
j_{!}j^!\to \op{id}\to i_\ast i^\ast \to j_!j^![1] \quad
\text{ and }\quad
i_! i^!\to \op{id}\to j_\ast j^\ast \to i_!i^![1].
\]
\end{Bemerkung}


\begin{Bemerkung}[{\bf Monoidal structure}]
\label{rem:monoidal}
\index{homotopical stable algebraic derivator!monoidal}
Definition 2.1.150 and 2.4.48 of \cite{ayoub:thesis1} introduce the notion of monoidal homotopical stable algebraic derivator, which requires that for each variety $X$, the category $\mathbb{D}(X)$ is a symmetric monoidal closed  triangulated category. Moreover, the functor $f^\ast:\mathbb{D}(Y)\to\mathbb{D}(X)$ is strong monoidal for $f:X\to Y$ a morphism of varieties, and some additional projection formulas hold, cf. the recollection in Appendix~\ref{sec:derivators}. 

As a consequence, if $\mathbb{D}$ is a monoidal homotopical stable algebraic derivator, then  the categories $\mathbb{D}^+_G(X)$ of equivariant motives are symmetric monoidal closed, and the pullback functor $(\phi,f)^\ast$ for a morphism $(\phi,f):(G\looparrowright  X) \to (H\looparrowright Y)$ is a strong monoidal functor. The usual projection formulae hold:
\begin{align*}
 f_!(M)\otimes N&\simeq  f_!(M\otimes f^\ast(N))\\
f^!\op{Hom}(M,N) &\simeq \op{Hom}(f^\ast M,f^!N)\\
f_\ast\op{Hom}(f^\ast M,N)&\simeq \op{Hom}(M,f_\ast N)\\
\op{Hom}(f_!M,N)&\simeq f_\ast\op{Hom}(M,f^! N).
\end{align*}
In particular, the forgetful functor is strong monoidal.
\end{Bemerkung}

\begin{Bemerkung}
\label{gmonderivator}
As discussed in Section~\ref{sec:derivators}, if we restrict a monoidal homotopical stable algebraic derivator to any fixed variety, the result is a monoidal derivator in the sense of \cite{groth:ponto:shulman} since the projection formula enforces homotopy cocontinuity of the monoidal structure. This implies, in particular, that $\mathbb{D}_G^+(X,-)$ will be a monoidal stable derivator in the sense of \cite{groth:ponto:shulman}. This will be relevant for our discussion of compatibility of tilting with such monoidal structures, cf. Section~\ref{sec:tilting}.
\end{Bemerkung}

\begin{Bemerkung}[{\bf Constructible objects}]
\label{rem:constructible}
\index{homotopical stable algebraic derivator!separated}
We recall the discussion of constructibility properties from Section 2.3 of \cite{ayoub:thesis1}. Constructibility is  best-behaved for a $\Lambda$-linear separated homotopical stable algebraic derivator $\mathbb{D}$ with $\Lambda$ a field of characteristic zero, where $\mathbb{D}$  is called separated if $f^\ast:\mathbb{D}(Y)\to\mathbb{D}(X)$ is conservative for $f:X\to Y$  surjective. Note that we are working over a base field $k$, which implies in particular that singularities can be resolved by alterations. 

In this case, the compact and the constructible objects of $\mathbb{D}(X)$ coincide for each $k$-variety $X$, and the six functors preserve constructibility, cf. \cite[Scholie 2.2.34]{ayoub:thesis1}. 

This implies in particular that for each equivariant morphism $(\phi,f):(H\looparrowright Y)\to(G\looparrowright X)$, the functors $(\phi,f)^\ast$ and $(\phi,f)_\ast$ preserve constructible objects. If $\phi=\id$, this also holds for $(\id,f)^!$ and $(\id,f)_!$. Moreover,  constructible objects are also preserved by $\otimes$ and $\op{Hom}$.
\end{Bemerkung}

\begin{Bemerkung}[{\bf Dualizing objects}]
\label{abs:dual}
\index{dualizing object!equivariant}
We continue to assume that $\mathbb{D}$ is a $\Lambda$-linear separated homotopical stable algebraic derivator over a field $k$. We outline how the Verdier duality formalism of Theorems 2.3.73 and 2.3.75 in \cite{ayoub:thesis1} provides a complete analogue of the equivariant Verdier duality \cite[Section 3.5, 3.6]{BeLu} in the motivic situation. 

For a variety with action $G\looparrowright X$, let $\const{\pt}_{G}\in\mathbb{D}^+_G(\pt)$ be the constant equivariant motive on the point, and  define the equivariant dualizing object to be  
\[
D_{X,G}:=(\id,\op{fin}_X)^!(\const{\pt}_{G})\in \mathbb{D}^+_G(X).
\]
Since the forgetful functor commutes with exceptional pullback by \ref{CompRES}, we find that $\op{Res}^1_G(D_{X,G})$ is isomorphic to the non-equivariant dualizing motive $\op{fin}_X^!\const{\pt}$. 

If $X$ happens to be smooth of dimension $d$, then purity \ref{BCs} implies 
\[
(\op{id},\op{fin}_X)^!(\const{\pt}_{G})\cong (\op{id},\op{fin}_X)^\ast(\const{\pt}_{G})(d)[2d]\cong \const{X}_{G}(d)[2d],
\]
i.e., up to twist and shift, the dualizing object is isomorphic to the constant equivariant motive. 
\end{Bemerkung} 

\begin{proposition}[{\bf Verdier duality}]
\label{equiv:verdier}
\index{Verdier duality!equivariant}
\begin{enumerate}
\item For a variety with action $G\looparrowright X$, the functors 
\[
D(-):=\op{Hom}_{\mathbb{D}^+_G(X)}(-,D_{X,G})
\]
are duality functors on the categories $\mathbb{D}_G^{\op{c}}(X)$ of constructible objects in $\mathbb{D}^+_G(X)$.
\item
For a morphism $(\op{id},f):(G\looparrowright X)\to (G\looparrowright Y)$ of varieties with $G$-action, the duality functors satisfy the following formulae:
\begin{align*}
D\circ (\id,f)^\ast&\mapright{\sim} (\id,f)^!\circ D,\\
(\id,f)^\ast \circ D&\mapright{\sim} D\circ (\id,f)^!,\\
D\circ (\id,f)_!&\mapright{\sim} (\id,f)_\ast\circ D,\\
(\id,f)_! \circ D&\mapright{\sim} D\circ (\id,f)_\ast.
\end{align*}
\item The Verdier duality commutes with restriction functors $\op{Res}_G^H:\mathbb{D}^+_G(X)\to\mathbb{D}^+_H(X)$. 
\end{enumerate}
\end{proposition}

\begin{proof}
(i) There is a canonical biduality morphism $M\to D(D(M))$. This can be obtained from the internal Hom into the dualizing object for diagrams in the same way it is obtained for schemes in \cite[Theorem 2.3.73]{ayoub:thesis1}. The fact that the biduality morphism is an isomorphism for constructible objects is then obtained by evaluation over individual resolutions, where it follows from the non-equivariant Verdier duality. 

(ii) Similarly, the exchange morphisms for duality and pullbacks resp. push-forwards is obtained as in \cite[Theorem 2.3.75]{ayoub:thesis1}, and we can check that these are isomorphisms over individual resolutions. 

The final statement (iii) follows from \ref{CompRES}. 
\end{proof}

\begin{Bemerkung}[{\bf Relative dualizing objects}]
\label{rel:dual}
\index{dualizing object!equivariant!relative}
More generally, for a morphism $(\op{id},f):(G\looparrowright X)\to (G\looparrowright Y)$ of varieties with action, one can define the equivariant relative dualizing object 
\[
D_{f,G}:=(\id,f)^!(\const{Y})\in\mathbb{D}^+_G(X).
\] 
Again this corresponds to the non-equivariant relative dualizing motive under the forgetful functor.  

In even greater generality, following \cite[Definition 7.4]{BeLu}, we can define equivariant relative dualizing objects for morphisms $(\phi,f):(H\looparrowright X)\to (G\looparrowright Y)$. Here we consider $(\op{id},f):(H\looparrowright X)\to (H\looparrowright Y)$ and defined 
\[
D_{f,H}:=(\op{id},f)^!(\const{Y}_H)\in\mathbb{D}^+_H(X). 
\]

As for the absolute case \ref{abs:dual}, if $f$ happens to be smooth the equivariant relative dualizing object is, up to twist and shift by the relative dimension, isomorphic to the constant equivariant motive.
\end{Bemerkung}

\begin{proposition}
\label{belu741}
Let $(\phi,f):(H\looparrowright X)\to(G\looparrowright Y)$ be a morphism of varieties with action, and assume that $f$ is smooth. Then the inverse image functors 
\[
(\phi,f)^\ast:\mathbb{D}^+_G(Y)\to\mathbb{D}_H^+(X)
\]
commute with Verdier duality up to a twist by the equivariant relative dualizing object of \ref{rel:dual}:
\[
D\circ (\phi,f)^\ast\approx D_{f,H}\otimes (\phi,f)^\ast\circ D.
\]
\end{proposition}

\begin{proof}
This follows from \ref{belu71} and \ref{CompRES}. 
\end{proof}

\subsection{Exterior product and exchange morphisms}

\begin{definition}[{\bf Exterior product}]
\label{def:exterior}
Assume we have varieties with action $G\looparrowright X$ and $H\looparrowright Y$. Then we define the \emph{exterior product functor  $\boxtimes$}\index{exterior product}\index{$\boxtimes$} by
\[
\boxtimes:\mathbb{D}^+_G(X)\times\mathbb{D}^+_H(Y)\to\mathbb{D}^+_{G\times H}(X\times Y):(M,N)\mapsto M\boxtimes N:=\op{pr}_1^\ast(M)\otimes \op{pr}_2^\ast(N). 
\]
\end{definition}

\begin{Bemerkung}
\label{comm} 
We discuss compatibility of the exterior product with the $\sharp$-functors, cf. \cite[Lemme 3.2.21]{ayoub:thesis2}. In the following, we use the definition of the category $\mathbb{D}^+_G(X)$ as cartesian motives over the Borel construction for $G\looparrowright X$. The functors $f_\sharp$, $f^\ast$ and $\otimes$ which we are using here are all defined in the simplicial setting, cf.~Section~\ref{sec:sixfunctors}.

Assume we have two smooth morphisms of varieties with action $f:(G\looparrowright X)\to(G\looparrowright X')$ and $g:(H\looparrowright Y)\to(H\looparrowright Y')$. Let $M\in\mathbb{D}^+_G(X)$ and $N\in\mathbb{D}^+_H(Y)$. Then there is a canonical morphism 
\[
(f\times g)_\sharp(M\boxtimes N)\to f_\sharp M\boxtimes g_\sharp N
\]
which is defined as follows. The functor $(f\times g)_\sharp$ is colax symmetric monoidal because it is adjoint to the symmetric monoidal functor $(f\times g)^\ast$, cf. \cite[1.1.24.1]{cisinski:deglise}. This means that we have a canonical morphism 
\[
(f\times g)_\sharp(M\boxtimes_{X\times Y}N)\to (f\times g)_\sharp(\op{pr}_1^\ast M)\otimes_{X'\times Y'} (f\times g)_\sharp(\op{pr}_2^\ast N).
\]
Smooth base change provides an isomorphism
\begin{eqnarray*}
(f\times g)_\sharp(\op{pr}_1^\ast M)\otimes_{X'\times Y'} (f\times g)_\sharp(\op{pr}_2^\ast N) & \stackrel{\cong}{\longrightarrow} &
f_\sharp(M)\boxtimes_{X'\times Y'}g_\sharp(N)\\&=&\op{pr}_1^\ast\circ f_\sharp(M) \otimes_{X'\times Y'}\op{pr}_2^\ast\circ g_\sharp(N),
\end{eqnarray*}
and the composition provides the required canonical morphism. 

Since the category of motives is generated by motives of smooth varieties, we can check that the canonical morphism is an isomorphism in the special case of motives of smooth varieties. In this case, i.e., when $M=\op{M}_X(V)$ and $N=\op{M}_Y(W)$ are motives of smooth varieties, we have that $f_\sharp$ is simply given by composition of the structure morphism with $f$, and $f^\ast$ is given by fiber product. The base change isomorphism is then simply 
the identification $\op{M}_{X'\times Y'}(V\times Y')\cong \op{pr}_1^\ast(\op{M}_{X'}(V))$. The second morphism, coming from compatibility of tensor with $\sharp$-functors is the canonical map 
\begin{eqnarray*}
\op{M}_{X'\times Y'}(V\times W) & \cong & 
\op{M}_{X'\times Y'}\left((V\times Y')\times_{X'\times Y'}(X'\times W)\right)\\
&\stackrel{\cong}{\longrightarrow}&
\op{M}_{X'\times Y'}(V\times Y')\otimes \op{M}_{X'\times Y'}(X'\times W),
\end{eqnarray*}
proving that the morphism exchanging $\sharp$-functors and exterior product is an isomorphism.
\end{Bemerkung}

\begin{Bemerkung}
\label{extprodshriek}
Some remarks concerning the relation between exceptional push-forward functors and exterior products are in order. In the end, we want to show that an exchange map is an isomorphism. However, the definition of the morphism requires passing through the ordinary push-forward functor which uses the definition of $\mathbb{D}^+_G(X)$ in terms of cartesian motives over appropriate diagrams of resolutions. Some care is needed to check that the natural non-equivariant argument also works for the categories of equivariant  motives. 

Assume we have two morphisms of varieties with action $f:(G\looparrowright X)\to(G\looparrowright X')$ and $g:(H\looparrowright Y)\to(H\looparrowright Y')$. Note that the morphisms $f$, $g$ and $f\times g$ do not change the group, which is relevant for the definition of the equivariant direct image functors. Let $M\in\mathbb{D}^+_G(X)$ and $N\in\mathbb{D}^+_H(Y)$. 

As right adjoint of the symmetric monoidal functor $f^\ast$, the functor $f_\ast$ is lax monoidal. Therefore,  we also get a canonical morphism
\[
(f\times g)_\ast\circ \op{pr}_1^\ast(M) \otimes_{X\times Y} (f\times g)_\ast\circ\op{pr}_2^\ast(N) 
\to (f\times g)_\ast(M\boxtimes N)
\]
It follows from the resolution definition and the corresponding non-equivariant statement, cf.  \cite[Lemme 2.3.6]{ayoub:thesis1}, that this morphism is in fact an isomorphism whenever $f$ is a closed immersion.  Combined with the proper base change above, we get an exchange isomorphism
\[
i_\ast M\boxtimes i'_\ast N \stackrel{\cong}{\longrightarrow} (i\times i')_\ast(M\boxtimes N)
\]
for closed immersions $i$ and $i'$.

There is also a morphism 
\[
f_! M\boxtimes g_! N\to (f\times g)_!(M\boxtimes N),
\]
arising either via the resolution or simplicial definition from the corresponding non-equivariant exchange morphism defined as one of the compositions in the commutative square of \cite[Corollary 2.3.49]{ayoub:thesis1}. As before, this morphism is an isomorphism whenever $f$ and $g$ are closed immersions. 

Assume now that $f$ and $g$ are smooth projective morphisms of relative dimensions $d$ and $d'$, respectively. In this situation, we have the purity isomorphisms, giving rise, by naturality, to a commutative diagram
\[
\xymatrix{
  (f\times g)_\sharp(M\boxtimes N)\ar[r]^\cong \ar[d]_\cong & f_\sharp(M)\boxtimes g_\sharp(N)\ar[d]^\cong \\
  (f\times g)_!(M\boxtimes N)(d+d')[2(d+d')] &
  f_!(M)(d)[2d]\boxtimes g_!(N)(d')[2d'] \ar[l] 
}
\]
Finally, if $f$ and $g$ are arbitrary quasi-projective morphisms, we can factor them as compositions of an open immersion, a closed immersion and a relative projective space. For the open immersion, we have $f_\sharp\cong f_!$, hence we get the comparison isomorphism from the one for $f_\sharp$. For the closed immersion, we noted above that the comparison morphism is an isomorphism for $f_\ast\cong f_!$. Finally, the relative projective space is smooth, and we also saw above that the comparison morphism for $f_!$ is an isomorphism. This shows that the comparison morphism 
\[
f_! M\boxtimes g_! N\to (f\times g)_!(M\boxtimes N)
\]
is an isomorphism for arbitrary finite type quasi-projective morphisms.

Alternatively, one can use the projection formula 
\begin{eqnarray*}
  (f_! M)\boxtimes  N &\cong &\op{pr}_1^*(f_! M)\otimes \op{pr}_2^* N\\
  &\cong& ((f\times\op{id})_!\op{pr}_1^* M)\otimes \op{pr}_2^* N\\
  &\cong& (f\times\op{id})_!(\op{pr}_1^* M \otimes (f\times\op{id})^*\op{pr}_2^* N)\\
  &\cong& (f\times\op{id})_!(\op{pr}_1^* M \otimes \op{pr}_2^* N)\\
  &\cong& (f\times \op{id})_!(M\boxtimes N)
\end{eqnarray*}
to show the case where one of the maps is the identity. The general case is then obtained by as composition of 
\begin{eqnarray*}
(f_! M)\boxtimes (g_!N)&\cong& (f_!\times\op{id})(M\boxtimes g_!(N)) \\ 
&\cong& (f_!\times\op{id})(\op{id}\times g_!)(M\boxtimes N) \\
&\cong& (f_!\times g_!)(M\boxtimes N). 
\end{eqnarray*}
\end{Bemerkung}

\section{Further consequences}		
\label{sec:further}

In this section, we discuss a couple of further consequences of the six functor formalism for equivariant motives. In particular, we provide some more refined remarks on quotient equivalences and the construction of integration functors. 

\subsection{Quotient and induction equivalences} 

\begin{proposition}[{\bf Generalized quotient equivalence}]
\label{prop:quotientequiv}
\index{quotient equivalence}
Let $G\looparrowright X$ be a variety with action and let $N\subset G$ be a closed normal subgroup such that the restricted action $N\looparrowright X$ is free. Then the morphism $(\pi,p)\colon(G,X)\to (G/N,N\backslash X)$ 
induces an equivalence of categories
\[
(\pi,p)^\ast\colon\mathbb{D}^+_{G/N}(N\backslash X)\sirra \mathbb{D}^+_G(X).  
\] 
The equivalence is compatible with all the six functors.
\end{proposition}

\begin{proof}
We adapt the proof of \cite[Theorem 2.6.2]{BeLu}. There is a functor on the categories of resolutions:
\[
\mathscr{E}:\op{Res}(G\looparrowright X)\to\op{Res}(G/N\looparrowright N\backslash X):
(p:P\to X)\mapsto (\bar{p}:N\backslash P\to N\backslash X).
\] 
Now \ref{prop:belu211} implies that this functor of resolutions is an equivalence of categories with inverse mapping $Q\to N\backslash X$ to $X\times_{/(N\backslash X)}Q\to X$. This equivalence is compatible with the quotient functors because $(G/N)\backslash (N\backslash P)\cong G\backslash P$. This implies that we get an equivalence of categories of cartesian motives
\[
\mathbb{D}^{\op{cart},+}(\mathscr{Q}_X\circ \mathscr{F})\approx
\mathbb{D}^{\op{cart},+}(\mathscr{Q}_{N\backslash X}\circ \mathscr{E}\circ \mathscr{F})
\]
where $\mathscr{F}:\mathcal{I}\to\op{Res}(G\looparrowright X)$ is a choice of skeleton. Via the comparison equivalences of Proposition~\ref{prop:richenough}, this proves the claim.

\end{proof}

The following is a version of \cite[Propositions 7.5.1 and 7.5.2]{BeLu}. 

\begin{proposition}
\label{prop:quotverdier}
Let $G\looparrowright X$ be a variety with action and let $N\subset G$ be a closed normal subgroup such that the restricted action $N\looparrowright X$ is free. Then the quotient equivalence 
\[
(\pi,p)^\ast\colon\mathbb{D}^+_{G/N}(N\backslash X)\sirra \mathbb{D}^+_G(X)
\]
commutes with the Verdier duality $D$ up a twist by the relative dualizing object $D_p=p^!(\const{N\backslash X}_G)$ of \ref{rel:dual}, i.e.,
\[
D\circ (\pi,p)^\ast \approx D_p\otimes (\pi,p)^\ast\circ D.
\]
Moreover, we have $D_p\cong \const{X}_G(d)[2d]$ where $d=\dim N$. In particular, the quotient equivalence commutes, up to twist and shift, with Verdier duality.
\end{proposition}

\begin{proof}
The first statements are a special case of Proposition~\ref{belu741}. For the specific description of the relative dualizing object, let $Y:=N\backslash X$, and let $P\to Y$ be a resolution of the quotient, viewed as $G$-variety via $G\to G/N\looparrowright Y$. Denote by $p^0(P)\to X$ the induced resolution of $X$, and consider the pullback diagram
\[
\xymatrix{
p^0(P) \ar[r]^q \ar[d]_{\tilde{p}} & p^0(P)/G \ar[d]^{\overline{p}} \\
P \ar[r]_q & P/G.
}
\]
Since $\overline{p}$ is smooth of relative dimension $d=\dim N$, we have 
\[
D_{\overline{p}}=\overline{p}^!\const{P/G}\cong \overline{p}^\ast\const{P/G}(d)[2d]\cong\const{p^0(P)/G}(d)[2d]
\]
by purity for (non-equivariant) motives. This implies, that the equivariant relative dualizing motive is isomorphic to $\const{X}_G(d)[2d]$ as claimed. 
\end{proof}

\begin{remark}
The explicit description of the equivariant relative dualizing object in \cite[Proposition 7.5.2]{BeLu} requires $G$ to be connected. This is due to orientability issues, cf. the proofs of \cite[Proposition 7.5.2, Lemma 7.5.3]{BeLu}. In the motivic setting, such issues do not arise, leading to a stronger result with a simpler proof -- the formalism has a built-in orientability for smooth maps, which is expressed exactly in the purity property of \ref{BCs}. 
\end{remark}

\begin{proposition}[{\bf Induction equivalence}]
\label{cor:indequiv}
\index{induction equivalence}
Let  $i\colon H\hookrightarrow G$ be the inclusion of a closed subgroup into a linear algebraic group $G$ and let $H\looparrowright X$ be a variety with action such that the diagonal $H$-action on  $G\times X$ is free. Then pullback along the obvious morphism $(i,s)\colon (H\looparrowright X)\to (G\looparrowright G\times_{/H} X)$ given by $s\colon x\mapsto [e,x]$ induces an equivalence of categories
\[
(i,s)^\ast \colon \mathbb{D}^+_{G}(G\times_{/H}X)\sirra \mathbb{D}^+_H(X)
\]
These equivalences are compatible with all six functors. 
\end{proposition}

\begin{proof}
We adapt \cite[Theorem 2.6.3]{BeLu}. We have an equivalence of categories
\[
\op{Res}(H\looparrowright X)\to\op{Res}(G\looparrowright G\times_{/H}X):(p:P\to X)\mapsto (G\times_{/H}P\to G\times_{/H} X).
\]
Moreover, the natural inclusions induce canonical isomorphisms on quotients 
\[
H\backslash P\cong G\backslash(G\times_{/H}P).
\]
The functor $(i,s)^\ast$, which is given by restriction for the above morphism of diagrams, is an equivalence on categories of cartesian motives. This implies the claim.
\end{proof}

Again, we have an analogue of \cite[7.6]{BeLu} on how the induction equivalence commutes with Verdier duality. 

\begin{proposition}
\label{prop:indverdier}
Let  $i\colon H\hookrightarrow G$ be the inclusion of a closed subgroup into a linear algebraic group $G$ and let $H\looparrowright X$ be a variety with action such that the diagonal $H$-action on  $G\times X$ is free. Then the induction equivalence 
\[
(i,s)^\ast \colon \mathbb{D}^+_{G}(G\times_{/H}X)\sirra \mathbb{D}^+_H(X)
\]
commutes with Verdier duality up to a twist by the equivariant relative dualizing object $D_s=s^!(\const{G\times_{/H}X}_H)\in\mathbb{D}_H^+(X)$ of \ref{rel:dual}, i.e.,
\[
D\circ (i,s)^\ast \approx D_s\otimes (i,s)^\ast\circ D.
\]
Moreover, we have $D_s\cong \const{X}_H(d)[2d]$ where $d=\dim H-\dim G$. In particular, the induction equivalence commutes, up to twist and shift, with Verdier duality.
\end{proposition}

\begin{proof}
Since $(i,s)^\ast$ commutes with inner Homs, it suffices to prove that $D_{X,H}\cong D_s\otimes (i,s)^\ast \left(D_{G\times_{/H}X,G}\right)$. By \ref{CompRES}, we can identify $\op{Res}_G^H \left(D_{G\times_{/H}X,G}\right)\cong D_{G\times_{/H}X,H}$. Hence, we need to show 
\[
D_{X,H}\cong D_s\otimes s^\ast \left(D_{G\times_{/H}X,H}\right). 
\]
This follows from \ref{belu1472}. 

For the explicit description of the relative dualizing motive, consider the pullback diagram
\[
\xymatrix{
X \ar[r]^s \ar[d]_{\op{fin}} & G\times_{/H}X \ar[d]^p \\
\op{pt} \ar[r]_i & G/H.
}
\]
Note that $p:G\times_{/H}X\to G/H$ is a locally trivial fibration with fiber $X$, and consequently we have by the base change formula \ref{belu1473} (resp. its equivariant analogue) 
\[
s^!\circ p^\ast \approx \op{fin}^\ast\circ i^!.
\]
Hence it suffices to show that $i^!(\const{G/H}_H)\cong \const{\op{pt}}_H(d)[2d]$ with $d=\dim H-\dim G$. Note that $i:\op{pt}\to G/H$ is a closed immersion of codimension $-d$ of smooth varieties. By absolute purity, we have $\const{\op{pt}}_H\cong i^!(\const{G/H}_H)(-d)[-2d]\in \mathbb{D}_H^+(\op{pt})$ which proves the claim. 
\end{proof}

\begin{proposition}
\label{prop:suk} 
Let $k$ be a perfect field. Let $\phi\colon G\sra H$ be a surjective homomorphism of linear groups over $k$, with connected unipotent kernel $U=\ker(\phi)$, and let $H\looparrowright X$ be a variety with $H$-action. Then the restriction functor  
\[
(\phi,\op{id})^\ast\colon\mathbb{D}_H(X)\to \mathbb{D}_G(X)
\]
is an equivalence of categories. 
\end{proposition}

\begin{proof}
Under the assumptions, $U$ is split and hence its underlying variety is $\mathbb{A}^n$. Now we can choose $U$ itself to be the space which is $\mathbb{A}^1$-contractible with free $U$-action. In particular, ${\op{B}}U$ is $\mathbb{A}^1$-contractible. The space $({\op{E}}G)/U$ has a free $H$-action and, by the above, all maps in ${\op{E}}G\to ({\op{E}}G)/U\to \pt$ are $\mathbb{A}^1$-weak equivalences. From the diagram
\[
\xymatrix{
  {\op{E}}G\times_{/G} X\ar[r]^-\sim & (({\op{E}}G)/U)\times_{/H}X\ar[r] &
  {\op{E}}H\times_{/H}X\\
  {\op{E}}G\times X \ar[r] \ar[u] \ar[d] & ({\op{E}}G)/U\times X \ar[r] \ar[u]
  \ar[d] & 
  {\op{E}}H\times X \ar[u] \ar[d] \\
  X \ar@{=}[r] & X\ar@{=}[r] & X
}
\]
we obtain equivalences
\[
\mathbb{D}_H(X)\stackrel{\approx}{\longrightarrow} \mathbb{D}^{\op{cart}}(X,H\backslash(({\op{E}}G)/U\times X)) \stackrel{\approx}{\longrightarrow} \mathbb{D}_G(X).
\qedhere
\]
\end{proof}

\begin{remark}
By Proposition~\ref{prop:suk}, general statements about linear algebraic groups can be reduced to reductive groups. We can assume that the
base field is perfect since purely inseparable extensions induce
equivalences on $\mathbb{D}$, and over perfect fields, Proposition~\ref{prop:suk} implies that $G$-equivariant motives are equivalent to $G/{\op{R}_{\op{u}}}G$-equivariant ones, and the latter is a reductive group.
\end{remark}

\begin{proposition}
Let $\phi\colon H\hra G$ be an injective homomorphism of linear algebraic groups. Assume that $G/H\cong\mathbb{A}^n$, and let $G\looparrowright X$ be a variety with $G$-action. Then the restriction functor  
\[
(\phi,\op{id})^\ast \colon \mathbb{D}_G(X)\to
\mathbb{D}_H(X)
\]
is fully faithful. 
\end{proposition}

\begin{proof}
We have an $\mathbb{A}^1$-fiber sequence $G/H\to {\op{E}}G\times_{/H}X\to {\op{E}}G\times_{/G}X$, the claim on full faithfulness then follows from the homotopy axiom for $\mathbb{D}$. 
\end{proof}

\subsection{Integration functors}
The six functor formalism for equivariant motives allows to provide analogues of the integration functors of \cite[Section 3.7]{BeLu}. 

\begin{Bemerkung}[{\bf Construction of integration functors}]
\label{PBPF} 
\index{integration functors}
\index{induction}
We now provide a construction of the integration functors $\op{Ind}_H^G$, using the induction equivalence. Therefore, let $i\colon H \hra G$ be the inclusion of a closed subgroup of the linear group $G$, and let $G\looparrowright X$ be a  variety with action. Then the balanced product $G \times_{/H} X$ exists and the morphism 
\[
G \times_{/H} X\xrightarrow{\cong} G/H\times X: [g,x]\mapsto (gH,gx)
\]
is an isomorphism. We also have the commutative diagram
\[
\xymatrix{
  & (G\looparrowright (G \times_{/H} X))\ar[dr]^{(\op{id},m)} & \\
  (H\looparrowright X)\ar[rr]^{(i,\op{id})}\ar[ur]^{(i,s)}& 
  & (G\looparrowright  X)
}
\]
and by Proposition~\ref{cor:indequiv} the functor $(i,s)_\ast$ is an equivalence. 

The right adjoint of the restriction functor $\op{Res}_G^H=(i,\id)^\ast$ is then the ordinary integration functor 
\[
\op{Ind}_H^G=(i,\op{id})_\ast\cong (\op{id},m)_\ast(i,s)_\ast
\]

Since $(i,s)^\ast$ is the induction
equivalence of Proposition~\ref{cor:indequiv}, we can also construct a left adjoint to the restriction functor $\op{Res}_G^H=(i,\op{id})^\ast= (i,s)^\ast  (\op{id},m)^\ast$. It is given as the composition 
\[
\op{Ind}_!\pdef  (\op{id},m)_{\sharp}(i,s)_\ast
\]
using \ref{def:ordinary}. This provides a motivic version of the exceptional integration functors of \cite[Theorem 3.7.1]{BeLu}. 
\end{Bemerkung}

\begin{remark}
 Arguably it might seem more natural to denote these integration functors $\op{Int}$ rather than $\op{Ind}$, but we decided to follow the notation established in \cite{BeLu}.
\end{remark}


\begin{proposition}[{\bf Compatibilities for integration}]
\label{prop:integration}
\index{integration functors, $\op{Ind}_{\ast}$}
Assume we have a variety with action $G\looparrowright X$, and let $i\colon H\hookrightarrow G$ be the inclusion of a closed subgroup in the linear algebraic group $G$. Then the restriction of the group action $\op{Res}_G^H\colon\mathbb{D}_G(X)\to\mathbb{D}_H(X)$ admits a right adjoint $\op{Ind}_\ast$ and a left adjoint $\op{Ind}_!$. 

The functor  $\op{Ind}_\ast$ commutes with $f_\ast$ and $f^!$, and the functor $\op{Ind}_!$ commutes with $f_!$ and $f^\ast$.  
\end{proposition}

\begin{proof}
Existence of the adjoints was established above in \ref{PBPF}. 

The compatibilities of $\op{Ind}_\ast$ follow from the identity  $(\op{id},f)(\phi,\op{id})=(\phi,f)=(\phi,\op{id})(\op{id},f)$ and from base change \ref{BC}. 
For the compatibilities of $\op{Ind}_!$, we 
recall our definition $\op{Ind}_!=(\op{id},m)_{\sharp}(i,s)_\ast$. In our situation, $(i,s):(H\looparrowright X)\to (G\looparrowright G\times_{/H} X)$ is a closed immersion, and hence  $(i,s)_\ast=(i,s)_!$ commutes with $f_!$. It also commutes with $f^\ast$ by base change \ref{BC}. On the other hand, the morphism $(\op{id},m):(G\looparrowright G\times_{/H} X)\to (G\looparrowright X)$ is smooth, and hence $(\op{id},m)_\sharp\approx (\op{id},m)_!(d)[2d]$ where $d=\dim (G/H)$. In particular, $(\op{id},m)$ commutes with $f_!$, and $f^\ast$ by  base change \ref{BC}.
\end{proof}


\begin{proposition}[{\bf Integration from parabolics}] 
\label{Lasp}
Assume that we have a variety with action $G\looparrowright X$, and let $P\subset G$ be a parabolic subgroup. Denote by $d=\dim G/P$ the codimension of $P$ in $G$. Then the integration functors $\op{Ind}_\ast,\op{Ind}_!:\mathbb D_P(X)\ra \mathbb D_G(X)$ from $P$-equivariant motives to $G$-equivariant ones are connected by the following natural isomorphisms for each $M\in\mathbb{D}_P(X)$ 
\[
\op{Ind}_!M \xrightarrow{\cong}   \op{Ind}_\ast M(d)[2d].
\]
\end{proposition} 

\begin{proof}
Recalling the construction of the integration functors in \ref{PBPF}, this boils  down to showing that there is a natural isomorphism $(\op{id},m)_{\sharp}\approx(\op{id},m)_\ast (d)[2d]$ where $m\colon G\times_{/P} X\ra X$ is the multiplication map coming from the action. However, as a variety  over $X$, $G\times_{/P} X$ is isomorphic to the projection  $(G/P)\times X\ra X$. Since $G/P$ is smooth, purity implies an isomorphism $(\op{id},m)_{\sharp}\approx (\op{id},m)_!(d)[2d]$, and since $G/P$ is proper, we have the natural isomorphism $(\op{id},m)_!\approx (\op{id},m)_\ast$. This proves the claim.
\end{proof}



\subsection{Projective bundle formula}

We shortly discuss an equivariant version of the projective bundle formula for motives. This result will be useful later in the formalism of Bott--Samelson motives of Section~\ref{sec:BS} where it will be employed to reduce statements about parabolic subgroups to the case of Borel subgroups.

Recall that the projective bundle formula for motives states the following: if $f:Y\to X$ is a projective bundle, i.e., locally in the \'etale topology on $X$, it is isomorphic to $\op{pr}_2:\mathbb{P}^n\times U\to U$, then $f_\ast(\underline{Y})\cong \bigoplus_{i=0}^n\underline{X}(i)[2i]$. The same is true more generally: if $f:Y\to X$ is an \'etale fiber bundle with fiber a smooth projective homogeneous space $G/P$, then $f_\ast(\underline{Y})\cong\op{M}(G/P)\otimes\underline{X}$. A similar statement is now true in the equivariant setting as well: 

\begin{proposition}
\label{prop:projectivebundle}
Let $G$ be a reductive group with Borel subgroup $B\subset G$, let $H$ be a closed algebraic subgroup and let $P\subset G$ be a parabolic subgroup containing $B$. For the natural projection $\pi:G/B\to G/P$ and the constant $H$-equivariant motive $\underline{G/B}\in \mathbb{D}_H^+(G/B)$, the motive $\pi_\ast(\underline{G/B})$ splits as direct sum of copies of $\underline{G/P}(i)[2i]$.
\end{proposition}

\begin{proof}
There are two possibilities to describe $\pi_\ast(\underline{G/B})$. By Proposition~\ref{prop:simpcart}, the equivariant push-forward functor $\pi_\ast$ can be described on the simplicial level, and we can also use resolutions. Both approaches work similarly, we choose the simplicial one: the map ${\op{E}}H\times_{/H}G/B\to{\op{E}}H\times_{/H}G/P$ is, in each of the simplicial degrees, a $P/B$-bundle. The constant motive $\underline{G/B}$ is obtained by pulling back the non-equivariant $\underline{G/B}$ in $\mathbb{D}(G/B)$ to the schemes $H^n\times G/B$ of $n$-simplices of ${\op{E}}H\times_{/H}G/B$. In each simplicial degree, the pushforward of this motive along the map $\op{id}\times \pi:H^n\times G/B\to H^n\times G/P$ splits as required, by the non-equivariant projective bundle formula. It remains to show that the splitting is compatible with the structure maps of the simplicial object, but this follows from the proper base change isomorphisms.
\end{proof}

A direct consequence of the projective bundle formula above is now the following computation: for an arbitrary equivariant motive $M\in \mathbb{D}_H^+(G/P)$, we have the following chain of isomorphisms:
\begin{eqnarray*}
\pi_\ast \pi^\ast M &\cong& \pi_\ast\left(\underline{G/B}\otimes \pi^\ast M\right) \\
&\cong& \pi_!\left(\underline{G/B}\otimes \pi^\ast M\right) \\
&\cong & \pi_!(\underline{G/B})\otimes M
\end{eqnarray*}
The first isomorphism follows since $\underline{G/B}$ is the unit object for the symmetric monoidal structure on $\mathbb{D}_H(G/B)$, the second isomorphism follows from properness of $\pi$, and the last isomorphism is the projection formula. The last term can now be written as a direct sum of motives $\underline{G/P}(i)[2i]\otimes M\cong M(i)[2i]$ by the above equivariant projective bundle formula.

\section{Convolution}
\label{sec:convolution}

In this section, we define (via the usual formulas) the convolution of equivariant motives, prove associativity and provide a reinterpretation of induction and restriction via convolution with suitably extended constant sheaves. 

For this, we assume throughout the section that we have a homotopical stable algebraic derivator $\mathbb{D}$ satisfying the conditions of \ref{derivator:new}. Particularly important here is that $\mathbb{D}$ is monoidal.

\subsection{Definition of the convolution bifunctor}

\begin{Bemerkung}
  Let $G \looparrowright X \looparrowleft R$ be a variety with two commuting actions of affine algebraic groups, with $G$ acting from the left and $R$ from the right. Let $P, Q \subset G$ be two closed subgroups. We introduce names for relevant maps: 
\begin{itemize}
\item
Denote by $\op{diag}=\op{id_P}\times\Delta\times\op{id}_R:P\times Q\times R\to P\times Q\times Q\times R$ the map induced from the diagonal $\Delta:Q\to Q\times Q$,
\item
denote by $\op{quot}:G\times X\to G\times_{/Q} X$ the quotient map from the product to the balanced product, and by 
\item $\op{mult}:G\times_{/Q} X\to X$ the multiplication map.
\end{itemize}
\end{Bemerkung}

\begin{definition}
\label{def:convolution}
\index{convolution}
Using these maps, we define the \emph{convolution bifunctor} 
\begin{eqnarray*}
-\conv_Q -\colon\mathbb{D}^+_{P \times Q}(G) \times \mathbb{D}^+_{Q\times R} (X) &\ra & \mathbb{D}^+_{P\times R}(X):\\
(M,N)&\mapsto& M \conv_Q N\pdef \op{mult}_! (M \ttimes N), 
\end{eqnarray*}
where $M \ttimes N \in \mathbb{D}^+_{P\times R}(G\times_{/Q}X)$ denotes the object obtained from the exterior product $M \boxtimes N$ under the composition
\[ 
\mathbb{D}^+_{P\times Q \times Q \times R}(G\times X) \mapright{\op{diag}^\ast} \mathbb{D}^+_{P\times  Q  \times R}(G\times X) \stackrel{\approx}{\longrightarrow} \mathbb{D}^+_{P\times  R}(G\times_{/Q} X).
\]  
Here the equivalence on the right is the generalized quotient equivalence of  Proposition~\ref{prop:quotientequiv}. 
\end{definition}


Now we will give a proof of associativity of the convolution. Note that we will give the construction of a natural isomorphism (alias an associativity constraint), but refrain from writing out a proof for the pentagon property. 

\begin{proposition}[\textbf{Associativity}]
\label{prop:convassoc}
Let $G \looparrowright X \looparrowleft Q$ be a variety with two commuting actions of affine algebraic groups, one from the left and one from the right. Let furthermore $P_1,P_2,P_3 \subset G$ be closed subgroups. Then for arbitrary objects $M_1\in \mathbb{D}^+_{P_1 \times P_2} (G)$, $M_2\in \mathbb{D}^+_{P_2 \times P_3} (G)$, $N\in \mathbb{D}^+_{P_3 \times Q} (X)$ there exists an isomorphism
\[
M_1\conv_{P_2}(M_2\conv_{P_3}N) \xrightarrow{\cong} (M_1\conv_{P_2}M_2)\conv_{P_3}N
\]
in $\mathbb{D}^+_{P_1\times Q}(X)$. 
\end{proposition}

\begin{proof}
(1) We first note that the expression $M_1\boxtimes M_2\boxtimes N$ is well-defined; the isomorphism
\[
\left(\op{pr_1}^\ast M_1\otimes \op{pr_2}^\ast M_2\right)\otimes \op{pr_3}^\ast N \cong \op{pr_1}^\ast M_1\otimes \left(\op{pr_2}^\ast M_2\otimes \op{pr_3}^\ast N\right)
\]
in $\mathbb{D}^+_{P_1\times P_2\times P_2\times P_3\times P_3\times Q}(G\times G\times X)$ is given by the associator of the monoidal structure $\otimes$. 

(2) 
We now want to show that there is a natural associativity constraint for $\tilde{\boxtimes}$. Recall that by definition we have 
\[
(M_1\tilde{\boxtimes} M_2)\tilde{\boxtimes} N:=\op{quot}\circ\op{diag}^\ast((\op{quot}\circ\op{diag}^\ast(M_1\boxtimes M_2))\boxtimes N).
\]
Now we have a diagram: 
\[
\xymatrix{
\mathbb{D}^+({}_{P_1}G_{P_2} \times {}_{P_2} G_{P_3}) \ar[rrr]^{(-)\otimes\op{pr}_3^\ast N\circ \op{pr}_{12}^\ast}\ar[d]_{\op{diag}(3)^\ast} &&& \mathbb{D}^+({}_{P_1}G_{P_2} \times {}_{P_2} G_{P_3}\times {}_{P_3}X_Q) \ar[d]^{\op{diag}(3)^\ast} \\
\mathbb{D}^+({}_{P_1}G \times_{\Delta P_2} G_{P_3}) \ar[d]_{\op{quot}} \ar[rrr]^{(-)\otimes\op{pr}_3^\ast N\circ \op{pr}_{12}^\ast} &&& \mathbb{D}^+({}_{P_1}G \times_{\Delta P_2} G_{P_3}\times {}_{P_3}X_Q) \ar[d]^{\op{quot}} \\ \mathbb{D}^+({}_{P_1}G\times_{/P_2} G_{P_3}) \ar[rrr]_{(-)\otimes\op{pr}_3^\ast N\circ \op{pr}_{12}^\ast} &&& \mathbb{D}^+({}_{P_1}G\times_{/P_2} G_{P_3}\times {}_{P_3}X_Q)
}
\]
Commutativity of the upper square follows from 2-functoriality of ordinary pushforwards and strong monoidality of $\otimes$. Commutativity of the lower square follows since the quotient equivalence is compatible with the six functors by Proposition~\ref{prop:quotientequiv}. Now the lower composition, i.e., $(-)\otimes\op{pr}_3^\ast N\circ\op{pr}_{12}^\ast\circ \op{quot}\circ\op{diag}^\ast$ applied to $M_1\boxtimes M_2$ computes exactly $(M_1\tilde{\boxtimes} M_2)\boxtimes N$. If we take the upper composition, i.e., $\op{quot}\circ\op{diag}^\ast\circ ((-)\otimes\op{pr}_3^\ast N)\circ\op{pr}_{12}^\ast$ and further compose this with $\op{quot}\circ \op{diag}(2)^\ast$, we get an object which we could call  $\op{quot}\circ\op{diag}^\ast(M_1\boxtimes M_2\boxtimes N)$. The above commutative diagram provides then an isotransformation
\[
\op{quot}\circ\op{diag}^\ast(M_1\boxtimes M_2\boxtimes N) \to (M_1\tilde{\boxtimes} M_2)\boxtimes N. 
\]
Note that this isotransformation is natural since it involves only natural transformations from the six functor formalism. A similar diagram then shows that there is also a natural isotransformation 
\[
\op{quot}\circ\op{diag}^\ast(M_1\boxtimes M_2\boxtimes N) \to M_1\tilde{\boxtimes} (M_2\boxtimes N). 
\]
This of course uses that $\op{quot}\circ\op{diag}(2)^\ast$ commutes with $\op{quot}\circ\op{diag}(2)^\ast$ and the fact that $M_1\boxtimes M_2\boxtimes N$ is actually well-defined, given by the natural associativity constraint step 1 above. Now we can compose the latter associativity constraint with the inverse of the former (and use the associativity constraint for $\boxtimes$ in the middle) to obtain a natural associativity constraint for $\tilde{\boxtimes}$. 

(3) Now we want to compare the pushforward of $(M_1\tilde{\boxtimes} M_2) \tilde{\boxtimes} N$ to the actual convolution. Recall that the motive $(M_1\star_{P_2} M_2)\star_{P_3}N$ would be defined as 
\[
(M_1\star_{P_2} M_2)\star_{P_3}N:=\op{mult}_! \left((\op{mult}_! M_1\tilde{\boxtimes} M_2)\tilde{\boxtimes} N\right).
\] 
Now we first have a commutative square
\[
\xymatrix{
\mathbb{D}^+({}_{P_1}G\times_{/P_2} G_{P_3}\times {}_{P_3}X_Q) \ar[d]_{\op{mult}(2)_!} \ar[r]^{\op{diag}(3)^\ast} &  \mathbb{D}^+({}_{P_1}G\times_{/P_2} G\times_{\Delta P_3}X_Q) \ar[d]^{\op{mult}(2)_!}  \\
\mathbb{D}^+({}_{P_1}G_{P_3}\times {}_{P_3}X_Q) \ar[r]_{\op{diag}(3)^\ast} & \mathbb{D}^+({}_{P_1}G\times_{\Delta P_3}X_Q) 
}
\]
which follows from the general base-change. We also have a commutative diagram which follows from the compatibility of the quotient equivalence with the six functors, cf. Proposition~\ref{prop:quotientequiv}:
\[
\xymatrix{
  \mathbb{D}^+({}_{P_1}G\times_{/P_2} G\times_{\Delta P_3}X_Q) \ar[d]^{\op{mult}(2)_!} \ar[r]^{\op{quot}} & \mathbb{D}^+({}_{P_1}G\times_{/P_2} G\times_{/P_3}X_Q) \ar[d]^{\op{mult}(2)_!} \\
 \mathbb{D}^+({}_{P_1}G\times_{\Delta P_3}X_Q) \ar[r]^\cong_{\op{quot}} & \mathbb{D}^+({}_{P_1}G\times_{/P_3}X_Q)
}
\]
Note that in particular the isotransformations making the diagrams commute are natural (this is relevant for the naturality of the associativity constraint constructed as composition of these isotransformations or their inverses). Using these commutative squares, there is a natural isotransformation between $(M_1\star_{P_2} M_2)\star_{P_3}N$ and 
\[
\op{mult}_!((M_1\tilde{\boxtimes} M_2)\tilde{\boxtimes} N)
\]
where here $\op{mult}_!$ actually means the composition for the multiplication maps for both groups $P_2$ and $P_3$. 

There is a small addendum: the fact that in the above we can actually compute the $\op{mult}_!$ on the triple product follows from the following diagram, whose commutativity follows from base change and the projection formula:
\[
\xymatrix{
\mathbb{D}^+({}_{P_1}G\times_{/P_2} G_{P_3}) \ar[d]_{\op{mult}_!} \ar[rrr]^{(-)\otimes \op{pr}_3^\ast N\circ \op{pr}_1^\ast} &&& \mathbb{D}^+({}_{P_1}G\times_{/P_2} G_{P_3}\times {}_{P_3} X_Q) \ar[d]^{\op{mult}_!} \\ \mathbb{D}^+({}_{P_1}G_{P_3}) \ar[rrr]^{(-)\otimes \op{pr}_3^\ast N\circ \op{pr}_1^\ast} &&& \mathbb{D}^+({}_{P_1}G_{P_3}\times {}_{P_3} X_Q)
}
\]

(4) As a consequence, we have a natural isotransformation 
\[
\op{mult}_!((M_1\tilde{\boxtimes} M_2)\tilde{\boxtimes} N)\to (M_1\star_{P_2} M_2)\star_{P_3}N.
\]
This provides one half of the associativity constraint. A symmetric argument  also gives a natural isotransformation $\op{mult}_!(M_1\tilde{\boxtimes} (M_2\tilde{\boxtimes} N))\to M_1\star_{P_2} (M_2\star_{P_3}N)$. Note that steps (1) and (2) showed that there is a natural associativity constraint for $\tilde{\boxtimes}$. Combining the inverse of the latter isotransformation, the associativity constraint for $\tilde{\boxtimes}$ in the middle and the former isotransformation, we get the associativity constraint for the convolution.
\end{proof}

\subsection{Induction and restriction}

\label{irc}
Let $G$ be a connected split reductive group, and let $Q\subset P$ be an inclusion of standard parabolic subgroups, and let $R\subset G$ be an arbitrary closed subgroup. We show how induction and restriction functors can be seen as applications of convolution. 


\begin{definition}
\index{constant motive $\leftidx{_Q}{\underline{P}}{_P}$}
\label{def:qpp}
Denote by $\leftidx{_Q}{\underline{P}}{_P}\in \mathbb D_{Q\times  P}(G)$ the constant $Q\times P$-equivariant motive on $P$, pushed forward along the $(Q\times P)$-equivariant inclusion $P\hookrightarrow G$. 
\end{definition}

\begin{proposition}
\label{prop:convres}
With this notation we have, for any equivariant motive $M\in\mathbb{D}^+_{P\times  R}(G)$, a natural isomorphism in $\mathbb{D}^+_{Q\times R}(G)$:
\[
\leftidx{_Q}{\underline{P}}{_P}\conv_P M\stackrel{\cong}{\longrightarrow} \op{Res}_{P\times R}^{Q\times R}M.
\]
\end{proposition}

\begin{proof}
We trace through the definition of the convolution, cf. Definition~\ref{def:convolution}. First, we have the motive $\leftidx{_Q}{\underline{P}}{_P}\boxtimes M\cong \op{pr}_1^\ast\leftidx{_Q}{\underline{P}}{_P}\otimes \op{pr}_2^\ast M$ in $\mathbb{D}^+_{Q\times P\times P\times R}(G\times G)$. By base change, this motive can be identified with the pushforward of $\op{pr}_2^\ast M$ along the $Q\times P\times P\times R$-equivariant inclusion $P\times G\hookrightarrow G\times G$. After restricting the equivariance to $Q\times P\times R$, we use the induction equivalence $\mathbb{D}^+_{Q\times P\times R}(G\times G)\xrightarrow{\simeq}\mathbb{D}^+_{Q\times R}(G\times_{/P}G)$. Note that the equivalence is induced from pullback along the projection $G\times G\to G\times_{/P}X$. Now consider the pushforward of the motive $\iota_\ast M$ on $P\times_{/P} G$ along the inclusion $\iota:P\times_{/P}G\hookrightarrow G\times_{/P}G$. The smooth base change implies that the pullback of $\iota_\ast M$ along $G\times G\to G\times_{/P}G$ is isomorphic to the pushforward of $\op{pr}_2^\ast M$ on $P\times G$ along $\tilde{\iota}:P\times G\hookrightarrow G\times G$. Therefore, the motive $\op{Res}_{\Delta}(\leftidx{_Q}{\underline{P}}{_P}\boxtimes M)$ corresponds, under the induction equivalence to $\iota_\ast M$ in $\mathbb{D}^+_{Q\times R}(G\times_{/P} G)$. Now the final step is an exceptional pushforward along $\op{mult}:G\times_{/P}G\to G$. However, the motive $\iota_\ast M$ we consider is concentrated on the closed subscheme $P\times_{/P} G\subseteq G\times_{/P}G$. The composition $G\cong P\times_{/P}G\subseteq G\times_{/P}G\xrightarrow{\op{mult}} G$ is the identity. In particular, the result of the proper pushforward $\op{mult}_!\iota_\ast M$ is simply the motive $M$ we started with (but of course, as $Q\times R$-equivariant motive). 
\end{proof}

Conversely, we have a statement concerning induction functors. Here the exceptional integration $\op{Ind}_!$ is understood to be from $Q\times R$ to $P\times R$. 

\begin{proposition}
\label{prop:convind}
For an equivariant motive $N\in\mathbb{D}^+_{Q\times  R}(G)$, we have natural isomorphisms in $\mathbb{D}^+_{P\times  R}(G)$
\[
\leftidx{_P}{\underline{P}}{_Q}(d)[2d]\conv_Q N\xrightarrow{\cong} \op{Ind}_!N
\] 
where $d=\op{dim}(P/Q)$. 
\end{proposition}

\begin{proof}
This follows from the construction in Proposition~\ref{prop:integration} resp. \ref{PBPF}. All we have to do is rewrite (as in the proof of Proposition~\ref{prop:convres}) the induction equivalence between $(Q\times R\looparrowright G)$ and $(P\times R\looparrowright G\times_{Q} G)$ as a composition
\[
\mathbb{D}^+({}_Q G_R)\xrightarrow{\op{pr}_2^\ast} \mathbb{D}^+({}_P G_Q\times {}_Q G_R) \xrightarrow{\op{diag^\ast}}\mathbb{D}^+({}_P G\times_{\Delta Q} G_R) \xrightarrow{\op{quot}} \mathbb{D}^+({}_PG\times_{/Q} G_R)
\]
and then use the relative purity comparison between $\op{mult}_\sharp$ and $\op{mult}_!$. 
\end{proof}

Moreover, since $Q\subset P$ is a parabolic subgroup, by Proposition~\ref{Lasp} we can rewrite $\op{Ind}_!\cong \op{Ind}_\ast(d)[2d]$ and thus get additional natural isomorphisms 
\[
\leftidx{_P}{\underline{P}}{_Q}\conv_Q N\xrightarrow{\cong}\op{Ind}_{Q\times R}^{P\times R} N.
\]

\begin{Bemerkung}
The above identification of restriction and exceptional induction as convolution with specific equivariant motives prompts the question about general adjointness properties of the convolution functors.  We shortly discuss these. Fix an equivariant motive $M\in \mathbb{D}_{P\times Q}^+(G)$ and consider the induced functor $M\star_P-:\mathbb{D}_{Q\times R}^+(X)\to \mathbb{D}_{P\times R}^+(X)$. The functor is given as the composition
\begin{eqnarray*}
\mathbb{D}_{Q\times R}^+(X)&\xrightarrow{\op{pr}_2^\ast}& \mathbb{D}^+_{P\times Q\times Q\times R}(G\times X)\\ &\xrightarrow{\op{pr}_1^\ast M\otimes-} & \mathbb{D}^+_{P\times Q\times Q\times R}(G\times X) \\ &\xrightarrow{\op{Res}_\Delta}& \mathbb{D}^+_{P\times Q\times R}(G\times X)\\ & \xrightarrow{\approx}& \mathbb{D}^+_{P\times R}(G\times_{/Q}X) \\&\xrightarrow{\op{mult}_!}& \mathbb{D}^+_{P\times R}(X). 
\end{eqnarray*}
All these functors are left adjoints, so the convolution $M\star_P-$ is also a left adjoint. The right adjoint is given as the composition of the respective right adjoints. The right adjoint of the composition $\mathbb{D}^+_{P\times Q\times R}(G\times X) \xrightarrow{\approx} \mathbb{D}^+_{P\times R}(G\times_{/Q}X) \xrightarrow{\op{mult}_!} \mathbb{D}^+_{P\times R}(X)$ is simply $\op{pr}_2^\ast(-)(d)[2d]$ where $d=\dim G/Q$. This follows since the inverse of the induction equivalence is the pullback, cf. Proposition~\ref{cor:indequiv}, and the right adjoint of $\op{mult}_!$ is $\op{mult}^!$. The latter can be identified with $\op{mult}^\ast(d)[2d]$ since $G\times_{/Q}X\to X$ is smooth of relative dimension $d$. Then the right adjoint of $\op{Res}_\Delta$ is the corresponding (ordinary) induction functor $\op{Ind}_{P\times Q\times R}^{P\times Q\times Q\times R}$; the right adjoint of $\op{pr}_1^\ast M\otimes -$ is $\iHom(\op{pr}_1^\ast M,-)$. Note that if $\op{pr}_1^\ast M$ is strongly dualizable, then $\iHom(\op{pr}_1^\ast M,-)\cong (\op{pr}_1^\ast M)^\vee\otimes -$. For the right adjoint to $\op{pr}_2^\ast$, we could either use the general direct image functors $(\op{pr}_2)_{\ast}$ (note that there is a change of groups involved here); the better way is to realize that in the above discussion we identified the right adjoint of the pullback along the projection in terms of an induction equivalence followed by a proper pushforward. Putting all the above together, we find that the adjoint of the convolution by $M$ (assumed strongly dualizable) can be written as convolution with the dual. The observation concerning restriction and induction as convolution above is a simple case of this.
\end{Bemerkung}

\section{Realization functors} 
\label{sec:realization}

We now provide a construction of realization functors on the equivariant motivic categories. 

\begin{proposition}
\label{prop:realexists}
\index{realization functor!equivariant}
Let $r:\mathbb{D}\to \mathbb{E}$ be a morphism of homotopical stable algebraic derivators on varieties over $k$ such that the derivators satisfy the conditions of \ref{derivator:new}. Then for every variety with action $G\looparrowright X$ there is a functor
\[
\mathbb{D}^+_G(X)\to\mathbb{E}^+_G(X). 
\]
These functors are compatible with the equivariant six functor formalism.
\end{proposition}

\begin{proof}
Note that morphism of homotopical stable algebraic derivators implies that we have for every diagram of varieties $(\mathscr{F},\mathcal{I})$ a functor $r_{\mathscr{F}}:\mathbb{D}(\mathscr{F},\mathcal{I})\to\mathbb{E}(\mathscr{F},\mathcal{I})$ and these functors commute with the six functors in the non-equivariant setting. We can apply this to the diagrams given by the simplicial Borel construction or a category of resolutions for the variety with action $G\looparrowright X$. Since the morphism $r$ commutes in particular with ordinary pullback $f^\ast$, it preserves cartesian objects. Then naturally, the functors $r_{{\op{E}}G\times_{/G}X}$ and $r_{\op{Res}(G\looparrowright X)}$ will restrict from the full categories of motives to the subcategories of cartesian objects, i.e., to equivariant motives. This provides the construction. 

By assumption, the morphism $r$ commutes with all the non-equivariant six functors. Since the equivariant six functors are given basically by evaluation of the non-equivariant functors over diagrams, the morphism $r$ will also commute with the equivariant six functors.
\end{proof}

\begin{proposition}
\label{prop:realidentify}
Let $\mathbb{K}$ be a coefficient field of characteristic zero and let $k$ be any field. Then there is a homotopical stable algebraic derivator $\mathbb{DK}$ sending a $k$-variety $X$ to the unbounded derived category of $\mathbb{K}$-sheaves on $X$. For every variety with action $G\looparrowright X$, we have an equivalence $\mathbb{DK}^+_G(X)\cong\op{Der}^+_G(X,\mathbb{K})$, compatible with the six functor formalism.
\end{proposition}

\begin{proof}
The existence of the derivator is clear. The construction of $\mathbb{DK}^+_G(X)$ of Section~\ref{sec:resolutions} restricts exactly to the construction of the equivariant derived category in \cite{BeLu} by definition.
\end{proof}

While the result doesn't quite apply to $\ell$-adic sheaves, we can use the results on \'etale realization in \cite{ayoub:realisation} and an argument as in Proposition~\ref{prop:realidentify}, to identify categories $\op{Der}_G(X,\mathbb{Q}_\ell)$ with the ``usual'' equivariant $\ell$-adic derived categories. 

Applying the previous results on equivariant versions of motivic categories and six functors, we obtain the following diagram of equivariant motivic categories and realization functors:
\begin{center}
  \begin{minipage}[c]{10cm}
    \xymatrix{
      \mathbf{DA}^{\et,+}_G(X;\Lambda) \ar[r]^{\mathsf{Real}_\ell}
      \ar[d]_{\mathsf{Real}_{\text{Hodge}}} &  \op{Der}^+_G(X,\mathbb{Q}_\ell)
        \\
      \mathbf{DH}^+_G(X) \ar[d]_{\op{Gr_W}} \\
      \op{MDer}^+_G(X;\mathbb{C}) 
      \ar[r]_{\mathsf{Real}_{\text{Betti}}} &  
      \op{Der}^+_G(X;\mathbb{C})
    }
  \end{minipage}
\end{center}

In the above, $\op{Der}_G(X,\mathbb{Q}_\ell)$ is an equivariant version of the $\ell$-adic derived category. However, the construction we use here - locally constant $\ell$-adic sheaves on the Borel construction - is a lot more transparent than the constructions usually found in the literature. 

The lower half of the diagram only applies for fields $k\subseteq\mathbb{C}$, as it is based on the Hodge realization. The categories $\mathbf{DH}_G(X)$ are equivariant versions of Drew's categories of mixed Hodge modules \cite{drew:thesis} which are built using sheaves of mixed Hodge structures on $X^{\op{an}}$. Passing to the associated graded, we obtain the categories $\op{MDer}_G(X,\mathbb{C})$ which are equivariant graded versions of the usual equivariant derived categories $\op{Der}_G(X;\mathbb{C})$ of sheaves (in the analytic topology on $X$) of $\mathbb{C}$-vector spaces.

All these categories have a fully working six functor formalism. This does not quite imply all six functors for arbitrary morphisms $(\phi,f):(G,X)\to (H,Y)$. The full six functors are only defined when $H=G$ and $\phi$ is the identity, but the pair $(\phi,f)^\ast\dashv(\phi,f)_\ast$ is defined in the full generality. All the basic formulas relevant for the development of equivariant derived categories in \cite{BeLu} also work in the motivic context. Moreover, the realization functors are compatible with the six functor formalism. In a sense, the above diagram is a diagram of equivariant motivic triangulated categories with pre-motivic adjunctions (in the sense of the framework of \cite{cisinski:deglise}) or a diagram of equivariant versions of homotopical stable algebraic derivator with appropriate morphisms, although we have chosen not to precisely axiomatize what these notions should be.

It is worth pointing out that even the construction of the $\ell$-adic equivariant derived category $\op{Der}_G(X;\mathbb{Q}_\ell)$ is more transparent in the setting of motivic categories. It is not given in terms of inverse limits of triangulated categories - rather, the usual $\ell$-adic derived category can be obtained as homotopy category of modules over a spectrum representing $\ell$-adic cohomology, cf. \cite{cd:weil}; and the equivariant version is given by considering these categories of modules over the Borel construction or over the category of resolutions. The point of introducing all these constructions here is that the equivariant versions of the functors  $\mathsf{Real}_\ell$ (over finite fields) and $\mathsf{Real}_{\text{Betti}}$ (over $\mathbb{C}$) will provide $\mathbb{Z}$-graded versions of the  usual equivariant derived categories as considered e.g. in \cite{BeLu} - with the $\mathbb{Z}$-grading simply and naturally given by the Tate twist readily available in any motivic category. 

We formulate the claims on existence and properties of the equivariant realization functors for later reference: 

\begin{proposition}
\label{prop:realization}
\begin{enumerate}
\item Let $k$ be a finite field or its algebraic closure, let $\Lambda\subseteq\mathbb{Q}_\ell$ be a field of characteristic zero and let $\mathbb{D}=\mathbf{DA}^{\et}(-;\Lambda)$. Then for every variety with action $G\looparrowright X$ over $k$, the \'etale realization induces $\ell$-adic realization functors
\[
\mathsf{Real}_\ell:\mathbf{DA}^{\et,+}_G(X;\Lambda)\to \op{Der}^+_G(X;\mathbb{Q}_\ell). 
\]
These commute with the six functor formalism, the quotient and induction equivalence and Verdier duality.
\item For every variety with action $G\looparrowright X$ over $\mathbb{C}$, there are Betti realization functors
\[
\mathsf{Real}_{\op{Betti}}:\op{MDer}^+_G(X;\mathbb{C})\to \op{Der}^+_G(X;\mathbb{C}). 
\]
These commute with the six functor formalism, the quotient and induction equivalence and Verdier duality.
\end{enumerate}
\end{proposition}

\begin{proof}
Existence and the fact that the realization commutes with all the six functors follows from Proposition~\ref{prop:realexists}. Identification of the target follows from Proposition~\ref{prop:realidentify}. The claimed compatibility with quotient and induction equivalences and Verdier duality follows then from the construction of these equivalences. 
\end{proof}

\chapter{Equivariant mixed Tate motives and purity}

The second part of the paper develops more specifically the categories of equivariant mixed Tate motives and their weight structures. In Section~\ref{sec:tatemotives} we define equivariant mixed Tate motives and introduce the equivariant Whitney--Tate condition necessary for these categories to be well-behaved. In Sections~\ref{sec:motcohom}, we provide some computations of morphisms between equivariant mixed Tate motives and equivariant cohomology. This is used in \ref{sec:tilting} to describe, via tilting, categories of equivariant mixed Tate motives over the point in terms of complexes of modules over the Chow ring of the classifying space. We also discuss compatibility of tilting with the six-functor formalism. Section~\ref{sec:weights} discusses conditions for existence of weight structures on equivariant mixed Tate motives. Actual examples of varieties with actions satisfying the equivariant Whitney--Tate conditions and applications of the framework developed in this part will be discussed in Chapter~\ref{chap:repthy}. 

\begin{convention}
From now on we will assume that the base field $k$ is algebraically closed of characteristic unequal to $2$. Though this is not strictly necessary for the definition of equivariant mixed Tate motives, it will become relevant for the discussion of our key examples. Assuming that the base field is algebraically closed simplifies the discussion of algebraic subgroups in $\op{PGL}_2$ (which is relevant for the theory of Bott--Samelson motives in Section~\ref{sec:BS}) and it implies that any two $k$-points in a homogeneous space $G/H$ are $G(k)$-conjugate (which is relevant for nice behaviour of our definition of equivariant mixed Tate motives). 
\end{convention}

\section{Equivariant mixed Tate motives} 
\label{sec:tatemotives}

In this section, we discuss the definition and basic properties of equivariant mixed Tate motives. As the name suggests, these should be equivariant motives whose underlying motives are pointwise mixed Tate. Our choice of making this a precise definition is the following: for a trivial group action on the point, we want equivariant motives whose underlying motive is mixed Tate in the usual sense. Then we can extend this to group actions on homogeneous spaces with rational points using the induction equivalence. Finally, for a variety with finitely many orbits, we can define equivariant mixed Tate motives as equivariant motives such that the $*$- and $!$-restrictions to orbits are equivariant mixed Tate. This approach only works for varieties with finitely many orbits, but this is sufficient for our application to graded categorification of Hecke algebras and their modules.

\subsection{Definition of equivariant mixed Tate motives}

\begin{Bemerkung}
\label{mtder}
\index{motives!mixed Tate}
Let $k$ be a field and 
let $\mathbb{D}$ be a homotopical stable algebraic derivator on $\mathrm{Var}_k$ satisfying the conditions of \ref{derivator:new}. Let $\pt = \op{Spec} k$ be the final object of $\mathrm{Var}_k$. Recall that the Tate object, denoted by 
$\const{\pt}(1)$ if the coefficients are clear, is the object of $\mathbb{D}(\pt)$ defined by 
\[
\const{\pt}(1) := \mathrm{cone}\left(\op{M}(\pt) \to \op{M}(\GG_{\op{m}})\right)[-1]. 
\]

The subcategory 
\[ 
\DMT(\pt)\subset \mathbb{D}(\pt) 
\]
of  mixed Tate objects over the point is defined to be the  strictly full thick, i.e., triangulated and closed under summands, subcategory of  $\mathbb{D}(\pt)$ generated by the tensor powers $\const{\pt}(n) \pdef \const{\pt}(1)^{\otimes n}$ of the Tate object for $n\in\ZZ$. This category is a tensor triangulated category, with 
\[
\const{\pt}(m)\otimes\const{\pt}(n)\cong\const{\pt}(m+n).
\]

The objects of $\DMT(\pt)$ are compact in $\mathbb{D}(\pt)$. 

The above definitions and statements more generally provide tensor triangulated categories of mixed Tate motives over arbitrary schemes. For any morphism of schemes $f:X\to Y$, there is a canonical isomorphism $\const{X}(1)\cong f^\ast\const{Y}(1)$. In particular, mixed Tate motives are preserved by ordinary pullbacks. 
\end{Bemerkung}

 We introduce a further condition on the derivator; this will be used in several places. It implies that equivariant motives on the point are basically objects from an equivariant derived category plus an additional grading.  It will also imply that the relevant equivariant motivic cohomology rings can be identified with the classical cohomology rings of the classifying space of the appropriate Lie group. 

\begin{convention}
\label{conditions:grading}
For explicit computations in the remainder of the paper and the resulting representation-theoretic applications, we will consider a more restricted setting in which the Tannaka group of mixed Tate motives over the point is just the multiplicative group: a homotopical stable algebraic derivator $\mathbb{D}$ over $\op{Var}/k$ satisfying the conditions of \ref{derivator:new} is said to satisfy the \emph{grading condition} if it has coefficients in a field $\Lambda$ of characteristic $0$ and $\DMT(\pt)$ is equivalent (as tensor-triangulated category) to the bounded derived category of finite-dimensional $\mathbb{Z}$-graded $\Lambda$-vector spaces. Moreover, for a finite field extension $f:\op{Spec} L\to \op{Spec} k$ of the base field $k$, we want that the ordinary pullback functors $f^\ast$ are fully faithful, i.e., induces equivalences $\DMT(\op{Spec} k)\simeq \DMT(\op{Spec} L)$.

The two main situations satisfying this condition are 
\begin{enumerate}
\item  $k=\mathbb{F}_q$ is a finite field, and $\mathbb{D}=\mathbf{DA}^{\et}(-,\DQ)$ is the homotopical stable algebraic derivator of motives with rational coefficients. 
\item $k=\mathbb{C}$ and $\mathbb{D}=\op{MDer}(-;\mathbb{C})$ is the homotopical stable algebraic derivator coming from the semisimplification of the Hodge realization, cf. \cite{drew:thesis}. 
\end{enumerate}
The condition on full faithfulness is satisfied for $\op{MDer}$ (if we define the derivator using a complex embedding $K\hookrightarrow\mathbb{C}$ followed by Hodge realization) since this ignores finite field extensions anyway. For rational motives over finite fields, this is satisfied because the morphisms between mixed Tate motives are given by $\op{K}_0$ and this is rationally also insensitive to finite field extensions.
\end{convention} 


\begin{definition}
\label{mtderdef1}
\index{motives!equivariant mixed Tate, $\DMT_G^?(X)$}
\index{motives!orbitwise mixed Tate}
Fix a base field $k$ and let $\mathbb{D}$ be a homotopical stable algebraic derivator satisfying the conditions of \ref{derivator:new} and assume $\DMT(\pt)\subset \mathbb{D}^+(\pt)$. We will now define \emph{equivariant mixed Tate motives} for varieties with action having finitely many orbits separably defined over $k$ in the sense of \ref{orbitsoverk}. 

\begin{enumerate}
\item 
First, consider the case where $G\looparrowright\pt$ is a linear group acting trivially on the  point. In this case, we say that an equivariant motive $M\in\mathbb{D}^+_G(\pt)$ is $G$-equivariantly mixed Tate if  the underlying motive  $\op{Res}^1_GM\in\mathbb{D}^+(\pt)$ is mixed Tate in the sense of \ref{mtder}. 
\item 
Next, consider the case where $G\looparrowright G/H$ is a group acting via left multiplication on the homogeneous space $G/H$ for $H$ a closed reduced subgroup defined over $k$. In this case, we say that an equivariant motive $M\in\mathbb{D}_G^+(G/H)$ is $G$-equivariantly mixed Tate if  it corresponds to an $H$-equivariant mixed Tate motive under the induction equivalence 
\[
\mathbb{D}_G^+(G/H)\approx \mathbb{D}^+_H(\pt)
\]
of Proposition~\ref{cor:indequiv}. 
\item 
Now assume that $G\looparrowright X$ is a variety with action having finitely many $G$-orbits separably defined over $k$. In this case, an equivariant motive $M\in\mathbb{D}_G^+(X)$ is called \emph{$?$-orbitwise mixed Tate} if for each orbit inclusion $j:G/H\hookrightarrow X$ the $?$-restriction $j^?M$ is $G$-equivariantly mixed Tate as defined above. Substituting $*$ and $!$ for $?$ provides two a priori different notions of orbitwise mixed Tate motives. A motive $M$ is called \emph{orbitwise $G$-equivariant mixed Tate} if it is both $*$-orbitwise and $!$-orbitwise $G$-equivariant mixed Tate. 
\end{enumerate}
For a variety with action $G\looparrowright X$ having finitely many $G$-orbits separably defined over $k$, we denote by
\begin{itemize}
\item $\DMT_G^?(X)$ the category of $?$-orbitwise $G$-equivariant mixed Tate motives, and by
\item $\DMT_G(X)$ the category of orbitwise $G$-equivariant mixed Tate motives.
\end{itemize}
\end{definition}

\begin{Bemerkung}
The definitions of equivariant categories of motives and the associated six-functor formalism only provided us with categories $\mathbb{D}^+_G(X)$ of motives which are bounded below in the homotopy t-structure. In general, it is not known if mixed Tate motives are bounded below for the homotopy t-structure. One way to make sure they are is to assume the Beilinson--Soul{\'e} conjectures. The essence of the Beilinson--Soul{\'e} conjectures is the assumption that mixed Tate motives are bounded for the homotopy t-structure, with an explicitly specified range where the cohomology is allowed to be non-trivial. The grading condition \ref{conditions:grading} makes sure that the constant mixed Tate motives $\const{X}(i)[j]$ are always in $\mathbb{D}^+(\pt)$ in all the situations we consider, as required in Definition~\ref{mtderdef1}.
\end{Bemerkung}

\begin{remark}
One of the motivations of the present work is to provide graded versions of usual equivariant derived categories (defined e.g. using $\ell$-adic sheaves). In the applications to categorification of Hecke algebras, conditions on action of Frobenius were usually imposed pointwise. For this reason, one would like to consider equivariant motives such that the restriction to each point is mixed Tate over the base. Generally, such pointwise notions are not particularly well-behaved. The orbitwise definition  above seems to provide the appropriate notion. 


Note also that the orbitwise definition appropriately captures pointwise notions because we are assuming that the base field is algebraically closed. 
If $k$ is a field which is not algebraically closed and $X=G/H$ is a homogeneous space, then it is not necessarily true that the induced $G(k)$-action on the $k$-rational points of $X$ is transitive, i.e., the stabilizer groups of different $k$-rational points might not be conjugate over $k$. If we then take an orbitwise mixed Tate motive (which produces a mixed Tate motive when restricted to the point used for identifying the homogeneous space as $G/H$), it is not clear that restriction to a $k$-rational point in a different $G(k)$-orbit will also be a mixed Tate motive. At least this nuisance is avoided by our assumption that the base field is algebraically closed.
\end{remark}




\begin{example}
\label{ex:locsys}
\index{local system}
For any group $G$, there is a constant equivariant Tate motive $\op{Res}_1^G\underline{\op{pt}}(i)$. To write down some non-trivial ``local systems'' in mixed Tate motives, consider  the morphism of varieties with $\mathbb{G}_{\op{m}}$-actions $p:\mathbb{G}_{\op{m}}\sra \mathbb{G}_{\op{m}}/(\pm 1)$. Then $p_\ast$ of the constant equivariant mixed Tate motive is an equivariant  mixed Tate motive. Essentially, an application of smooth base-change shows that the restriction of $p_\ast\const{\mathbb{G}_{\op{m}}}$ to any point of the orbit is a mixed Tate motive. The claim then follows from the definition of equivariant mixed Tate motives and the fact that the induction equivalence is induced by the restriction functor. It can also be shown that $p_\ast\const{\mathbb{G}_{\op{m}}}$ splits as a direct sum of a constant equivariant mixed Tate motive and an equivariant mixed Tate motive which corresponds to the ``sign representation'' of $\mu_2$ in $\DMT(\pt)$, cf.~\ref{finitetorsorNew}. 
\end{example}

\begin{example}
We will see in Proposition~\ref{Las} that the induction functors preserve equivariant mixed Tate motives. In particular, we will get an induction functor $\op{Ind}_1^G:\DMT(\pt)\to \DMT_G(\pt)$; for $G$ a finite group, the result of applying this functor to some mixed Tate motive on the point can be considered as a local system of mixed Tate motives on the classifying space ${\op{B}}G$. 

For a natural number $n$ we can consider the variety $\mathbb{G}_{\op{m}}\langle n \rangle$ with the $\mathbb{G}_{\op{m}}$-action of weight $n$. 
Assuming that the base field contains the $n$-th roots of unity $\mu_n$, the induction equivalence
\[ 
\mathbb{D}^+_{\mu_n}(\pt) \mapright{\approx} \mathbb{D}^+_{\mathbb{G}_{\op{m}}}(\mathbb{G}_{\op{m}}\langle n \rangle)
\]
provides an equivalence between $\mu_n$-representations in mixed Tate motives and ``local systems with $\mu_n$-monodromy''. 

We will also frequently encounter local systems of the form $\op{Ind}_{G^0}^G\underline{\pt}_{G^0}$ where $G$ is a linear group whose connected component of the identity is $G^0$.\footnote{We refrain from giving a definition of local system in mixed Tate motives here. In the triangulated setting we are in, the closest things to local systems would be objects of the equivariant derived category, and there would appear to be no need for introducing just another name for these. Whenever we refer to local systems later, this is only to remind the reader that we are thinking of equivariant mixed Tate motives as local systems in mixed Tate motives.}
\end{example}

\begin{Bemerkung}
\index{local system!rank of}
\label{locsysrank}
It follows from the grading condition in \ref{conditions:grading} there is an identification of the category of mixed Tate motives on the point with the category of graded vector spaces over the coefficient field. In this situation, we can also talk about the rank of local systems -- this would simply be the rank of the underlying mixed Tate motives, viewed as an element in the category of graded vector spaces. In particular, in this setting, rank one local systems would necessarily be indecomposable. 
\end{Bemerkung}

\begin{lemma}
\label{lem:idempotent}
In the situation of Definition~\ref{mtderdef1}, let $k$ be a field, let $G$ be a linear algebraic group and let $G\looparrowright X$ be a variety with action having finitely many orbits separably defined over $k$. Then the categories $\DMT_G^\ast(X)$ and $\DMT_G^!(X)$ are idempotent complete. 
\end{lemma}

\begin{proof}
We first discuss the case $X=\pt$. Let $p:M\to M$ be a projector in $\DMT_G(\pt)$. As a morphism in $\mathbb{D}_G^+(\pt)$, it splits. The forgetful functor $\op{Res}_G^1:\mathbb{D}_G^+(\pt)\to\mathbb{D}^+(\pt)$, being a left adjoint, commutes with the cokernel of the projector, i.e., $\op{coker}\op{Res}_G^1 p\simeq \op{Res}_G^1\op{coker} p$. In particular, the underlying motive of $\op{coker}p$ is isomorphic to the cokernel of the underlying projector of $p$. Since the category of mixed Tate motives is idempotent complete by Definition~\ref{mtder}, the cokernel of $p$ is an equivariant mixed Tate motive. The same argument, using the orthogonal projector, shows that the kernel is an equivariant mixed Tate motive.

The idempotent completeness transfers in the obvious way along the induction equivalence. Hence we get the claim in case of a homogeneous space $X=G/H$. 

In the general case, let $p:M\to M$ be a projector in $\DMT_G^?(X)$. As a projector in $\mathbb{D}_G^+(X)$ it splits. For any orbit inclusion $j:G/H\to X$, the left adjoint functor $j^\ast$ commutes with taking the cokernel and the right adjoint functor $j^!$ commutes with taking the kernel. In particular, $j^!\ker p\simeq \ker j^!p$ implies that the $!$-restriction of the kernel of $p$ along $j$ is a mixed Tate motive, because $\ker j^!p$ is. Thus $\ker p$ is a $!$-orbitwise mixed Tate motive, and similarly $\op{coker}p$ is a $\ast$-orbitwise mixed Tate motive. The same argument, using the orthogonal projector, finishes the proof.
\end{proof}

\begin{proposition}
\label{lkio}
In the situation of Definition~\ref{mtderdef1}, let $k$ be a field, let $G$ be a linear algebraic group and let $G\looparrowright X$ be a variety with action having finitely many orbits separably defined over $k$. 
\begin{enumerate}
\item The triangulated category $\DMT_G^*(X)$ of $G$-equivariant $*$-orbitwise mixed Tate motives is generated by objects of the form $j_!M$, for  $j\colon G/H\hookrightarrow X$ the inclusion of a $G$-orbit and $M\in \DMT_G(G/H)$.
\item Dually the triangulated category $\DMT^!_G(X)$ is generated by objects of the form $j_*M$, for  $j\colon G/H\hookrightarrow X$ the inclusion of a $G$-orbit and $M\in \DMT_G(G/H)$.
\end{enumerate}
\end{proposition}

\begin{proof}
The proof is essentially the same as that of \cite[Lemma 4.4]{SoWe}, by induction on the number of orbits. 
For $j:G/H\hookrightarrow X$ the inclusion of an open orbit and its closed $G$-stable complement $i:Z\hookrightarrow X$, we have the localization triangle 
\[
j_!j^\ast M\to M\to i_!i^\ast M\to j_!j^\ast M[1].
\]
The first term has the appropriate form by definition, and the third term has the required form by induction hypothesis. This proves (1), (2) is proved similarly.
\end{proof}

In general, we do not know whether $\ast$-pointwise and $!$-pointwise mixed Tate  motives are  the same. 

\begin{definition}
\label{def:mtderdef2}
\index{equivariantly Whitney--Tate}
\index{motives!equivariant mixed Tate}
Let $k$ be a field and let $G\looparrowright X$ be a variety with action having finitely many $G$-orbits separably defined over $k$. Then $X$ is called \emph{$G$-equivariantly Whitney--Tate} if 
\[
\DMT_G(X)=\DMT_G^\ast(X)=\DMT_G^!(X).
\]
If this condition is satisfied, $\DMT_G(X)$ will be called the \emph{category of equivariant mixed Tate motives}.
\end{definition}

\begin{remark}
Obviously $G\looparrowright G/H$ with action by left multiplication is equivariantly Whitney--Tate. Some nontrivial examples of equivariantly Whitney--Tate varieties will be exhibited in Sections~\ref{sec:BS} resp. \ref{sec:tiltingapp}.
\end{remark}

\subsection{Equivariant mixed Tate motives and six functors}

In this section, we discuss compatibility of the mixed Tate conditions with the six-functor formalism. While restriction is compatible with mixed Tate conditions (which however requires that restriction commutes with Verdier duality), the induction functors are only compatible with one of the orbitwise conditions.

\begin{lemma}
\label{lem:twistdmt}
In the situation of Definition~\ref{mtderdef1}, let $G\looparrowright X$ be a variety with action having finitely many separably defined orbits. Then both categories $\DMT^?_G(X)$, $?=\ast,!$, are stable under $M\mapsto M(i)[j]$.
\end{lemma}

\begin{proof}
Clear, since all the functors involved in the definition of $?$-orbitwise mixed Tate motives commute with twists and shifts.
\end{proof}

\begin{proposition}
\label{prop:dmtverdier}
In the situation of Definition~\ref{mtderdef1}, let $G\looparrowright X$ be a variety with action having finitely many separably defined orbits. Then the Verdier duality $D:\mathbb{D}^+_G(X)\to\mathbb{D}^+_G(X)$ restricts to an equivalence 
\[
D:\DMT_G^\ast(X)^{\op{op}}\sirra \DMT_G^!(X). 
\]
\end{proposition}

\begin{proof}
(1) By Proposition~\ref{equiv:verdier}, Verdier duality commutes with the forgetful functor. Hence an equivariant motive on the point is mixed Tate if and only if its Verdier dual is. 

(2) Now if $X=G/H$ is a single $G$-orbit, we have a diagram
\[
\xymatrix{
\mathbb{D}^+_G(X)^{\op{op}} \ar[r]^{(i,s)^\ast} \ar[d]_{D} & \mathbb{D}^+_H(\op{pt})^{\op{op}} \ar[d]^D \\
\mathbb{D}^+_G(X) \ar[r]_{(i,s)^\ast} & \mathbb{D}^+_H(\op{pt})
}
\]
By Proposition~\ref{prop:indverdier}, this diagram commutes up to twist and shift, i.e., 
\[
D\circ (i,s)^\ast \simeq (i,s)^\ast\circ D(d)[2d]
\]
with $d=\dim H-\dim G$. If $M\in\DMT_G(X)$, this means $(i,s)^\ast M\in \DMT_H(\op{pt})$, and then $(i,s)^\ast D(M) \cong D((i,s)^\ast M)(-d)[-2d]$. The latter is in $\DMT_H(\op{pt})$ because the category of equivariant mixed Tate motives on the point is closed under Verdier duality (by step 1) as well as twisting and shifting. This implies that $\DMT_G(X)$ is also closed under Verdier duality.

(3) Now we consider the general case $G\looparrowright X$. Let $j:G/H\hookrightarrow X$ be an orbit inclusion. By the equivariant Verdier duality, cf. Proposition~\ref{equiv:verdier}, $D\circ j^\ast\simeq j^!\circ D$. Hence, if $M\in\DMT^\ast_G(X)$, then $j^! D(M)\cong D(j^\ast M)$ is a $G$-equivariant mixed Tate motive on the orbit $j:G/H\hookrightarrow X$, by Step (2). This concludes the proof.
\end{proof}

\begin{proposition}
\label{prop:dmtres}
In the situation of Definition~\ref{mtderdef1}, let $k$ be a field, and let $G\looparrowright X$ be a variety with action. Let $H\to G$ be a homomorphism of algebraic groups and assume that the variety with action $H\to G\looparrowright X$ has finitely many orbits separably defined over $k$. Then the restriction functor $\op{Res}_G^H:\mathbb{D}^+_G(X)\to\mathbb{D}^+_H(X)$ restricts to $*$- and $!$-orbitwise equivariant mixed Tate motives, respectively. 
\end{proposition}

\begin{proof}
First, the claim in the case $X=\pt$ with the trivial $G$-action follows from the functoriality of restriction, more precisely $\op{Res}_H^1\circ \op{Res}_G^H\simeq \op{Res}_G^1$: the underlying motive of $\op{Res}_G^HM$ will be mixed Tate if and only if the underlying motive of $M$ is.

Now assume $X$ is a single $G$-orbit $G/G'$. In general, $H\to G\looparrowright G/G'$ will have several $H$-orbits. For a point $x\in G/G'$ with stabilizer subgroup $G_x\subset G$, we have a morphism of varieties with action $\iota_x:(G_x\looparrowright \op{pt})\to (G\looparrowright G/G')$. There is a commutative (even cartesian) diagram of varieties with action 
\[
\xymatrix{
(H_x\looparrowright \op{pt})\ar[r]^{\iota_x}\ar[d]  & (H\looparrowright G/G')\ar[d] \\
(G_x\looparrowright \op{pt})\ar[r]^{\iota_x}  & (G\looparrowright G/G')
}
\]
This induces the following diagram of functors on equivariant motivic categories.
\[
\xymatrix{
\mathbb{D}^+_G(G/G') \ar[r]^{\iota_x^\ast} \ar[d]_{\op{Res}_G^H} & \mathbb{D}^+_{G_x}(\op{pt}) \ar[d]^{\op{Res}_{G_x}^{H_x}} \\
\mathbb{D}^+_H(G/G')\ar[r]_{\iota_x^\ast} & \mathbb{D}^+_{H_x}(\op{pt})
}
\]
By Proposition~\ref{cor:indequiv}, $\iota_x^\ast$ induces the induction equivalence. Now we note that the lower vertical arrow factors as follows 
\[
\iota_x^\ast:\mathbb{D}^+_H(G/G')\xrightarrow{j^\ast}\mathbb{D}^+_H(H\cdot x)\to\mathbb{D}^+_{H_x}(\op{pt})
\]
where the first functor is restriction to the $H$-orbit $H\cdot x\cong H/H_x$ and the second is further restriction to the point $x$. Now a motive $M\in \mathbb{D}^+_G(G/G')$ is equivariant mixed Tate if and only if $\iota_x^\ast M\in\mathbb{D}^+_{G_x}(\op{pt})$ is. Similarly, the motive $\op{Res}_G^HM\in\mathbb{D}^+_H(G/G')$ is $*$-orbitwise mixed Tate if $\iota_x^\ast M\in\mathbb{D}_{H_x}^+(\op{pt})$ is. The commutativity of the above diagram together with the previous remarks on equivariant mixed Tate motives on the point thus implies that the restriction induces a functor
\[
\op{Res}_G^H:\DMT_G(G/G') \to \DMT_H^\ast(G/G').
\]

Now we still need to deal with the $!$-orbitwise $H$-equivariant mixed Tate motives. Since restriction commutes with Verdier duality by Proposition~\ref{equiv:verdier}, we have a commutative diagram
\[
\xymatrix{
\DMT_G(G/G') \ar[r]^{\op{Res}_G^H} \ar[d]_{D} & \DMT_H^\ast(G/G') \ar[d]^D \\
\mathbb{D}^+_G(G/G') \ar[r]_{\op{Res}_G^H} & \mathbb{D}_H(G/G')
}
\]
The upper horizontal arrow is the restriction functor to $\ast$-orbitwise $H$-equivariant mixed Tate motives discussed earlier; the lower horizontal is the ordinary restriction for equivariant motives. The vertical arrows are Verdier duality  functors. By Proposition~\ref{prop:dmtverdier}, the essential image of the left vertical is $\DMT_G(G/G')$, and the essential image of the right vertical is $\DMT_H^!(G/G')$. This shows that the restriction functor on equivariant mixed Tate motives induces a functor 
\[
\op{Res}_G^H:\DMT_G(G/G')\to\DMT_H^!(G/G'),
\]
which concludes our discussion of the single-orbit case $X=G/G'$.

Because restriction commutes with all the six functors, in particular $j^?$ for the inclusion of a $G$-orbit $j:G/G'\to X$, the claim for $*$- and $!$-orbitwise mixed Tate motives on varieties with several orbits follows from the case where $X$ is a single $G$-orbit discussed above. 
\end{proof}

\begin{proposition}
\label{prop:pullbackdmt}
In the situation of Definition~\ref{mtderdef1}, let $f:X\to Y$ be a morphism of $G$-varieties both having finitely many orbits separably defined over $k$. Then we have induced functors 
\[
f^\ast:\DMT^\ast_G(Y)\to \DMT^\ast_G(X)\quad\textrm{ and }\quad f^!:\DMT^!_G(Y)\to \DMT^!_G(X).
\]
If $f$ is additionally smooth, then both functors $f^\ast$ and $f^!$ preserve both categories $\DMT^\ast_G$ and $\DMT^!_G$, respectively. 
\end{proposition}

\begin{proof}
The second statement follows from the first by Verdier duality, cf. Proposition~\ref{prop:dmtverdier}. To prove the first, let $j:Z\hookrightarrow Y$ be an orbit in $Y$, and let $i:W\hookrightarrow X$ be an orbit in $f^{-1}(j(Z))$. Then we have a commutative diagram of pullback functors
\[
\xymatrix{
\mathbb{D}^+_G(Y) \ar[r]^{j^\ast} \ar[d]_{f^\ast} & \mathbb{D}^+_G(Z) \ar[d]^{f^\ast} \\
\mathbb{D}^+_G(X) \ar[r]_{i^\ast} & \mathbb{D}^+_G(W).
}
\]
In particular, it suffices to establish the case where $X$ and $Y$ are both homogeneous $G$-varieties. Note that in this case $f:G/G_X\to G/G_Y$ is a smooth surjection with fiber $G_Y/G_X$. Since the induction equivalence commutes with pullbacks, cf. Proposition~\ref{cor:indequiv}, it suffices to show that the restriction functor  $\op{Res}_{G_Y}^{G_X}:\mathbb{D}^+_{G_Y}(\op{pt})\to\mathbb{D}^+_{G_X}(\op{pt})$ preserves mixed Tate motives. But this was established in Proposition~\ref{prop:dmtres}. 

If $f$ is additionally smooth of relative dimension $d$, then purity, cf. \ref{BCs}, implies $f^!\simeq f^\ast(d)[2d]$. In particular, using Lemma~\ref{lem:twistdmt}, $f^!$ will preserve $\ast$-orbitwise mixed Tate motives and $f^\ast$ will preserve $!$-orbitwise ones. 
\end{proof}

We now show that the definition of equivariant mixed Tate motives is compatible with the generalized quotient equivalence of Proposition~\ref{prop:quotientequiv}. 

\begin{proposition}
\label{prop:quotientdmt}
In the situation of Definition~\ref{mtderdef1}, let $k$ be a field and let $G\looparrowright X$ be a variety with action having finitely many $G$-orbits separably defined over $k$. Let $N\hookrightarrow G$ be a normal subgroup which acts freely on $X$. We denote the quotient maps $\pi:G\to G/N$ and $p:X\to N\backslash X$. Then for $?=\ast,!$, the generalized quotient equivalence of Proposition~\ref{prop:quotientequiv} restricts to an equivalence
\[
(\pi,p)^\ast: \DMT_{G/N}^?(N\backslash X) \approx \DMT_G^?(X).
\]
\end{proposition}

\begin{proof}
We can factor the quotient equivalence functor $(\pi,p)^\ast$ as follows:
\[
\mathbb{D}^+_{G/N}(N\backslash X)\xrightarrow{\op{Res}_{G/N}^G} \mathbb{D}^+_G(N\backslash X)\xrightarrow{(\op{id},p)^\ast} \mathbb{D}^+_G(X),
\]
and it suffices to show that each of the constituent functors in the composition preserve $?$-orbitwise equivariant mixed Tate motives. For the restriction functor $\op{Res}_{G/N}^G$, this is Proposition~\ref{prop:dmtres}. For $(\op{id},p)^\ast$, this follows from Proposition~\ref{prop:pullbackdmt}, using that the projection morphism $p:X\to N\backslash X$ is smooth of relative dimension $\dim G/N$. 
\end{proof}

\begin{Bemerkung}
At this point, we want to remark that equivariant mixed Tate motives are compatible with the induction equivalence: in the situation of Proposition~\ref{cor:indequiv}, the induction equivalence induces an equivalence
\[
(i,s)^\ast:\DMT_G(G\times_{/H}X)\sirra \DMT_H(X).
\]
In the case where $X=H/H'$ has a single $H$-orbit, this is the definition of equivariant mixed Tate motives, because in both cases the categories of equivariant mixed Tate motives will be equivalent to $H'$-equivariant motives on the point via the respective induction equivalence. In the case of several orbits, only the $!$-orbitwise mixed Tate motives need an argument. But these are naturally preserved by $(\op{id},s)^!\circ\op{Res}_G^H$ which by a relative form of absolute purity in \ref{belu1473}, is isomorphic to $(\op{id},s)^\ast\circ\op{Res}_G^H$ up to twist and shift.
\end{Bemerkung}

\begin{proposition}
\label{Las} 
In the situation of Definition~\ref{mtderdef1}, let $k$ be a field and let $G\supset H$ be a  group with a closed subgroup.  Let $(G\looparrowright X)$ be a variety with action such that the restricted action $(H\looparrowright X)$ has finitely many orbits separably defined over $k$. Then the integration functor of \ref{PBPF} for the group action induces a functor 
\[
\op{Ind}_\ast=\op{Ind}_H^G:\DMT_H^!(X)\ra \DMT_G^!(X).
\]
\end{proposition}

\begin{proof} 
  By Corollary~\ref{lkio} it is sufficient to show that $\op{Ind}_H^Gj_{Z,\ast}M$ is contained in $\DMT_G^!(X)$ for any embedding $j_Z: Z\hra X$ of an $H$-orbit and any $H$-equivariant mixed Tate motive $M\in\DMT_H(Z)$. 

Now consider the composition $(H\looparrowright Z)\hra (H\looparrowright X)\hra(G\looparrowright X)$.
 It can be written as the  composition
\[
(H\looparrowright Z)\hra (G\looparrowright (G\times_{/H}Z)) \stackrel{\mu}{\ra} (G\looparrowright (G\cdot Z))\hra(G\looparrowright X),
\]
where $\mu$ denotes the action map and $j_{G\cdot Z}:G\cdot Z\hra X$ the obvious embedding of the $G$-orbit containing $Z$ in $X$. The corresponding commutative diagram of compositions shows that $\op{Ind}_H^G(j_{Z,\ast}M)$ is isomorphic to $j_{G\cdot Z,\ast}\mu_\ast\tilde M$, where $\tilde M\in \DMT_G^!(G\times_{/H}Z)$ is the motive on $G\times_{/H}Z$ corresponding to $M$ under the induction equivalence of Proposition~\ref{cor:indequiv}.  
By another application of \ref{lkio}, it suffices to show that for an $H$-equivariant mixed Tate motive on $Z$, the motive $\mu_\ast\tilde{M}$ is a $G$-equivariant mixed Tate motive on $G\cdot Z$. 

Under the induction equivalence of Proposition~\ref{cor:indequiv}, an $H$-equivariant mixed Tate motive in $\DMT_H(Z)$ corresponds to an $H_z$-equivariant motive on the point, where $H_z$ denotes the isotropy group of the orbit $Z\cong H/H_z$. Similarly, a $G$-equivariant motive in $\DMT_G(G\cdot Z)$ corresponds to a $G_z$-equivariant motive on the point, where now $G_z$ denotes the isotropy group of the orbit $G\cdot Z\cong G/G_z$.

We get a commutative diagram
\[
\xymatrix{
\DMT_{H_z}(\pt)\ar[r] \ar[d] & \DMT_H(Z)\ar[d] \\
\mathbb{D}^+_{G_z}(\pt)\ar[r] & \mathbb{D}^+_G(G\cdot Z)
}
\]
where the horizontal arrows are induction equivalences, the left vertical arrow is an integration functor for the inclusion $H_z\to G_z$, and the right vertical arrow is the integration functor for the inclusion $H\to G$ restricted to the appropriate orbits. To show that the right vertical arrow lands in $\DMT_G(G\cdot Z)$, it suffices to show that the left vertical arrow lands in $\DMT_{G_z}(\pt)$, i.e., that the induction induces a functor $\op{Int}_{H_z}^{G_z}:\DMT_{H_z}(\pt)\to\DMT_{G_z}(\pt)$. Now we can use the definition of induction, cf. \ref{PBPF}: take an $H_z$-equivariant motive on the point, use the induction equivalence to get a $G_z$-equivariant motive on $G_z\times_{/H_z}\pt\cong G_z/H_z$, and then push forward along $(G\looparrowright G_z/H_z)\to (G\looparrowright \pt)$. The induction equivalence $\mathbb{D}^+_{H_z}(\op{pt})\approx \mathbb{D}^+_{G_z}(G_z/H_z)$ is compatible with mixed Tate motives, by the very definition.

It remains to show that $\op{Res}_{G_z}^1\op{fin}_{G_z/H_z,\ast} M$ is a mixed Tate motive whenever $M$ corresponds to an $H_z$-equivariant mixed Tate motive on the point under the induction equivalence. By smooth base-change, we can first forget about the equivariance, and just consider the non-equivariant push-forward $\op{fin}_{G_z/H_z,\ast} M$ where $M$ is the underlying motive corresponds to an $H_z$-equivariant mixed Tate motive on the point under the induction equivalence. A look at the induction equivalence \ref{cor:indequiv} tells us that such a motive must become a constant mixed Tate motive after pullback along a resolution $Q\to G_z/H_z$. Now we can take a resolution $p:P\to \pt$ for $(G_z\looparrowright \pt)$ and consider the pullback to a resolution $f^\circ(p):G_z/H_z\times P\to G_z/H_z$ of $(G_z\looparrowright G_z/H_z)$. Note that the induced map on resolutions $G_z/H_z\times P\to P$ is an \'etale fiber bundle with fiber $G_z/H_z$. By an application of smooth base change, it suffices for an \'etale fiber bundle $q:E\to B$ with fiber $G_z/H_z$ that $q_\ast$ maps constant mixed Tate motives on $E$ to mixed Tate on $B$. Compatibility with shift and twist reduces us to show that $q_\ast\underline{X}$ is mixed Tate on $Y$ if $q:X\to Y$ is an \'etale fiber bundle with fiber $G/H$ an arbitrary homogeneous space. This follows from the well-known fact that the motive of $G_z/H_z$ is mixed Tate, cf. Proposition~\ref{prop:tatehomogeneous}. 
\end{proof}

\begin{remark}
In the end, we will only use this in very specific cases (for inclusions of parabolic groups) for the study of Bott--Samelson motives, cf. Section~\ref{sec:BS}.
\end{remark}

\begin{remark}
\label{Last} 
Since Verdier duality $D$ commutes with restrictions of the group action, compare \ref{equiv:verdier}, we can write the left adjoint of the restriction functor $\op{Res}_G^H$ on the subcategories of constructible objects as 
\[
\op{Ind}_!=D\circ \op{Ind}_\ast\circ D.
\]
Here, $\op{Ind}_\ast=\op{Ind}_P^G$ is the right adjoint of restriction, the ordinary integration functor. Using Verdier duality in Proposition~\ref{prop:dmtverdier} together with Proposition~\ref{Las} above, this exceptional integration functor will induce a functor 
\[
\op{Ind}_!\colon \DMT_P^\ast(X)\ra \DMT_G^\ast(X).
\]
\end{remark}

\begin{theorem}
\label{ifpt} 
In the situation of Definition~\ref{mtderdef1}, let $k$ be a field and let $G$ be a split reductive group with parabolic subgroups $P\subset Q\subset G$. Let $(G\looparrowright X)$ be a variety with action such that the restricted action $(P\looparrowright X)$ has finitely many orbits separably defined over $k$. Then both integration functors $\op{Ind}_\ast$ and $\op{Ind}_!$ (for integration from $P$-equivariant to $Q$-equivariant motives) preserve both subcategories $\DMT^\ast$ and $\DMT^!$ of $\ast$-pointwise and of $!$-pointwise mixed Tate motives.
\end{theorem}

\begin{proof}
  This follows immediately from Proposition~\ref{Las} together with the facts that the integration functors coincide up to twist and shift, cf.~\ref{Lasp}, and are also Verdier dual, cf.~\ref{Last}.
\end{proof}

\begin{proposition}
\label{prop:monodmt}
In the situation of Definition~\ref{mtderdef1}, let $G\looparrowright X$ be a variety with action having finitely many separably defined orbits. Then the monoidal structure $\otimes:\mathbb{D}^+_G(X)\times\mathbb{D}^+_G(X)\to\mathbb{D}^+_G(X)$ restricts to $\DMT^\ast_G(X)$. 
\end{proposition}

\begin{proof}
The statement is true in case $X=\op{pt}$, since the forgetful functor is strongly monoidal by \ref{rem:monoidal} and the category of non-equivariant mixed Tate motives is a $\otimes$-subcategory. The statement for $X=G/H$ a homogeneous variety follows, since the induction equivalence is a strongly monoidal functor, cf. Proposition~\ref{cor:indequiv}. For $\ast$-orbitwise mixed Tate motives on general $X$, the claim follows since $j^\ast$ is strongly monoidal, for orbit inclusions $j:G/H\to X$. 
\end{proof}

\subsection{Conservativity for equivariant mixed Tate motives} 

Now we discuss the interaction between equivariant mixed Tate motives and realization functors of Section~\ref{sec:realization}. Under suitable assumptions on the derivator, the realization functors are conservative on equivariant mixed Tate motives, which will be used for some computations later.

\begin{proposition}
\label{prop:conservative}
Let $G\looparrowright X$ be a variety with action, and assume it is $G$-equivariantly Whitney--Tate. 
\begin{enumerate}
\item Assume that everything is defined over a finite field $\mathbb{F}_q$ and that $\mathbb{D}=\mathbf{DA}^{\et}(-;\Lambda)$ with $\Lambda$ a field of characteristic zero contained in $\mathbb{Q}_\ell$. Then the $\ell$-adic realization functor $\op{Real}_\ell:\DMT_G(X)\to \op{Der}^{\op{b}}_G(X;\mathbb{Q}_\ell)$ is conservative. 
\item Assume that everything is defined over $\mathbb{C}$ and that $\mathbb{D}=\op{MDer}(-;\mathbb{C})$ is the derivator associated to the semi-simplified Hodge realization. Then the Hodge realization functor $\op{Real}_{\op{H}}:\DMT_G(X)\to\op{Der}^{\op{b}}_G(X;\mathbb{C})$ is conservative. 
\end{enumerate}
\end{proposition}

\begin{proof}
We prove (1), statement (2) is proved similarly. 

(i) We first establish the special case $X=\pt$. Consider the diagram
\[
\xymatrix{
\DMT_G(\pt)\ar[r]^{\op{Real}^G_\ell} \ar[d]_{\op{For}} & \op{Der}^{\op{b}}_G(\pt;\mathbb{Q}_\ell) \ar[d]^{\op{For}} \\
\DMT(\pt)\ar[r]_{\op{Real}_\ell} & \op{Der}^{\op{b}}(\pt;\mathbb{Q}_\ell)
}
\]
which commutes because realization is compatible with the forgetful functors.  The forgetful functors in the verticals are conservative, cf.~\ref{forget}. By assumption on the derivator, the lower horizontal is isomorphic to the functor 
\[
\op{Der}^{\op{b}}(\Lambda\op{-Mod}^{\mathbb{Z}})\to\op{Der}^{\op{b}}(\mathbb{Q}_\ell\op{-Mod})
\]
which extends coefficients to $\mathbb{Q}_\ell$ and forgets the $\mathbb{Z}$-grading, which is also obviously conservative. Therefore, the composition $\op{For}\circ\op{Real}^G_\ell$ is conservative, which implies that $\op{Real}^G_\ell$ is conservative, as claimed.

(ii) Now we consider the case where $X=G/H$ is a homogeneous space with the natural left $G$-action. In this case, we can consider the diagram 
\[
\xymatrix{
\DMT_G(G/H)\ar[r]^{\op{Real}^G_\ell} \ar[d]_{\approx} & 
\op{Der}^{\op{b}}_G(G/H;\mathbb{Q}_\ell) \ar[d]^{\approx}\\
\DMT_H(\pt)\ar[r]_{\op{Real}^H_\ell} &\op{Der}^{\op{b}}_H(\pt;\mathbb{Q}_\ell)
}
\]
which is commutative because the realization functor is compatible with the functors appearing in the induction equivalence of Proposition~\ref{cor:indequiv}. The lower horizontal is conservative, as established in (i) above. Since the induction equivalences in the verticals are obviously conservative, this establishes conservativity in the case $X=G/H$. 

(iii) For the general case, we use an induction, noting that the assumptions imply that $X$ has finitely many $G$-orbits separably defined over $k$. The $G$-variety $X$ will have an open orbit $j:U\hookrightarrow X$ with $G$-stable closed complement $i:Z\hookrightarrow X$, and the localization sequence implies that the pair $(i^\ast, j^\ast)$ is conservative. 

Now assume that an equivariant mixed Tate motive $M\in\DMT_G(X)$ has trivial realization, and consequently the restriction of $\op{Real}_\ell(M)$ to $U$ is also trivial. Since the realization functor is compatible with $j^\ast$, we have $j^\ast_\ell\circ\op{Real}_\ell(M)\cong \op{Real}_\ell(j^\ast M)$, and the assumption says that this is trivial. By the orbit case of conservativity established in (ii) above, the restriction $j^\ast M$ is also trivial. 
For the closed complement $i:Z\hookrightarrow X$, the inductive assumption states that the realization $\op{Real}_\ell:\DMT_G(Z)\to\op{Der}^{\op{b}}_G(Z;\mathbb{Q}_\ell)$ is conservative. Since the realization functor is compatible with $i^\ast$, an argument as above deduces from the inductive assumption that $i^\ast M$ is trivial. Now conservativity of the pair $(i^\ast, j^\ast)$ implies that the motive $M$ is trivial. 
\end{proof}

\begin{remark}
More generally, for a derivator satisfying the grading condition \ref{conditions:grading}, the morphism of derivators which forgets the grading would always be conservative, by an argument as in the above proof.
\end{remark}

\subsection{Generators for equivariant mixed Tate motives}

In this section, we show that the categories of equivariant mixed Tate motives over the point are generated (as thick triangulated subcategories of the equivariant motives) by inductions of constant motives. Moreover, this induction is only relevant in the case of non-connected algebraic groups. For a connected algebraic group $G$, the category $\DMT_G(\pt)$ of $G$-equivariant mixed Tate motives on the point is generated by motives whose evaluation over ${\op{B}}G$ is constant mixed Tate. This is basically the statement that there are no ``local systems'' on ${\op{B}}G$ with values in mixed Tate motives whenever $G$ is a connected algebraic group. 

\begin{proposition}
\label{prop:locsysgm}
Assume that the underlying derivator satisfies conditions~\ref{derivator:new} and the grading condition~\ref{conditions:grading}. Let $M\in\mathbb{D}(\mathbb{P}^\infty)$ be a motive whose restriction along $p:\mathbb{A}^\infty\setminus\{0\}\to\mathbb{P}^\infty$ is constant mixed Tate. Then $M$ is already a constant mixed Tate motive. 
\end{proposition}

\begin{proof}
The claim is that the evaluation of $M$ on $\mathbb{P}^\infty$ is already a constant mixed Tate motive on $\mathbb{P}^\infty$. More precisely, denoting $p:\mathbb{A}^\infty\setminus\{0\}\to\mathbb{P}^\infty$ the natural projection, we have a constant mixed Tate motive $N:=p^\ast M(\mathbb{P}^\infty)\in \DMT(\mathbb{A}^\infty\setminus\{0\})\approx\DMT(\pt)$, and we claim that there is an isomorphism $M\cong N$ in $\DMT(\mathbb{P}^\infty)$. 

This is done by descent theory. The motives $M$ and $N$ correspond to descent data on $p^\ast M$ and $p^\ast N$, respectively, and we want to show that these descent data are compatible with the identification $p^\ast M\cong p^\ast N$. The descent datum for $p^\ast M$ is basically given by an action of $\mathbb{G}_{\op{m}}$ on $p^\ast M$: for every extension field $F$ of the base field and every $F$-point $\lambda:\op{Spec} F\to\mathbb{G}_{\op{m}}$, there is an automorphism $\lambda:\mathbb{A}^\infty\setminus\{0\}_F\to \mathbb{A}^\infty\setminus\{0\}_F$. The descent datum now provides an isomorphism $\lambda^\ast (p^\ast M)\to p^\ast M$ and these isomorphisms have to be compatible with field extensions and the multiplication in $\mathbb{G}_{\op{m}}$. Fixing a contraction of $\mathbb{A}^\infty\setminus\{0\}$, the descent datum can be translated into an action of $\mathbb{G}_{\op{m}}$ on the motive $p^\ast M$. Note that action here means that for every $F$-point $\lambda$ of $\mathbb{G}_{\op{m}}$, there is an action of $\lambda$ on the image of $p^\ast M$ under $\DMT(\pt)\to\DMT(\op{Spec} F)$. 

Now we claim that there cannot be any nontrivial action of $\mathbb{G}_{\op{m}}$ on $p^\ast M$. By the assumptions on the grading condition~\ref{conditions:grading}, the mixed Tate motive $p^\ast M$ can be viewed as a complex of graded vector spaces over a field of characteristic $0$ and the action has to preserve the grading. Moreover, the action has to be compatible with field extensions, and we can consider how units over the generic point of $\mathbb{G}_{\op{m}}$ act on the base-change of $p^\ast M$ to $k(\mathbb{G}_{\op{m}})=k(T)$. Since the degree $0$ endomorphisms of $p^\ast M$ and $p^\ast M\otimes k(T)$ are the same, the multiplication by scalars can only be via the trivial character. Now over the generic point, a generic unit doesn't have to preserve the grading: for example, in the case of $\mathbf{DA}^{\et}_{\mathbb{Q}}$ over a finite field $\mathbb{F}_q$,  $\op{K}_1(\mathbb{F}_q(T))\cong\mathbb{F}_q(T)^\times$ is rationally nontrivial and hence there are nontrivial morphisms $\mathbb{Q}\to\mathbb{Q}(1)[1]$. In principle, this means that there could be a unipotent action of $\mathbb{G}_{\op{m}}$ on the mixed Tate motive $p^\ast M$. However, the evaluation of this unipotent action at every closed point of $\mathbb{G}_{\op{m}}$ has to be trivial, again by the assumption on the grading condition which prohibits such unipotent action. Therefore, the action of $\mathbb{G}_{\op{m}}$ on the motive $p^\ast M$ must necessarily be trivial. 

As a consequence, the descent data for $p^\ast M$ and $p^\ast N$ are really the same and therefore compatible with the identification $p^\ast M\cong p^\ast N$. This implies that the evaluation of $M$ on $\mathbb{P}^\infty$ is really a constant mixed Tate motive. 
\end{proof}

\begin{Bemerkung}
The same proof also works in the case $G=\mathbb{G}_{\op{m}}^n$; the only change is that now there are several units acting. The argument with the generic point also shows explicitly how the connectedness of the algebraic group $\mathbb{G}_{\op{m}}$  implies that there cannot be nontrivial actions of $\mathbb{G}_{\op{m}}$ on mixed Tate motives and consequently no nontrivial local systems of mixed Tate motives over $\mathbb{P}^\infty$. 
\end{Bemerkung}

\begin{proposition}
\label{prop:locsysconn}
Assume that the underlying derivator $\mathbb{D}$ satisfies conditions~\ref{derivator:new} and the grading condition~\ref{conditions:grading}. Let $G$ be a connected split reductive group and let $M\in\mathbb{D}({\op{B}}G)$ be a motive whose restriction along $p:{\op{E}}G\to{\op{B}}G$ is constant mixed Tate. Then  $M$ is already a constant mixed Tate motive.
\end{proposition}

\begin{proof}
For a connected split reductive group $G$, we can choose an inclusion $T\subseteq B\subseteq G$ of a maximal torus $T$ and a Borel subgroup $B$. Then the universal $G$-torsor $\pi:{\op{E}}G\to{\op{B}}G$ can be factored as
\[
{\op{E}}G\xrightarrow{p} {\op{E}}G/B\xrightarrow{q}{\op{B}}G.
\]
The first morphism $p$ is $\mathbb{A}^1$-equivalent to ${\op{E}}T\to{\op{B}}T$ with $T\subseteq B$ a choice of maximal torus in $G$. The second morphism $q$ is a $G/B$-bundle. 

Now consider the motive $M\in\mathbb{D}({\op{B}}G)$. By assumption, 
\[
(q\circ p)^\ast M\in \DMT({\op{E}}G)\approx\DMT(\pt),
\]
i.e., the pullback of $M$ is a constant mixed Tate motive. By Proposition~\ref{prop:locsysgm} (resp. its extension to the case $T=\mathbb{G}_{\op{m}}^n$), the motive $q^\ast M\in \mathbb{D}({\op{E}}G/B)$ is constant mixed Tate. By the projective bundle formula, cf. Proposition~\ref{prop:projectivebundle}, $q_\ast q^\ast M$ is a direct sum of constant mixed Tate motives; basically, it is the pullback of $\op{M}(G/B)\otimes q^\ast M$ along the structure morphism of ${\op{B}}G$. Moreover, the unit morphism $M\to q_\ast q^\ast M$ of the adjunction is a split injection, embedding $M$ as a direct summand of $q_\ast q^\ast M$. Therefore, $M$ is a constant mixed Tate motive. 
\end{proof}

\begin{remark}
Recall from Definition~\ref{mtderdef1} that an equivariant motive $M\in \mathbb{D}_G^+(\pt)$ is mixed Tate if its underlying motive is. This means in particular, that the evaluation of $M$ on ${\op{B}}G$ is a motive whose pullback along ${\op{E}}G\to{\op{B}}G$ is a constant mixed Tate motive. The above results now imply that if $M\in\DMT_G(\pt)$ is a $G$-equivariant mixed Tate motive with $G$ a connected split reductive group, then its evaluation over ${\op{B}}G$ is a constant mixed Tate motive.
\end{remark}

\begin{proposition}
\label{prop:generatingtateconn}
Assume that the underlying derivator $\mathbb{D}$ satisfies conditions~\ref{derivator:new} and the grading condition~\ref{conditions:grading}. Let $G$ be a connected split reductive group. Then evaluation at ${\op{B}}G$ induces an equivalence of categories
\[
\DMT_G(\pt)\to \DMT({\op{B}}G).
\]
In particular, $\DMT_G(\pt)$ is the smallest thick subcategory of $\mathbb{D}^+_G(\pt)$ containing the constant equivariant motives $\const{\pt}(n)[2n]$, $n\in\mathbb{Z}$. 
\end{proposition}

\begin{proof}
This follows directly from Proposition~\ref{prop:locsysconn}. Once we know that the evaluation functor lands in  $\DMT({\op{B}}G)$, the definition of morphisms in the equivariant motivic categories implies the evaluation functor is fully faithful, i.e., an equivalence onto its essential image. The motives $\const{\pt}(n)[2n]$, $n\in\mathbb{Z}$, generate $\DMT({\op{B}}G)$ (as a thick triangulated subcategory of $\mathbb{D}^+({\op{B}}G)$). Since these motives are contained in the essential image (they are the images of the corresponding constant equivariant mixed Tate motives), the evaluation at ${\op{B}}G$ is also essentially surjective, giving the claimed equivalence. 

For the second statement, note that the constant motives $\const{\pt}(n)[2n]$ are clearly equivariantly mixed Tate. Since the category of equivariant mixed Tate motives is closed under triangles and summands, cf. \ref{mtder} and Lemma~\ref{lem:idempotent}, the thick subcategory generated by these constant motives is contained in $\DMT_G(\pt)$. For the other inclusion, if the thick triangulated subcategory of $\DMT_G(\pt)$ generated by $\const{\pt}(n)[2n]$ were strictly contained in $\DMT_G(\pt)$, then $\DMT({\op{B}}G)$ would not be generated by the constant mixed Tate motives, by the equivalence above. This is obviously a contradiction, proving the second claim.
\end{proof}

\begin{Bemerkung}
This is the motivic analogue of the results in  \cite[Sections 10--12]{BeLu} (more specifically 12.4.3) for connected Lie groups. 
\end{Bemerkung}

\begin{proposition}
\label{prop:generatingtate}
Assume that the underlying derivator $\mathbb{D}$ satisfies conditions~\ref{derivator:new} and the grading condition~\ref{conditions:grading}. Let $G$ be a split reductive group with connected component $G^0$ and denote by $p:{\op{B}}G^0\to{\op{B}}G$ the natural projection induced from the inclusion $G^0\hookrightarrow G$. Then $\DMT_G(\pt)$ is the smallest thick subcategory of $\mathbb{D}^+_G(\pt)$ containing the equivariant motives 
\[
p_\ast \const{\pt}(n)[2n]=\op{Ind}_{G^0}^G\const{\pt}(n)[2n], n\in\mathbb{Z}. 
\]
\end{proposition}

\begin{proof}
This follows from Proposition~\ref{prop:generatingtateconn}. For an equivariant motive $M\in\DMT_G(\pt)$, the unit of the adjunction $M\to p_\ast p^\ast M$ embeds $M$ as a direct summand of $p_\ast p^\ast M$, cf. also Lemma~\ref{finitetorsorNew} in Appendix~\ref{sec:motivebg}. But $p^\ast M\in \DMT_{G^0}(\pt)$ which therefore is contained in the smallest thick triangulated subcategory generated by $\const{\pt}_{G^0}(n)[2n]$, $n\in\mathbb{Z}$. Therefore, $p_\ast p^\ast M$ is contained in the smallest thick triangulated subcategory generated by $p_\ast \const{\pt}_{G^0}(n)[2n]$, $n\in\mathbb{Z}$, as claimed.
\end{proof}

\section{Equivariant motivic cohomology}
\label{sec:motcohom}

In this section, we will define various versions of equivariant motivic homology and cohomology theories. We show that the equivariant motivic cohomology for smooth varieties with action recovers the definition of equivariant higher Chow groups of Edidin and Graham. Then we turn to a more restricted setting, where the motives over the point are nothing but graded vector spaces. In this setting, the equivariant motivic cohomology of the point recovers the classical computations of the cohomology of compact Lie groups, and in the special situations we are interested in the isomorphism is induced by the realization functor. Finally, we discuss how these computations of equivariant cohomology of the point behave under the six functors. This will be the basic input for the tilting theorems in Section~\ref{sec:tilting}. 

We would like to take this opportunity to point out and emphasize once again that the equivariant motivic cohomology we consider here is of Borel-type, as opposed to the Bredon-type equivariant motivic cohomology considered e.g. in \cite{hoyois}. 

\subsection{Definition of equivariant cohomology}

We first give the definition of equivariant motivic homology and cohomology as well as the versions with compact supports. These are equivariant versions of the usual definitions of motivic homology theories, cf. e.g. \cite{mazza:voevodsky:weibel}, also \ref{def:motcohom}. 


\begin{definition}
\label{def:eqmot}
\index{equivariant motivic cohomology}
Let $\mathbb{D}$ be a $\Lambda$-linear derivator satisfying the conditions of \ref{derivator:new}, and let $G\looparrowright X$ be a variety with action. Denote by $\mathbb{D}^+_G:=\mathbb{D}^+_G(\op{pt})$.

We define the \emph{equivariant motivic cohomology of $G\looparrowright X$} as
\[
\op{H}^{n,i}(G\looparrowright X,\mathbb{D}):=\mathbb{D}^+_G(\op{M}(G\looparrowright X),\Lambda(i)[n]). 
\]
Dually, we define the \emph{equivariant motivic homology of $G\looparrowright X$} as 
\[
\op{H}_{n,i}(G\looparrowright X,\mathbb{D}):=\mathbb{D}^+_G(\Lambda(i)[n],\op{M}(G\looparrowright X)). 
\]
Here, $\op{M}(G\looparrowright X)\in \mathbb{D}^+_G(\op{pt})$ denotes the equivariant motive of $G\looparrowright X$, cf. Definition~\ref{def:equivmotive}. 
\end{definition}

\begin{definition}
\label{def:eqmotsupport}
\index{equivariant motivic cohomology!with compact supports}
Let $\mathbb{D}$ be a $\Lambda$-linear derivator satisfying the conditions of \ref{derivator:new}, and let $G\looparrowright X$ be a variety with action. Denote by $\mathbb{D}^+_G:=\mathbb{D}^+_G(\op{pt})$.

We define the \emph{equivariant motivic cohomology with compact supports of $G\looparrowright X$} as
\[
\op{H}^{n,i}_{\op{c}}(G\looparrowright X,\mathbb{D}):=\mathbb{D}^+_G(\op{M}^{\op{c}}(G\looparrowright X),\Lambda(i)[n]). 
\]
Dually, we define the \emph{equivariant (Borel--Moore) motivic homology with compact supports of $G\looparrowright X$} as 
\[
\op{H}_{n,i}^{\op{BM}}(G\looparrowright X,\mathbb{D}):=\mathbb{D}^+_G(\Lambda(i)[n],\op{M}^{\op{c}}(G\looparrowright X))\cong \mathbb{D}^+_G(\op{M}^{\op{BM}}(G\looparrowright X),\Lambda(i)[n]). 
\]
Here, $\op{M}^{\op{c}}(G\looparrowright X)\in \mathbb{D}^+_G(\op{pt})$ denotes the equivariant motive with compact support of $G\looparrowright X$ and $\op{M}^{\op{BM}}(G\looparrowright X)\in \mathbb{D}^+_G(\op{pt})$ denotes the Borel--Moore motive of $G\looparrowright X$, cf. Definition~\ref{def:equivmotive}. 
\end{definition}

\begin{Bemerkung}
For $\mathbb{D}=\op{Der}^{\op{b}}(-,\mathbb{K})$ the derivator given by derived categories of sheaves of $\mathbb{K}$-vector spaces, the above definitions recover the usual definitions of equivariant cohomology, cf. \cite[Definition 13.1]{BeLu}, and equivariant cohomology with compact supports, cf. \cite[Definition 13.2]{BeLu}; at least the special case where the coefficients are given by constant sheaves. In this vein, we could consider more generally equivariant cohomology of $G\looparrowright X$ with coefficients in a $G$-equivariant motive $M$ on $X$ to be defined as 
\[
\mathbb{D}^+_G((\op{fin}_X)_\ast M,\Lambda(i)[n]). 
\]
\end{Bemerkung}

\begin{Bemerkung}
\label{eqmotlocal}
\index{equivariant motivic cohomology!local coefficients}
The above definitions work well to detect most of the relevant phenomena in the category of equivariant mixed Tate motives via cohomology in the case of connected groups. However, when the group $G$ is not connected, there exist non-trivial local systems on ${\op{B}}G$, cf.~Example~\ref{ex:locsys}, given by the direct summands of $\op{Ind}_{G^0}^G\underline{\pt}_{G^0}\in\mathbb{D}^+_G(\pt)$ where $G^0$ is the connected component of the identity. In these cases, we also want to consider the equivariant motivic cohomology with coefficients in such a local system $L$; this would be an analogue of the classical equivariant cohomology with coefficients in a local system. More generally, 
\[
\mathbb{D}^+_G(\op{M}(G\looparrowright X),\op{Ind}_{G^0}^G \Lambda(i)[n])
\]
could be interpreted as the direct sum of the equivariant motivic cohomology groups with coefficients in all relevant local systems with monodromy in $\pi_0(G):=G/G^0$. Similarly, 
\[
\mathbb{D}^+_G(\op{M}^{\op{c}}(G\looparrowright X),\op{Ind}_{G^0}^G\Lambda(i)[n])
\]
would be the direct sum of equivariant motivic cohomology with compact support in all the relevant local systems with monodromy in $\pi_0(G)$.

In the case of equivariant motivic cohomology, the cohomology with coefficients in the induced motive above can also be computed over the connected component of the identity $G^0$, using the induction-restriction adjunction:
\begin{eqnarray*}
\op{H}^{n,i}(G^0\looparrowright X,\mathbb{D})&=& \mathbb{D}^+_{G^0}(\op{M}(G^0\subset G\looparrowright X), \Lambda(i)[n])\\
&\cong &
\mathbb{D}^+_{G^0}(\op{Res}_G^{G^0} \op{M}(G\looparrowright X), \Lambda(i)[n])\\ &\cong& \mathbb{D}^+_G(\op{M}(G\looparrowright X),\op{Ind}_{G^0}^G \Lambda(i)[n]).
\end{eqnarray*}
\end{Bemerkung}

\begin{Bemerkung}
In the situations where we have realization functors, cf. Section~\ref{sec:realization}, we would get induced morphisms from equivariant motivic cohomology to whatever cohomology the target of the realization functor computes. To be more specific, for $\ell$-adic realization on $\mathbf{DA}^{\et}$ we would get induced $\ell$-adic equivariant cycle class maps 
\[
\op{H}^{n,i}(G\looparrowright X,\mathbf{DA}^{\et})\xrightarrow{\op{Real}_\ell} \op{H}^{n}(G\looparrowright X,\mathbb{Q}_\ell)
\] 
which take values in equivariant $\ell$-adic cohomology. Similarly, for the Hodge realization on $\op{MDer}(-;\mathbb{C})$ we would get induced maps
\[
\op{H}^{n,i}(G\looparrowright X,\op{MDer})\xrightarrow{\op{Real}_{\op{H}}} \op{H}^{n}(G\looparrowright X,\mathbb{C})
\]
which take values in equivariant singular cohomology of the complex points. We will see in the explicit computations later on that the cycle class maps for these realization functors will be isomorphisms in the situations relevant for us. 

Although the statements above were made for equivariant motivic cohomology,  analogous statements would also hold for the other versions of equivariant motivic homology theories. 
\end{Bemerkung}

\begin{Bemerkung}
\label{agnotation}
\index{$\mathcal{A}_G$}
In the special case where $X=\pt$ and the derivator $\mathbb{D}$ is clear from the context, the cohomology ring $\op{H}^{2i,i}(G\looparrowright \pt;\mathbb{D})$ will be abbreviated to $\mathcal{A}_G$. This follows the notation for cohomology of classifying spaces used in \cite{BeLu}. 
\end{Bemerkung}

\subsection{Equivariant motives and equivariant Chow groups}
\label{equiv:chow}

Now we want to compute morphisms between some equivariant mixed Tate motives in terms of equivariant higher Chow groups in the sense of Edidin and Graham \cite{edidin:graham}. This is the equivariant version of the identification 
\[
\mathbf{DA}^{\et}_{X}(\Lambda,\Lambda(p)[q])\cong
\op{gr}_\gamma^p\op{K}_{2p-q}(X)_{\Lambda}\cong 
\op{CH}^p(X;2p-q)_{\Lambda}
\]
of morphisms between mixed Tate motives over a smooth scheme $X$ in terms of Bloch's higher Chow groups of $X$, by Voevodsky's comparison theorem \cite{voevodsky:comparison}, cf. also \cite[Theorem 19.1]{mazza:voevodsky:weibel}. The results below are formulated with coefficients in a field $\Lambda$ of characteristic zero, but some integral statements are also possible. 

\begin{Bemerkung}
Recall from \cite{totaro} and \cite{edidin:graham} that the equivariant higher Chow groups of a variety with action $(G\looparrowright X)$ are defined via resolutions, cf. also Lemma~\ref{lem:resexist}: for $s\geq 0$ take a representation $V_s$ of $G$ over $k$ such that $G$ acts freely outside a $G$-invariant closed subset $Z_s\subseteq V_s$ of codimension $\geq s$. Then the equivariant higher Chow groups are defined as 
\[
\op{CH}^i_G(X;j)\cong \lim_{s\to\infty}\op{CH}^i((X\times(V_s\setminus Z_s))/G;j).
\]
By the double fibration argument \cite[Remark 1.4]{totaro} resp. \cite[Definition-Proposition 1]{edidin:graham}, these groups stabilize as $s\to\infty$. 
\end{Bemerkung}

\begin{Bemerkung}
Now in the situation of $(G\looparrowright \pt)$, \cite[Remark 1.4]{totaro} shows that for a representation $V$ of $G$ as above with non-free closed subset $Z\subseteq V$ of codimension $\geq s$, the quotient $(V\setminus Z)/ G$ is an affine bundle over a homogeneous space $\op{GL}_{N+n}/(H\times G)$ where $H$ is an extension of $\op{GL}_N$ by a unipotent group. In particular, the quotient $(V\setminus Z)/ G$ is smooth. A similar argument shows that for \emph{smooth} $X$ with $G$-action, and a representation as above, the quotient $Y=(X\times(V\setminus Z))/G$ is smooth. 
By Voevodsky's comparison result mentioned above, 
\[
\mathbf{DA}^{\et}_Y(\Lambda,\Lambda(p)[q])\cong \op{CH}^p(Y,2p-q).
\]
For high enough $s$, this is independent of the choice of representation $V$, and hence the latter Chow group can be identified with the equivariant Chow group $\op{CH}^p_G(X,2p-q)$. In particular, the morphisms between Tate motives over the quotient $(X\times(V\setminus Z))/G$ compute equivariant higher Chow groups in sufficiently low degrees. 
\end{Bemerkung}

Now we still need to identify the morphisms between mixed Tate motives over quotients of individual resolutions with morphisms of mixed Tate motives in the category of motives over the resolution. For this, we use the categories $\mathbb{D}_G(X,P)$ and their identification with motives over diagrams from Lemma~\ref{lem:resolindep}. 

By Definition~\ref{def:equivmotres1}, a morphism in $\mathbb{D}_G(X,P)$ is a pair $\alpha=(\alpha_X,\overline{\alpha})$ with $\alpha_X:M_X\to N_X$ and $\overline{\alpha}:\overline{M}\to\overline{N}$ such that $\beta_N\circ p^\ast(\alpha_X)=q^\ast(\overline{a})\circ \beta_M$. For the constant mixed Tate motives $\Lambda(i)[j]$ in $\mathbb{D}_G(X,P)$ the components over $X$ and $G\backslash P$ are both $\Lambda(i)[j]$, and the comparison isomorphism $\beta$ is the identity. In particular, we obtain a cartesian square
\[
\xymatrix{
\mathbf{DA}^{\et}_G(X,P)(\Lambda,\Lambda(i)[j])\ar[r] \ar[d] & \mathbf{DA}^{\et}(X)(\Lambda,\Lambda(i)[j])\ar[d] \\
\mathbf{DA}^{\et}(G\backslash P)(\Lambda,\Lambda(i)[j])\ar[r] & \mathbf{DA}^{\et}(P)(\Lambda,\Lambda(i)[j]).
}
\]
For $P=X\times(V\setminus Z)$ as above, the projection $P\to X$ is $(s-2)$-acyclic, hence induces an isomorphism on motivic cohomology in degrees  below $(s-2)$. In particular, in these degrees, the pullback square provides isomorphisms
\[
\mathbf{DA}^{\et}_G(X,P)(\Lambda,\Lambda(i)[j])\cong 
\mathbf{DA}^{\et}(G\backslash P)(\Lambda,\Lambda(i)[j])
\cong\op{CH}^i_G(X;2i-j).
\]
The double fibration argument shows again stabilization, which implies that we get an isomorphism for all degrees
\[
\mathbf{DA}^{\et,\op{res},\op{b}}_G(X)(\Lambda,\Lambda(i)[j])\cong \op{CH}^i_G(X;2i-j).
\]
By Corollary~\ref{cor:resRes}, we get a similar identification for the category $\mathbf{DA}^{\et,\op{Res}}$. Putting all these together, we finally obtain the required computation of morphisms between constant equivariant mixed Tate motives. We also formulate a version for non-smooth varieties with action, pointed out by Jens Eberhardt; for the proof, Voevodsky's comparison theorem has to be replaced by \cite[Proposition 19.18]{mazza:voevodsky:weibel}. 

\begin{theorem}
\label{thm:higherchow}
Let $G\looparrowright X$ be a variety with action, and assume that $X$ is smooth. Then we have a functorial identification
\[
\op{H}^{n,i}(G\looparrowright X;\mathbf{DA}^{\et}_\Lambda)=\mathbf{DA}^{\et}_G(X)(\Lambda,\Lambda(i)[n])\cong \op{CH}^i_G(X;2i-n;\Lambda).
\]

Let $k$ be a field admitting resolution of singularities, and let $G\looparrowright X$ be a variety with action over $k$, equidimensional of dimension $d$ but not necessarily smooth. Then for every positive $i\leq d$ we have a functorial identification
\[
\op{H}^{\op{BM}}_{2i+n,i}(G\looparrowright X;\mathbf{DA}^{\et}_\Lambda)= \mathbf{DA}^{\et}_G(\op{pt})(\Lambda(i)[2i+n],\op{M}^{\op{c}}(G\looparrowright X))\cong \op{CH}^{d-i}_G(X;n;\Lambda).
\]
\end{theorem}

\begin{remark}
Very likely the requirement of resolution of singularities can be omitted in the above, and replaced by alterations. This would allow to get the above result for all coefficients in which the characteristic of the base field is invertible; that would be by far enough for our applications with rational coefficients. 
\end{remark}

\subsection{Equivariant motivic cohomology over the point}

Our next goal is to show that equivariant motivic cohomology induces an equivalence from a suitable category of mixed Tate motives to a suitable category of modules over the equivariant motivic cohomology ring of the point. These statements will be used in Section~\ref{sec:tilting} to relate the category of equivariant mixed Tate motives on the point to a (homotopy category of the) category of modules over the equivariant cohomology ring. Most of the relevant computations pertain to the (non-equivariant) motive of classifying spaces; these are deferred to Appendix~\ref{sec:motivebg}. We also discuss the compatibility of these identifications with the restriction and integration functors and the monoidal structure from the six functor formalism.

\begin{remark}
Recall from \ref{agnotation} that $\mathcal{A}_G$ denotes the equivariant motivic cohomology ring (computed in some underlying derivator $\mathbb{D}$). By the results in Section~\ref{equiv:chow},  $\mathcal{A}_G$ is isomorphic to the Chow ring $\op{CH}^\bullet({\op{B}}G;\Lambda)$ of the classifying space of $G$ in case $\mathbb{D}\approx\mathbf{DA}^{\et}(-;\Lambda)$. In the cases where the derivator satisfies the grading condition~\ref{conditions:grading}, we will see that the ring $\mathcal{A}_G$ is isomorphic to the singular cohomology of the classifying space of the corresponding complex Lie group. 
\end{remark}

\begin{lemma}
\label{lem:addemb}
\begin{enumerate}
\item Let $\mathcal{A}$ be an idempotent complete additive category and $T\in \mathcal A$ be an object. Then the additive subcategory of $\mathcal{A}$ generated by $T$ is equivalent to the additive subcategory of right $\op{End}_{\mathcal{A}}(T)$-modules generated by the ring $\op{End}_{\mathcal{A}}$ itself under the functor $\op{Hom}_{\mathcal A}(T,-)$.
\item If $\mathcal{A}$ is in addition equipped with an auto-equivalence $(-)[1]$, then similarly the additive subcategory generated by the $T[n]$ for $n\in\DZ$ is equivalent to the additive subcategory of $\DZ$-graded right $H$-modules generated by the $H[n]$, for the $\DZ$-graded ring 
\[
H\pdef\bigoplus_n\op{Hom}_{\mathcal A}(T,T[n]).
\]
\end{enumerate}
\end{lemma}

\begin{proof}
(1) is classical, a proof can be found e.g. in \cite[Proposition 2.3]{krause:ks}.\footnote{Probably, this would be due to Morita, Freyd or Mitchell.} (2) is proved similarly.
\end{proof}

\begin{theorem}
\label{thm:equivcoh}
Let $k$ be a field and let $\mathbb{D}$ be a derivator satisfying the conditions~\ref{derivator:new} and the grading condition~\ref{conditions:grading}. Let $G$ be a linear algebraic group, and denote the connected component of the identity by $G^0$ and set $\pi_0(G)=G/G^0$. Denote by $\mathcal{T}$ the additive subcategory of $\mathbb{D}^+_G(\pt)$ generated by the equivariant mixed Tate motives $\ind_{G^0}^G\const{\pt}(n)[2n]$ for $n\in \ZZ$. Then the equivariant motivic homology functor
\begin{align*}
\mathbb{H}_G^{\op{mot}}:\mathbb{D}^+_G(\pt) &\to S^W\rtimes \Lambda[\pi_0(G)]\op{-Mod}^{\ZZ}, \\
M &\mapsto \bigoplus_i \mathbb{D}^+_G( \ind^G_{G^0}\Lambda, M(i)[2i]),
\end{align*}
induces an equivalence of categories
\[ 
\cT \sirra S^W\rtimes\Lambda[\pi_0(G)]\op{-fModfg}^{\ZZ}, 
\]
where $R\op{-fModfg}^{\ZZ}$ denotes the category of graded free finitely generated $R$-modules and $S^W\rtimes\Lambda[\pi_0(G)]$ denotes the appropriate twisted group ring, cf. \ref{bem:coinv} in Appendix~\ref{sec:motivebg}. If $G$ is connected, then this equivalence maps $\const{\pt}(n)[2n]$ to $S^W\langle n\rangle$.
\end{theorem}

\begin{proof}
This is a consequence of part (2) of Lemma~\ref{lem:addemb} above, applied to the object $T=\op{Ind}_{G^0}^G\const{\pt}$ and the auto-equivalence $M\mapsto M(1)[2]$. The required idempotent completeness of the motivic categories was discussed in Remark~\ref{rem:triangulated}. To identify the relevant endomorphism ring for the target category, we have 
\[
\bigoplus_n\mathbb{D}_G^+\left(\op{Ind}_{G^0}^G\const{\pt}, \op{Ind}_{G^0}^G\const{\pt}(n)[2n]\right)\cong S^W\rtimes\Lambda[\pi_0(G)]
\]
by Proposition~\ref{prop:swcomp}. This implies the claim about the equivalence of categories. The additional claim in the connected case follows, since Lemma~\ref{lem:addemb} is compatible with the autoequivalences and $\const{\pt}(n)[2n]$ is the choice of additive generators of $\mathcal{T}$.
\end{proof}

Next, we will establish several compatibility statements describing the behaviour of equivariant motivic cohomology under restriction, integration and the monoidal structure. These are motivic versions of the computations leading to \cite[Theorem 12.7.2]{BeLu}.

\begin{proposition}
\label{prop:compatotimes}
In the situation of Theorem~\ref{thm:equivcoh}, 
\begin{enumerate}
\item the category $\mathcal{T}$ is closed under $\otimes$, and 
\item the restriction of the equivariant motivic homology functor 
\[
\mathbb{H}^{\op{mot}}_G:\mathcal{T}_G\hookrightarrow \mathbb{D}_G^+(\pt)\to S^W\rtimes\Lambda[\pi_0(G)]\op{-Mod}^\mathbb{Z}
\] 
is a $\otimes$-functor. 
\end{enumerate}
\end{proposition}

\begin{proof}
(1)  Since the monoidal structure on $\mathbb{D}_G^+(\pt)$ is closed, $\otimes$ is a left adjoint functor and therefore commutes with left adjoints. In particular, it distributes over direct sums. Therefore, it suffices to show that tensor products of additive generators of $\mathcal{T}$ can be decomposed as direct sums of generators. In the case of connected $G$, this is clear since $\const{\pt}(i)[2i] \otimes \const{\pt}(j)[2j] \cong \const{\pt}(i+j)[2i+2j]$. In the case where the group $G$ is not connected, we need to decompose 
\[
\op{Ind}_{G^0}^G\const{\pt}(i)[2i]\otimes \op{Ind}_{G^0}^G\const{\pt}(j)[2j]
\]
as direct sum of additive generators. Note that $G/G^0$ is a finite group, and hence $\op{Ind}_!\simeq \op{Ind}_\ast$ by an argument as in the proof of Proposition~\ref{Lasp}. Therefore, we have
\[
\op{Ind}_{G^0}^G\const{\pt}(i)[2i]\cong \op{Ind}_!\const{\pt}(i)[2i] \cong \op{fin}_!\const{G/G^0}(i)[2i].
\]
Using the projection formula, we obtain 
\[
\op{Ind}_{G^0}^G\const{\pt}(i)[2i]\otimes \op{Ind}_{G^0}^G\const{\pt}(j)[2j]\cong 
\op{Ind}_{G^0}^G\left(\op{Res}_G^{G^0} \op{Ind}_{G^0}^G\const{\pt}(i+j)[2i+2j]\right).
\]
But then $\op{Res}_G^{G^0}\op{Ind}_{G^0}^G\const{\pt}$ is a mixed Tate motive by Propositions~\ref{prop:dmtres} and \ref{Las}. It is also constant, by Proposition~\ref{prop:locsysconn}, and therefore decomposes as a direct sum of copies of $\const{\pt}(n)[2n]$. Since $\op{Ind}_{G^0}^G$ is an additive functor, this provides the required decomposition. 

(2)
Recall that the equivariant motivic homology maps 
\[
M \mapsto \bigoplus_i \mathbb{D}^+_G( \ind^G_{G^0}\Lambda, M(i)[2i])\cong \mathbb{D}^+_{G^0}( \const{\pt}, \op{Res}_G^{G^0} M(i)[2i]). 
\]
Here we use again that $\op{Ind}_!\simeq \op{Ind}_\ast$ since $G/G^0$ is a finite group, cf. Step (1). Since restriction commutes with the monoidal structure, cf.~\ref{CompRES} resp. \ref{rem:monoidal}, and preserves mixed Tate motives by Proposition~\ref{prop:dmtres}, it suffices to prove the claim for $G$ connected. 

In the case where $G$ is connected, we know from Theorem~\ref{thm:equivcoh} that $\const{\pt}(n)[2n]$ maps to $S^W\langle n\rangle$. Then $\const{\pt}(i)[2i]\otimes\const{\pt}(j)[2j]\cong\const{\pt}(i+j)[2i+2j]$ implies that $\const{\pt}(i)[2i]\otimes\const{\pt}(j)[2j]$ maps to $S^W\langle i+j\rangle\cong S^W\langle i\rangle\otimes S^W\langle j\rangle$. Now we note that $S^W\langle n\rangle$ is generated by the identity map $\const{\pt}\to \const{\pt}(n)[2n](-n)[-2n]$, viewed as element in degree $-n$. Then it is clear that the tensor product of the two generators for $S^W\langle i\rangle$ and $S^W\langle j\rangle$ maps to a generator of $S^W\langle i+j\rangle$. So the natural comparison map 
\[
\mathbb{D}^+_G(\const{\pt},\const{\pt}(i)[2i])\otimes \mathbb{D}^+_G(\const{\pt},\const{\pt}(j)[2j])\to \mathbb{D}^+_G(\const{\pt},\const{\pt}(i+j)[2i+2j])
\]
induced from the injections $M_i\to M_1\otimes M_2$ induces an isomorphism so that equivariant motivic homology is a tensor functor on the additive category $\mathcal{T}$ generated by $\const{\pt}(i)[2i]$.  
\end{proof}

\begin{remark}
The above compatibility of motivic homology with the tensor structures can only be true for mixed Tate motives, since the K\"unneth formula doesn't generally hold for motivic homology theories. 
\end{remark}

\begin{proposition}
\label{prop:compatres}
In the situation of Theorem~\ref{thm:equivcoh}, let $\phi:H\to G$ be a homomorphism of linear algebraic groups, and let $\phi^\ast:\mathcal{A}_G\to\mathcal{A}_H$ be the induced 
morphism on equivariant motivic cohomology rings of classifying spaces, cf. \ref{agnotation} for the notation. 
\begin{enumerate}
\item
The following diagram commutes 
\begin{displaymath}
\xymatrix{
\mathcal{T}_G\ar[r]^-{\approx}\ar[d]_-{\op{Res}^H_G} &\mathcal{A}_G\op{-fModfg}^\DZ\ar@{=>}_-\sim[dl] \ar[d]^-{\mathcal{A}_H\otimes_{\mathcal{A}_G}(-)}\\
\mathcal{T}_H \ar[r]_-{\approx}&  \mathcal{A}_H\op{-fModfg}^\DZ,
}
\end{displaymath}
where the categories $\mathcal{T}_G$ and $\mathcal{T}_H$ are the additive subcategories of $\mathbb{D}^+_G(\pt)$ and $\mathbb{D}^+_H(\pt)$ generated by $\op{Ind}_{G^0}^G\const{\pt}(n)[2n]$ and $\op{Ind}_{H^0}^H\const{\pt}(n)[2n]$, $n\in\mathbb{Z}$, respectively. 
\item Let $\mathcal{T}_G^\infty$ and $\mathcal{T}_H^\infty$ be the additive subcategories of $\mathbb{D}_G^+(\pt)$ and $\mathbb{D}_H^+(\pt)$ generated by  countable direct sums of copies of the motives $\op{Ind}_{G^0}^G\const{\pt}(n)[2n]$ and $\op{Ind}_{H^0}^H\const{\pt}(n)[2n]$, $n\in\mathbb{Z}$, respectively. Then the following diagram commutes
\[
\xymatrix{
\mathcal{T}^\infty_G\ar[r]^-{\approx}\ar[d]_-{\op{Res}^H_G} 
&\mathcal A_G\op{-fModcg}^\DZ\ar@{=>}_-\sim[dl]
\ar[d]^-{\mathcal A_H\otimes_{\mathcal A_G}(-)}\\
\mathcal{T}^\infty_H \ar[r]_-{\approx}&  \mathcal A_H\op{-fModcg}^\DZ
}
\]
\end{enumerate}
\end{proposition}

\begin{proof}
(1) 
We first define the natural transformation between the two compositions, and then prove that it is an isotransformation. The right-hand composition is
\[
M\mapsto \mathcal{A}_H\otimes_{\mathcal{A}_G}\bigoplus_i \mathbb{D}^+_G\left(\const{\pt}_G,M(i)[2i]\right),
\]
and this $\mathcal{A}_H$-module is generated by elements of the form $1\otimes \left(\const{\pt}_G\to M(i)[2i]\right)$. The left-hand composition maps 
\[
M\mapsto \bigoplus_i\mathbb{D}^+_H\left(\const{\pt}_H,\op{Res}_G^H M(i)[2i]\right)
\]
and elements of that $\mathcal{A}_H$-module are of the form $\const{\pt}_H\to \op{Res}_G^H M(i)[2i]$. We can now send an element $1\otimes \left(\const{\pt}_G\to M(i)[2i]\right)$ to the composition 
\[
\const{\pt}_G\to M(i)[2i]\to\op{Ind}_H^G\op{Res}_G^H M(i)[2i],
\]
where the second map is the unit of the adjunction between induction and restriction; this composition can be interpreted as an element in the appropriate $\mathcal{A}_H$-module via the adjunction
\[
\mathbb{D}^+_G\left(\const{\pt}_G,\op{Ind}_H^G\op{Res}_G^H M(i)[2i]\right)\cong \mathbb{D}^+_H\left(\const{\pt}_H,\op{Res}_G^H M(i)[2i]\right). 
\]
This is well-defined, since whenever two maps $\const{\pt}_G\to M(i)[2i]$ give the same element after scalar extension to $\mathcal{A}_H$, they differ by an element in $\ker\phi^\ast$, but that means that their compositions with the unit of the Ind-Res-adjunction are equal. Hence we get a well-defined and additive natural transformation. 

Since we are only considering the additive categories generated by the motives $\op{Ind}\const{\pt}(n)[2n]$ and the natural transformation is additive, we only need to show the commutativity for the generating motives $\op{Ind}_{G^0}^G\const{\pt}(n)[2n] \in\mathcal{T}_G\subset \mathbb{D}^+_G(\pt)$. 

In the case where $G$ and $H$ are connected, we have  $\op{Res}_G^H\const{\pt}_G\cong\const{\pt}_H$ because $\const{\pt}_G=\op{Res}_1^G\const{\pt}$ and 2-functoriality of restriction. Using Theorem~\ref{thm:equivcoh}, $\mathbb{H}^{\op{mot}}_H\circ \op{Res}_G^H$ maps $\const{\pt}_G(n)[2n]$ to $\mathcal{A}_H\langle n\rangle$, viewed as free rank one $\mathcal{A}_H$-module generated by the identity of $\const{\pt}_H$ in the appropriate degree.  By the same reasoning, $\mathbb{H}^{\op{mot}}_G(\const{\pt}_G(n)[2n])\cong\mathcal{A}_G\langle n\rangle$ is the free rank one $\mathcal{A}_G$-module generated by the identity of $\const{\pt}_G$ in the appropriate degree. Therefore, tensoring with $\mathcal{A}_H$ also yields the free rank one module generated by $1\otimes\op{id}_{\const{\pt}_G}$. Now we want to know that the above natural transformation maps a generator to a generator. By construction, the natural transformation maps $1\otimes\op{id}_{\const{\pt}_G}$ to the composition of the identity with the unit of the Ind-Res-adjunction; under the adjunction, this corresponds exactly to the identity of $\const{\pt}_H$ in the appropriate degree. Therefore, the natural transformation is an isotransformation in the case of connected groups.

The argument in the non-connected case is similar. The only thing to note is that 
\[
\op{Res}_G^H\op{Ind}_{G^0}^G \const{\pt}_{G^0}(n)[2n]\cong \op{Ind}_{\tilde{H}}^H\const{\pt}_{\tilde{H}}(n)[2n],
\] 
where $\tilde{H}\cong G^0\times_G H$ is the preimage of $G^0$ under $\phi$. The latter now splits as $\#\ker(\pi_0(H)\to\pi_0(G))$ many copies of $\op{Ind}_{H^0}^H\const{\pt}_{H^0}(n)[2n]$. With these modifications, the arguments above go through to prove the claim in the non-connected case.

(2) This follows from part (1) using that $\op{Ind}_{G^0}^G \const{\pt}(n)[2n]$ are compact objects in $\mathbb{D}_G^+(\op{pt})$.
\end{proof}

\begin{remark}
For the right-hand vertical arrow in Proposition~\ref{prop:compatres}, we could also take the derived tensor product $\mathcal{A}_H\otimes^{\op{L}}_{\mathcal{A}_G}(-)$ instead of the underived tensor product, because we are only making statements about free $\mathcal{A}_G$-modules. 
\end{remark}

We can also formulate the splitting principle as a consequence of the above compatibility. 

\begin{corollary}
Let $G$ be a reductive group and let $T\subseteq G$ be a maximal torus. Then the restriction functor $\mathbb{D}^+_{G} (\pt)\ra \mathbb{D}^+_{T} (\pt)$ factors as  
\[
\mathbb{D}^+_{G}(\underline{\op{pt}},\underline{\op{pt}}(n)[2n])\xrightarrow{\cong} (S^n)^W\hookrightarrow
S^n\xrightarrow{\cong}  \mathbb{D}^+_{T}(\underline{\op{pt}},\underline{\op{pt}}(n)[2n])
\] 
for $W$ the Weyl group acting on the polynomial ring $S$ of characters of  $T$ as above. 
\end{corollary}

\begin{proposition}
\label{prop:countable}
In the situation of Theorem~\ref{thm:equivcoh}, let $\phi:H\hookrightarrow G$ be the inclusion of a closed subgroup in a linear algebraic group, and let $\phi^\ast:\mathcal{A}_G\to\mathcal{A}_H$ be the induced morphisms on equivariant motivic cohomology rings, cf. \ref{agnotation} for the notation. Then the ordinary integration functor $\op{Ind}_H^G:\mathbb{D}^+_H(\pt)\to \mathbb{D}^+_G(\pt)$ restricts to a functor $\langle\mathcal{T}_H\rangle_\Delta\to \langle\mathcal{T}_G\rangle_\Delta$.
\end{proposition}

\begin{proof}
It suffices to show that $\op{Ind}_H^G\op{Ind}_{H^0}^H\const{\pt}(n)[2n]\cong\op{Ind}_{H^0}^G\const{\pt}(n)[2n]$ is contained in $\mathcal{T}_G$. By definition of $\op{Ind}_{H^0}^G$, we first use the induction equivalence to go from $\const{\pt}_{H^0}(n)[2n]\in\mathbb{D}^+_{H^0}(\pt)$ to $\const{\pt}_G(n)[2n]\in\mathbb{D}^+_G(G/H^0)$ and then push forward along the structure map $\op{fin}:G/H^0\to\pt$, using $(\op{id},\op{fin})_\ast$. By Proposition~\ref{prop:tatehomogeneous} we have that  $\op{M}(G/H^0)$ is a mixed Tate motive. Note also that $\op{M}(G/H^0)$ is the underlying motive of $\op{Ind}_{H^0}^G\const{\pt}(n)[2n]$ and therefore $\op{Ind}_{H^0}^G\const{\pt}(n)[2n]\in \DMT_G(\pt)$. By Proposition~\ref{prop:generatingtate}, $\DMT_G(\pt)\approx\langle \mathcal{T}_G\rangle_\Delta$, proving the claim. 
\end{proof}

\begin{remark}
Note that the motive $\op{M}(G/H)$ is not necessarily pure. Therefore, the induction functor will not necessarily induce a functor $\op{Ind}_H^G:\mathcal{T}_H\to\mathcal{T}_G$. Even the fact that it lands in the subcategory $\mathcal{T}_G$ requires that equivariant mixed Tate motives are generated (as a triangulated category) by $\op{Ind}_{G^0}^G\const{\pt}(i)[2i]$. However, this would be true for induction from a parabolic subgroup because in this case $\op{Ind}_\ast\cong\op{Ind}_!$ preserves pure equivariant Tate motives, cf.~\ref{Lasp}. This case can also be seen as an instance of the projective bundle formula, cf. Proposition~\ref{prop:projectivebundle}. 
\end{remark}

\begin{Bemerkung}
\label{countable}
More generally, the general direct image functors discussed in Section~\ref{sec:qfstar} preserve countably generated Ind-mixed Tate motives. In the situation of Theorem~\ref{thm:equivcoh}, let $\phi:H\twoheadrightarrow G$ be a surjective homomorphism of linear algebraic groups, and let $\phi^\ast:\mathcal{A}_G\to\mathcal{A}_H$ be the induced morphisms on equivariant motivic cohomology rings, cf. \ref{agnotation} for the notation. Let $\mathcal{T}_G^\infty$ and $\mathcal{T}_H^\infty$ be the additive subcategories of $\mathbb{D}_G^+(\pt)$ and $\mathbb{D}_H^+(\pt)$ generated by  countable direct sums of copies of $\op{Ind}_{G^0}^G\const{\op{pt}}(n)[2n]$ and $\op{Ind}_{H^0}^H\const{\op{pt}}(n)[2n]$, $n\in\mathbb{Z}$, respectively. Then the ordinary integration functor $\op{Ind}_H^G:\mathbb{D}^+_H(\pt)\to \mathbb{D}^+_G(\pt)$ restricts to a functor $\langle\mathcal{T}^\infty_H\rangle_\Delta\to \langle\mathcal{T}^\infty_G\rangle_\Delta$. This follows from the definition of the functor $\op{Ind}_H^G$ which is basically given by push-forward along $\op{fin}:{\op{B}}\ker \phi\to\pt$. By Proposition~\ref{prop:motbg},  $\fin_\ast \const{{\op{B}}\ker\phi}\cong\op{M}({\op{B}}\ker\phi)$ is a pure Ind-mixed Tate motive. Moreover, by the explicit formulas in Section~\ref{sec:motivebg}, we see that this motive is in fact a countable direct sum of constant motives $\const{\op{pt}}(i)[2i]$. In particular, $\op{Ind}_H^G\const{\pt}_H(n)[2n]\in\mathcal{T}^\infty_G$. The same then holds for countable sums of constant mixed Tate motives, proving the claim.

The general case of an arbitrary homomorphism $\phi:H\to G$ can now be obtained by factoring $H\twoheadrightarrow\op{Im}\phi\hookrightarrow G$ and combining the above statement for the surjection $H\twoheadrightarrow \op{Im}\phi$ with Proposition~\ref{prop:countable} for the injection $\op{Im}\phi\hookrightarrow G$. 
\end{Bemerkung}

We finally want to discuss how equivariant homology behaves under the exceptional integration functors $\op{Ind}_!$. For this, we first need to discuss the algebraic functor on cohomology rings which should correspond to the exceptional integration. Since exceptional integration is the left adjoint to restriction and restriction corresponds to scalar extension by Proposition~\ref{prop:compatres}, we are looking for a left adjoint of scalar extension. 

\begin{Bemerkung}
\label{EXR}
If we have  rings $A,B$ and an $A$-$B$-bimodule $D$, which is finitely generated and projective over $A$, then the $B$-$A$-bimodule  $D^\ast\pdef \op{Hom}_A(D,A)$ is a finitely generated projective right $A$-module. Furthermore, we get natural isomorphisms $D^\ast\otimes_A M\xrightarrow{\cong} \op{Hom}_A(D,M)$ of $B$-modules for any $A$-module $M$. Thus under our assumptions we get an adjoint pair
\[
D\otimes_B(-)\dashv D^\ast\otimes_A(-).
\]
Now we can use the analogous construction $^\ast E\pdef \op{Hom}_{A}(E,A)$ to  get an $A$-$B$-bimodule from a $B$-$A$-bimodule. If we have a $B$-$A$-bimodule $E$ which is projective and of finite rank as a right $A$-module, the evaluation map will be an isomorphism $E\xrightarrow{\cong} ({}^\ast E)^\ast$. Therefore, we get an adjoint pair
\[
{}^\ast E\otimes_B(-)\vdash E\otimes_A(-)
\]
Similarly if $E$ is a bounded complex of $B$-$A$-bimodules, which are projective and of finite rank as right $A$-modules, then ${}^\ast E\otimes_B^{\op{L}}(-):\op{Der}(B\op{-Mod})\ra \op{Der}(A\op{-Mod})$
will be a left adjoint to $E\otimes_A^{\op{L}}(-)$. 
\end{Bemerkung}

\begin{Bemerkung}
\label{EXR2}
We discuss the analogous statements for graded rings, modules and bimodules. Given $\DZ$-graded rings $A$ and $B$ and a $\DZ$-graded $A$-$B$-bimodule $D$, 
the functor $D\otimes_B(-): B\op{-Mod}^\DZ\ra A\op{-Mod}^\DZ$ is left adjoint to the functor 
\[
\op{Hom}_A^\circledast(D,-): A\op{-Mod}^\DZ\ra B\op{-Mod}^\DZ.
\]
Here we use the notation $\op{Hom}^\circledast$ to indicate that
we take the direct sum of the spaces of homogeneous homomorphisms
of various degrees rather than the full space of all homomorphisms, which
need not admit a natural grading in general.  
\end{Bemerkung}

\begin{Bemerkung}
\label{excres}
Let $\phi:H\hookrightarrow G$ be the inclusion of a closed connected subgroup in a connected linear algebraic group and let $\phi^\ast:\mathcal{A}_G\to\mathcal{A}_H$ be the induced morphism of equivariant motivic cohomology rings, cf. \ref{agnotation} for the notation. We now want to apply the discussion in \ref{EXR2} above to determine the left adjoint of the extension of scalars along the homomorphism $\phi^\ast$. Note that we always consider derivators with coefficients $\Lambda$ which are fields, and our assumptions imply that $\mathcal{A}_H$ is a finitely generated $\mathcal{A}_G$-module, cf. Lemma~\ref{fgre}. To get a left adjoint $\op{Res}_!$ of the usual scalar extension $\mathcal{A}_H\otimes^{\op{L}}_{\mathcal{A}_G}(-)$, we can take (in the notation of \ref{EXR}) $E$ to be a finite resolution of $\mathcal{A}_H$ by graded free finitely generated modules over $\mathcal{A}_H\otimes_\Lambda\mathcal{A}_G$. The resulting adjunction would be 
\[
\op{Res}_!={}^\ast E\otimes^{\op{L}}_{\mathcal{A}_H}(-) \vdash E\otimes^{\op{L}}_{\mathcal{A}_G}(-)\approx \mathcal{A}_H\otimes_{\mathcal{A}_G}(-). 
\]
In the special case where $\mathcal{A}_H$ is free of finite rank over
$\mathcal{A}_G$, for example if $H=P$ is a parabolic in the reductive group $G$, we don't even need to resolve at all and get ${}^\ast\mathcal{A}_G \otimes_{\mathcal{A}_H}^{\op{L}}(-)$ as our looked-for left adjoint  $\op{Res}_!$.
\end{Bemerkung}

\begin{proposition}
\label{prop:compexcind}
In the situation of Theorem~\ref{thm:equivcoh}, let $\phi:H\hookrightarrow G$ be the inclusion of a closed subgroup in a linear algebraic group, and let $\phi^\ast:\mathcal{A}_G\to\mathcal{A}_H$ be the induced morphism on equivariant motivic cohomology rings of classifying spaces. Then the exceptional integration functor $\op{Ind}_!:\mathbb{D}_H^+(\pt)\to \mathbb{D}_G^+(\pt)$ restricts to a functor $\langle\mathcal{T}_H\rangle_\Delta\to \langle\mathcal{T}_G\rangle_\Delta$. 
\end{proposition}

\begin{proof}
Since $\op{Ind}_!$ is additive, this follows from the fact that $\op{Ind}_!\const{\pt}_H$ is by definition $\op{fin}_{G/H,!}\const{\pt}_{G/H}\cong D\op{fin}_{G/H,\ast}\const{\pt}_{G/H}$. But $\op{Res}_G^1\op{fin}_{G/H,\ast}\const{\pt}_{G/H}\cong\op{M}(G/H)$ is a mixed Tate motive, and the category of mixed Tate motives is closed under Verdier duality. The rest of the proof is the same as for the ordinary integration functor, cf. Proposition~\ref{prop:countable}. 
\end{proof}

\begin{Bemerkung}
As for the ordinary integration functor $\op{Ind}_H^G$, the exceptional integration functor $\op{Ind}_!:\mathbb{D}^+_H(\pt)\to\mathbb{D}^+_G(\pt)$ doesn't necessarily restrict to a functor $\mathcal{T}_H\to\mathcal{T}_G$. This only happens whenever the motive $\op{M}(G/H)$ is a pure Tate motive, e.g. in the case where $H=P$ is a parabolic subgroup. 
\end{Bemerkung}

\begin{Bemerkung}
It is also possible to prove compatibility results between induction functors and proper pushforwards as in \cite[Section 13.11]{BeLu}, but we won't need those in the present work. 
\end{Bemerkung}

\section{Motives as modules via tilting}
\label{sec:tilting}

In this section, we apply the general tilting results from Appendices~\ref{sec:fritz} and \ref{sec:tiltcompat} to identify, in suitable situations, equivariant mixed Tate motives over the point in terms of complexes of modules over the equivariant cohomology ring. Applying the compatibility results from Appendix~\ref{sec:tiltcompat} will show that these identifications are compatible with the relevant parts of the six functor formalism.

\subsection{The tilting result}

From the computations we established in Section~\ref{sec:motcohom}, we obtain a nicely behaved functor from equivariant motives to modules over the cohomology ring, given by equivariant cohomology (computed via the motivic derivator $\mathbb{D}$). We have also seen that the equivariant cohomology functor induces an equivalence on a suitable subcategory of equivariant pure Tate motives. The general tilting formalism allows to turn this into a fully faithful embedding of the homotopy category of modules into the category of equivariant motives whose essential image is the category of equivariant mixed Tate motives. 

The following result recovers, in particular, the known formality of the equivariant derived category of $G\looparrowright \pt$ of \cite[(12.7.2)]{BeLu}. 

\begin{theorem} 
\label{thm:tiltpoint}
Assume the situation of Theorem~\ref{thm:equivcoh}. 
\begin{enumerate}
\item Then tilting gives a
fully faithful embedding 
\[
\op{Hot}^{\op{b}}(S^W\rtimes\Lambda[\pi_0(G)]\op{-fModfg}^\DZ)  \stackrel{\approx}{\hra}\mathbb{D}^{\op{c}}_G (\op{pt}).
\]  
\item After idempotent completion, tilting provides a fully faithful embedding  
\[
\op{Hot}^{\op{b}}(S^W\rtimes\Lambda[\pi_0(G)]\op{-pModfg}^\DZ)  \stackrel{\approx}{\hra}\mathbb{D}^{\op{c}}_{G} (\op{pt}),
\]
where $S^W\rtimes\Lambda[\pi_0(G)]\op{-pModfg}^\DZ$ denotes the category of finitely generated graded projective modules over the twisted group ring.
\item
The essential image of the above embedding coincides with the category $\DMT_G(\op{pt})$ of $G$-equivariant mixed Tate motives over the point.
\end{enumerate}
\end{theorem}

\begin{proof}
This is  an application of the tilting result Theorem~\ref{thm:derivatortilting}. By Proposition~\ref{prop:equivderivator}, $\mathbb{D}^+_G(\op{pt},-)$ is a stable derivator. So we need to show that the category $\mathcal{T}$ of Theorem~\ref{thm:equivcoh} is a tilting subcategory in the sense of Definition~\ref{def:tiltsubcat}. For this, it suffices to show 
\[ 
\mathbb{D}^+_G(\ind^G_{G^0}\const{\pt}, \ind^G_{G^0}\const{\pt}(j)[i]) \neq 0
\]
implies that $i=2j$. By Theorem \ref{endring}, this reduces to showing 
\[ \op{H}^i_{\mathbb{D},G^0}(\pt; \Lambda(j)) \neq 0 \]
implies $i=2j$. In view of \ref{weaksplitting}, we may assume $G^0$ is a split torus. This, in turn, reduces to the case of $\GG_{\op{m}}$. For $\GG_{\op{m}}$ this is now obvious from the computation of $\op{H}^i_{\mathbb{D},\mathbb{G}_{\op{m}}}(\pt,\Lambda(j))$ in terms of the approximations $\mathbb{P}^n$ of ${\op{B}}\mathbb{G}_{\op{m}}$.

Now we use that Theorem~\ref{thm:equivcoh} provides an equivalence between $\mathcal{T}$ and the category of modules $S^W\rtimes\Lambda[\pi_0(G)]\op{-fModfg}^\DZ$. Since the conditions of the tilting result Theorem~\ref{thm:derivatortilting} are satisfied by the above discussion, we get a fully faithful functor $\op{Hot}^{\op{b}}(S^W\rtimes\Lambda[\pi_0(G)]\op{-fModfg}^\DZ)   \hra\mathbb{D}^{\op{c}}_G (\op{pt})$ as claimed in (1). 

For (2), we note that $\mathbb{D}^+_G(\op{pt})$ is idempotent complete, cf. Remark~\ref{rem:triangulated}. By \cite[Theorem 3.4(3)]{WtS}, the category $\op{Hot}^{\op{b}}(S^W\rtimes\Lambda[\pi_0(G)]\op{-pModfg}^\DZ)$ is idempotent complete because the category of projective modules is the idempotent completion of the category of free modules. Therefore, the idempotent completion of the homotopy category of bounded complexes of free modules is the homotopy category of complexes of projective modules.

It remains to establish the claim about the essential image in (3). By Proposition~\ref{prop:essimg} (and the equivalence of Theorem~\ref{thm:equivcoh}), the essential image of the tilting functor in (1) is the triangulated subcategory of $\mathbb{D}^+_G(\pt)$ generated by $\mathcal{T}$, and the essential image of the extension in (2) is the respective thick subcategory. The claim then follows from Proposition~\ref{prop:generatingtate}. 
\end{proof}

\begin{remark}
In the connected case, $\mathcal{A}_G=S^W$ is a polynomial ring. In this situation, projective modules are free, whence $\op{Hot}^{\op{b}}(S^W\textrm{-fModfg}^\mathbb{Z})$ is already idempotent complete and point (2) in Theorem~\ref{thm:tiltpoint} is not necessary. 
\end{remark}

\begin{remark}
\label{sts}   
\index{tilting t-structure}  
Since polynomial rings in finitely many variables over a field have finite homological dimension, the obvious functor 
\[
\op{Hot}^{\op{b}}(S^W\op{-fModfg}^\DZ) \sirra  \op{Der}^{\op{b}}(S^W\op{-Modfg}^\DZ)
\]
is an equivalence. The obvious t-structure on the derived category then corresponds to a t-structure on $\DMT_G(\pt)$, viewed as the thick subcategory of $\mathbb{D}^+_G(\pt)$ generated by  $\mathcal T$. We call it the \emph{tilting t-structure}. 
\end{remark}

\begin{corollary}
\label{cor:tiltpoint}
Assume the situation of Theorem~\ref{thm:equivcoh}. Then tilting gives an equivalence
\[
\op{Der}^{\op{b}}(S^W\rtimes\Lambda[\pi_0(G)]\op{-Modfg}^\DZ)  \stackrel{\approx}{\hra}\DMT_G (\op{pt}).
\]  
\end{corollary}

\begin{proof}
For the bounded homotopy category of finitely generated graded projective  $S^W\rtimes\Lambda[\pi_0(G)]$-modules, this is the statement of (2) and (3) of Theorem~\ref{thm:tiltpoint}. Replacing the bounded homotopy category of projective modules by the bounded derived category follows from Remark~\ref{sts} above. 
\end{proof}

\begin{remark}
In the situation of Corollary~\ref{cor:tiltpoint}, we have a commutative diagram
\[
\xymatrix{
\DMT_G(\op{pt}) \ar[r]^{\op{Real}} & \op{Der}^{\op{b}}_G(\op{pt}) \ar[d]^\approx \\
\op{Der}^{\op{b}}(\mathcal{A}_G{\op{-Modfg}^{\mathbb{Z}}})  \ar[u]^{\op{tilt}}_{\approx} \ar[r] & \op{dgDer}(\mathcal{A}_G,d=0).
}
\]
The left vertical arrow is the tilting equivalence of Corollary~\ref{cor:tiltpoint}. The top horizontal arrow is one of the realization functors of Proposition~\ref{prop:realization} and the right vertical arrow is the identification of the equivariant derived category with the category of dg-modules over the cohomology ring $(\mathcal{A}_G,d=0)$ established in \cite[Section 12.4]{BeLu}. The lower horizontal arrow is basically forgetting the grading, by mapping a complex $C_\bullet$ of $\mathbb{Z}$-graded $\mathcal{A}_G$-modules $C_{n,\bullet}$ to the complex $(\bigoplus_{i+j=n}C_{i,j})_n$ of dg-modules for $(\mathcal{A}_G,d=0)$. Similar diagrams will be established in Chapter~\ref{chap:repthy} for certain varieties $(G\looparrowright X)$ with action relevant for representation theory. 
\end{remark}


\subsection{Compatibility with restriction and induction}
Now we describe the relation between tilting and the restriction functors. Essentially, under the tilting equivalence between equivariant motives and modules over the motivic cohomology ring of the classifying space, restriction corresponds to extension of scalars along the relevant morphism of equivariant cohomology rings. 

\begin{proposition}[{\bf Compatibility with restriction}]
\label{tilting} 
Let $\phi:H\to G$ be a homomorphism of linear algebraic groups, and let $\phi^\ast:\mathcal{A}_G\to\mathcal{A}_H$ be the induced 
morphism on equivariant cohomology rings of classifying spaces. Then the diagram from Proposition~\ref{prop:compatres} induces a commutative diagram
\begin{displaymath}
\xymatrix{
\langle\mathcal{T}_G\rangle_\Delta\ar[r]^-{\approx}\ar[d]_-{\op{Res}^H_G} &\op{Der}^{\op{perf}}(\mathcal{A}_G\op{-Mod}^\DZ)\ar@{=>}_-\sim[dl]
\ar[d]^-{\mathcal{A}_H\otimes^{\op{L}}_{\mathcal{A}_G}(-)}
\\
\langle\mathcal{T}_H\rangle_\Delta\ar[r]_-{\approx} &  \op{Der}^{\op{perf}}(\mathcal{A}_H\op{-Mod}^\DZ) 
}
\end{displaymath}
where the categories $\langle\mathcal{T}_G\rangle_\Delta$ and $\langle\mathcal{T}_H\rangle_\Delta$ are the thick triangulated subcategories of $\mathbb{D}^+_G(\pt)$ and $\mathbb{D}^+_H(\pt)$ generated by $\op{Ind}_{G^0}^G\const{\pt}(n)[2n]$ and $\op{Ind}_{H^0}^H\const{\pt}(n)[2n]$, $n\in\mathbb{Z}$, respectively.
\end{proposition}

\begin{proof}
We want to deduce this from Theorem~\ref{thm:funtilt} and Proposition~\ref{prop:compatres}. Note that the horizontal equivalences arise from Theorem~\ref{thm:tiltpoint}.

First, by Proposition~\ref{prop:6fmorderivator}, the functor $\op{Res}_G^H:\mathbb{D}^+_G(\pt,-)\to\mathbb{D}^+_H(\pt,-)$ is a morphism of stable derivators. Since restriction is the left adjoint of $\op{Ind}_\ast$, cf. Proposition~\ref{prop:integration}, it preserves homotopy colimits. By \cite[Proposition 2.4]{groth}, $\op{Res}_G^H$ preserves homotopy left Kan extensions. 

Second, it follows as in the proof of Theorem~\ref{thm:tiltpoint} that the categories $\mathcal{T}_G$ and $\mathcal{T}_H$ are tilting subcategories. The fact that restriction $\op{Res}_G^H$ maps $\mathcal{T}_G$ into $\mathcal{T}_H$ follows from  Proposition~\ref{prop:compatres}.

From the above discussion, the conditions of Theorem~\ref{thm:funtilt} are satisfied, which provides us with a diagram, commutative up to isotransformation
\[
 \xymatrix{
\op{Hot}^{\op{b}}(\mathcal{T}_G)\ar[d]_{\op{Res}^H_G} \ar[r] & \mathbb{D}^+_G(\pt) \ar[d]^{\op{Res}_G^H} \ar@{=>}[dl]_{\sim}\\
\op{Hot}^{\op{b}}(\mathcal{T}_H) \ar[r] 
& \mathbb{D}^+_H(\pt). 
}
\]
We already showed in Theorem~\ref{thm:tiltpoint} that the essential images on the right-hand side are given by $\langle\mathcal{T}_G\rangle_\Delta$ and $\langle\mathcal{T}_H\rangle_\Delta$, respectively. This provides a square like the one claimed, except that we have homotopy categories of complexes instead of derived categories, and the functor hasn't been correctly identified. 

We first identify the functor. Using Proposition~\ref{prop:compatres}, we know that equivariant cohomology provides a diagram, commutative up to isotransformation
\[
\xymatrix{
  \op{Hot}^{\op{b}}(\mathcal{T}_G)  \ar[d]_-{\op{Res}_G^H} \ar[r]^-{\approx} & \op{Hot}^{\op{b}}(\mathcal{A}_G\op{-pModfg}^{\mathbb{Z}})\ar@{=>}_-\sim[dl]
  \ar[d]^-{\mathcal{A}_H\otimes_{\mathcal{A}_G}(-)} 
  \\
  \op{Hot}^{\op{b}}(\mathcal{T}_H) \ar[r]_-{\approx} &  \op{Hot}^{\op{b}}(\mathcal{A}_H\op{-pModfg}^{\mathbb{Z}}) 
}
\]
In particular, we get a diagram, commutative up to isotransformation
\[
 \xymatrix{
\op{Hot}^{\op{b}}(\mathcal{A}_G\textrm{-pModfg}^{\mathbb{Z}})\ar[d]_{\mathcal{A}_H\otimes_{\mathcal{A}_G}(-)} \ar[r] & \langle\mathcal{T}_G\rangle_\Delta \ar[d]^{\op{Res}_G^H} \ar@{=>}[dl]_{\sim}\\
\op{Hot}^{\op{b}}(\mathcal{A}_H\textrm{-pModfg}^{\mathbb{Z}}) \ar[r] 
& \langle\mathcal{T}_H\rangle_\Delta. 
}
\]

To replace the homotopy categories with derived categories is another application of tilting. The morphism of derivators 
\[
\mathcal{A}_H\otimes^{\op{L}}_{\mathcal{A}_G}(-):\op{Der}^{\op{perf}}(\mathcal{A}_G\textrm{-pModfg}^{\mathbb{Z}})\to \op{Der}^{\op{perf}}(\mathcal{A}_H\textrm{-pModfg}^{\mathbb{Z}}).
\] 
preserves homotopy left Kan extensions because it is a left adjoint. Applying Theorem~\ref{thm:funtilt} to the tilting subcategory $\mathcal{A}_G\op{-pModfg}^\mathbb{Z}\hookrightarrow \op{Der}^{\op{perf}}(\mathcal{A}_G\op{-Mod}^\mathbb{Z})$ of projective modules and using the fact that tensoring is a left adjoint, we get the following commutative diagram, commutative up to isotransformation:
\[
 \xymatrix{
\op{Hot}^{\op{b}}(\mathcal{A}_G\textrm{-pModfg}^{\mathbb{Z}})\ar[d]_{\mathcal{A}_H\otimes_{\mathcal{A}_G}(-)} \ar[r]^\approx & \op{Der}^{\op{perf}}(\mathcal{A}_G\textrm{-Mod}^{\mathbb{Z}}) \ar[d]^{\mathcal{A}_H\otimes^{\op{L}}_{\mathcal{A}_G}(-)} \ar@{=>}[dl]_{\sim}\\
\op{Hot}^{\op{b}}(\mathcal{A}_H\textrm{-pModfg}^{\mathbb{Z}}) \ar[r] 
& \op{Der}^{\op{perf}}(\mathcal{A}_H\textrm{-Mod}^{\mathbb{Z}}). 
}
\]
Using the derived tensor product on the right is the only sensible thing to do, but since the homotopy category contains projective objects which are flat, we can use the underived tensor product on the left-hand side. The only thing that needs to be justified is the use of the perfect complexes on the right-hand side. However, perfect complexes are quasi-isomorphic to bounded complexes of finitely generated projective modules. Using Proposition~\ref{prop:essimg}, the essential image of the homotopy category of bounded complexes of finitely generated free graded $\mathcal{A}_G$-modules inside the derived category of $\mathcal{A}_G$-modules is in fact the category of perfect complexes, as claimed. 
\end{proof}

Next we discuss ordinary integration, the right adjoint to restriction. 

\begin{proposition}[{\bf Compatibility with induction}]
\label{tiltingv} 
Let $\phi:H\hookrightarrow G$ be the inclusion of a closed subgroup into a linear algebraic group and let $\phi^\ast:\mathcal{A}_G\to\mathcal{A}_H$ be the induced morphism of equivariant cohomology rings of classifying spaces. By passing to right adjoints, the diagram of Proposition~\ref{tilting} induces a diagram, commutative up to isotransformation
\[
\xymatrix{
\langle\mathcal{T}_G\rangle_\Delta\ar[r]^-{\approx} \ar@{=>}^-\sim[dr] & \op{Der}^{\op{perf}}(\mathcal{A}_G\op{-Mod}^\DZ)
\\
\langle\mathcal{T}_H\rangle_\Delta \ar[r]_-{\approx}\ar[u]^-{\op{Ind}_H^G} &  
\op{Der}^{\op{perf}}(\mathcal{A}_H\op{-Mod}^\DZ) \ar[u]_-{\op{Res}_{\mathcal{A}_H}^{\mathcal{A}_G}}
}
\]
\end{proposition}

\begin{proof}
The left-hand vertical arrow $\op{Res}_G^H$ in the diagram of Proposition~\ref{tilting} has a right adjoint $\op{Ind}_H^G:\mathbb{D}^+_H(\pt)\to\mathbb{D}^+_G(\pt)$. Moreover, the induction functor maps equivariant mixed Tate motives to equivariant mixed Tate motives by Proposition~\ref{prop:countable}.  In particular, the Res-Ind-adjunction restricts to the subcategories $\mathcal{T}_G$ and $\mathcal{T}_H$. Similarly, the right vertical in the diagram of Proposition~\ref{tilting} above also admits a right adjoint, namely restricting the action from $\mathcal{A}_H$ to $\mathcal{A}_G$. Uniqueness of right adjoints now implies the claimed commutativity up to isotransformation of the diagram. 
\end{proof}

The more general case of arbitrary group homomorphisms $\phi:H\to G$ is slightly more problematic, as can already be seen in the special case of the final group morphism $\phi:G\to\{1\}$. The push-forward of the constant equivariant motive on ${\op{B}}G$ to the point will be the motive of the classifying space. Strictly speaking, this is not a mixed Tate motive because it is (usually) not a compact object since the cohomology of ${\op{B}}G$ is not a finite module over the coefficient ring. Typically, the motive of ${\op{B}}G$ will only be an Ind-mixed Tate motive. Therefore, to discuss the possible commutativity of diagrams involving the right adjoints to restriction functors and the monoidal structure, we need to pass to a slightly larger category containing such Ind-mixed Tate motives. 

\begin{remark}
\label{rem:countable}
Given a $\DZ$-graded ring  $A$ let $A\op{-Modcg}^\DZ\supset A\op{-fModcg}^\DZ$ denote the categories of $\DZ$-graded modules with a countable set of generators and the full subcategory of graded free such modules respectively. If $B\ra A$ is a homomorphism such that $A$ is a countably generated $B$-module, then the extension of scalars $B\op{-Modcg}^\DZ\ra A\op{-Modcg}^\DZ$ admits a right adjoint, the restriction of scalars. Similar statements would hold for other cardinalities. We choose here to work with a rather restricted setting to keep clear as much as possible of set-theoretic considerations.
\end{remark}

\begin{remark}
\label{MGHT} 
Note that there is no finiteness condition in the tilting theorem~\ref{thm:fritz}; it applies to diagrams indexed by directed categories. Therefore, we can also get a fully faithful embedding
\[
\op{Hot}^{+}(S^W\op{-fModcg}^\DZ) \stackrel{\sim}{\hra}\mathbb{D}_G^+(\pt)
\]
where $S^W\op{-fModcbg}^\DZ$ denotes the category of free countably generated graded modules bounded below in their degrees. 
\end{remark}

\begin{proposition}[{\bf Compatibility with general direct image}]
\label{tiltingv:countable} 
Let $\phi:H\to G$ be a homomorphism of connected linear algebraic groups and let $\phi^\ast:\mathcal{A}_G\to\mathcal{A}_H$ be the induced morphism of equivariant cohomology rings of classifying spaces.
\begin{enumerate}
\item 
Let $\langle\mathcal{T}_G^\infty\rangle_\Delta \subset\mathbb{D}_G^+(\pt)$ and $\langle\mathcal{T}_H^\infty\rangle_\Delta \subset\mathbb{D}_H^+(\pt)$ denote the full thick triangulated subcategories generated by $\mathcal{T}_G^\infty$ and $\mathcal{T}_H^\infty$ respectively.
Then the above induces a commutative diagram
\begin{displaymath}
\xymatrix{
\langle\mathcal{T}_G^\infty\rangle_\Delta \ar[r]^-{\approx}\ar[d]_-{\op{Res}^H_G} 
&\op{Der}^{\op{b}}(\mathcal A_G\op{-Modcg}^\DZ)\ar@{=>}_-\sim[dl] \ar[d]^-{\mathcal A_H\otimes^{\op{L}}_{\mathcal A_G}(-)}
\\
\langle\mathcal{T}_H^\infty\rangle_\Delta\ar[r]_-{\approx} &  \op{Der}^{\op{b}}(\mathcal A_H\op{-Modcg}^\DZ) 
}
\end{displaymath}
\item
By passing to right adjoints the above induces a commutative diagram
\begin{displaymath}
\xymatrix{
\langle\mathcal{T}_G^\infty\rangle_\Delta\ar[r]^-{\approx} \ar@{=>}^-\sim[dr]
&\op{Der}^{\op{b}}(\mathcal A_G\op{-Modcg}^\DZ)
\\
\langle\mathcal{T}_H^\infty\rangle_\Delta
\ar[r]_-{\approx}\ar[u]^-{\op{Ind}_H^G} &  
\op{Der}^{\op{b}}(\mathcal A_H\op{-Modcg}^\DZ) \ar[u]_-{\op{Res}}
}
\end{displaymath}\label{TInd} 
\end{enumerate}
\end{proposition}

\begin{proof}
(1) follows from part (2) of Proposition~\ref{prop:compatres}, basically as in the proof of Proposition~\ref{tilting}. 

For (2), the left-hand vertical arrow $\op{Res}_G^H$ has a right adjoint $\op{Ind}_H^G:\mathbb{D}^+_H(\pt)\to\mathbb{D}^+_G(\pt)$. Moreover, the induction functor maps mixed Tate motives to Ind-mixed Tate motives by \ref{countable}.  In particular, the Res-Ind-adjunction restricts to the countably generated subcategories. Similarly, the right vertical in the second diagram above also admits a right adjoint, namely restricting the action from $\mathcal{A}_H$ to $\mathcal{A}_G$, by Remark~\ref{rem:countable}. Actually the whole point of working with countably generated modules was to ensure that such an adjoint would exist. Uniqueness of right adjoints now implies the claimed commutativity of the diagram. 
\end{proof}

\begin{remark}
Alternatively, the identification of right adjoints can also be seen on realization: denote by $\op{TInd}_H^G$ the functor on the left-hand side corresponding under tilting to the right adjoint of the right vertical. Given $N\in \langle\mathcal{T}_H^\infty\rangle_\Delta$ the  counit of the adjunction $\op{Res}_G^H\op{TInd}_H^GN\ra N$ will now induce a natural morphism 
$$
\op{TInd}_H^GN\ra \op{Ind}_H^GN
$$
When we pass to the Betti realization, this has to become an isomorphism by \cite[12.7.2]{BeLu}.  Therefore it must have been an isomorphism in the first place, by a countable extension of the conservativity statement from Proposition~\ref{prop:conservative}.
\end{remark}

We now want to discuss the compatibility of tilting with the exceptional integration functors of \ref{PBPF}. 

\begin{proposition}[{\bf Compatibility with exceptional integration}]
\label{tilting:exc} 
\hspace{1cm}\newline 
Let $\phi:H\hookrightarrow G$ be the inclusion of a closed subgroup in a linear algebraic group and let $\phi^\ast:\mathcal{A}_G\to\mathcal{A}_H$ be the induced morphism of equivariant cohomology rings of classifying spaces. Let  $\langle\mathcal{T}_G\rangle_\Delta \subset\mathbb{D}_G^+(\pt)$ and $\langle\mathcal{T}_H\rangle_\Delta \subset\mathbb{D}_H^+(\pt)$ denote the full thick triangulated subcategories generated by $\mathcal{T}_G$ and $\mathcal{T}_H$ respectively. Then the diagram of Proposition~\ref{prop:compexcind} induces a diagram commutative up to isotransformation
\[
\xymatrix{
\langle\mathcal{T}_G\rangle_\Delta \ar[r]^-{\approx}\ar[d]_-{\op{Ind}_!} & \op{Der}^{\op{perf}}(\mathcal{A}_G\op{-Mod}^\DZ)\ar@{=>}_-\sim[dl] \ar[d]^-{{}^\ast\mathcal{A}_G\otimes^{\op{L}}_{\mathcal{A}_H}(-)}
\\
\langle\mathcal{T}_H\rangle_\Delta\ar[r]_-{\approx} & \op{Der}^{\op{perf}}(\mathcal{A}_H\op{-Mod}^\DZ) 
}
\]
\end{proposition}

\begin{proof}
The claim follows from Theorem~\ref{thm:funtilt}. Exceptional integration is a morphism of derivators which preserves homotopy left Kan extensions, because it is a left adjoint by Proposition~\ref{prop:integration}. The rest of the argument follows closely the one in the proof of Proposition~\ref{tilting}.
\end{proof}

\subsection{Compatibility with monoidal structure}

\begin{proposition}[{\bf Compatibility with tensor}]
\label{prop:tilttensor}
In the situation of Theorem~\ref{thm:tiltpoint}, the tilting equivalence
\[
\op{Hot}^{\op{b}}(S^W{\rtimes}\Lambda[\pi_0(G)]\op{-fModfg}^\DZ) 
\stackrel{\sim}{\hra}\DMT_G (k)
\]
is a $\otimes$-functor. On the source, we use the monoidal structure coming from the tensor product of graded $S^W{\rtimes}\Lambda[\pi_0(G)]$-modules; on the target, we use the monoidal structure of equivariant motives from Remark~\ref{rem:monoidal}. 
\end{proposition}

\begin{proof}
This will follow from Theorem~\ref{thm:tiltmonoid}. By Proposition~\ref{gmonderivator}, $\mathbb{D}_G^+(\op{pt},-)$ is a monoidal stable derivator. We already discussed in the proof of Theorem~\ref{thm:tiltpoint} that the category $\mathcal{T}$ from Theorem~\ref{thm:equivcoh} is a tilting subcategory. Note that the category $\mathcal{T}$ is closed under $\otimes$ by Proposition~\ref{prop:compatotimes}. 

Applying Theorem~\ref{thm:tiltmonoid} will then produce a monoidal tilting functor. However, the monoidal structure used in Theorem~\ref{thm:tiltmonoid} is the one induced from the monoidal structure on $\mathcal{T}$ given by Remark~\ref{rem:monoidal}. To get the claim, we need that the equivalence between motives and modules -- given by equivariant cohomology -- is a $\otimes$-functor when we consider the above monoidal structure on $\mathcal{T}$ and the usual tensor product on modules. This is guaranteed by Proposition~\ref{prop:compatotimes}. 
\end{proof}

\begin{Bemerkung}[{\bf Tilting and bimodules}]
Let $\mathbb{D}$ be a derivator satisfying the conditions \ref{derivator:new} and the grading condition~\ref{conditions:grading}. Let $G$ be a connected affine algebraic group. Using the shorthand notation $\mathcal{A}_G$ for the equivariant motivic cohomology ring of the classifying space, cf. \ref{agnotation}, the tilting result Theorem~\ref{thm:tiltpoint}, or better its Corollary~\ref{cor:tiltpoint}, gives an equivalence of categories
\[
\op{tilt}:\op{Der}^{\op{b}}(\mathcal{A}_G\op{-Modfg}^\DZ) \sirra  \op{MTDer}_G(\op{pt})
\]
between the category of $G$-equivariant mixed Tate motives on the point and the bounded derived category of the category of finitely generated graded $\mathcal{A}_G$-modules.

For two connected affine algebraic groups $G$ and $H$, 
the K\"unneth formula holds (under our assumptions) for the equivariant motivic cohomology of the classifying spaces, hence there is a canonical isomorphism
$\mathcal A_{G}\otimes_\Lambda \mathcal A_{H}\xrightarrow{\cong} \mathcal A_{G\times H}$.
In this way we obtain an equivalence of categories 
\[
\mathcal A_{G\times H}\op{-Modfg}^\DZ\sirra  \mathcal A_{G}\op{-Modfg_\Lambda^\DZ-}\mathcal A_{H}.
\]
Here on the right we put the index $\Lambda$ to indicate that we consider only those bimodules, for which the left and right actions of the coefficient field $\Lambda$ coincide. In this way, for a product of two connected groups, our tilting equivalence takes the form of an equivalence 
\[
\op{tilt}: \op{Der}^{\op{b}}(\mathcal A_G\op{-Modfg_\Lambda^\DZ-}\mathcal A_H) \sirra \op{MTDer}_{G\times H}(\op{pt}).
\]
\end{Bemerkung}

With this notation, we can now discuss the correspondence between the exterior products on the motivic and the module side under the tilting equivalences. 

\begin{proposition}[{\bf Compatibility with exterior products}]
\label{prop:exterior}
\hspace{1cm}\newline 
Let $H$ and $G$ be linear algebraic groups and let $\mathbb{D}$ be a derivator satisfying the conditions~\ref{derivator:new} and \ref{conditions:grading}. 
\begin{enumerate}
\item The following diagram commutes 
\begin{displaymath}
\xymatrix{
\mathcal{T}_G\times \mathcal{T}_H\ar[r]^-{\approx}\ar[d]_-{\boxtimes} &\mathcal{A}_G\op{-fModfg}^\DZ\times 
\mathcal{A}_H\op{-fModfg}^\DZ
\ar@{=>}_-\sim[dl]\ar[d]^-{\boxtimes}\\
\mathcal{T}_{G\times H} \ar[r]^-{\approx}&  
\mathcal{A}_{G\times H}\op{-fModfg}^\DZ.
}
\end{displaymath}
Again, the categories $\mathcal{T}_G$ and $\mathcal{T}_H$ are the additive subcategories of $\mathbb{D}^+_G(\pt)$ and $\mathbb{D}^+_H(\pt)$ generated, for  $n\in\mathbb{Z}$, by $\op{Ind}_{G^0}^G\const{\pt}(n)[2n]$ and $\op{Ind}_{H^0}^H\const{\pt}(n)[2n]$, respectively. The right-hand vertical map is defined using the canonical isomorphism 
$\mathcal{A}_{G\times H}\xrightarrow{\cong} \mathcal{A}_G\otimes_\Lambda \mathcal{A}_H$.
\item The following diagram commutes
\begin{displaymath}
\xymatrix{
\mathbb{D}^+_G(\pt)\times \mathbb{D}^+_H(\pt)\ar[d]_-{\boxtimes} &\op{Hot}^{\op{b}}(\mathcal{A}_G\op{-fModfg}^\DZ)\times 
\op{Hot}^{\op{b}}(\mathcal{A}_H\op{-fModfg}^\DZ)
\ar@{=>}_-\sim[dl]\ar[d]^-{\boxtimes}
\ar@{_{(}->}[l]\\
\mathbb{D}_{G\times H}^+(\pt) &  \op{Hot}^{\op{b}}(\mathcal{A}_{G\times H}\op{-fModfg}^\DZ). \ar@{_{(}->}[l]
}
\end{displaymath}
\end{enumerate}
\end{proposition}

\begin{proof}
Recall from Definition~\ref{def:exterior} that the exterior product on the motivic side is given by taking $M\in\mathbb{D}^+_G(\pt)$ and $N\in\mathbb{D}^+_H(\pt)$ and setting 
\[
M\boxtimes N:=p_1^\ast M\otimes p_2^\ast N\in\mathbb{D}^+_{G\times H}(k).
\]
Here $p_1:(G\times H\looparrowright \pt)\to (G\looparrowright \pt)$ and $p_2:(G\times H\looparrowright \pt)\to(H\looparrowright\pt)$ are the projection morphisms from the product; in particular, $p_1^\ast=\op{Res}_G^{G\times H}$ and $p_2^\ast=\op{Res}^{G\times H}_H$ are restriction functors. The restriction functors preserve equivariant mixed Tate motives by Proposition~\ref{prop:dmtres} and they even preserve the categories $\mathcal{T}_?$ by Proposition~\ref{prop:compatres}. Moreover, the tensor product preserves equivariant mixed Tate motives by Proposition~\ref{prop:monodmt}, and it even preserves the categories $\mathcal{T}_{G\times H}$ by Proposition~\ref{prop:compatotimes}. In particular, the exterior product of equivariant motives restricts to a functor $\boxtimes:\mathcal{T}_G\times \mathcal{T}_H \to\mathcal{T}_{G\times H}$. The fact that the square in (1) can be filled by an isotransformation also follows from Propositions~\ref{prop:compatres} and \ref{prop:compatotimes}. 

Assertion (2) then follows by applying the tilting theorem and its compatibility with tensor, cf.~\ref{prop:tilttensor}, and restriction, cf.~\ref{tilting}. 
\end{proof}

\subsection{Descent along separable field extenions}
\label{descent}

We shortly discuss how the equivariant mixed Tate motives behave in separable field extensions. Assume that the underlying derivator satisfies the conditions~\ref{derivator:new} and the grading condition~\ref{conditions:grading}. As a consequence of the grading condition, we find that for any finite field extension $L/k$ of the base field $k$, the ordinary pullback functor $f^\ast:\mathbb{D}(\op{Spec} k)\to \mathbb{D}(\op{Spec} L)$ induces an equivalence $\DMT(\op{Spec} k)\simeq \DMT(\op{Spec} L)$. In particular, the categories of mixed Tate motives over the point don't change in the finite field extension $L/k$. If the derivator is continuous (as it happens for the derivators $\mathbf{DA}^{\et}_{\mathbb{Q}}$ and $\op{MDer}(-,\mathbb{C})$), then the same is also true for the extension $\overline{k}/k$.  

Assuming the grading condition~\ref{conditions:grading}, Proposition~\ref{prop:swcomp} implies that the equivariant motivic cohomology ring of the point $\mathcal{A}_G$ for a split reductive group $G$ is identified with the ordinary cohomology ring of the classifying space of the corresponding complex Lie group. For a finite field extension $L/k$ of the base field, the natural morphism from the equivariant motivic cohomology of $G$ computed over $k$ to the one computed over the extension field $L$ is an isomorphism of rings. Combining this with Corollary~\ref{cor:tiltpoint}, the ordinary pullback functor $\DMT_G(\op{Spec} k)\to\DMT_{G_L}(\op{Spec} L)$ is an equivalence of categories. Being induced from ordinary pullback, these equivalences are compatible with all the six functors, cf. Section~\ref{sec:further}. In particular, they are compatible with the induction equivalences. Consequently, for a homogeneous space $G\looparrowright G/H$, the ordinary pullback functor $\DMT_G(G/H)\to\DMT_{G_L}(G/H\times_kL)$ is an equivalence of categories. 

If we have a variety with action $G\looparrowright X$ having finitely many orbits separably defined over the base field $k$. An induction on the number of orbits, using the localization sequences implies that the ordinary pullback functor $\DMT_G(X)\to \DMT_{G_L}(X_L)$ induces an equivalence of categories. Moreover, as before, these equivalences for $X$ and the various unions of orbits are compatible with the equivariant six functors. The same thing can be said about the categories $\DMT_G^\ast(X)$ and $\DMT_G^!(X)$. As a consequence, the variety with action $G\looparrowright X$ satisfies the equivariant Whitney--Tate condition~\ref{def:mtderdef2} if the base change $G_{\overline{k}}\looparrowright X_{\overline{k}}$ to the algebraic closure satisfies the equivariant Whitney--Tate condition. This will simplify some of the discussion of equivariant Whitney--Tate conditions in the representation-theoretic applications in Chapter~\ref{chap:repthy}. 

The comparison between equivariant mixed Tate motives over a field $k$ and its algebraic closure $\overline{k}$ also allows to relate the motivic framework to earlier works on mixed geometry and gradings on equivariant derived categories. These works usually assumed that the variety with action is defined over a finite field and then worked over the algebraic closure to detect weights in terms of eigenvalues of Frobenius, cf. e.g. \cite{Mars-Springer}. The comparison between equivariant mixed Tate motives over finite fields and their algebraic closure now allows to relate the motivic weights to the weights of Frobenius actions in the $\ell$-adic realizations, cf. \ref{rem:fq}.

\section{Weight structures on equivariant motives}
\label{sec:weights}

In this section, we show that there is a weight structure on the category of equivariant mixed Tate motives. This weight structure is glued from weight structures on equivariant motives over the point, and the latter is a consequence of the tilting results established in Section~\ref{sec:tilting}. With this definition, the equivariant six functors exhibit the expected exactness properties for weights. 

Unfortunately, it is not clear at the moment if the subcategories induced from the weight structures on motives described by Bondarko and H{\'e}bert via the forgetful functor are part of a weight structure on the full category of equivariant motives. 

\subsection{Recollection on weight structures on motives}

We first recall, for the reader's convenience, the definition of weight structures from \cite[Definition 1.1.1]{bondarko:ktheory}. Note, however, that our sign convention for the weight is opposite to the one of loc.cit. We follow the sign convention used in most other works on weight structures, such as \cite{wildeshaus:intermediate} and \cite{hebert}. 

\begin{definition}
\label{defin:wtstruct}
\index{weight structure}
Let $\mathcal{C}$ be a triangulated category. A \emph{weight structure on $\mathcal{C}$} is a pair $w=(\mathcal{C}_{\op{wt}\leq 0},\mathcal{C}_{\op{wt}\geq 0})$ of full subcategories of $\mathcal{C}$ such that with the notations $\mathcal{C}_{\op{wt}\leq n}:=\mathcal{C}_{\op{wt}\leq 0}[n]$ and $\mathcal{C}_{\op{wt}\geq n}:=\mathcal{C}_{\op{wt}\geq 0}[n]$ the following conditions are satisfied: 
\begin{enumerate}
\item the categories $\mathcal{C}_{\op{wt}\leq 0}$ and $\mathcal{C}_{\op{wt}\geq 0}$ are closed under taking direct summands;
\item $\mathcal{C}_{\op{wt}\leq 0}\subset \mathcal{C}_{\op{wt}\leq 1}$ and $\mathcal{C}_{\op{wt}\geq 1}\subset \mathcal{C}_{\op{wt}\geq 0}$;
\item for any pair of objects $X\in\mathcal{C}_{\op{wt}\leq 0}$, $Y\in \mathcal{C}_{\op{wt}\geq 1}$, we have $\mathcal{C}(X,Y)=0$; 
\item for any object $X\in\mathcal{C}$ there is a distinguished triangle $A\to X\to B\to A[1]$ with $A\in\mathcal{C}_{\op{wt}\leq 0}$ and $B\in\mathcal{C}_{\op{wt}\geq 1}$. 
\end{enumerate}
The full subcategory $\mathcal{C}_{\op{wt}=0}=\mathcal{C}_{\op{wt}\leq 0}\cap\mathcal{C}_{\op{wt}\geq 0}$ is called the \emph{heart of the weight structure $w$}. 
\end{definition}

For an arbitrary ground field $k$ and any $k$-variety $X$, H{\'e}bert \cite[Theorems 3.3 and 3.8]{hebert} has constructed a canonical weight structure on the category of Beilinson motives over $X$. The following formulation for weight structures on $\mathbf{DA}^{\et}$ follows from H{\'e}bert's result together with the comparison theorems between $\mathbf{DA}^{\et}$ and $\op{DM}$.\footnote{Note that we are using rational coefficients. The weight structure is not expected to exist integrally for \'etale motives.}

\begin{theorem}
\label{thm:hebert}
Let $k$ be any base field, and let $\Lambda$ be a coefficient field of characteristic $0$. For any separated scheme $X$ of finite type over $k$, there is a canonical weight structure $w$ on $\mathbf{DA}^{\et,\op{c}}(X;\Lambda)$. The family of these weight structures on $\mathbf{DA}^{\et,\op{c}}(X;\Lambda)$ is characterized uniquely by the following  properties:  
\begin{enumerate}
\item if $X$ is regular, then $\const{X}(n)[2n]\in\mathbf{DA}^{\et,\op{c}}(X;\Lambda)_{\op{wt}=0}$ for all $n\in\mathbb{Z}$, and 
\item for any separated finite type morphism $f:X\to Y$, the functors $f^\ast$, $f_!$ (and $f_\sharp$ for $f$ smooth) are $w$-left exact, i.e., they preserve non-positivity of weights,
\item for any separated finite type morphism $f:X\to Y$, the functors $f_\ast$, $f^!$ (and $f^\ast$ for $f$ smooth) are $w$-right exact, i.e., they preserve non-negativity of weights. 
\end{enumerate}
\end{theorem}

\begin{remark}
\label{WDM} 
If $k$ is a field, then by \cite[Theorem 2.5]{wildeshaus:at} the restriction of the above weight structure $w$ on $\mathbf{DA}^{\et,\op{c}}(\pt)$ is a weight structure on $\DMT(\pt)$.
\end{remark}

By the work of Drew, cf. \cite[Theorem 2.3.2-2.3.4]{drew:thesis}, there are analogues of H{\'e}bert's theorems on weight structures for motives with coefficients in an enriched mixed Weil cohomology theory $\mathcal{E}$ provided the following axiom is satisfied:
\begin{description}
\item[(W6)] for all smooth affine schemes $X$ over the base, $r,s\in\mathbb{Z}$ with $2r<s$, we have
\[
\op{Hom}(\mathbb{Q}_{\mathcal{T}},E(X)(r)[s])=0.
\]
\end{description}
Whenever this axiom is satisfied, then there is, for each $k$-scheme $X$, a weight structure on $\op{Der}(\mathcal{E}_X)$ whose heart is generated by $\mathcal{E}_X$-motives of $Y$ for $f:Y\to X$ projective and $Y$ regular. This in particular applies to the motivic categories associated to the Hodge realization with values in $\mathcal{MHS}^{\op{pol}}_{\mathbb{Q}}$, cf. the remark on p.8 of \cite{drew:thesis}. The same also holds for the cohomology theory $E_{\op{GrH}}$ which takes the associated graded of the weight filtration on Hodge realization, cf. the discussion in Sections 2.3 and 2.5 of \cite{SoWe}.

\begin{convention}
\label{conditions:weight}
\index{weight condition}
For employing weight arguments in the later representation-theoretic applications, we will consider a more restricted setting. A homotopical stable algebraic derivator $\mathbb{D}$ over $\op{Var}/k$ satisfying the conditions of \ref{derivator:new} is said to satisfy the \emph{weight condition} if for each variety $X$ over $k$ there is a weight structure on $\mathbb{D}(X)$, such that this collection of weight structures satisfies the conclusion of the Bondarko--H{\'e}bert theorem, cf.~Theorem~\ref{thm:hebert}. The two main situations relevant for our applications are the ones mentioned in \ref{conditions:grading}. 
\end{convention}

\subsection{Weight structures for equivariant mixed Tate motives}  
Now we are ready to define weight structures on the categories $\DMT_G(X)$ of equivariant mixed Tate motives. The basic ingredient is the explicit identification of equivariant mixed Tate motives over the point in terms of $\mathcal{A}_G$-modules from Section~\ref{sec:tilting}. 

\begin{proposition}
\label{prop:wtpt}
Let $G$ be a linear group over the base field $k$, and let $\mathbb{D}$ be a derivator satisfying the conditions of \ref{derivator:new} and the grading condition~\ref{conditions:grading}. Then there is a weight structure on $\DMT_G(\pt)$ whose heart is the idempotent completion of the additive subcategory $\mathcal{T}_G$ of $\mathbb{D}^+_G(\pt)$ generated by the objects $\op{Ind}_{G^0}^G\const{\pt}(n)[2n]$, $n\in\mathbb{Z}$. 
\end{proposition}

\begin{proof}
By Theorem~\ref{thm:equivcoh}, the functor 
\[
M\mapsto \bigoplus_i\mathbb{D}^+_G\left(\op{Ind}_{G^0}^G\Lambda,M(i)[2i]\right)
\]
induces an equivalence between $\mathcal{T}_G$ and the category $\mathcal{A}_G\op{-fModfg}^{\mathbb{Z}}$ of finitely generated free graded modules over the twisted group ring $\mathcal{A}_G=S^W{\rtimes}\Lambda[\pi_0(G)]$. By Theorem~\ref{thm:tiltpoint}, this equivalence can be extended to a fully faithful embedding 
\[
\op{Hot}^{\op{b}}(\mathcal{A}_G\op{-fModfg}^{\mathbb{Z}}) \stackrel{\sim}{\hra} \mathbb{D}^{\op{c}}_G(\pt)
\]
whose essential image is exactly the category $\DMT_G(\pt)$ of equivariant mixed Tate motives over the point. Now we can make use of the basic example of weight structure, namely that on the homotopy category of complexes over an additive category. This example is mentioned in \cite[Remark 1.6(2)]{bondarko:imrn} and the idempotent completeness of the category of complexes is proved in \cite[Theorem 1.2]{WtS}. As a consequence, there is a weight structure on the category $\op{Hot}^{\op{b}}(\mathcal{A}_G\op{-fModfg}^{\mathbb{Z}})$, and the equivalence to $\DMT_G(\pt)$ above proves the claim.
\end{proof}

\begin{remark}
Let us sketch an alternative proof. Under the grading condition~\ref{conditions:grading}, the  tilting subcategory $\mathcal{T}_G$ is a negative collection in the sense of \cite[Definition 1.5.VII]{bondarko:imrn} by Theorem~\ref{thm:equivcoh}. Moreover, by Proposition~\ref{prop:generatingtate}, the category $\mathcal{T}_G$ generates $\DMT_G(\pt)$ in the sense that $\DMT_G(\pt)$ is the smallest idempotent complete triangulated subcategory of $\mathbb{D}^+_G(\pt)$ containing $\mathcal{T}_G$. Note that we established in Lemma~\ref{lem:idempotent} that the category $\DMT_G(X)$ is idempotent complete. Therefore, applying \cite[Proposition 1.7(6)9]{bondarko:imrn}, we get a unique weight structure on $\DMT_G(\pt)$ whose heart is the idempotent completion of $\mathcal{T}_G$. This weight structure will necessarily coincide with the one above. 
\end{remark}

\begin{corollary}
\label{cor:wtorb}
Let $\mathbb{D}$ be a derivator satisfying the conditions of \ref{derivator:new} and the grading condition~\ref{conditions:grading}. Let $G$ be a linear group over the base field $k$, and let $H\leq G$ be a closed subgroup defined over $k$ such that the quotient $G/H$ exists as a variety. Then there is a weight structure on $\DMT_G(G/H)$ whose heart is generated by those motives whose restriction to a point is of the form $\op{Ind}_{H^0}^H\const{\pt}(n)[2n]$, $n\in\mathbb{Z}$. 
\end{corollary}

\begin{proof}
This follows from the existence of weight structures on equivariant mixed Tate motives over the point, cf. Proposition~\ref{prop:wtpt}, and the fact that equivariant mixed Tate motives are stable under induction equivalence, cf. Propositions~\ref{cor:indequiv} and \ref{prop:quotientdmt}.
\end{proof}

Now we can glue the weight structures for orbits above and define weight structures on $G$-varieties with finitely many orbits. 

\begin{proposition}
\label{prop:equivweight} 
Let $\mathbb{D}$ be a derivator satisfying the conditions of \ref{derivator:new} and the grading condition~\ref{conditions:grading}. Let $k$ be a field, let $G$ be a linear algebraic group, and let $G\looparrowright X$ be a variety with action having finitely many $G$-orbits separably defined over $k$. Assume that $G\looparrowright X$ is equivariantly Whitney--Tate, cf. Definition~\ref{def:mtderdef2}. Then there is a weight structure on the category $\DMT_G(X)$, given as follows:
\begin{enumerate}
\item The subcategory $\DMT_G(X)_{\op{wt}\leq 0}$ is the full subcategory on the motives $M\in\DMT_G(X)$ such that $j^\ast M \in \DMT_G(Y)_{\op{wt}\leq 0}$ for all $G$-orbits $j:Y\to X$.
\item The subcategory $\DMT_G(X)_{\op{wt}\geq 0}$ is the full subcategory on the motives $M\in\DMT_G(X)$ such that $j^! M \in \DMT_G(Y)_{\op{wt}\geq 0}$ for all $G$-orbits $j:Y\to X$.
\end{enumerate}
\end{proposition}

\begin{proof}
The argument is similar to the one for stratified mixed Tate motives in \cite[Proposition 5.1]{SoWe}. To prove the existence of such a weight structure, we proceed by induction on the number of orbits. If there is a single orbit, the claim is proved in Corollary~\ref{cor:wtorb}. Otherwise, decompose $X$ as the disjoint union of an open orbit $j:X_s\hra X$ and its closed $G$-stable complement $i:Z\hra X$.  Using Bondarko's result \cite[Proposition 1.7 (13), (15)]{bondarko:imrn} on glueing weight structures, we obtain a weight structure on $\DMT_G(X)$  by setting
\[
\begin{array}{l}
\DMT_G(X)_{\op{wt}\leq 0}\pdef \left\{M\mid i^\ast M\in \DMT_G(Z)_{\op{wt}\leq 0}, \;
  j^\ast M\in \DMT_G(X_s)_{\op{wt}\leq 0}\right\} \\[2mm]
\DMT_G(X)_{\op{wt}\geq 0}\pdef\left\{M\mid i^!M\in \DMT_G(Z)_{\op{wt}\geq 0}, \;
  j^! M\in \DMT_G(X_s)_{\op{wt}\geq 0}\right\}  
\end{array}
\]
This implies the existence of the weight structure as well as the characterization of non-positive resp. non-negative weight parts claimed.
\end{proof}

\begin{example}
Relevant examples for our later applications are parabolic group actions on partial flag varieties $\DMT_P(G/Q)$, double Borel actions on wonderful compactifications $\DMT_{B\times B}(X)$ or Borel actions on symmetric varieties $\DMT_B(G/K)$. We will see applications of the above weight structures as well as alternative arguments for their existence (using a specific generating collection, the Bott-Samelson motives) in Section~\ref{sec:BS}. 
\end{example}

Finally, we discuss the behaviour of the above weight structures under the six functors. The weights of equivariant mixed Tate motives can be checked after forgetting the equivariance:

\begin{proposition}
\label{prop:wtfor}
Let $\mathbb{D}$ be a derivator satisfying the conditions of \ref{derivator:new}, the grading condition~\ref{conditions:grading} and the weight condition~\ref{conditions:weight}. Let $k$ be a field, let $G$ be a linear algebraic group, and let $G\looparrowright X$ be a variety with action having finitely many $G$-orbits separably defined over $k$. Assume that $G\looparrowright X$ is equivariantly Whitney--Tate, cf. Definition~\ref{def:mtderdef2}. 

Let $\DMT_G(X)$ be equipped with the weight structure of Proposition~\ref{prop:equivweight}, and let $\mathbb{D}(X)$ be equipped with the weight structure coming from the weight condition~\ref{conditions:weight}. 
\begin{enumerate}
\item An equivariant motive $M\in\DMT_G(X)$ is in $\DMT_G(X)_{\op{wt}\leq 0}$ if and only if $\op{For}(M)\in\mathbb{D}(X)_{\op{wt}\leq 0}$. 
\item An equivariant motive $M\in\DMT_G(X)$ is in $\DMT_G(X)_{\op{wt}\geq 0}$ if and only if $\op{For}(M)\in\mathbb{D}(X)_{\op{wt}\geq 0}$. 
\end{enumerate}
\end{proposition}

\begin{proof}
This essentially follows from the definition. For $X$ a point, this is clear: by Proposition~\ref{prop:wtpt}, the underlying motive is constant mixed Tate, a direct sum of copies of  $\const{\pt}(n)[2n]$. On the other hand, it follows easily that for any weight structure satisfying the conditions of the Bondarko--H\'ebert theorem, a mixed Tate motive is of weight zero if and only if it is a direct sum of copies of  $\const{\pt}(n)[2n]$ with $n\in\mathbb{Z}$. 

To prove the claim for $G$-homogeneous varieties, we first note that the induction equivalence takes constant mixed Tate motives $\const{\pt}(n)[2n]$ on the point to constant mixed Tate motives $\const{G/H}(n)[2n]$ on $G/H$. Then we still need to deal with the possibility of local systems. Let $M\in\DMT_G(G/H)$ be a motive on a homogeneous space. The weight of $M$ for the weight structure of Proposition~\ref{cor:wtorb} is determined by restriction along $x:(H\looparrowright \pt)\to (G\looparrowright G/H)$. Now $\op{For}(M)$ will be some motive; by the assumption that the weight structure on $\mathbb{D}$ satisfies the conditions of the Bondarko--H\'ebert theorem, cf.~Theorem~\ref{thm:hebert}, $x^\ast \op{For}(M)$ will be non-positive if $\op{For}(M)$ is. Moreover, $x^!\op{For}(M)$ will be non-negative if $\op{For}(M)$ is. By absolute purity, $x^\ast\op{For}(M)\cong x^!\op{For}(M)(d)[2d]$ with $d=\dim G/H$. This implies that if an equivariant mixed Tate motive has underlying motive of weight $0$ in $\mathbb{D}(G/H)$, then it has weight $0$ in $\DMT_G(\pt)$. For the other direction, assume that the restriction $x^\ast M$ lands in $\DMT_H(\pt)_{\op{wt}=0}$; we need to show that $\op{For}(M)$ has weight $0$ in $\mathbb{D}(G/H)$. We can assume without loss of generality that $G$ is connected. After pullback to an \'etale covering, we can also assume that $H$ is connected. Pullback along an \'etale covering is conservative, hence it cannot forget pieces of the weight filtration of $M$. Therefore, it now suffices to show $\op{For}(M)\in\mathbb{D}(G/H)_{\op{wt}=0}$ in the special case where $G$ and $H$ are connected. In this case, $\op{For}(M)$ is a constant mixed Tate motive whose restriction to a point is pure of weight zero, because $x^\ast M\in\DMT_H(\pt)_{\op{wt}=0}$. Then $\op{For}(M)\in \DMT(G/H)_{\op{wt}=0}$ as required. 

Finally, by Proposition~\ref{prop:equivweight}, for a variety with finitely many $G$-orbits, we have $M\in\DMT_G(X)_{\op{wt}\leq 0}$ if and only if $j^\ast M\in \DMT_G(Y)_{\op{wt}\leq 0}$ for every $G$-orbit $j:Y\hookrightarrow X$. The same is true for the underlying motives, by the weight condition. This finishes the argument.
\end{proof}

\begin{proposition}
\label{prop:wtres}
Let $\mathbb{D}$ be a derivator satisfying the conditions of \ref{derivator:new}, the grading condition~\ref{conditions:grading} and the weight condition~\ref{conditions:weight}. Let $k$ be a field, and let $G\looparrowright X$ and $H\looparrowright Y$ be  varieties with action satisfying the conditions of Proposition~\ref{prop:wtfor}. Let $(\phi,f):(H\looparrowright Y)\to(G\looparrowright X)$ be a morphism of varieties with action.
\begin{enumerate}
\item 
The pullback functor $(\phi,f)^\ast$ preserves non-positivity of weights.
\item If the morphism of Borel constructions associated to $(\phi,f)$ is additionally smooth, i.e., $f$ is smooth and $\phi$ is a closed immersion, then $(\phi,f)^\ast$ preserves non-negativity of weights. 
\item In particular, the restriction functor $\op{Res}_G^H$ is weight exact if $\phi$ is a closed immersion.
\item 
For a morphism $f:(G\looparrowright Y)\to(G\looparrowright X)$ of $G$-varieties, the push-forward functor $f_\ast$ preserves non-negativity of weights.
\item If $\phi:H\hookrightarrow G$ is the inclusion of a closed subgroup in a linear algebraic group, then the ordinary integration functor $\op{Ind}_H^G$ preserves non-negativity of weights, and the exceptional integration functor $\op{Ind}_!$ preserves non-positivity of weights. 
\end{enumerate}
\end{proposition}

\begin{proof}
This follows from Proposition~\ref{prop:wtfor} and the corresponding statements for the weight structures on $\mathbb{D}(X)$, cf. Theorem~\ref{thm:hebert}. 
\end{proof}

\begin{proposition}
\label{prop:wtpush}
Let $\mathbb{D}$ be a derivator satisfying the conditions of \ref{derivator:new}, the grading condition~\ref{conditions:grading} and the weight condition~\ref{conditions:weight}. Let $k$ be a field, let $G$ be a linear group and let $G\looparrowright X$ and $G\looparrowright Y$ be  varieties with action satisfying the conditions of Proposition~\ref{prop:wtfor}. Let $(\op{id},f):(G\looparrowright Y)\to(G\looparrowright X)$ be a morphism of varieties with action. Then $(\op{id},f)_!$ preserves non-positivity of weights, and  $(\op{id},f)^!$ preserves non-negativity of weights.
\end{proposition}

\begin{proof}
As above, this follows from Proposition~\ref{prop:wtfor} and the corresponding statements for the weight structures on $\mathbb{D}(X)$, cf. Theorem~\ref{thm:hebert}. 
\end{proof}

\begin{proposition}
\label{prop:wttensor}
Let $\mathbb{D}$ be a derivator satisfying the conditions of \ref{derivator:new}, the grading condition~\ref{conditions:grading} and the weight condition~\ref{conditions:weight}. Let $k$ be a field, let $G$ be a linear algebraic group, and let $G\looparrowright X$ be a variety with action having finitely many $G$-orbits separably defined over $k$. Assume that $G\looparrowright X$ is equivariantly Whitney--Tate, cf. Definition~\ref{def:mtderdef2}. 
\begin{enumerate}
\item 
The tensor product functor 
\[
\otimes:\DMT_G(X)\times\DMT_G(X)\to\DMT_G(X)
\]
maps $\DMT_G(X)_{\op{wt}\leq n}\times \DMT_G(X)_{\op{wt}\leq m}$ to $\DMT_G(X)_{\op{wt}\leq n+m}$. 
\item The functors $(-)\otimes\const{X}(n)[2n]$ and $\const{X}(n)[2n]\otimes(-)$ are weight exact. More generally, tensoring with objects from $\DMT_G(X)_{\op{wt}=0}$ is weight exact.
\end{enumerate}
\end{proposition}

\begin{proof}
There are two ways to see this. One is to use Proposition~\ref{prop:wtfor} together with the corresponding assertion for the weight structure on $\mathbb{D}(X)$, cf. \cite[Theorem 3.7. $(iv)$ and $(iv_c)$]{hebert}. The other is to use the fact that the tilting functors of Theorem~\ref{thm:tiltpoint} are tensor functors, cf. Proposition~\ref{prop:tilttensor}, together with the obvious corresponding weight-exactness statements for complexes of modules over cohomology rings (plus an induction for the number of orbits).
\end{proof}

\begin{proposition}
\label{prop:wtdual}
Let $\mathbb{D}$ be a derivator satisfying the conditions of \ref{derivator:new}, the grading condition~\ref{conditions:grading} and the weight condition~\ref{conditions:weight}. Let $k$ be a field, let $G$ be a linear algebraic group, and let $G\looparrowright X$ be a variety with action having finitely many $G$-orbits separably defined over $k$. Assume that $G\looparrowright X$ is equivariantly Whitney--Tate, cf. Definition~\ref{def:mtderdef2}. Then Verdier duality induces functors
\[
D:\DMT_G(X)_{\op{wt}\leq n}^{\op{op}}\to \DMT_G(X)_{\op{wt}\geq -n}.
\]
In particular, Verdier duality preserves the heart of the weight structure.
\end{proposition}

\begin{proof}
Either one does an induction on the dimension of orbits, using the definition of the weight structure from Proposition~\ref{prop:equivweight}, eventually reducing via tilting to an identification of the weight structure on a point via the homotopy category of complexes. The other possibility is to use Proposition~\ref{prop:wtfor} together with the corresponding non-equivariant statement which can be found in \cite[Corollaire 3.8]{hebert}. 
\end{proof}

\begin{remark}
We discuss shortly the question of existence of weight structures on equivariant motives in general. It would be nice if such weight structures on $\mathbb{D}^+_G(X)$ existed, and we could just get the above weight structures on equivariant mixed Tate motives by restriction. 

On the one hand, it is fairly clear that the weight structure would be expected to be compatible with the forgetful functor. In particular, the natural expectation is that $\mathbb{D}^+_G(X)_{\op{wt}\leq 0}$ and $\mathbb{D}^+_G(X)_{\op{wt}\geq 0}$ should consist of those objects which are of non-positive and non-negative weight after forgetting the equivariance, respectively. Most of the axioms for weight structures hold in this case, but the existence of the weight filtration is a most problematic issue. Basically, it is not clear that the truncation (for the non-equivariant weight structure) of an equivariant motive will be equivariant again. Alternatively, one could try to use \cite[Proposition 1.7(6)]{bondarko:imrn} to construct a weight structure on $\mathbb{D}^+_G(X)$, but again the problem is to show that equivariant motives of weight $0$ generate the whole category of equivariant motives. These problems appear in various guises, no matter which of the definitions of equivariant motives in Section~\ref{sec:equivdef} one is using. In the end, the crux is the non-functoriality of the weight filtration.

In any case, if there is a weight structure on equivariant motives,
then its heart should be the equivariant Chow motives of Laterveer \cite{laterveer}. Maybe the existence of the weight structure could be
deduced from properties of equivariant Chow motives. But again, the most natural way of decomposing the motive of an arbitrary $G$-variety into equivariant Chow motives would be to write out a simplicial $G$-variety given by equivariant resolution of singularities, where again we run into problems which seem unsurmountable at this point.
\end{remark}

\subsection{Pointwise purity}
\index{pointwise purity}

For a general orbit inclusion $j:G/H\to X$, the restriction functors $j^\ast$ and $j^!$  only have a half-exactness for the weight structure; in particular, they do not necessarily preserve the heart of the weight structure. For the representation-theoretic applications later on, we are interested in those special motives whose restrictions along orbit inclusions remain pure. We call these motives pointwise pure, since we think of these as composed of local systems on orbits whose stalks are pure mixed Tate motives. As discussed in the definition of equivariant mixed Tate motives, cf. Definition~\ref{mtderdef1} and the subsequent remark, the proper way to formulate pointwise conditions is via the induction equivalence. The following is the weight-analogue of Definition~\ref{mtderdef1}, defining orbitwise purity properties. However, terminology will still call these ``pointwise purity'' properties, for better compliance with the existing literature.

\begin{definition}
\label{poip}
\index{purity!pointwise}
In the situation of Proposition~\ref{prop:equivweight}, let $G\looparrowright X$ be a variety with action which is $G\looparrowright X$ is equivariantly Whitney--Tate, and let $M\in\DMT_G(X)$. Then
\begin{enumerate}
\item $M$ will be called \emph{pointwise $\ast$-pure} if for each $G$-orbit $j\colon G/H \hookrightarrow X$, we have $j^\ast M\in\DMT_G(G/H)_{\op{wt}=0}$.
\item $M$ will be called \emph{pointwise $!$-pure} if for each $G$-orbit $j\colon G/H \hookrightarrow X$, we have $j^! M\in\DMT_G(G/H)_{\op{wt}=0}$.
\item $M$ will be called \emph{pointwise pure} if it is both $\ast$- and $!$-pointwise pure.
\end{enumerate}
\end{definition}

As in Proposition~\ref{prop:quotientdmt} for orbitwise mixed Tate motives, 
pointwise purity is stable under generalized quotient equivalence.

\begin{proposition}
\label{prop:quotientdmtx}
In the situation of Proposition~\ref{prop:equivweight}, let $k$ be a field and let $G\looparrowright X$ be a variety with action having finitely many $G$-orbits separably defined over $k$. Let $N\hookrightarrow G$ be a normal subgroup which acts freely on $X$. We denote the quotient maps $\pi:G\to G/N$ and $p:X\to N\backslash X$. Then the generalized quotient equivalence of Proposition~\ref{prop:quotientequiv} preserves $*$- and $!$-pointwise purity, respectively. 
\end{proposition}

\begin{proof}
The argument is the same as for Proposition~\ref{prop:quotientdmt}. 
\end{proof}


\begin{proposition}
\label{prop:wthom}
In the situation of Proposition~\ref{prop:equivweight}, let $G$ be a linear algebraic group. For any two motives $M,N\in \DMT_G(\pt)_{\op{wt}=0}$, the motive $\fin_{{\op{B}}G*}\iHom(M,N)$ is a pure Ind-Tate motive.
\end{proposition}

\begin{proof}
We first assume that $G$ is connected. In this case, $\iHom(M,N)\in \DMT_G(\op{pt})_{\op{wt}=0}$ by part (2) of Proposition~\ref{prop:wttensor}. Note here that inner Hom in $\DMT_G(\op{pt})$ is given by tensoring with the Verdier dual and Verdier duality preserves the heart by Proposition~\ref{prop:wtdual}. Now recall that the heart of the weight structure on $\DMT_G(\pt)$ is the idempotent completion of the additive subcategory generated by $\const{\pt}_G(i)[2i]$, by Proposition~\ref{prop:wtpt}. The claim now follows since $\op{fin}_{{\op{B}}G,\ast}\const{\pt}_G(i)[2i]\cong\op{M}({\op{B}}G)(i)[2i]$ and by Proposition~\ref{prop:motbg} the motive of a classifying space is a pure Ind-Tate motive.

In general, $\iHom(M,N)$ is a direct summand of $\ind^G_{G^0}\res^{G^0}_G\iHom(M, N)$, cf. Lemma~\ref{finitetorsorNew}. Further, using \ref{CompRES}, we get
\begin{align*}
\fin_{{\op{B}}G,*}\ind^G_{G^0}\res^{G^0}_G\iHom(M, N) &\cong  \fin_{{\op{B}}G^0,*}\res^{G^0}_G\iHom(M,N) \\
&\cong \fin_{{\op{B}}G^0,*}\iHom(\res^{G^0}_G M, \res^{G^0}_GN).
\end{align*}
Since restriction is weight-exact by Proposition~\ref{prop:wtres}, this reduces us to the connected case.
\end{proof}

The following provides analogues of \cite[6.3, 8.3]{SoWe}, which is required for showing that pointwise pure Bott--Samelson motives form a negative collection and generate a weight structure, cf. \ref{puBS}. 

\begin{corollary}
\label{cor:hompuretate}
In the situation of Proposition~\ref{prop:equivweight}, let $X$ be a variety on which $G$ acts with finitely many orbits. Let $M \in \DMT_G^*(X)$ and $N\in \DMT_G^!(X)$. Assume $M$ is $\ast$-pointwise pure and $N$ is $!$-pointwise pure. Then $(\fin_{{\op{E}}G\times_{/G}X})_\ast\iHom(M,N)$ is pure and Tate. 
\end{corollary}

\begin{proof}
Essentially, the idea is to filter $\iHom(M,N)$ by the orbit stratification to reduce to the assertion of Proposition~\ref{prop:wthom}. 

For any orbit inclusion $j:Z\hookrightarrow X$ we have $j^!\iHom(M,N)\cong \iHom(j^\ast M,j^! N)$, and the latter is pure of weight $0$, cf. \cite[proof of Proposition 6.3]{SoWe}. In particular, $\iHom(M,N)$ is pointwise $!$-pure. 

Analogous to \cite[Proposition 6.2]{SoWe}, $(\op{fin}_{{\op{E}}G\times_{/G}X})_\ast \iHom(M,N)$ is a pure equivariant Tate motive of weight $0$, i.e., $(\op{fin}_X)_\ast\iHom(M,N)\in \DMT_G(\op{pt})_{\op{wt}=0}$. The inductive proof is the same. The difference is that the respective open orbit doesn't have to be an affine space; nevertheless, the assumption on pointwise purity implies that the respective pushforward  $\op{fin}_!j^\ast\iHom(M,N)$ (aka the equivariant cohomology of the Hom-motive on the orbit) is pure of weight $0$. This follows, since the restriction $j^\ast\iHom(M,N)$ corresponds, under the induction equivalence, to a pure motive of weight $0$ on the point. But then, from the definition of $\op{Ind}_!$, we see that $\op{fin}_! j^\ast\iHom(M,N)$ is exactly the image of $j^\ast\iHom(M,N)$ under the induction equivalence, and this is pure of weight $0$ by assumption.
\end{proof}

\subsection{Springer's Homotopy Lemma and contracting slices}
The following is a disguised and generalized version of the familiar fact that higher cohomology groups of a contractible space vanish. In the context of $\ell$-adic cohomology, it is due to T.A. Springer \cite{Spp}. 

\begin{lemma}
\label{lemma:springer}
Let $X$ be a variety contracted by an action of the multiplicative group $\mathbb{G}_{\op{m}}$ to a closed subvariety $i\colon Z\hookrightarrow X$. Let $\op{lim}\colon X \to Z$ be the morphism mapping each point to its limit. Then, for each $M\in \mathbb{D}^+_{\mathbb{G}_{\op{m}}}(X)$, applying $\op{lim}_*$ to the adjunction $M \to i_*i^*M$ yields an isomorphism in $\mathbb{D}_{\mathbb{G}_{\op{m}}}^+(Z)$
\[ \op{lim}_*M \xrightarrow{\cong} i^*M. \]
\end{lemma}

\begin{proof}
This is essentially \cite[Proposition 6.3]{SoWe}. However, instead of working with equivariant derived categories, a weaker notion of equivariance  \cite[Definition 6.1]{SoWe} is used there. In our equivariant setting one can argue exactly as in the proof of \cite[Proposition 6.4]{SoWe} by replacing the relevant diagrams with their equivariant analogues. An alternative is to observe that forgetting the $\mathbb{G}_{\op{m}}$-action we get an object $\For(M)$ which is  weakly equivariant in the sense of \cite[Definition 6.1]{SoWe}. Consequently, as the forgetful functor $\For$ commutes with all the standard functors and adjunctions between them, \cite[Proposition 6.4]{SoWe} implies that the morphism under consideration yields an isomorphism upon applying $\For$. Thus, our morphism must have been an isomorphism already, since its mapping cone vanishes under $\For$ and therefore was zero to begin with, cf.~\ref{forget}.
\end{proof}

\begin{example}
Consider the $\mathbb{G}_{\op{m}}$-action on $\mathbb{A}^n$ by dilations. This contracts $\mathbb{A}^n$ to the origin. So, for the constant motive $\fin^*\const{\pt}\in \mathbb{D}^+_{\mathbb{G}_{\op{m}}}(\mathbb{A}^n)$, Lemma~\ref{lemma:springer} reduces to the statement that the adjunction map yields an isomorphism 
\[
\const{\pt} \xrightarrow{\cong} \fin_*\fin^*\const{\pt}. 
\]
\end{example}

\begin{definition}
\label{def:slice}
\index{contracting slice}
Let $G\looparrowright X$ be a variety with action by a linear algebraic group and let $x\in X$. A \emph{contracting slice} at $x$ (for the $G$-action) is a locally closed subvariety $S\subset X$ containing $x$ and satisfying:
\begin{enumerate}
\item there exists a one-parameter subgroup $\mathbb{G}_{\op{m}}\to G_x$ which stabilizes $S$ and contracts $S$ to $x$;
\item the map $a_S\colon G\times S \to X$, $(g,x)\mapsto gx$, is smooth.
\end{enumerate}
\end{definition}

The existence of contracting slices has implications for pointwise purity:

\begin{lemma}
\label{lemma:contractingslices}
In the situation of Proposition~\ref{prop:equivweight}, let $G\looparrowright X$ be a variety with action which is $G$-equivariantly Whitney--Tate, and such that each $G$-orbit contains a point which admits a contracting slice. Then any motive $M\in \DMT_G(X)_{\op{wt}=0}$ is pointwise pure.
\end{lemma}

\begin{proof}
Let $j:G/H\to X$ be a $G$-orbit and let $M\in\DMT_G(X)_{\op{wt}=0}$. By Definition~\ref{poip}, we need to show that $j^\ast M,j^!M\in \DMT_G(G/H)_{\op{wt}=0}$. By assumption there is a point $x\in j(G/H)$ in the orbit and a contracting slice $i\colon S\hookrightarrow X$ at $x$; we denote by $(\iota_H,\iota_x)\colon(H\looparrowright \pt)\to (G\looparrowright G/H)$ the inclusion of the point $x$ in the orbit. 

By definition of the weight structure, cf. Corollary~\ref{cor:wtorb}, it suffices to show $(\iota_H,\iota_x)^\ast j^\ast M, (\iota_H,\iota_x)^\ast j^!M\in\DMT_H(\pt)_{\op{wt}=0}$. The first of these motives is simply obtained by restriction along the composition $(\op{id},j)\circ(\iota_H,\iota_x)$. For the second, we can factor $(\iota_H,\iota_x)^\ast\simeq(\op{id}_H,\iota_x)^\ast\circ\op{Res}_G^H$; by absolute purity, $(\op{id}_H,\iota_x)^\ast\simeq (\op{id}_H,\iota_x)^!(d)[2d]$ with $d=\dim G/H$. Since Verdier duality commutes with $\op{Res}_G^H$, $(\iota_H,\iota_x)^\ast j^!M\cong (\op{id}_H,\iota_x)^!\op{Res}_G^H j^!M(d)[2d]$ and the latter is simply given by restriction $\op{Res}_G^H$ followed by $!$-restriction along the composition $j\circ\iota_x$. In slight abuse of notation, it suffices to show $(\iota_H,j\circ\iota_x)^?M\in\DMT_H(\pt)_{\op{wt}=0}$ for $?=\ast,!$. Since the hearts of the weight structure as well as the categories of mixed Tate motives are stable under Verdier duality, cf. Propositions~\ref{prop:wtdual} and \ref{prop:dmtverdier}, it suffices to prove just the statement about $\ast$-restriction.

Now we consider the contracting slice $i\colon S\hookrightarrow X$ at $x$; we abusively denote by $x$ the inclusion of the point $x$ in the contracting slice $S$. Let $\mathbb{G}_{\op{m}}\to G_x$ be a one-parameter subgroup and $a_S:G\times S\to X$ be as in Definition~\ref{def:slice}, and write $a\colon G\qtimes{/\mathbb{G}_{\op{m}}} S \to X$ for the map induced by $a_S$. Consider the following composition, which is well-defined by the various results concerning compatibility of Tate motives with the six functors, cf. Section~\ref{sec:tatemotives}
\[
\DMT_G(X)\xrightarrow{a^\ast} \DMT_G(G\qtimes{/\mathbb{G}_{\op{m}}}S) \xrightarrow{\mathrm{ind. equiv.}} \DMT_{\mathbb{G}_{\op{m}}}(S) \xrightarrow{x^\ast}  \DMT_{\mathbb{G}_{\op{m}}}(\pt) 
\]
As $a_S$ is smooth by assumption, so is $a$. In particular, it is weight exact. Note also that the induction equivalence is weight exact by the definition of the weight structure, cf.~Corollary~\ref{cor:wtorb}. Consequently, if $M\in\DMT_G(X)_{\op{wt}=0}$, then the image $M'$ of $a^\ast M$ under the induction equivalence lands in $\DMT_{\mathbb{G}_{\op{m}}}(S)_{\op{wt}=0}$. Springer's homotopy lemma~\ref{lemma:springer} now implies an isomorphism $\lim_\ast M'\cong x^\ast M'$ in $\DMT_{\mathbb{G}_{\op{m}}}(\pt)$, and this isomorphism of a left with a right adjoint implies that $x^\ast$ is weight-exact. In particular, $\For(x^\ast M')\in\DMT(\pt)_{\op{wt}=0}$ and this is exactly what we needed to show. 
\end{proof}

\subsection{Graded versions and degrading functors}

\label{bgs43}
\index{mixed category}

In \cite[4.3]{BGSo}, gradings were defined on artinian categories of modules. In this work, we want to show that equivariant mixed Tate motives provide a graded version of the equivariant derived category. This means that we need suitable modifications of the notions defined in \cite{BGSo}. 

In our context, the \emph{mixed categories} are triangulated category equipped with a weight structure in the sense of Bondarko, cf.~Definition~\ref{defin:wtstruct}. Recall from Section~\ref{sec:weights} that the categories of equivariant mixed Tate motives we consider satisfy this requirement. Note that all the properties discussed in \cite[4.1]{BGSo}, in particular the orthogonality and the weight filtration, follow from the very definition of weight structures. 

A \emph{degree $d$ Tate twist} of the category $\mathcal{M}$ in our context is an auto-equivalence $\langle d\rangle:\mathcal{M}\to\mathcal{M}$ which satisfies 
\[
\mathcal{M}_{\op{wt}\leq n}\langle d\rangle=\mathcal{M}_{\op{wt}\leq n+d}, \textrm{ and }
\mathcal{M}_{\op{wt}\geq n}\langle d\rangle=\mathcal{M}_{\op{wt}\geq n+d}.
\]
For mixed Tate motives, the functor $(-)\otimes\mathbb{Q}(1)[2]$ preserves motivic weights. This means that the Tate twist $(-)\otimes\mathbb{Q}(1)$ is a degree $-2$ Tate twist. 

With these modifications, we can now define degrading functors and graded versions. 
\begin{definition}
\index{degrading functor}
Let $\mathcal{C}$ be a triangulated category and let $\mathcal{M}$ be a mixed category with a degree $d$ Tate twist $\langle d\rangle$. A functor $v:\mathcal{M}\to\mathcal{C}$ is called \emph{degrading functor} if it is exact, faithful and is equipped with a choice of natural iso-transformation 
\[
\epsilon:v \stackrel{\approx}{\longrightarrow} v\circ\langle d\rangle. 
\]
\end{definition}

\begin{definition}
\index{graded version}
Let $\mathcal{C}$ be a triangulated category, let $\mathcal{M}$ be a mixed triangulated category with degree $d$ Tate twist $\langle d\rangle$, and let $v:\mathcal{M}\to\mathcal{C}$ be a degrading functor with iso-transformation $\epsilon:v\simeq v\circ\langle d\rangle$. The tuple $(\mathcal{M},v,\epsilon)$ is called a \emph{graded version} of $\mathcal{C}$, if in addition the degrading functor $v$ is essentially surjective and induces natural isomorphisms
\[
\bigoplus_{n\in\mathbb{Z}} \mathcal{M}\left(M,N\langle n d\rangle\right) \stackrel{\cong}{\longrightarrow} \mathcal{C}\left(v(M), v(N)\right)
\]
for any objects $M,N\in\mathcal{M}$. 
\end{definition}

\chapter{Representation theory}
\label{chap:repthy}

The third part of the paper applies the results on equivariant mixed Tate motives to situations relevant for geometric representation theory. We begin with a short recollection on Hecke algebras, the Schur algebroid and some of its modules in Section~\ref{sec:hecke}, along with a formulation of the main results that will be established in this part. In Section~\ref{sec:korbit} we provide some background on the combinatorial structures for Borel actions on spherical homogeneous spaces that have been discussed in the literature, along with a more detailed recollection on the geometry of orbit closures and generalizations of Schubert varieties. Then Section~\ref{sec:BS} develops a formalism of Bott--Samelson motives which provides a way to establish the equivariant Whitney--Tate and pointwise purity conditions for actions of parabolic groups. This general formalism is applied in Section~\ref{sec:tiltingapp} to the case of parabolic group actions on symmetric varieties. In section~\ref{sec:convolution2}, we discuss convolution and its compatibility with six functors and tilting. The first representation-theoretic application is given in Section~\ref{sec:parabolic}, where we explain how equivariant mixed Tate motives provide a graded categorification of the Schur algebroid. The second application is given in  Section~\ref{sec:symmetric} where we provide the graded categorification of Schur algebroid modules arising from symmetric varieties, with applications to gradings on equivariant derived categories relevant for representation theory of real Lie groups. Finally, Section~\ref{sec:wonderful} contains a short discussion of the case of wonderful compactifications.

\section{Recollections and outline of applications}
\label{sec:hecke}

In this section, we recall some objects related to the Hecke algebra and give an overview how these will be categorified using equivariant mixed Tate motives. The claims made in this introductory section will be proven in the remainder of the chapter. 

\subsection{Categorification of the Schur algebroid} 

We begin with a recollection of the Schur algebroid, following \cite{GWi}. This is a category associated to the system of standard parabolic subgroups of a reductive group and encompasses the Hecke algebra as well as the parabolic Hecke modules of \cite{Deopar}. We will provide a new categorification of the Schur algebroid using equivariant mixed Tate motives. This will, in particular, also provide motivic categorifications of the Hecke algebra as well as the parabolic Hecke modules.

\begin{definition}[\textbf{The Schur algebroid}]
\index{Schur algebroid}
Let $G$ be a connected split reductive group over a finite field $\mathbb{F}=\mathbb{F}_q$. Fix a Borel subgroup $B\subset G$ also defined over $\mathbb{F}$. Following \cite[Section 2]{GWi}, we define a $\mathbb{Z}$-linear category, the \emph{Schur algebroid} associated to the situation $(G,B)$, as follows:
\begin{itemize}
\item The objects are parabolic subgroups $P\subset G$ containing $B$. 
\item Given any two parabolic subgroups $P,Q\subset G$ containing $B$, the morphisms from $Q$ to $P$ are given by the abelian group
\[
{^P\mathcal H^Q_{q}}\pdef {^{P(\mathbb{F})}\DZ[G(\mathbb{F})]^{Q(\mathbb{F})}}
\]
of $\DZ$-valued functions on the finite group $G(\mathbb{F})$ which are invariant under left  translation by elements of $P(\mathbb{F})$ and right translation by elements of $Q(\mathbb{F})$. 
\item The composition of morphisms is given by (renormalized) convolution
\[
\ast_Q:{^P\mathcal H^Q_{q}}\times {^Q\mathcal H^R_{q}}\ra {^P\mathcal H^R_{q}}
\]
given by first taking the product in the group ring and then dividing by the cardinality of $Q(\mathbb{F})$. 
\end{itemize}
Note that the characteristic functions of the $P$-$Q$-double cosets form a basis of $^{P(\mathbb{F})}\mathbb{Z}[G(\mathbb{F})]^{Q(\mathbb{F})}$. 
\end{definition}

\begin{Bemerkung}
The endomorphism ring of the object $B$ in the Schur algebroid is the Hecke algebra associated to the inclusion of finite groups $B(\mathbb{F})\subset G(\mathbb{F})$. 
\end{Bemerkung}

\begin{Bemerkung}[\textbf{The universal Schur algebroid}]
\label{schuruniv}
\index{Schur algebroid!universal}
As usual, for the pair $(G,B)$ of reductive group $G$ and a choice of Borel subgroup $B\subset G$, we have an associated Coxeter system $(W,S)$, where $W$ is the Weyl group and $S\subset W$ a system of simple reflections. Denote by $W_R\subset W$  the subgroup corresponding to a standard parabolic subgroup $R\subset G$. With this notation, the set of parabolic subgroups $P\subset G$ containing $B$ is in natural bijection with the set $\mathcal{P}(S)$ of sets of simple reflections. Consequently, we get a natural bijection 
\[
W_P\backslash W/W_Q\xrightarrow{\cong} P(\mathbb{F})\backslash G(\mathbb{F})/Q(\mathbb{F}): w\mapsto T_w
\]
between the double cosets  $P(\mathbb{F})\backslash G(\mathbb{F})/Q(\mathbb{F})$ in the group $G(\mathbb{F})$ and the double cosets $\bar z\in W_P\backslash W/W_Q$ in the associated Weyl group. 

Using the above bases given by $T_w$ (and the fact that the relevant abelian groups are finitely generated), it is possible to write down structure constants for the convolution products 
\[
\ast_Q:{^P\mathcal H^Q_{q}}\times {^Q\mathcal H^R_{q}}\ra {^P\mathcal H^R_{q}},
\]
decomposing the product of basis elements in ${^P\mathcal H^Q_{q}}$ and ${^Q\mathcal H^R_{q}}$ as linear combinations of basis elements in ${^P\mathcal H^R_{q}}$. It turns out that these structure constants depend polynomially on $q=|\mathbb{F}|$. Therefore, viewing $q$ as a variable and writing the structure constants as polynomials in $q$, we can define a $\DZ[q]$-linear category $\mathcal{H}$ whose objects are still the standard parabolic subgroups, but where the morphism spaces ${^P\mathcal H^Q}$ are now free modules over the polynomial ring $\DZ[q]$, generated by $T_{\bar z}$ for ${\bar z}\in W_P\backslash W/W_Q$, with convolution given by the above structure constant polynomials. This category is still called the Schur algebroid or sometimes the \emph{universal  Schur algebroid}. 

In abuse of notation, we will also denote by $\mathcal{H}$ the $\mathbb{Z}[q,q^{-1}]$-linear category obtained by extension of scalars along the obvious morphism $\mathbb{Z}[q]\to\mathbb{Z}[q,q^{-1}]$. 
\end{Bemerkung}

\begin{Bemerkung}
\index{Iwahori--Hecke algebra}
The endomorphism ring ${^B\mathcal H^B}$ of $B$ is  called the \emph{Iwahori--Hecke algebra}. The parabolic Hecke modules of \cite{Deopar} can be recovered as morphism spaces ${^B\mathcal{H}^P}$. We refrain from giving the well-known presentations for the convolution product, cf.~e.g.~\cite{iwahori}.
\end{Bemerkung}

\begin{Bemerkung}
An alternative definition of the universal Schur algebroid for reductive groups over $\mathbb{C}$ is discussed in Geordie Williamson's post answering MathOverflow question 74017 ``Parabolic convolution of perverse sheaves in terms of the Hecke algebra''. The Schur algebroid is defined as the split Grothendieck group of an additive 2-category whose objects are still the standard parabolic subgroups, and the 1-morphisms and 2-morphisms are given by perversely semi-simple complexes exact in odd degrees and their morphisms, respectively, inside the equivariant derived category $\op{Der}^{\op{b}}_{P\times Q}(G)$. Convolution induces bifunctors
\[
-\star_Q-:\op{Der}^{\op{b}}_{P\times Q}(G)\times \op{Der}^{\op{b}}_{Q\times R}(G)\to \op{Der}^{\op{b}}_{P\times R}(G),
\]
and by the decomposition theorem these  induce a convolution  on said Grothendieck groups. In particular, this definition of the Schur algebroid already comes with a categorification given in terms of equivariant derived categories. 

Note that one could also consider a 2-category whose objects are the standard parabolic subgroups, and whose 1-morphisms and 2-morphisms are given by objects and morphisms, respectively, of the equivariant derived category $\op{Der}^{\op{b}}_{P\times Q}(G)$. Again, convolution induces bifunctors
\[
-\star_Q-:\op{Der}^{\op{b}}_{P\times Q}(G)\times \op{Der}^{\op{b}}_{Q\times R}(G)\to \op{Der}^{\op{b}}_{P\times R}(G),
\]
and these induce the convolution of functions on the Grothendieck groups. This would, however, not be the right thing: the Grothendieck group of $\op{Der}^{\op{b}}_{P\times Q}(G)$ is finite and the morphism spaces in the universal Schur algebroid are not; the information in the universal Schur algebroid is a $\mathbb{Z}[q,q^{-1}]$-module structure and the abelian group structure for the equivariant derived categories is obtained by evaluation at $q=1$.
\end{Bemerkung}

\begin{Bemerkung}
  To be clear about terminology, categorification of the Schur algebroid here will mean that for each of pair of standard parabolic subgroups $P,Q\subset G$ we are looking for a category $\mathcal{C}_{P,Q}$ whose Grothendieck group (or split Grothendieck group) is isomorphic to the module ${}^P\mathcal{H}^Q$. Moreover, we want that there are convolution functors $\star_Q:\mathcal{C}_{P,Q}\times\mathcal{C}_{Q,R}\to \mathcal{C}_{P,R}$ in such a way that the multiplication operations induced by the convolution functors on the Grothendieck groups are mapped to the convolution operation for the Schur algebroid under the above isomorphisms.  In fact, we will obtain graded categorifications, which means that our categorifications will have Tate twists and weight structures, and passing to Grothendieck groups these will corresponds to the mixed structure (wherever available). The main novelty in our work is that we will provide a categorification of the universal Schur algebroid via the ordinary Grothendieck group of suitable categories of equivariant mixed Tate motives, while previous work used the split Grothendieck group of perversely semi-simple complexes.
\end{Bemerkung}

\begin{Bemerkung}[\textbf{Categorification of the  Schur algebroid}]
  We first describe how to categorify the Schur algebroid $\mathcal{H}_q$ associated to $(G,B)$ over a specific finite field $\mathbb{F}_q$. Here we need to extend coefficients to $\DZ[q^{-1}]\subset \DQ$.  Recall (e.g. from \cite{kiehl:weissauer}) that the Grothendieck function--sheaf correspondence takes a constructible $\ell$-adic sheaf $\mathcal{F}$ on a scheme $X$ over the finite field $\mathbb{F}_q$ to the function mapping a point $x\in X(\mathbb{F}_q)$ to the Frobenius trace of its stalk $\mathcal{F}_x$. We can now categorify the Schur algebroid using the function--sheaf correspondence as follows: for two standard parabolic subgroups $P,Q\subset G$ consider the category $\DMT_{P\times Q}(G)$ of $(P\times Q)$-equivariant mixed Tate motives over $G$, with the underlying derivator $\mathbb{D}=\mathbf{DA}^{\et}(-;\mathbb{Q}_\ell)$ (or slightly more general with a coefficient field $\Lambda\subset \mathbb{Q}_\ell$). As in the case of derived categories, we have convolution bifunctors 
\[
-\star_Q-:\DMT_{P\times Q}(G)\times \DMT_{Q\times R}(G)\to \DMT_{P\times R}(G). 
\]
There are $\ell$-adic realization functors 
\[
\op{Real}_\ell:\DMT_{P\times Q}(G)\to \op{Der}_{P\times Q}(G;\mathbb{Q}_\ell)
\]
from mixed Tate motives to  equivariant  $\ell$-adic sheaves on $G$, and these are compatible with all six functors and convolution.
If we now apply the function--sheaf correspondence to the resulting Weil sheaf (i.e., an $\ell$-adic sheaf with an additional action of Frobenius), we obtain a morphism
\[
\op{K}_0(\DMT_{P\times Q}(G))\to {^P\mathcal{H}^Q_q}.
\]
\end{Bemerkung}

\begin{Bemerkung}[\textbf{Categorification of  universal Schur algebroid}]
\label{CUS} 
In fact, we have an isomorphism
\[
\op{K}_0(\DMT_{P\times Q}(G))\xrightarrow{\cong} {^P\mathcal{H}^Q}
\]
to the universal Schur algebroid which maps the motive $j_!\underline{PzQ}(i)$, given by extension by zero of the constant mixed Tate sheaf on the double coset with weight $2i$, to the element $q^iT_{\bar z}$. The morphism $\op{K}_0(\DMT_{P\times Q}(G))\to {^P\mathcal{H}^Q_q}$ above can be obtained by specializing the variable $q$ to the corresponding  prime power. Under this isomorphism, convolution maps to convolution, Verdier duality corresponds to the Kazhdan--Lusztig involution and the intersection cohomology complexes correspond to the elements of the Kazhdan--Lusztig selfdual basis.
\end{Bemerkung}

\begin{Bemerkung}[\textbf{Alternative categorification of universal Schur algebroid}] 
To obtain an alternative categorification of the universal Schur algebroid, one may use the additive category $\op{MTDer}_{P\times Q}(G;\DQ_\ell)_{\op{wt}=0}$ of $(P\times Q)$-equivariant mixed Tate motives on $G$ which are pure of weight zero. This is the heart of a weight structure whose existence is a consequence of pointwise purity. Again, these categories are connected by the appropriate restrictions of convolution functors. 

Base-changing to the algebraic closure and setting $\overline{G}=G\times_{\mathbb F_q}\overline{\mathbb F_q}$, $\ell$-adic realization functor induces the second of two isomorphisms
\[
\op{K}_0(\DMT_{P\times Q}(\overline{G}))\stackrel{\cong}{\leftarrow} \op{K}_0^{\op{split}}(\DMT_{P\times Q}(\overline{G})_{w=0})\xrightarrow{\cong} \op{K}_0^{\op{split}}\left(\op{Der}_{P\times Q}^{\op{ev},\op{ss}}(\overline{G};\DQ_\ell)\right)
\]
where $\op{Der}_{P\times Q}^{\op{ev},\op{ss}}(\overline{G};\DQ_\ell)$ denotes the category of geometric perversely semisimple complexes with only even cohomology. The restriction of the trace of Frobenius map then induces an isomorphism
\[
\op{K}_0^{\op{split}}\left(\op{Der}_{P\times Q}^{\op{ev},\op{ss}}(G;\DQ_\ell)\right) \stackrel{\cong}{\longrightarrow} {^P\mathcal H^Q}\otimes_{\DZ[q]}\DZ[q,q^{-1}]
\]
from the split Grothendieck group of this additive category to the module ${^P\mathcal H^Q}$ with scalars extended to the ring $\DZ[q,q^{-1}]$ of Laurent polynomials. If we formally adjoin a square root of $q$ and put for later convenience $v^{2}=q^{-1}$, we similarly get an isomorphism  
\[
\op{K}_0^{\op{split}}\left(\op{Der}_{P\times Q}^{\op{ss}}(G;\DQ_\ell)\right) \stackrel{\cong}{\longrightarrow} {^P\mathcal H^Q}\otimes_{\DZ[q]}\DZ[v,v^{-1}].
\]
Explicitely, it is given by a formula of the type
$\mathcal F\mapsto \sum_z\op{dim}\mathcal H^i(\mathcal F_z)v^iT_z$ for
$\mathcal F_z$ the stalk at any point of our double coset corresponding to $z$. Under the above isomorphisms, convolution maps to convolution, Verdier duality corresponds to the Kazhdan--Lusztig involution and the intersection cohomology complexes correspond to the elements of the Kazhdan--Lusztig selfdual basis. In particular, the categorification via equivariant mixed Tate motives (or  the heart of the relevant weight structure) recovers previous categorifications of the universal Schur algebroid via equivariant perverse sheaves. 
\end{Bemerkung}

Following the discussion in Williamson's answer to MO-question 74017, a categorification of the Schur algebroid can be achieved by simply writing degrading functors from motives to derived categories. This is proved in Theorem~\ref{thm:schur}. The two categorifications mentioned above are related; the tilting equivalence 
\[
\op{Hot}^{\op{b}}\left(\DMT_{P\times Q}(G)_{\op{wt}=0}\right)\stackrel{\simeq}{\longrightarrow} \DMT_{P\times Q}(G)
\]
induces an equivalence from the split Grothendieck group of the heart of the weight structure to the Grothendieck group of the full triangulated category. This can be seen as a formality result for the equivariant derived category, which recovers known formality results, cf. \cite{SchTH}.

\begin{Bemerkung}[\textbf{Categorification via bimodules}]
Another way to obtain a categorification of the Schur algebroid is via equivariant cohomology bimodules. For an equivariant motive $M\in\mathbb{D}^+_{P\times Q}(G)$, we can consider some motivic version of equivariant cohomology, given by first pushing down along $\op{fin}:G\to\pt$ and applying the functor 
\[
\mathbb{D}^+_{P\times Q}(\pt)\rightarrow \mathcal{A}_P\op{-Mod^{\mathbb{Z}}-}\mathcal{A}_Q
\]
of Theorem~\ref{thm:tilting} to $\op{fin}_\ast(M)$. This functor maps Bott--Samelson motives to Soergel bimodules, and in fact induces, via tilting,
 an equivalence from the category of equivariant mixed Tate motives to the homotopy category of Soergel bimodules:
\[
 \DMT_{P\times Q}(G) \stackrel{\approx}{\longrightarrow} \op{Hot}^{\op{b}}(\mathcal{A}_P\op{-SMod^{\mathbb{Z}}-}\mathcal{A}_Q) 
\]
All this is essentially proved in Theorem~\ref{thm:gradedparabolic}. Via the known categorification of the Hecke algebra using Soergel bimodules, this shows that $\op{K}_0(\DMT_{B\times B}(G))$ is isomorphic to the Hecke algebra.
\end{Bemerkung}

\begin{Bemerkung}[\textbf{Relation to knot homology}]  
For the needs of knot theory, Khovanov proposed to work with $\op{Hot}^{\op{b}}(\op{Der}_{P\times Q}^{\op{ss}}(G;\DQ_\ell))$ or rather (and equivalently) with the homotopy category of Soergel bimodules $\op{Hot}^{\op{b}}(\mathcal{A}_P\op{-SMod-}\mathcal{A}_Q)$. The coefficient field is no longer relevant now; in fact, Khovanov works with $\DC$-coefficients instead of $\DQ_\ell$-coefficients. As above, the homotopy category of Soergel bimodules admits a more geometric interpretation by equivariant mixed Tate motives over the group $G$, via an equivalence of triangulated categories 
\[
 \DMT_{P\times Q}(G) \stackrel{\approx}{\longrightarrow} \op{Hot}^{\op{b}}(\mathcal{A}_P\op{-SMod^{\mathbb{Z}}-}\mathcal{A}_Q) 
\]
compatible with convolution. In Section~\ref{sec:knot}, we show that the so-called Rouquier complexes on the right in the case $P=Q=B$ correspond to the standard objects $i_!\underline{P_s}$ and $i_\ast\underline{P_s}$ on the left, with $P_s$ the parabolic with Weyl group $\{1,s\}$ for a simple reflection $s$. As a consequence, the braid relations for Rouquier complexes, which are required for the definition of knot invariants, follow directly from obvious geometric isomorphisms of Bruhat cells. 
\end{Bemerkung}

\begin{Bemerkung}[\textbf{Relation to Harish-Chandra bimodules}]
Consider the special case of Harish-Chandra modules for a \emph{complex} Lie group (alias Harish-Chandra bimodules). In this case, the bounded homotopy category of Soergel bimodules (or equivalently, the category of equivariant mixed Tate motives on the group) is the ``graded version'' of the Langlands-parameter side of the Koszul-duality equivalence conjectured in \cite{So-L}. We may discuss the more general case of real reductive groups in an unspecified sequel, to be written in an intermediate future.
\end{Bemerkung}
    
\subsection{Hecke modules for symmetric varieties}

Now we want to discuss the more general situation of parabolic actions on spherical homogeneous spaces, i.e., homogeneous spaces $G/H$ where a Borel subgroup $B\subset G$ acts with finitely many orbits. In this situation, one can also define a module over the Hecke algebra which is generated by elements corresponding to local systems over orbits. This was first done by Lusztig and Vogan in \cite{lusztig:vogan}. The module can be described purely algebraically, based on an analysis of the combinatorics of orbit closures similar to the Weyl group case, cf. the explicit formulas detailed in \cite[Section 3.7]{springer:schubert} or \cite[Section 7]{richardson:springer:survey}. 

It is, however, conceptually more satisfying to interpret it as Grothendieck group of suitable categories of $\ell$-adic sheaves, as done in \cite{Mars-Springer}. Some of the more elaborate results concerning the $\ell$-adic sheaves require more precise geometric information, like existence of contracting slices, which is only available in the symmetric case. In this case, the categorical description of the Hecke module in \cite{Mars-Springer} is the following: 

\begin{Bemerkung}[{\bf Hecke modules over ``finite fields''}]
\label{symmmod}
Let $G$ be a connected reductive group over the algebraic closure of a finite field, let $K\subset G$ be a symmetric subgroup and set $Y=G/K$. 

In \cite[(3.1.2)]{Mars-Springer}, Mars and Springer define a category $\mathcal{A}_Y$ of perverse sheaves on $Y$ having a Frobenius structure and suitable equivariance for a Borel subgroup. There is also a definition of a category $\mathcal{C}_Y$ of constructible equivariant $\ell$-adic sheaves on $Y$ having a Frobenius structure. The Grothendieck group of $\mathcal{C}_Y$ has as natural basis the local systems on $B$-orbits, and the Grothendieck of $\mathcal{A}_Y$ has as natural basis the intersection cohomology sheaves associated to these local systems. Moreover, there is a natural map $\mathcal{H}:\op{K}(\mathcal{A}_Y)\to\op{K}(\mathcal{C}_Y)$ taking a perverse sheaf to the alternating sum of its constructible cohomology sheaves.

For the special case of the $B\times B$-action on $G\times G$ with the symmetric subgroup $\Delta G$, convolution defines algebra structures on $\op{K}(\mathcal{A}_G)$ and $\op{K}(\mathcal{C}_G)$, such that the cohomology morphism $\mathcal{H}$ above is an algebra isomorphism, cf. \cite[(3.2.2--3.2.4)]{Mars-Springer}. This recovers the Hecke algebra for $G/\mathbb{F}_q$ discussed above. In a similar way, convolution defines module structures 
\[
\op{K}(\mathcal{A}_G)\times\op{K}(\mathcal{A}_Y)\to \op{K}(\mathcal{A}_Y), \quad \textrm{ and } \quad \op{K}(\mathcal{C}_G)\times\op{K}(\mathcal{C}_Y)\to \op{K}(\mathcal{C}_Y)
\]
for general $Y=G/K$, such that the cohomology morphism $\mathcal{H}$ is a module isomorphism, cf. \cite[(3.2.5--3.2.8)]{Mars-Springer}. Verdier duality induces an involution of the algebras and modules, which is semilinear for the module structure, cf. \cite[Section 3.3]{Mars-Springer}. 
\end{Bemerkung}

\begin{Bemerkung}[{\bf Universal modules over the Hecke algebra}]
In analogy with the discussion of the universal Schur algebroid in \ref{schuruniv}, one can ask about the dependence on the prime power $q$ underlying the construction of \cite{Mars-Springer} outlined in \ref{symmmod}. The main result, \cite[Theorem 7.1.2]{Mars-Springer}, states that this dependence on $q$ is polynomial, so that the Hecke modules can actually be considered as specializations of modules over the Iwahori--Hecke algebra associated to a symmetric variety. 

The other main conclusion formulated in \cite[Theorem 7.1.2]{Mars-Springer} is a parity statement for intersection cohomology. This is the key statement, implying that everything is well-behaved and a theory of Kazhdan--Lusztig polynomials can be set up parallel to the Schur algebroid case. Note that for this it is necessary to have some sort of mixed structure, and this is obtained from the Frobenius action built into the objects. There are various technical difficulties that then arise; there is also (as far as we know) no analogue of the Mars--Springer construction in characteristic $0$, and the construction isn't quite categorical enough to provide a graded version of the equivariant derived category.
\end{Bemerkung}

What we provide in Section~\ref{sec:symmetric}, generalizing and extending \cite{Mars-Springer}, is a structure of module over the universal Schur algebroid associated to a symmetric subgroup $K\subset G$ in a connected reductive group $G$ over an algebraically closed field, using the formalism of equivariant mixed Tate motives. 

\begin{Bemerkung}[{\bf Categorification of symmetric moduloid}]
\label{motsym}
For a connected reductive group $G$ over an algebraically closed field of characteristic $\neq 2$, a symmetric subgroup $K\subset G$, and a parabolic subgroup $P\subset G$, we can consider the category $\DMT_{P\times K}(G)$ of $P\times K$-equivariant mixed Tate motives on $G$. That these categories are well-behaved follows by combining the formalism of equivariant mixed Tate motives with the geometric input developed in \cite{Mars-Springer}, cf. Theorem~\ref{cor:symmetricWT}. Convolution then provides bifunctors
\[
-\star_Q-:\DMT_{P\times Q}(G)\times \DMT_{Q\times K}(G)\to \DMT_{P\times K}(G).
\]
The collection of Grothendieck groups for the categories $\DMT_{P\times K}(G)$, $P$ a parabolic, provides a structure of module over the universal Schur algebroid for the group $G$, cf. Proposition~\ref{prop:schurmodulesymm}.

As in the case of the Schur algebroid, there are now two ways of seeing this as a categorification, resp. of setting up the relation with the objects defined in \cite{Mars-Springer}, provided we work over the algebraic closure of a finite field. 

The first way is to use the $\ell$-adic realization functor 
\[
\op{Real}_\ell:\DMT_{P\times K}(G)\to \op{Der}^{\op{b}}_{P}(G/K,\mathbb{Q}_\ell).
\]
Since the category of equivariant mixed Tate motives is very combinatorial, the objects in the target will descend to a finite field $\mathbb{F}_q$ (provided the symmetric subgroup is defined there) and the Frobenius structure will recover the weight information from the motive. The image of $\DMT_{B\times K}(G)$ will then be identified with the category $\mathcal{C}_Y$ of \cite{Mars-Springer}, with $Y=G/K$, cf. Theorem~\ref{thm:gradedsymm}. 

The second way is to note that pointwise purity implies that $\DMT_{P\times K}(G)$ has a weight structure. Pointwise purity follows from the existence of contracting slices (and therefore is known only in the case of symmetric varieties, not the general spherical case), and the existence of the weight structure is the motivic version of the parity statements for intersection cohomology established in \cite{Mars-Springer}. Applying $\ell$-adic realization to the heart $\DMT_{B\times K}(G)_{\op{wt}=0}$ sends pure equivariant Tate motives of weight zero to the perverse sheaf category $\mathcal{A}_Y$ of \cite{Mars-Springer}, cf. Proposition~\ref{prop:symmpurity} and Theorem~\ref{thm:gradedsymm}.

These two ways of relating mixed Tate motives to the sheaves of Mars and Springer turn out to be equivalent, again via a tilting result
\[
\op{Hot}^{\op{b}}(\DMT_{P\times K}(G)_{\op{wt}=0})\xrightarrow{\simeq} \DMT_{P\times K}(G), 
\]
which under the above functors turns exactly into the map $\mathcal{H}:\op{K}(\mathcal{A}_Y)\to \op{K}(\mathcal{C}_Y)$ used in \cite{Mars-Springer}, cf. Theorem~\ref{thm:comparisonMS}. 

Note that we are for now ignoring the $\mathbb{Z}[C]$-module structure in \cite{Mars-Springer} which seems to be an artifact of the technical eigenvalue problems introduced by the Frobenius action. Our framework allows to provide a graded Hecke module with $\mathbb{Z}[q,q^{-1}]$ coefficients, and the Mars--Springer modules can be recovered by extensions of scalars along $\mathbb{Z}\to \mathbb{Z}[C]$. 

The main point of the new categorification is then that we obtain streamlined formality results which recover the known parity statements for intersection cohomology sheaves, cf. \cite[Theorem 7.2.1]{Mars-Springer} \cite[3.9]{springer:schubert}. In particular, we get Kazhdan--Lusztig--Vogan polynomials whose coefficients are dimensions of equivariant mixed Tate motives. As an added bonus, our approach also isolates very clearly the relevant geometric inputs for proving formality of equivariant derived categories. We understand that these formality questions were also the subject of recent yet unpublished work of Brion and Joshua \cite{brion:joshua}. 
\end{Bemerkung}

\begin{Bemerkung}
The mixed geometry present in the categories $\DMT_{P\times K}(G)$ allows to obtain gradings on the corresponding equivariant derived categories $\op{Der}^{\op{b}}_{P\times K}(G)$, either with $\mathbb{C}$-coefficients for symmetric varieties over $\mathbb{C}$ or of $\ell$-adic sheaves for symmetric varieties over $\overline{\mathbb{F}_q}$. The degrading functors are simply given by the corresponding realization functors, cf. Theorem~\ref{thm:abvconjecture}.
\end{Bemerkung}

\begin{Bemerkung}
Unfortunately, the picture for now is not as complete as the one for the Schur algebroid. While we now have mixed versions of the equivariant derived categories, purity and formality results, there is no combinatorial model available that would be analogous to the Soergel bimodules for the Schur algebroid. Part of this seems due to the fact that natural resolutions for orbit closures are only available in multiplicity one cases, cf. \ref{symmBSresolution}. 
\end{Bemerkung}

\begin{Bemerkung}[{\bf Application to representation theory}]
The motivic categories for parabolic group actions on symmetric varieties have several applications relevant for the representation theory of real Lie groups. First of all, the tilting and formality results Theorem~\ref{thm:parabolicWT} and Corollary~\ref{tiFL} establish the Soergel--Lunts conjecture on formality of equivariant derived categories, cf. \cite[Conjecture 0.1.3]{Lu-tor} and \cite{So-L}. 

The motivic categories for symmetric varieties also solve the ``geometric'' part of Soergel's conjectures from \cite{So-L} which deals with graded versions $\mathcal{D}_g$ of the equivariant derived categories. More precisely, Theorem~\ref{thm:abvconjecture} shows that the categories $\DMT_P(X)$ for $X=G/K$ are graded versions of the respective equivariant derived categories. Via the Matsuki correspondence relating $G_{\mathbb{R}}$-orbits on flag varieties to $K$-orbits, cf. \cite[Section 6]{richardson:springer:survey}, the representations of a real Lie group (and its inner forms) can be related to the equivariant derived categories of suitable symmetric spaces $B\looparrowright G/K$. The graded versions of these categories, which were conjectured to exist in \cite[Conjectures 4.2.2 and 4.2.3]{So-L}, are now available via the formalism of equivariant mixed Tate motives. 

On the other hand, according to the conjectures in \cite{So-L}, these categories should be related to graded versions of derived categories of Harish-Chandra modules via a Koszul-type duality. At this point, even the existence of graded versions of such derived categories is unknown. However, we expect that the equivariant motivic approach outlined here can also be used to construct categories of motives satisfying suitable monodromy conditions; these categories should provide the required graded versions of the derived category of Harish-Chandra modules via a motivic version of the Bernstein--Lunts localization for $(\mathfrak{g},K)$-modules in \cite{BeLuK}. This would establish the representation-theoretic part of Soergel's conjectures from \cite{So-L}. Having graded versions for both sides of Soergel's conjectural Koszul duality should greatly simplify both the construction of explicit candidates for the Koszul duality functors as well as the proof of equivalence (e.g. by comparing finite-dimensional mapping spaces). We understand that Ben Zvi and Nadler have a yet unpublished construction of a candidate functor for the Koszul duality (in a non-motivic setting without the full graded versions). The construction and properties of the categories of monodromic motives will be subject of a sequel paper. 
\end{Bemerkung}

\subsection{Wonderful compactifications}

We discuss another case to which the formalism of mixed Tate motives applies. As in Example~\ref{ex:wonderful}, let $G$ be a connected adjoint semi-simple group with Borel subgroup $B$. Let $G\times G\looparrowright X$ be a wonderful compactification of $G$, viewed as $G\times G$-variety. Then the Borel subgroup $B\times B\subset G\times G$ acts with finitely many orbits. For a choice of maximal torus $T\subset B\subset G$, the relevant Hecke algebra is $\mathcal{H}:=\mathcal{H}(W,S)\otimes_{\mathbb{Z}[v,v^{-1}]}\mathcal{H}(W,S)$. 

\begin{Bemerkung}
As in the case of Hecke modules for symmetric varieties, if the whole situation is defined over a finite field $\mathbb{F}_q$, there is a module over the appropriate Hecke algebra which can be described algebraically but which is better interpreted as Grothendieck group of suitable categories of equivariant $\ell$-adic sheaves. This was discussed in \cite{SpCompact}. In fact, the definition of the Hecke module for a spherical variety of \cite{Mars-Springer}, cf. \ref{symmmod}, applies to this setting. In particular, there is one module defined via perverse sheaves on $X$ (equivariant and with Frobenius structure) and one module defined via constructible sheaves on $X$ (equivariant and with Frobenius structure). These are isomorphic via the cohomology morphism $\mathcal{H}$. 

The main results, cf. \cite[Theorem 4.2]{SpCompact}, are again that the module structure depends polynomially on $q$, and there are parity statements for the intersection cohomology sheaves associated to $B\times B$-orbits. 
\end{Bemerkung}

\begin{Bemerkung}[{\bf Categorification of the wonderful moduloid}]
For a wonderful compactification as above, we can now define more generally a module over the Schur algebroid by considering the categories $\DMT_{P\times Q}(X)$ of $P\times Q$-equivariant mixed Tate motives over $X$, with $P\times Q\subset G\times G$ a parabolic subgroup. Again, it follows from the formalism of equivariant mixed Tate motives, combined with classical geometric statements that this is a well-behaved category, cf. Theorem~\ref{thm:wonderfulWT}. Convolution produces bifunctors 
\[
\DMT_{(P_1\times Q_1)_{\op{left}}\times (P_2\times Q_2)_{\op{right}}}(G\times G)\times \DMT_{P_2\times Q_2}(X)\to \DMT_{P_1\times Q_1}(X).
\]
The collection of the Grothendieck groups for the categories $\DMT_{P\times Q}(X)$ provides a module over the universal Schur algebroid for the group $G\times G$, cf. Proposition~\ref{prop:schurmodulewonderful}. 

As in \ref{motsym}, we have two ways to recover the module defined by Springer in \cite{SpCompact}. On the one hand, we have the $\ell$-adic realization functor 
\[
\op{Real}_\ell:\DMT_{B\times B}(X)\to \op{Der}^{\op{b}}_{B\times B}(X,\mathbb{Q}_\ell),
\]
which recovers the definition via the categories $\mathcal{C}_Y$; on the other hand, we can restrict the $\ell$-adic realization to the heart $\DMT_{B\times B}(X)_{\op{wt}=0}$ of a weight structure to recover the definition via the categories $\mathcal{A}_Y$, cf. Theorem~\ref{thm:comparisonSpCompact}.

Then, as before, the parity statements for intersection cohomology sheaves can be recovered from the purity and formality results of the motivic formalism, and moreover the $\ell$-adic realization induces a grading on the equivariant derived category of $G\times G\looparrowright X$, Theorem~\ref{thm:gradedwonderful}. This can again be interpreted as a formality statement. 
\end{Bemerkung}

\section{Borel orbits on spherical varieties}
\label{sec:korbit}

In this section, we will provide a dense recollection on the geometry of Schubert varieties in flag varieties and their generalizations in spherical homogeneous spaces. We recall the combinatorical and geometric structures related to orbit closures that have been considered in the literature because these will be the relevant geometric input for our motivic formalism.

\subsection{Parabolic groups and Schubert varieties}

We first begin with a recollection on parabolic subgroups in reductive groups and Schubert varieties in the corresponding flag varieties. For the standard definition, cf. the standard textbooks \cite{BorAG,HumAG,SpLAG}, or Brion's lectures on the geometry of flag varieties \cite{brion}.

Let $G$ be a connected split reductive group, let $T\subset G$ be a
split maximal torus, and let $B\subset G$ be a Borel subgroup containing $T$. The Weyl group is $W=\op{N}_G(T)/T$. 

For a choice of maximal torus, we can talk about the corresponding root system. The choice of Borel subgroup $B$ containing $T$ corresponds to a choice of a system $\Delta=\{\alpha_1,\dots,\alpha_r\}$ of positive simple roots, where $r$ is the rank of the maximal torus. Taking the root reflection for a simple root $\alpha_i$ provides a system of simple reflections $s_i$ in the Weyl group $W$ of $G$. The length $l(w)$ of an element $w\in W$ is the smallest number $n$ such that $w$ is the product of $n$ reflections. 

Via the Bruhat decomposition $G=\bigsqcup_{w\in W}BwB$, the $B$-orbits of $G/B$ can be parametrized by elements of the Weyl group $W$. 

For a subset of simple positive roots $I\subset\Delta$, we can consider the subgroup $W_I\leq W$ generated by the simple reflections $s_i$ for $i\in I$. Then we can use the Bruhat decomposition to obtain a correspondence between subsets $I\subset \Delta$ and standard parabolic subgroups $B\subset P\subset G$ via 
\[
P_I=\bigsqcup_{w\in W_I} BwB.
\]
In particular, the minimal standard parabolic subgroups are given by $P_s=B\cup BsB$ for $s\in W$ a simple reflection. Then there is also a Bruhat decomposition for $P_I\supset B$ a standard parabolic: 
\[
G=\bigsqcup_{\overline{w}\in W/W_I} B\overline{w}P_I, 
\]
and this provides a correspondence between the $B$-orbits on $G/P_I$ and the elements of $W/W_I$. Even more generally, the $P_I$-orbits on the flag variety $G/P_J$ are parametrized by the double cosets $W_I\backslash W/W_J$. 

Note that the orbits on flag varieties $G/P$ are all isomorphic to affine spaces. The closures of $B$-orbits on $G/P$ are called \emph{Schubert varieties}. \index{Schubert varieties}

\begin{Bemerkung}
\label{BSresolution}
\index{Bott--Samelson!resolution}
Let $X=G/B$ be the flag variety. Identify $B$-orbits in $X$ with the Weyl group in the usual way, i.e., $G/B= \bigsqcup_{w\in W} BwB/B$. For any sequence of simple reflections $\underline w = (s_1, \ldots, s_n)$, let
\[ 
\op{BS}(\underline w) = P_{s_1} \qtimes{B} \cdots \qtimes{B} P_{s_n} \qtimes{B} B/B, 
\]
where $P_{s_i}$ denotes the minimal parabolic $P_{s_i}\supset B$ corresponding to $s_i$. Multiplication induces a map $\pi\colon \op{BS}(\underline w) \to G/B$. Note that if $s_1 \cdots s_n$ is a reduced word, then $\pi$ is the classical \emph{Bott--Samelson resolution} of singularities of the closure (i.e., Schubert variety) $\overline{Bs_1\cdots s_nB/B}$. 
\end{Bemerkung}

\begin{Bemerkung}
\index{Bruhat order}
Using the geometry of flag varieties, we can define a partial order on the set of $B$-orbits in $G/P$. The \emph{Bruhat order} on the orbits is defined by $V\leq W$ if and only if $\overline{V}\subseteq \overline{W}$. This can be expressed combinatorially in terms of the Weyl group (whose elements parametrize $B$-orbits), cf. \cite[Proposition 1.2.1 and Corollary 2.2.2]{brion}. 
\end{Bemerkung}

\begin{definition}
\index{Demazure product}
For a ring $R$ and $W$ the Weyl group of a semisimple algebraic group, we have the \emph{Demazure product} on $R[W]$, given as follows: 
\[
s\star w=\left\{\begin{array}{ll}
sw & \textrm{if }l(sw)>l(w)\\
w & \textrm{otherwise}\end{array}\right.
\]
\end{definition}

The Demazure algebra $(R[W],\star)$ can then be realized as an algebra of differential operators on the flag manifold, it can be seen as degeneration of the Hecke algebra, cf. \cite[Section 7]{richardson:springer:survey}.

\begin{definition}[Richardson--Springer monoid]
\label{def:rsmonoid}
\index{Richardson--Springer monoid}
Fix a Weyl group $W$. Define the \emph{Richardson--Springer monoid} $\mathcal{M}(W)$ to be the monoid with elements $m(w)$ indexed by elements $w\in W$ such that $m(s)m(w)=m(s\star w)$ for $s$ a simple reflection. 
\end{definition}

\begin{remark}
The Richardson--Springer monoid is generated by $m(s)$ for simple reflections, subject to $m(s)^2=m(s)$ and the braid relations in the Weyl group:
\[
m(s)m(t)m(s)\cdots=m(t)m(s)m(t)\cdots 
\]
where $s,t$ are simple reflections in $W$ and the number of factors on both sides is the order of the product $st$ in $W$.
\end{remark}

\subsection{Spherical subgroups and generalizations of Schubert varieties}

Now we will provide a recollection on the generalization of Schubert varieties in spherical homogeneous spaces. We also outline the combinatorics of closures of orbits of symmetric groups on flag varieties. Standard references for these topics are \cite{lusztig:vogan}, \cite{richardson:springer:survey} and  \cite{springer:schubert}. 

\begin{definition}
\index{spherical}
Let $G$ be a reductive group. A variety with action $G\looparrowright X$ is called \emph{spherical} if a Borel subgroup $B\subset G$ acts with finitely many orbits. A subgroup $H\subset G$ is called \emph{spherical} if the corresponding homogeneous space $G/H$ is spherical for the natural left $G$-action.
\end{definition}

\begin{example}
\begin{enumerate}
\item
If $B\subset G$ is a Borel subgroup in a connected split reductive group, then $B$ is spherical; this is the basis of Example~\ref{ex:parabolic}. 
\item
If $\theta$ is an involution of a connected split reductive group $G$, then the symmetric subgroup $K=G^\theta$ of $\theta$-fixed points is spherical. A special case is the maximal torus $T$ inside $\op{SL}_2$, cf. \cite{matsuki,springer:involution}.
\item 
The Borel subgroup case can be recovered from the symmetric case by using the switch involution on $G\times G$. 
\item If $G$ is an adjoint semisimple group and $X$ is a $G\times G$-equivariant compactification of $G$, then $G\times G\looparrowright X$ is a spherical variety.
\end{enumerate}
\end{example}


There are again partial orderings induced by inclusion of orbit closures. A further discussion of the Bruhat orders in the parabolic and symmetric cases can be found in \cite{yee}. 

\begin{definition}
\label{def:order}
\index{Bruhat order}
Let $G$ be a connected split reductive group, let $H\subset G$ be a spherical subgroup, and let $B\subset G$ be a Borel subgroup. The set of $B$-orbits on $G/H$ carries two natural partial orderings. 
\begin{enumerate}
\item the \emph{strong Bruhat order}, given by 
\[
V\leq W \textrm{ if and only if } \overline{V}\subset\overline{W}. 
\]
\item 
  The \emph{weak order} on $K$-orbit closures is defined using the Richardson--Springer monoid action as follows:
\[
Y\leq Y' \textrm{ if and only if } Y'=w\star Y.
\] 
\end{enumerate}
\end{definition}

\begin{Bemerkung}
For the strong Bruhat order, there is a unique maximal element -- the open $B$-orbit. The minimal elements are the closed $B$-orbits. Those are the orbits $BgH/H$ where $H\cap g^{-1}Bg$ is a Borel subgroup in $H$. Minimal orbits all have the same dimension. The example $T\subset \op{SL}_2$ shows that there can be several minimal orbits.

Inclusion in the weak order implies inclusion in the strong Bruhat order. More precisely, by \cite[Proposition 4.10]{richardson:springer:survey}, the strong Bruhat order is the weakest partial order on the set of orbits which is compatible with the action of the Richardson--Springer monoid $\mathcal{M}(W)$. 
\end{Bemerkung}

\begin{Bemerkung}
For $G$ a reductive group with a symmetric subgroup $K\leq G$ and a standard parabolic subgroup $P\leq G$, the structure of $P\times K$-orbits on $G$ can be deduced from the action of the Richardson--Springer monoid on the $B\times K$-orbits on $G$, cf. \cite[Section 3.4]{richardson:springer:survey,brion:helminck}. In particular, there is also an analogue of the Bruhat decomposition of $G$ into $P\times K$-double cosets. 
\end{Bemerkung}

The following discussion of orbit types for actions of the Borel group on spherical varieties, cf. \cite{Mars-Springer} or \cite[Section 2]{richardson:springer:survey}. 

\begin{Bemerkung}[{\bf $\mathbb{P}^1$-modifications}]
\label{ss:p1mod}
Let $G$ be a connected split reductive group. Fix a Borel subgroup $B\subset G$ and a maximal torus $T\subset B$. Let $W$ denote the Weyl group. The pair $(B, T)$ determines a set of simple reflections in $W$. Given a simple reflection $s$, let 
\[ P_s = B\sqcup BsB \]
denote the corresponding minimal parabolic. Let $X$ be a $G$-variety. Assume that there are only finitely many $B$-orbits in $X$. So, if $Z$ is a $B$-orbit, then 
\[ 
\mbox{$P_s \cdot Z$ contains a unique open $B$-orbit, denoted $s\star Z$.}
\index{$s\star Z$ - unique open orbit in $P_s\cdot Z$} 
\]
Pick a point $z$ in the orbit $Z$, and let $H$ denote the isotropy group of $z$ in $P_s$. Then
\[ 
\mathbb{D}^+_B(P_s \cdot z) \approx \mathbb{D}^+_B(P_s/H) \approx \mathbb{D}^+_{B\times H}(P_s)\approx \mathbb{D}^+_H(B\backslash P_s) \approx \mathbb{D}^+_H(\PP^1).
\]
As $B$ acts on $P_s \cdot Z$ with finitely many orbits, $H$ must act on $\PP^1$ with finitely many orbits. Let $\bar H$ be the image of $H$ in $\op{Aut}(\PP^1)\simeq \op{PGL}_2$. The only possibilities are the following, cf. also \cite[4.1.3]{Mars-Springer}:
\begin{description}
\item[Case G] $\bar H = \op{PGL}_2$;
\item[Case U] $\bar H$ is a proper subgroup of $\op{PGL}_2$ containing a unipotent subgroup;
\item[Case T] $\bar H$ is a maximal torus in $\op{PGL}_2$;
\item[Case N] $\bar H$ is the normalizer of a maximal torus in $\op{PGL}_2$.
\end{description}
Modulo conjugation, the corresponding $H$-orbit decompositions of $\PP^1$ are:
\begin{description}
\item[Case G] $\PP^1$;
\item[Case U] $\mathbb{A}^1 \sqcup \{\infty\}$;
\item[Case T] $\{0\} \sqcup \mathbb{G}_{\op{m}} \sqcup \{\infty\}$;
\item[Case N] $\{0,\infty\} \sqcup \mathbb{G}_{\op{m}}$.
\end{description}
We will say that \emph{the orbit $Z$ is of type \textbf{G}, \textbf{U}, \textbf{T} or \textbf{N} relative to the simple reflection $s$} depending on which of these cases actually occurs.
\index{type of orbit relative to simple reflection}
\end{Bemerkung}

\begin{remark}
Note that the orbit decomposition as above is more complicated when we consider arbitrary parabolic subgroups $P\subset G$. The approach to formality results below first establishes properties for the action of $B$ (via the simple orbit decomposition as above) and then deduces the results for actions of standard parabolics $P\supset B$ by push-forward along the natural projection $G/B\to G/P$, using e.g. the projective bundle formula of Proposition~\ref{prop:projectivebundle}.
\end{remark}

\begin{example}
If $X=G/B$ is the flag variety, then the Bruhat decomposition implies that each $B$-orbit in $G/B$ is of type \textbf{G} or \textbf{U} relative to any simple reflection. 
\end{example}

\begin{Bemerkung}
Let $K\subset G$ be a connected spherical subgroup. There is an action of reduced words on closures of $K$-orbits: assume $Y\subset G/B$ is the closure of a $K$-orbit, then for a simple reflection $s_i\in W$ we obtain another $K$-orbit closure by $s_i\star Y:=\pi_i^{-1}\pi_i(Y)$ where $\pi_i:G/B\to G/P_i$ is the natural projection for the minimal parabolic $P_i$. This extends to an action of the Richardson--Springer monoid $\mathcal{M}(W)$ on the set of $K$-orbit closures. 
\end{Bemerkung}

\begin{Bemerkung}
\label{symmBSresolution}
Using the action of the Richardson--Springer monoid on the set of $B\times K$-orbits in $G$, we can consider the \emph{reduced decomposition} of an orbit, cf. \cite[Section 3]{richardson:springer:survey}. A reduced decomposition of the $B$-orbit $V\subset G/K$ is a pair of a sequence $(V_0,\dots,V_r)$ of distinct $B$-orbits and a sequence $(s_1,\dots,s_r)$ of simple reflections such that the orbit $V_0$ is minimal, $V_r=V$ and $V_i=m(s_i)\star V_{i-1}$. If $K\leq G$ is a symmetric subgroup, then all $B$-orbits in $G/K$ have reduced decompositions. This can be used to prove statements about $B\times K$-orbits in $G$ by starting from closed orbits and iteratively go up $\mathbb{P}^1$-bundles. 

Let $K\subset G$ be a symmetric subgroup in a connected split reductive group, and let $B\subset G$ be a Borel subgroup. Let $V\subset G/K$ be a $B$-orbit, the closures $\overline{V}$ of $B$-orbits are analogues of the Schubert varieties for the symmetric case. Let $(\mathbf{v},\mathbf{s})$, $\mathbf{v}=(V_0,\dots,V_r)$ and $\mathbf{s}=(s_1,\dots,s_r)$ be a reduced decomposition of the $B$-orbit $V$. Then we define 
\[
Z_{(\mathbf{v},\mathbf{s})}:= P_{s_r}\times_{/B} P_{s_{r-1}}\times_{/B} \cdots \times_{/B} P_{s_1}\times_{/B} V_0
\]
Then there is a proper surjective morphism $\psi:Z_{(\mathbf{v},\mathbf{s})}\to \overline{V}$. This morphism is generically finite, and its degree is $2^{c(V)}$ where $c$ is the number of times the case $\mathbf{N}$ appears in the reduced decomposition of the orbit $V$, cf. \cite[3.10]{springer:schubert}. In particular, this is a resolution in the multiplicity-one cases, similar to the Bott--Samelson resolutions for Schubert varieties in flag varieties.\footnote{In general, this might still be an alteration in the sense of de Jong, so that we can use these to prove results for motives with rational coefficients. For now we are not aware of results in the literature discussing when the fibers of these resolutions have pavings by affine spaces.}
\end{Bemerkung}

\section{Bott--Samelson motives}
\label{sec:BS}

In this section, we discuss a formalism of Bott--Samelson motives which provides a way of showing that a variety with action satisfies the equivariant Whitney--Tate condition.

\index{motives!cuspidal}
\index{motives!clean}
\index{motives!Bott--Samelson}
\index{Bott--Samelson!motives}
The basic idea of the formalism can be described as follows: in the examples we are interested in, we always have an action of some parabolic in a reductive group. Given a collection of (equivariant) motives on orbits, we can generate a collection of $P$-equivariant motives, with $P$ running through all the standard parabolic subgroups, by using the six-functor formalism developed previously. There are two interesting collections of motives to apply this procedure to: if we start with \emph{cuspidal motives}, then we get a collection containing all $*$-orbitwise equivariant mixed Tate motives, cf. Proposition~\ref{prop:bsgenerates}; if we start with \emph{clean motives}, then we get a collection contained in the equivariant mixed Tate motives, cf. Proposition~\ref{prop:bstate}. \emph{If cuspidals are clean}, we have a very small collection, called \emph{Bott--Samelson motives}, which generate all equivariant mixed Tate motives. Since the collection of Bott--Samelson motives is closed under Verdier duality, we get the equivariant Whitney--Tate condition. 

If furthermore Bott--Samelson motives are pointwise pure, we can compute morphisms between them and see that they satisfy the orthogonality for weight structures. Consequently, there is a weight structure on equivariant mixed Tate motives whose heart is exactly the idempotent completion of the Bott--Samelson motives (and that weight structure coincides with the one discussed in Section~\ref{sec:weights}). This gives rise to tilting results and can be interpreted as a stronger form of known formality results for the equivariant derived category. 

The formalism will be applied to our standard examples in Sections~\ref{sec:parabolic}, \ref{sec:symmetric} and \ref{sec:wonderful}. The proof of the two pivotal assertions -- that cuspidals are clean and that Bott--Samelson motives are pointwise pure -- will be done in Section~\ref{sec:tiltingapp}. 

The arguments provided in this section can in principle be found already in works of Lusztig--Vogan \cite{lusztig:vogan} and Mars--Springer \cite{Mars-Springer}. We hope that the translation into the motivic language has the advantage of making the geometric inputs clearly visible.

\begin{Bemerkung}
From this moment on, we will really enforce our conventional assumption that the base field $k$ is algebraically closed of characteristic unequal to $2$.\footnote{Most of the results also go through over finite fields of characteristic unequal to $2$. However, for finite fields one ends up having to deal with some annoying separability issues that only serve to distract from the main ideas.} `Point' will always mean a `geometric point', and we set $\pt = \mathrm{Spec}(k)$. 
\end{Bemerkung}

\begin{Bemerkung}
In the following, we consider a homotopical stable algebraic derivator $\mathbb{D}$ satisfying the conditions of \ref{derivator:new} and the grading condition \ref{conditions:grading}. 
\end{Bemerkung}

\begin{Bemerkung}
We will always be concerned with varieties with action $G\looparrowright X$ over $k$ where $G$ is a connected reductive group such that for a Borel subgroup $B\subset G$ the variety with action  $B\subset G\looparrowright X$ has finitely many separably defined orbits. The goal of the formalism of Bott--Samelson motives discussed below is to establish the equivariant Whitney--Tate property for all varieties $P\looparrowright X$, where $B\subset P\subset G$ is a standard parabolic of $G$. 
\end{Bemerkung}

The following definition provides a way of generating equivariant mixed Tate motives for parabolic groups, by starting from a collection of motives on the various Borel orbits, via (essentially) successive induction and restriction. 

\begin{definition}[BS-closure]
\label{def:BSclosure}
\index{BS-closure}
Let $\mathbb{D}$ be a derivator satisfying the conditions of \ref{derivator:new} and the grading condition \ref{conditions:grading}, let $G$ be a connected reductive group and $B\subset G$ be a Borel subgroup and let $G\looparrowright X$ be a variety with $G$-action having finitely many $B$-orbits separably defined over $k$. Assume we are given, for each $B$-orbit $j:W\hookrightarrow X$ a collection $\mathcal{S}_W$ of motives in $\DMT_B(W)$. We define the \emph{BS-closure of $\mathcal{S}_W$} to be the smallest collection $\left\{\langle\mathcal{S}\rangle^{\op{BS}}\subset\mathbb{D}^+_{P}(X)\right\}_{P\supset B}$  of strictly full subcategories of $\mathbb{D}^+_P(X)$, for $P\subset G$ running through the parabolic subgroups of $G$ containing $B$, such that:
\begin{enumerate}
\item $j_*M \in \langle\mathcal{S}\rangle^{\op{BS}}\subset \mathbb{D}_B^+(X)$ for each $B$-orbit $j\colon W\hookrightarrow X$ and each $M\in \mathcal{S}_W$;
\item The categories $\langle\mathcal{S}\rangle^{\op{BS}}\subset\mathbb{D}^+_{P}(X)$ are stable under $M\mapsto M(n)[2n]$,  $n\in\mathbb{Z}$;
\item The categories $\langle\mathcal{S}\rangle^{\op{BS}}\subset\mathbb{D}^+_{P}(X)$ are stable under taking direct summands;
\item The categories $\langle\mathcal{S}\rangle^{\op{BS}}\subset\mathbb{D}^+_{P}(X)$ are extension stable, i.e., if there is a distinguished triangle 
\[ 
L\to M \to N \to L[1]
\] in 
$\mathbb{D}^+_P(X)$ with $L,N\in \langle\mathcal{S}\rangle^{\op{BS}}$, then $M\in\langle\mathcal{S}\rangle^{\op{BS}}$.
\item for any two parabolics $P,Q$ with $B\subset P\subset Q\subset G$, the functors  $\op{Res}_Q^P$ and $\op{Ind}_P^Q$ preserve the categories $\langle\mathcal{S}\rangle^{\op{BS}}$.
\end{enumerate}
\end{definition}

\begin{remark}
By Proposition~\ref{Lasp}, we could also use exceptional integration functors in the above without changing much. This will be used later for establishing that the BS-closure is stable under Verdier duality, cf. Theorem~\ref{thm:wtcond}.
\end{remark}

\begin{remark}
In the above definition, we actually need to include $G$ with the parabolic subgroups. Otherwise the BS-closure of cuspidals below doesn't produce enough mixed Tate motives, even in the case $\op{SL}_2\looparrowright \mathbb{P}^1$. 
\end{remark}

\begin{definition}[{\bf Hecke-connectedness}]
\label{ref:heckeconn}
\index{Hecke-connected}
\index{Hecke-simply-connected}
Let $G$ be a connected reductive group with Borel subgroup $B$, and let $G\looparrowright X$ be a variety with action. Recall that the Bruhat order $\leq$ is given by inclusion of orbit closures, cf. Definition~\ref{def:order}. The $G$-variety $X$ will be called \emph{Hecke-connected} if the following hold: 
\begin{enumerate}
\item there are finitely many $B$-orbits in $X$;
\item if $v$ is not a closed $B$-orbit, then there exists a $B$-orbit $u\lneq v$ and a simple reflection $s$ such that $s\star u =v$, i.e., for any $B$-orbit which is not closed there is a smaller orbit in the weak order.
\end{enumerate}
If, in addition, the isotropy group $B_x$ is connected for each point $x\in X$, then we will call $X$ \emph{Hecke-simply-connected}.
\footnote{The terminology Hecke-simply connected refers to the topological fact that the orbits for the $B$-action will be simply-connected.}
\end{definition}

\begin{example}
\label{ex:exgbhecke}
The flag varieties $X=G/P$ are Hecke-simply-connected.
\end{example}

\begin{example}
\label{SLSOexample}
Take $G=\op{SL}_2$, $B\subset G$ to be the usual Borel subgroup consisting of upper triangular matrices, $K$ to be the maximal torus consisting of diagonal matrices, and set $X = G/K$. Then $X$ is Hecke-connected but not Hecke-simply-connected: instead of dealing with $B$ acting on $G/K$ it is more convenient to visualize $K$ acting on $G/B$. Identify $G/B$ with $\PP^1$ so that $K$ acts via
\[ \begin{pmatrix} t & 0 \\ 0 & t^{-1} \end{pmatrix}\cdot x = t^2x. \]
There are three orbits: $\{0\}, \{\infty\}$ and $\PP^1 - \{0,\infty\}$. The isotropy group of a point in the open orbit is $\mu_2=\{1,-1\}$.
\end{example}

\begin{proposition}
\label{prop:gkhecke}
Let $G$ be a connected reductive group, let $\theta:G\to G$ be an involution, and let $B\subset G$ be a $\theta$-stable Borel subgroup. Denoting $K=G^\theta$ the fixed subgroup, the symmetric variety $G\looparrowright G/K$ is Hecke-connected. 
\end{proposition}

\begin{proof}
That $G/K$ is Hecke-connected is \cite[Theorem 4.6]{Richardson-Springer}. 
\end{proof}

\begin{remark}
Note that being Hecke connected is a requirement for the $B$-action, more precisely a statement about the weak order. The requirement for the $B$-action implies similar statements for the corresponding orders of orbits of parabolic groups containing $B$, cf. the discussion in Section~\ref{sec:korbit}. 
\end{remark}

\subsection{Cuspidal motives}

We now define cuspidal motives. The category of equivariant mixed Tate motives will be a subcategory of the BS-closure of the cuspidal motives. In particular, every equivariant mixed Tate motive can be obtained by induction from cuspidal motives on lower-dimensional strata, and cuspidals are those Tate motives that cannot be obtained inductively from lower-dimensional orbits.

\begin{definition}
\label{def:cuspidal}
\index{motives!cuspidal}
Let $G$ be a connected reductive group, let $B\subset G$ be a Borel subgroup and let $G\looparrowright X$ be a variety with action which has finitely many $B$-orbits separably defined over $k$. Given a $B$-orbit $j\colon B/H=V\hookrightarrow X$ and $M \in \DMT_B(V)_{\op{wt}=0}$, we will say that $M$ is \emph{cuspidal} if $\ind^{P_s}_B(j_*M)= 0$ for each simple reflection $s$ such that $V$ is a proper open subvariety of $P_s\cdot V$. 
\end{definition}

\begin{example}
\label{ex:cuspidal}
Note that if the orbit $V=B/H$ is closed in $X$, then there is no simple reflection satisfying the conditions of Definition~\ref{def:cuspidal}. So for closed $B$-orbits $V$, each $M\in \DMT_B(V)_{\op{wt}=0}$ is cuspidal. This is in particular true for $X=G/B$ the flag variety and $V$ the one-point orbit. Moreover, in the case of flag varieties, there cannot be cuspidal motives on orbits which are not closed: these are all of type \textbf{G} or \textbf{U} and these do not support cuspidal motives which is explained in more detail in \ref{RTZl}. Therefore, in the case of flag varieties, there is only one type of cuspidal motives - constant Tate motives pure of weight $0$ on the one-point orbit.
\end{example}

\begin{example}
\label{SLSOexamplecuspidal}
In the situation of Example \ref{SLSOexample} we certainly have cuspidals corresponding to the closed orbits. But there is also a cuspidal on the open orbit - the isotropy group of a point in the open orbit is isomorphic to $\mu_2$. The motive corresponding to the sign representation is cuspidal.
The cuspidality of this motive is not quite obvious (a more general situation is described in Lemma \ref{lemma:gmlocalsystem}; also see Lemma \ref{lemma:reducetop1}). For the moment we hope to assuage the reader with the following intuition: if we had been working topologically (i.e., with sheaves instead of motives) this would correspond to the `M\"obius band' local system on $\CC^{\times}$, i.e., the 1-dimensional local system with eigenvalue of monodromy equal to $-1$. This local system has no (derived) global sections.
\end{example}

Eventually, we want to show that the BS-closure of the cuspidal motives contains all equivariant mixed Tate motives. For this, we need to study the structure of cuspidal motives. The relevant features already appear in the case of $\mathbb{P}^1$, as discussed in the lemma below. Subsequent results will then extend the characterization of cuspidal motives to general Hecke-connected varieties. We also note that our assumptions on the base field being algebraically closed implies that the groups $\mu_N$ of roots of unity have $N$ distinct points and the same is then true for the fibers of the Kummer covering $\mathbb{G}_{\op{m}}\xrightarrow{x\mapsto x^N}\mathbb{G}_{\op{m}}$. 

\begin{lemma}
\label{lemma:gmlocalsystem}
Let $H$ be a linear algebraic group acting transitively on $\mathbb{G}_{\op{m}}$ via a group homomorphism $H\twoheadrightarrow \mathbb{G}_{\op{m}}$. Let $V\in \DMT_H(\mathbb{G}_{\op{m}})$ be a rank one local system in the sense of\ref{locsysrank}. Then $\fin_*V \neq 0$ if and only if $\For(V) = \const{\mathbb{G}_{\op{m}}}(n)$ for some $n\in \ZZ$.
\end{lemma}

\begin{proof}
Denote by $H'$ the stabilizer group of $H\looparrowright \mathbb{G}_{\op{m}}$ of some point $x\in\mathbb{G}_{\op{m}}$ and let $V\in \DMT_H(\mathbb{G}_{\op{m}})$ be a rank one local system. By \ref{locsysrank}, this means that $V$ corresponds under the induction equivalence $\DMT_H(\mathbb{G}_{\op{m}})\approx \DMT_{H'}(\pt)$ to an equivariant mixed Tate motive of the form $\mathbb{Q}(i)[j]$. The stabilizer subgroup $H'$ maps to a closed subgroup in $\mathbb{G}_{\op{m}}$, therefore we can assume that $H'$ is actually identified with this subgroup, i.e., we can assume that $H'=\mu_m$ for some $m$. Denoting the morphism 
\[
\pi_m\colon(1:\mathbb{G}_{\op{m}}\looparrowright \mathbb{G}_{\op{m}})\to (m:\mathbb{G}_{\op{m}}\looparrowright \mathbb{G}_{\op{m}}):t\mapsto t^m,
\]
we have $\pi_m^\ast\For(V)=\const{\mathbb{G}_{\op{m}}}(n)$ for some $n$.\footnote{Note here that the target is the variety with action we consider, since we have reduced to assuming that the stabilizer is $\mu_m$.} Taking the cone of the canonical unit map $\const{\mathbb{G}_{\op{m}}} \to \pi_{m,*}\pi_m^*\const{\mathbb{G}_{\op{m}}}$ yields a distinguished triangle
\[ 
\const{\mathbb{G}_{\op{m}}} \to \pi_{m,*}\const{\mathbb{G}_{\op{m}}} \to M \to \const{\mathbb{G}_{\op{m}}}[1].
\]
Note that $\pi_{m,\ast}\const{\mathbb{G}_{\op{m}}}$ can be viewed as the group ring for $\mu_m$ over the coefficient field $\Lambda$ of the underlying derivator $\mathbb{D}$. The distinguished triangle maps the trivial representation into the regular representation, so that $M$ contains all the summands with non-trivial $\mu_m$-representations.

The unit map $\const{\mathbb{G}_{\op{m}}} \to \pi_{m,*}\const{\mathbb{G}_{\op{m}}}$ becomes an isomorphism on applying $\fin_*$: both source $\op{fin}_\ast \const{\mathbb{G}_{\op{m}}}$ and target $\op{fin}_\ast \pi_{m,*}\const{\mathbb{G}_{\op{m}}}$  are isomorphic to the (cohomological) motive of $\mathbb{G}_{\op{m}}$. Composing with the augmentation map from the group ring $\op{fin}_\ast\pi_{m,*}\const{\mathbb{G}_{\op{m}}}\to \op{fin}_\ast \const{\mathbb{G}_{\op{m}}}$ we get an endomorphism of $\op{fin}_\ast\const{\mathbb{G}_{\op{m}}}$ which can be identified with multiplication by the degree of $\pi_m$, which is $m$. Since we are working with rational coefficients, this is an isomorphism. Moreover, since the objects are finite-dimensional, the unit map is an isomorphism.\footnote{We could also deduce this using the conservativity statements in Proposition~\ref{prop:conservative} for the derivators we are interested in.} Therefore, we deduce $\fin_*M = 0$ from the above distinguished triangle. 

Suppose $\For(V)\neq \const{\mathbb{G}_{\op{m}}}(n)$ for any $n$. Then $\For(V)$ occurs as a direct summand of $M$ up to some Tate twist. This basically follows because $\pi_{m,*}\const{\mathbb{G}_{\op{m}}}$ contains $\const{\mathbb{G}_{\op{m}}}$ as a direct summand, and if the underlying motive is not (up to twist) isomorphic to $\const{\mathbb{G}_{\op{m}}}$, then it must occur in $M$. Consequently, since $\op{fin}_\ast$ is additive, $\fin_*V = 0$. Finally, $\op{fin}_\ast V\neq 0$ if  $\op{For}(V)=\const{\mathbb{G}_{\op{m}}}(n)$ as in this case $\op{fin}_\ast V$ is the cohomological motive of $\mathbb{G}_{\op{m}}$. 
\end{proof}

\begin{lemma}
\label{lem:cuspreduction}
Suppose $H$ acts on $\PP^1$ with finitely many orbits. Let $j\colon U \hookrightarrow \PP^1$ be the open orbit and let $i\colon Y \hookrightarrow \PP^1$ be the closed complement (i.e., $Y=\PP^1-U$). Let $N\in \DMT_H(U)_{\op{wt}=0}$. Then in the notation of \ref{ss:p1mod}, we have the following cases:
\begin{description}
\item[Case U] there is a distinguished triangle
\[ 
i_*\fin_*N[-1] \to j_!N \to \fin^*_{\PP^1}\fin_*N \to i_*\fin_*N 
\]
\item[Case T and N] if $N$ is a rank one local system with $\fin_*N\neq 0$, then $\fin_*N \sim N_0 \oplus N_1$ with $N_0$ pure of weight $0$ and $N_1$ pure of weight $1$; there is a distinguished triangle
\[ 
i_*\fin_Y^*N_0[-1] \to j_!N \to \fin_{\PP^1}^* N_0 \to i_*\fin_Y^* N_0.
\]
\end{description}
\end{lemma}

\begin{proof}
In case \textbf{U}, $U\cong\mathbb{A}^1$ and consequently the underlying motive of $N$ is a constant mixed Tate motive. In this case, $\op{fin}_{\mathbb{P}^1}^\ast \op{fin}_\ast N$ also has underlying motive constant mixed Tate; moreover, the restriction of this motive to $U$ is isomorphic to $N$. The distinguished triangle claimed is simply the localization triangle $j_! j^\ast \to \op{id}\to i_\ast i^\ast$. 

In case \textbf{T} and $\textbf{N}$, we have $U\cong\mathbb{G}_{\op{m}}$. By Lemma~\ref{lemma:gmlocalsystem}, $N$ must be of the form $\const{\mathbb{G}_{\op{m}}}$. This implies the statement about the form of $\op{fin}_\ast N$. The distinguished triangle is then again simply the localization triangle $j_!j^\ast\to\op{id}\to i_\ast i^\ast$ for the constant motive $\op{fin}_{\mathbb{P}^1}^\ast N_0$. 
\end{proof}

In general, testing for cuspidality is no more complicated than the situation for $\PP^1$:

\begin{lemma}
\label{lemma:reducetop1}
Let $G$ be a reductive group, and let $G\looparrowright X$ be a variety with action which has finitely many $B$-orbits. Let $Z$ be a $B$-orbit, and denote by $P_s$ the standard parabolic for the simple reflection $s$. Let $\op{inc}\colon P_s\cdot  Z\hookrightarrow X$ be the inclusion of the $P_s$-orbit generated by $Z$, let $z\in Z$ be a point, and let $H\subset P_s$ be its isotropy group. Then we have a commutative diagram:
\[ 
\xymatrix{
  \mathbb{D}^+_B(X) \ar[r]^-{\ind^{P_s}_B} & \mathbb{D}^+_{P_s}(X) \ar[r]^-{\res^{B}_{P_s}}& \mathbb{D}^+_B(X) \\
  \mathbb{D}^+_B(P_s \cdot Z)\ar[u]_-{\op{inc}_*} \ar[r]^-{\ind^{P_s}_B} \ar[d]_-{\approx} & \mathbb{D}^+_{P_s}(P_s \cdot Z)\ar[u]_-{\op{inc}_*} \ar[r]^-{\res^B_{P_s}} \ar[d]_-{\approx}& \mathbb{D}^+_B(P_s \cdot Z)\ar[u]_-{\op{inc}_*}\ar[d]_-{\approx} \\
  \mathbb{D}^+_H(\PP^1) \ar[r]^-{\fin_*} & \mathbb{D}^+_H(\pt) \ar[r]^-{\fin^*} & \mathbb{D}^+_H(\PP^1)
}
\]
\end{lemma}

\begin{proof}
Restriction commutes with all the standard functors, by \ref{CompRES}, whence the commutativity of the top right square. As induction is defined to be the right adjoint to restriction, cf. Proposition~\ref{prop:integration}, and restriction commutes with $\op{inc}^*$, we obtain the commutativity of the top left square via adjunction (and the fact that compositions of right adjoints is right adjoint to the composition). 

We have to deal with the commutativity of the bottom squares. Recall from \ref{PBPF} that $\op{Ind}_\ast$ is given by push-forward along the composition 
\[
(B\looparrowright P_s\cdot Z) \xrightarrow{(i,s)} (P_s\looparrowright (P_s\times_{/B} P_s\cdot Z)) \xrightarrow{(\op{id},m)} (P_s\looparrowright P_s\cdot Z). 
\]
The lower vertical equivalences are the ones considered in \ref{ss:p1mod}, i.e., they are given by a composition of equivalences: the first and last equivalences are simply induced from isomorphisms of varieties and will therefore obviously commute with the pushforwards in the definition of induction. The two equivalences in the middle are obtained by changing the equivariance via the quotient equivalence, using that both $B$ and $H$ are normal subgroups in $B\times H$. By Proposition~\ref{prop:quotientequiv}, the quotient equivalences are compatible with the six functors, and in particular with the induction functors. This implies the commutativity of the lower left square. The commutativity of the lower right square follows similarly, since the quotient equivalence is compatible with restriction functors. 
\end{proof}

\begin{Bemerkung}
\label{RTZl}
Let $G$ be a reductive group with Borel subgroup $B\subset G$, and let $G\looparrowright X$ be a variety with action having finitely many $B$-orbits. Let $Z$ be a $B$-orbit of $X$, and let $Y$ be the open $B$-orbit in $P_s\cdot Z$. Denote the inclusions by $Y\xrightarrow{\tilde{\jmath}} P_s\cdot Z \xrightarrow{\op{inc}} X$, the composite is the inclusion $j:Y\to X$. Let $M$ be a cuspidal local system on $Y\subset P_s\cdot Z$, which means that 
\[
\ind^{P_s}_B (j_*M)=\op{inc}_*\ind^{P_s}_B (\tilde\jmath_*M)=0.
\]
Under the equivalences of \ref{ss:p1mod} and the commutative diagrams of Lemma~\ref{lemma:reducetop1}, this means necessarily $\op{fin}_*u_*N=0$ for $u:U\hra \mathbb P^1$ the inclusion of the open $H$-orbit and $N$ the local system on it corresponding to $M$.\footnote{This is the way in which Lemma~\ref{lemma:reducetop1} combined with Lemma~\ref{lemma:gmlocalsystem} can be used to check cuspidality of motives.}  Clearly this cannot happen if $U$ is all of $\mathbb P^1$ or the complement of a single point, corresponding to the cases \textbf{G} and \textbf{U} in \ref{ss:p1mod}. More generally, it cannot happen if $N$ is constant on $U$, but only if it has nontrivial monodromy (in the sense that it corresponds to a motive on the point with non-trivial stabilizer representation), cf. Lemma~\ref{lemma:gmlocalsystem}.
\end{Bemerkung}

\begin{proposition}
\label{prop:bsgenerates}
Let $\mathbb{D}$ be a derivator satisfying the conditions of \ref{derivator:new}. Let $G$ be a reductive group, let $G\looparrowright X$ be a variety with action which is Hecke-connected. Then for every standard parabolic $P$, the category $\DMT^\ast_P(X)$ is contained in the smallest triangulated subcategory of $\mathbb{D}^+_P(X)$ generated by the BS-closure of cuspidal motives.
\end{proposition}

\begin{proof}
Let $\cT_P$ denote the smallest triangulated subcategory of $\mathbb{D}^{\op{c}}_P(X)$ containing the BS-closure of the cuspidal motives. By Proposition~\ref{lkio}, the category $\DMT_P^\ast(X)$ is generated by objects of the form $j_!M$ for $j:V\hookrightarrow X$ the inclusion of a $P$-orbit and $M\in \DMT_P(V)$. In particular, we need to show that $j_!M\in \mathcal{T}_P$ for every parabolic $P$, every $P$-orbit $j:V\to X$ and every $M\in \DMT_P(V)$. 

First, we concentrate on the case of the Borel subgroup. Let $j:V\to X$ be a $B$-orbit, and let $M\in \DMT_B(V)$ be a motive. By Proposition~\ref{prop:wtpt}, there is a weight structure on $\DMT_B(V)$, and by using the weight decomposition triangles we can reduce to the case where $M\in \DMT_B(V)_{\op{wt}=0}$. 

Now the claim is proved by induction on the weak order on the set of $B$-orbits on $G/K$. If $V$ is a closed orbit, then $j_! M\cong j_\ast M$ and then $j_!M$ is in the BS-closure by definition and every weight $0$ motive is cuspidal by Example~\ref{ex:cuspidal}. Otherwise, since $X$ is Hecke-connected, we may find a simple reflection $s$ and an orbit $W\lneq V$ such that $s\star W = V$. This situation corresponds, via the remarks in \ref{ss:p1mod}, to an algebraic group $H$ acting on $\mathbb{P}^1$ with finitely many orbits, with $M$ corresponding to a motive on the open orbit. By Lemma~\ref{lem:cuspreduction}, to show that $j_!M$ is in the BS-closure, it suffices to show that all the other motives in the exact sequence are in the BS-closure. 

In case \textbf{U}, $M$ corresponds to a motive $N$ on the open orbit $\mathbb{A}^1$. In particular, $N$ is a constant pure Tate motive of weight $0$ and therefore $\op{fin}_\ast N$ is a constant pure Tate motive of weight $0$ such that $j^\ast \op{fin}_{\mathbb{P}^1}^\ast \op{fin}_\ast N\cong N$. By induction on the weak order we know that $i_\ast \op{fin}_\ast N$ is in the BS-closure. To show that $\op{fin}_{\mathbb{P}^1}^\ast\op{fin}_\ast N$ is in the BS-closure, we can use the localization sequence with $\tilde{N}=\op{fin}^\ast_{\mathbb{P}^1}\op{fin}_\ast N$ which is a constant pure Tate motive on $\mathbb{P}^1$: 
\[
i_! i^! \tilde{N} \to \tilde{N} \to j_\ast j^\ast \tilde{N} \to i_! i^! \tilde{N}[1] 
\]
For the first motive in the sequence, we can use absolute purity which implies that $i^! \tilde{N}(1)[2]$ is a constant pure Tate motive on the point, and the third motive is $j_\ast N$. Now we apply $\op{fin}_{\mathbb{P}^1,\ast}\op{fin}^\ast_{\mathbb{P}^1}$ to this sequence. Under this, the third motive maps to $\op{fin}^\ast_{\mathbb{P}^1}\op{fin}_\ast N$. We need to show that the other two summands are in the BS-closure. By construction, these will correspond to $\op{Res}_{P_s}^B\op{Ind}_B^{P_s}$ of the first two motives on $\mathbb{P}^1$; so it suffices to show that the motives corresponding to the skyscraper motive and the constant motive on $\mathbb{P}^1$ are in the BS-closure. For the first, this follows from induction on the weak Bruhat order, and the second is given by induction-restriction from the first. 

For case \textbf{T} and \textbf{N}, we can argue the same way. The only difference is that the relevant motives appearing in the sequence of Lemma~\ref{lem:cuspreduction} are only direct summands of the motives we would have considered in case \textbf{U}. 

To establish our claims for general $P$, we have to remark that by the projective bundle formula  Proposition~\ref{prop:projectivebundle} every $\mathcal F\in \mathbb D_{P\times K}(G)$ occurs as a direct summand in $\op{Ind}_B^P\op{Res}_P^B\mathcal F$, hence every $\mathcal F\in \op{MTDer}_{P\times K}^*(G)$ occurs as a direct summand  of $\op{Ind}_B^P\mathcal G$ for some $\mathcal G\in \op{MTDer}^*_{B\times K}(G)$  and therefore (1) and (2) for general $P$ follow from (1) and (2) in case $P=B$.   
\end{proof}

\subsection{Clean motives}

Now we define clean motives. The cleanness condition implies that the category generated by induction from clean motives is contained in the category of equivariant mixed Tate motives. 

\begin{definition}
\label{def:clean}
\index{motives!clean}
Let $G$ be a connected reductive group, let $B\subset G$ be a Borel subgroup and let $G\looparrowright X$ be a variety with action which has finitely many $B$-orbits separably defined over $k$. Given a $B$-orbit $j\colon B/H=V\hookrightarrow X$ and $M\in \DMT_B(V)_{\op{wt}=0}$ we will say $M$ is \emph{clean} if the canonical map $j_! M \to j_*M$ is an isomorphism. 
\end{definition}

\begin{example}
Take $X=G/B$ to be the flag variety. By Example~\ref{ex:cuspidal}, the cuspidals are exactly the motives $M\in \DMT_B(V)_{\op{wt}=0}$ where $V$ is the one-point orbit. These motives are clean because the inclusion $j:V\to G/B$ is proper. 
\end{example}

\begin{example}
In general, orbits that are not closed will often admit clean motives. For instance, the cuspidal motive on the open orbit in Example \ref{SLSOexamplecuspidal} is clean. Unfortunately, unlike cuspidality, cleanness cannot be completely reduced to a question on $\PP^1$, i.e., for cleanness it doesn't suffice to check that $i^\ast j_\ast M$ vanishes for $i$ the inclusion of a codimension 1 orbit. See Lemma~\ref{lemma:cleanlemma} and Proposition~\ref{prop:symmtilt} for how to establish cleanness of certain motives in the case of Borel-actions on symmetric varieties. 
\end{example}

\begin{proposition}
\label{prop:cleanverdier}
Let $G$ be a connected reductive group with Borel subgroup $B\subset G$ and let $G\looparrowright X$ be a variety with action which has finitely many $B$-orbits separably defined over $k$. Let $\mathbb{D}$ be a derivator satisfying the conditions of \ref{derivator:new}, the grading condition~\ref{conditions:grading} and the weight condition~\ref{conditions:weight}. Then the BS-closure of the clean motives is stable under Verdier duality.
\end{proposition}

\begin{proof}
If $M$ is a clean motive on a $B$-orbit $j:Z\hookrightarrow X$, then its Verdier dual $D(M)$ is clean, since Verdier duality preserves $\DMT_B(Z)_{\op{wt}=0}$, cf. Proposition~\ref{prop:wtdual}, and the composition
\[
D(j_\ast M)\cong j_!D(M)\to j_\ast D(M)\cong D(j_! M)
\]
is the dual of the canonical map $j_!M \to j_\ast M$. But then also motives of the form $j_\ast M$ for $M$ clean will be closed under Verdier duality, since $D(j_\ast M)\cong j_! D(M)\cong j_\ast D(M)$. The operations $M\mapsto M(n)[2n]$, taking direct summands and extensions are obviously compatible with Verdier duality. 

It remains to discuss compatibility with restriction and induction. For two standard parabolics $P\subset Q$, we have $D(\op{Ind}_P^Q M)\cong \op{Ind}_! D(M)$, by \ref{Last}. Moreover, by \ref{Lasp} we have $\op{Ind}_! D(M)\cong \op{Ind}_\ast D(M)(d)[2d]$ where $d=\dim Q/P$, in particular exceptional integration preserves the Bott--Samelson closure of the clean motives. Combining the two statements, we find that Verdier duality commutes, up to appropriate twist and shift, with the ordinary induction functors.

A similar discussion can be obtained for the restriction functors. We want to show that $D(\op{Res}_Q^P)$ preserves the Bott--Samelson closure of the clean motives, for an inclusion  $P\subset Q$ of standard parabolics. As in \ref{PBPF}, we can write the restriction as pullback along the composition
\[
\left(P\looparrowright X\right) \xrightarrow{(i,s)} \left(Q\looparrowright(Q\qtimes{P} Q)\right) \xrightarrow{(\op{id},m)} \left(Q\looparrowright X\right)
\]
The functor $(i,s)^\ast$ is an equivalence and compatible with Verdier duality, by the induction equivalence of Proposition~\ref{cor:indequiv}. On the other hand, since $(\op{id},m)$ is smooth of relative dimension $d=\dim Q/P$, we have 
\[
D\circ (\op{id},m)^\ast\approx (\op{id},m)^!\circ D\approx (\op{id},m)^\ast(d)[2d]\circ D.\]
We find that restriction $\op{Res}_Q^P$ commutes with Verdier duality up to appropriate twist and shift and therefore $D(\op{Res}_Q^P)$ preserves the Bott--Samelson closure of the clean motives. 
\end{proof}

\begin{proposition}
\label{prop:bstate}
Let $G$ be a connected reductive group with Borel subgroup $B\subset G$ and let $G\looparrowright X$ be a variety with action which has finitely many $B$-orbits separably defined over $k$. Let $\mathbb{D}$ be a derivator satisfying the conditions of \ref{derivator:new} and the grading condition~\ref{conditions:grading}. Then the motives in the BS-closure of the clean motives are orbitwise mixed Tate. 
\end{proposition}

\begin{proof}
By definition, if $j\colon Z\hookrightarrow X$ is a $B$-orbit and $M\in\DMT_{B}(Z)_{\op{wt}=0}$ is a clean motive, we have $j_!M\cong j_\ast M$. If $i:Y\to X$ is a different $B$-orbit, then $i^\ast j_! M\cong 0$ by an induction using the localization sequence. This implies that $j_\ast M$ is $\ast$-orbitwise mixed Tate: it is mixed Tate on $Z$ by definition and has trivial restrictions to any other orbit. Similarly, $j_\ast M$ is $!$-orbitwise mixed Tate because $j^! j_\ast M\cong j^\ast j_\ast M\cong M$ is mixed Tate by assumption and $i^! j_\ast M=0$ by induction using the localization sequence.

Further, the orbitwise mixed Tate properties are stable under $M\mapsto M(n)[2n]$, as well as taking direct summands and extensions. It is also stable under restriction functors $\op{Res}_P^Q$ by Proposition~\ref{prop:dmtres} and induction functors $\op{Ind}_Q^P$ by Theorem~\ref{ifpt}. This proves the claim.
\end{proof}


Let us point out here that the statements above concern motives in $\mathbb{D}^+_P(X)$ for every standard parabolic $B\subset P\subset G$.

\subsection{Bott--Samelson motives}

Now we have established the basic comparisons between the BS-closures of cuspidal and clean motives, we can identify the relevant condition: if cuspidals are clean, then these closures are equal and they coincide with the equivariant mixed Tate motives. 

\begin{proposition}
\label{prop:BScoincide}
Let $G$ be a reductive group, let $B\subset G$ be a Borel subgroup and let $X$ be a $G$-variety with finitely many $B$-orbits. Let $\mathbb{D}$ be a derivator satisfying the conditions of \ref{derivator:new}, the grading condition~\ref{conditions:grading} and the weight condition~\ref{conditions:weight}. If cuspidals are clean, then the Bott--Samelson closures of cuspidal and clean motives coincide.
\end{proposition}

\begin{proof}
By Proposition~\ref{prop:bsgenerates}, $\DMT_P^\ast(X)_{\op{wt}=0}$ is contained in the BS-closure of the cuspidal motives. By Proposition~\ref{prop:bstate}, the BS-closure of the clean motives is contained in $\DMT_P(X)$. By the assumption, the BS-closure of the cuspidal motives is contained in the clean motives. Consequently, all the above categories have to coincide (and coincide with $\DMT_P(X)_{\op{wt}=0}$).
\end{proof}

\begin{definition}
\index{Bott--Samelson!motives}
\index{motives!Bott--Samelson}
\label{BSmot}
In the situation of Proposition~\ref{prop:BScoincide}, the BS-closure, cf. Definition~\ref{def:BSclosure}, of clean (or equivalently, cuspidal) motives will be called the category of \emph{Bott--Samelson motives}. These categories will be denoted by $\DMT_P^{\op{bs}}(X)$, where $P$ runs through the standard parabolics of $(G,B)$.
\end{definition}


\begin{remark}
We will later provide a different interpretation of the Bott--Samelson motives, by replacing the restriction and induction in the definition of BS-closure by convolution functors, cf. \ref{BSoo}.
\end{remark}

\begin{example}
\label{ex:BSflag}
Let $X=G/B$ be the flag variety, let $\underline w = (s_1, \ldots, s_n)$ be a reduced sequence of simple reflections in the Weyl group,  and consider the Bott--Samelson resolution of the Schubert variety $\overline{Bs_1\cdots s_nB/B}$: 
\[ 
\pi\colon\op{BS}(\underline w) = P_{s_1} \qtimes{B} \cdots \qtimes{B} P_{s_n} \qtimes{B} B/B\to G/B. 
\]
Each $\pi_*\const{\op{BS}(\underline w)}$ is a Bott--Samelson motive. In fact, each Bott--Samelson motive on $G/B$ is of this form (modulo isomorphism, shift, Tate twist and direct sums). 
\end{example}

\begin{theorem}
\label{thm:wtcond}
Let $\mathbb{D}$ be a derivator satisfying the conditions of \ref{derivator:new}, the grading condition~\ref{conditions:grading} and the weight condition~\ref{conditions:weight}. Let $G$ be a reductive group, let $B\subset G$ be a Borel subgroup and let $G\looparrowright X$ be a variety with action having finitely many $B$-orbits. Assume that $X$ is Hecke-connected and that cuspidals are clean. Then the equivariant Whitney--Tate condition holds for $P\subset G\looparrowright X$, for any standard parabolic $P$, i.e., 
\[
\DMT^*_P(X) = \DMT^!_P(X).
\]
\end{theorem}

\begin{proof}
Under the assumptions, we can speak about categories $\DMT^{\op{bs}}_P(X)$ of Bott--Samelson motives. By Proposition~\ref{prop:cleanverdier}, we know that $\DMT^{\op{bs}}_P(X)$ is closed under Verdier duality. Then the respective triangulated subcategories of $\mathbb{D}^+_P(X)$ generated by $\DMT^{\op{bs}}_P(X)$ must also be stable under duality. Now, Propositions~\ref{prop:BScoincide},  \ref{prop:bsgenerates} and \ref{prop:bstate} imply that the triangulated subcategory of $\mathbb{D}^+_P(X)$ generated by Bott--Samelson motives is  $\DMT_P^\ast(X)$. Now Verdier duality interchanges the categories $\DMT^*_P(X)$ and $\DMT^!_P(X)$, which concludes the proof.
\end{proof}

\subsection{Purity and formality}

The remainder of the section will be spent on discussing how the formalism of Bott--Samelson motives implies various purity and formality results. The additional assumption required is the pointwise purity. If we have that, pushforwards of Bott--Samelson motives will be pure Tate motives, and the resulting orthogonality shows that the Bott--Samelson motives generate the heart of the weight structure. Then we can show the tilting result which expresses equivariant mixed Tate motives as complexes of Bott--Samelson motives.  

\begin{proposition}
\label{prop:BSinheart}
In the situation of Theorem~\ref{thm:wtcond}, the Bott--Samelson motives are contained in the heart of the weight structure of Proposition~\ref{prop:equivweight}.
\end{proposition}

\begin{proof}
  If $M$ is a clean motive on a $B$-orbit $j\colon Z\hookrightarrow X$, then by definition $M\in\DMT_B(X)_{\op{wt}=0}$. By definition $j_!M\cong j_\ast M$, and in particular $i^\ast j_\ast M=0$ for any other orbit inclusion $i\colon Y\hookrightarrow X$. Therefore, $j_\ast M\in \DMT_B(X)_{\op{wt}\leq 0}$. An argument as in the proof of Proposition~\ref{prop:bstate} implies that $j^!j_\ast M\cong M$ and $i^!j_\ast M=0$ for any other orbit $i:Y\hookrightarrow X$. Therefore, $j_\ast M\in\DMT_B(X)_{\op{wt}=0}$, which provides the base of the induction. 

The zero-weight categories are clearly preserved under $M\mapsto M(n)[2n]$, taking direct summands and extensions. 

It remains to deal with induction and restriction for standard parabolics $P\subset Q$. By Proposition~\ref{prop:wtres}, $\op{Res}_Q^P$ is weight exact. The induction functor $\op{Ind}_P^Q$ has left adjoint $\op{Res}_Q^P$, hence it is weight right-exact. On the other hand, Proposition~\ref{Lasp} implies that $\op{Ind}_\ast(d)[2d]\cong\op{Ind}_!$ has a right adjoint $\op{Res}_Q^P$, where $d=\dim Q/P$. Therefore, $\op{Ind}_\ast$ also has a right adjoint $\op{Res}_Q^P(d)[2d]$ which implies that it is weight exact. This proves the claim.
\end{proof}

\begin{remark}
Note that the above argument doesn't prove pointwise purity. While we show that all the functors involved in the definition of Bott--Samelson motives are weight-exact, we are not making any assertion about the restriction of the Bott--Samelson motives to orbits. For all we know at this point, the relevant restriction functors need not be weight-exact.
\end{remark}

\begin{proposition}
\label{puBS} 
Assume the situation of Theorem~\ref{thm:wtcond}, and assume that all Bott--Samelson motives are pointwise pure. Then for each standard parabolic $P$ there is a weight structure whose heart is $\DMT^{\op{bs}}_P(X)$. This weight structure coincides with the one defined in Proposition~\ref{prop:equivweight}, i.e., 
\[
\DMT^{\op{bs}}_P(X) = \DMT_P(X)_{\op{wt}=0}.
\]
\end{proposition}

\begin{proof}
Let $M, N\in \DMT^{\op{bs}}_P(X)$. Then
\[ 
\mathbb{D}^+_P(X)(M,N[n]) \cong \mathbb{D}^+(\pt)(\const{\pt}, \fin_{{\op{E}}P\times_{/P}X,*}\iHom(M,N)[n]). 
\]
By Corollary \ref{cor:hompuretate} (together with the assumption on pointwise purity), the motive $\fin_{{\op{E}}P\times_{/P}X,*}\iHom(M,N)$ is mixed Tate, and pure of weight $0$. Since the category $\DMT(\pt)$ is graded semi-simple by assumption, we find 
\[ 
\mathbb{D}^+_P(X)(M,N[n]) = 0 \quad \mbox{unless $n=0$}. 
\]
In particular, the collection of Bott--Samelson motives $\DMT_P^{\op{bs}}(X)$ is negative in the sense of \cite[Definition 1.5.VII]{bondarko:imrn}. By Theorem~\ref{thm:wtcond} and Proposition~\ref{prop:bstate}, the Bott--Samelson motives generate the category $\DMT_P^{\op{bs}}(X)$ as a triangulated category. Recall that by Lemma~\ref{lem:idempotent}, the category $\DMT_P(X)$ is idempotent complete. By \cite[Proposition 1.7(6)]{bondarko:imrn}, there exists a unique weight structure on $\DMT_P(X)$ which satisfies $\DMT_P^{\op{bs}}(X)\subset \DMT_P(X)_{\op{wt}=0}$. Since Bott--Samelson motives form an additive idempotent complete category by definition, the heart of that weight structure will be $\DMT_P^{\op{bs}}(X)$. By Proposition~\ref{prop:BSinheart}, this weight structure also coincides with the one of Proposition~\ref{prop:equivweight}. 
\end{proof}

Now we can deduce the following tilting result. This is basically a formal consequence of the negativity just established.

\begin{theorem}
\label{thm:tilting}
Assume the situation of Theorem~\ref{thm:wtcond}, and assume that all Bott--Samelson motives are pointwise pure. Then the tilting functor of Theorem~\ref{thm:derivatortilting} induces an equivalence 
\[ 
\Ho(\DMT^{\op{bs}}_P(X)) \mapright{\approx} \DMT_P(X). 
\]
Under the equivalence, the weight structure on $\DMT_P(X)$ corresponds to the natural weight structure on complexes of \cite[Remark 1.6(2)]{bondarko:imrn}.
\end{theorem}

\begin{proof}
The category of Bott--Samelson motives is a tilting subcategory by Proposition~\ref{puBS}. By Proposition~\ref{prop:equivderivator}, $\mathbb{D}^+_P(X,-)$ is a stable derivator. Applying Theorem~\ref{thm:derivatortilting}, we get a fully faithful functor 
\[
\op{Hot}^{\op{b}}(\DMT_P^{\op{bs}}(X))\to \mathbb{D}^+_P(X). 
\]
So the general tilting formalism yields an equivalence between $\Ho(\DMT_P^{\op{bs}}(X))$ and the smallest triangulated subcategory of $\mathbb{D}^+_P(X)$ containing the Bott--Samelson motives. By Proposition \ref{prop:bsgenerates} this latter subcategory is $\DMT_P(X)$. 

The claim concerning the weight structures follows from \cite[Proposition 1.7(6)]{bondarko:imrn} since $\DMT_P^{\op{bs}}(X)$ is a negative generating subcategory contained in both sides. 
\end{proof}

\section{Application of the Bott--Samelson formalism}
\label{sec:tiltingapp}

In this section, we apply the Bott--Samelson formalism developed in Section~\ref{sec:BS} to the relevant examples $P\times K\looparrowright G$. The central statements that need to be established are ``cuspidals are clean'' and ``pointwise purity''. This will imply that the corresponding varieties with action satisfy the equivariant Whitney--Tate condition of Definition~\ref{def:mtderdef2}, the category of equivariant mixed Tate motives has a well-behaved weight structure generated by Bott--Samelson motives, and the tilting functor of Theorem~\ref{thm:tilting} gives an equivalence
\[
\op{tilt}:\op{Hot}^{\op{b}}(\DMT^{\op{bs}}_G(X))\sirra\DMT_G(X).
\]

\begin{remark}
The equivariant motivic categories we will be interested in (for the applications to representation theory) are of the form $\mathbb{D}^+_{P\times K}(G)$, for $G$ reductive, $K\subset G$ a symmetric subgroup acting on the right and $P\subset G$ a parabolic subgroup acting on the left. Under the quotient equivalence, these are equivalent to $\mathbb{D}^+_P(G/K)\approx\mathbb{D}^+_K(G/P)$. The Bott--Samelson formalism applies to the latter, but it is really used to deduce the equivalent assertion that $P\times K\looparrowright G$ has the equivariant Whitney--Tate property.
\end{remark}

\subsection{Flag varieties}

Assume the situation of Example~\ref{ex:parabolic}, i.e., $G$ is a connected reductive group over an algebraically closed field $k$, $B\subseteq G$ a Borel subgroup, and $P,Q\supseteq B$ are parabolic subgroups containing $B$. Consider the variety with action $((P\times Q)\looparrowright G)$ where $P$ acts on $G$ by left multiplication and $Q$ acts on $G$ by right multiplication. The relevant geometric assertions are the following:

\begin{Bemerkung}
\label{pgqorbits}
Recall the combinatorial classification of $P$-orbits on partial flag varieties from Section~\ref{sec:korbit}. From the discussion there, we find that all $B$-orbits on $G/Q$ are of the types $\mathbf{G}$ or $\mathbf{U}$ with respect to the simple reflections, in the terminology of \ref{ss:p1mod}. By Example~\ref{ex:exgbhecke} the flag variety $G/B$ is Hecke-simply-connected in the sense of \ref{ref:heckeconn}. 
\end{Bemerkung}

We begin by listing two general situations in which cuspidals are clean. 

\begin{proposition}
\label{prop:simplyconnectedclean}
Let $G$ be a connected split reductive group and let $G\looparrowright X$ be a variety with action which is Hecke-simply-connected. Then all cuspidals are clean.\footnote{A similar criterion is also given in Lemma \ref{lemma:easypurity}.}
\end{proposition}

\begin{proof}
Let $V$ be a $B$-orbit. Assume $V$ admits a cuspidal motive. By assumption, the isotropy group $B_{v}$ is connected for each point $v$ in $V$. Using the induction equivalence $\DMT_B(V)\cong \DMT_{B_v}(\pt)$ of Definition~\ref{mtderdef1} and Proposition~\ref{prop:generatingtate}, we find that $\DMT_B(V)$ is generated (as a triangulated category) by the constant motive and Tate twists thereof. Thus, we may assume that the cuspidal in question is the constant motive. By Lemma~\ref{lemma:reducetop1}, it suffices to examine the cases of the Borel action on $\mathbb{P}^1$. In each of the orbit types \textbf{G}, \textbf{U}, \textbf{T}, \textbf{N}, we cannot have a constant cuspidal on the open orbit: by a localization argument, the constant motive on the orbit closure contributes to $j_\ast M$ and therefore the corresponding induction $\op{Ind}_B^{P_s}(j_\ast M)$ cannot be trivial, cf. \ref{RTZl}. We deduce that $V$ cannot be open in $P_s \cdot V$ for any simple reflection $s$. Expressed via the Richardson--Springer monoid structure of Definition~\ref{def:rsmonoid}, there is no pair consisting of a $B$-orbit $W\lneq V$ and a simple reflection $s$ such that $s\star W = V$. By definition of Hecke-connected, this means $V$ must be a closed orbit. Cuspidals on closed orbits are obviously clean.
%
\end{proof}

\begin{lemma}\label{lemma:easypurity}
Let $G$ be a connected reductive group and let $G\looparrowright X$ be a variety with action which is Hecke-connected. Assume that all orbits are of type \textbf{G} or \textbf{U} relative to each simple reflection. Then
\begin{enumerate}
\item all cuspidals are clean;
\item all Bott--Samelson motives are pointwise pure.
\end{enumerate}
\end{lemma}

\begin{proof}
As all orbits are of type \textbf{G} or \textbf{U}, cuspidals must correspond to closed orbits. For \textbf{G} this is clear since there is only one orbit, for \textbf{U} it follows since the open orbit is $\mathbb{A}^1$ and then a motivic local system on that necessarily has to be constant and therefore cannot be cuspidal. By a proof similar to the one of Proposition~\ref{prop:simplyconnectedclean}, cuspidals on closed orbits are obviously clean.

The proof of pointwise purity is similar to the proof of Proposition \ref{prop:bstate}. Details are left to the reader. 
\end{proof}

\begin{theorem}
\label{thm:parabolicWT}
Let $k$ be an algebraically closed field, let $G$ be a connected reductive group, and let $Q$ be a parabolic subgroup containing a Borel subgroup $B$. Let $\mathbb{D}$ be a derivator satisfying the conditions of Theorem~\ref{thm:wtcond}. 

Then for every parabolic $P$ containing $B$, the variety with action $P\looparrowright G/Q$ satisfies the equivariant Whitney--Tate condition of Definition~\ref{def:mtderdef2}, i.e., we have
\[
\DMT_{P\times Q}^\ast(G) =\DMT_{P\times Q}^!(G)= \DMT_{P\times Q}(G).
\]
\end{theorem}

\begin{proof}
By Example~\ref{ex:exgbhecke}, the variety $G\looparrowright G/Q$ is Hecke-simply-connected. We can either apply Lemma~\ref{lemma:easypurity}, based on the discussion of orbit structure in Section~\ref{sec:korbit} or apply Proposition~\ref{prop:simplyconnectedclean} to deduce that cuspidals are clean. Then we can apply Theorem~\ref{thm:wtcond} to get the second claim. The final claim follows then using the quotient equivalence of Proposition~\ref{prop:quotientequiv} and its compatibility with mixed Tate motives from Proposition~\ref{prop:quotientdmt}:
\[
\DMT_P^?(G/Q)\cong \DMT_{P\times Q}^?(G). \qedhere
\]
\end{proof}

Recall that Theorem~\ref{thm:parabolicWT} implies in particular that we have a subcategory $\DMT_P^{\op{bs}}(G/Q)\subset \mathbb{D}^+_P(G/Q)$ of Bott--Samelson motives, as defined in \ref{BSmot}. We will denote by  
\[
\DMT_{P\times Q}^{\op{bs}}(G)\subset \mathbb D^+_{P\times Q}(G)
\]
the subcategory corresponding to $\DMT_P^{\op{bs}}(G/Q)$  under the generalized quotient equivalence $\mathbb{D}_P^+(G/Q)\cong \mathbb{D}^+_{P\times Q}(G)$ of Proposition~\ref{prop:quotientequiv}. 


\begin{corollary}
\label{tiFL}
In the situation of Theorem~\ref{thm:parabolicWT}, Bott--Samelson motives are pointwise pure. There is a weight structure on the category $\DMT_{P\times Q}(G)$ whose heart is exactly given by the subcategory of Bott--Samelson motives:
\[
\DMT_{P\times Q}^{\op{bs}}(G)\cong\DMT_{P\times Q}(G)_{\op{wt}=0}.
\]
Consequently, in the situation of Theorem~\ref{thm:parabolicWT}, the tilting functor induces an equivalence 
\[
\op{tilt}: \op{Hot}^{\op{b}}(\DMT_{P\times Q}(G)_{\op{wt}=0}) = \op{Hot}^{\op{b}}(\DMT^{\op{bs}}_{P\times Q}(G))  \stackrel{\simeq}{\longrightarrow} \DMT_{P\times Q}(G).
\]
\end{corollary}         

\begin{proof}
In our situation,  pointwise purity of Bott--Samelson motives follows from Lemma~\ref{lemma:easypurity}. The claim concerning the weight structure then follows from Corollary~\ref{puBS}. The claim concerning tilting then follows from Theorem~\ref{thm:parabolicWT} and Theorem~\ref{thm:tilting}. 
\end{proof}

\begin{Bemerkung}
This is a motivic version of the formality result for flag varieties of Olaf Schn\"urer, cf. \cite[Theorem 1]{SchTH}. The tilting result of Theorem~\ref{tiFL} implies that we can write mixed Tate motives as a category of complexes whose entries are Bott--Samelson motives. The fact that the derived automorphisms of intersection motives for orbits are formal follows by weight arguments using the weight structure of Theorem~\ref{tiFL}. Then we could replace Bott--Samelson motives by modules over the Ext-algebra of intersection motives. Theorem~\ref{thm:gradedparabolic} below will then imply the formality of the equivariant derived categories in the Hodge and $\ell$-adic setting and thus recover Schn\"urer's results. 
\end{Bemerkung}

We now show that Bott--Samelson motives push forward to mixed Tate motives on the point. This is a version of the statement that Bott--Samelson resolutions of parabolic orbit closures have resolutions whose motive is mixed Tate. 

\begin{proposition}
\label{prop:bsmt}
Let $k$ be a field, let $G$ be a connected reductive algebraic group over $k$. Let $B\subset G$ be a Borel subgroup, and let $P,Q\subset G$ be two parabolic subgroups containing $B$. Then the push-forward functor $\op{fin}_\ast:\mathbb{D}^+_{P\times Q}(G)\to\mathbb{D}^+_{P\times Q}(\pt)$ maps Bott--Samelson motives to equivariant mixed Tate motives, and therefore induces a functor
\[
\op{fin}_\ast:\DMT_{P\times Q}^{\op{bs}}(G)\to\DMT_{P\times Q}(\pt).
\]
\end{proposition}

\begin{proof}
Let $j:Z\hookrightarrow X$ be a $B$-orbit and let $M\in\DMT_B(Z)_{\op{wt}=0}$ be a clean motive. It suffices to show $\op{fin}_{Z,\ast} M$ is mixed Tate, since $\op{fin}_Z=\op{fin}_X\circ j$. This follows, as in the proof of Proposition~\ref{Las}, from the fact that the motive of $Z$ is mixed Tate, cf. Proposition~\ref{prop:tatelocalsystem}. 

The push-forward $\op{fin}_\ast$ commutes with $M\mapsto M(n)[2n]$, direct sums and extensions, and $\DMT_P(\op{pt})$ is also closed under direct summands and extensions. The restriction functors $\op{Res}_P^Q$ commute with all the functors, cf. \ref{CompRES}, in particular with $\op{fin}_\ast$. 

To show the claim, it therefore suffices to discuss the relation between $\op{fin}_\ast$ and the induction functors. Recall from \ref{PBPF} that the induction functors are given by the composition 
\[
\mathbb{D}^+_P(X)\xrightarrow{\approx} \mathbb{D}^+_Q(Q\qtimes{P}X)\xrightarrow{(\op{id},m)_\ast} \mathbb{D}^+_Q(X)
\]
where the first equivalence is the induction equivalence and the second morphism is push-forward along the multiplication. By Proposition~\ref{cor:indequiv}, the induction equivalence is compatible with $\op{fin}_\ast$. By 2-functoriality, $\op{fin}_\ast$ commutes with the second functor $(\op{id},m)_\ast$ in the composition. Therefore, $\op{Ind}_P^Q$ commutes with $\op{fin}_\ast$.

By induction on the Weyl group, we get the claim. 
\end{proof}

\begin{remark}
An alternative proof would proceed along the following lines: the category $\DMT_{P}^{\op{bs}}(G/Q)$ is generated by pushforwards $\pi_\ast\const{\op{BS}(\underline{w})}$ of constant motives from Bott--Samelson resolutions of Schubert varieties in $G/Q$, where $\pi:\op{BS}(\underline{w})\to G/Q$ is the resolution of the Schubert variety associated to the reduced word $\underline{w}$. Under the quotient equivalence of Proposition~\ref{prop:quotientdmt}, we identify $\DMT_{P}(G/Q)\cong \DMT_{P\times Q}(G)$ and similarly 
\[
\DMT_P(\op{BS}(\underline{w}))\cong \DMT_{P\times Q}(\op{BS}(\underline{w})\times_{G/Q}G).
\]
Then the composition 
\[
\DMT_{P\times Q}(\op{BS}(\underline{w})\times_{G/Q}G)\xrightarrow{\pi_\ast}  \DMT_{P\times Q}(G)\xrightarrow{\op{fin}_\ast} \DMT_{P\times Q}(\pt)
\]
is the push-forward along $\op{fin}:\op{BS}(\underline{w})\times_{G/Q}G\to\pt$. The push-forward of the constant motive will have underlying motive $\op{M}(\op{BS}(\underline{w})\times_{G/Q}G)$. Now we use a version of the projective bundle formula \ref{prop:projectivebundle} together with the facts that the motives of $G$ and $\op{BS}(\underline{w})$ are mixed Tate to deduce that $\op{M}(\op{BS}(\underline{w})\times_{G/Q}G)$ is a mixed Tate motive.
\end{remark}

\subsection{Symmetric varieties}

In the following, we will discuss the Bott--Samelson formalism in the setting of parabolic actions on symmetric varieties, cf. Example~\ref{ex:symmetric}. 

\begin{Bemerkung}
\label{symmsit}
To fix the situation, let $G$ be a connected reductive group over an algebraically closed field $k$ of characteristic $\neq 2$. Let $\theta:G\to G$ be a non-trivial algebraic involution, $T$ be a $\theta$-stable maximal torus, and $B\supseteq T$ be a $\theta$-stable Borel subgroup. Denote by $K=G^\theta$ the fixed subgroup. We are interested in the varieties with action $P\times K\looparrowright G$ where $P$ is a standard parabolic of $(G,B)$.
\end{Bemerkung}

\begin{Bemerkung}
In this situation, the usual requirements of having finitely many orbits separably defined over $k$ are satisfied. Recall from Section~\ref{sec:korbit} that there are finitely many $B\times K$-orbits on $G$ and there is a natural partial order on $B\backslash G/K$ given by inclusion of orbit closures. It is also known that the groups of components of the isotropy groups of points have exponent 2, i.e., every element is its own inverse. By a slight extension of \cite[6.3]{Mars-Springer}, given a parabolic subgroup $P\subset G$, we have that for all $g\in G$ the map $P\times K\ra G$ given by $(p,k)\mapsto pgk$ is separable. 
\end{Bemerkung}

The major work is in proving that cuspidals are clean. A related statement concerning vanishing of intersection cohomology can be found in \cite[Lemma 7.4.1]{Mars-Springer}. The following is an adaptation of their argument to the current motivic setting.

First, a lemma dealing with the codimension 1.

\begin{lemma}
\label{lemma:cleanlemma}
Let $H$ be an algebraic group acting on $\PP^1$ with finitely many orbits. Assume that we are in case \textbf{T} or \textbf{N} of Section~\ref{ss:p1mod}. Let $j\colon U\hookrightarrow \PP^1$ be the inclusion of the open orbit, and let $V\in \DMT_H(U)$. Then $\fin_*V = 0$ if and only if the canonical map $j_!V\to j_*V$ is an isomorphism.\footnote{For intuition consider the topological context - complex local systems on $U = \PP^1 - \{0,\infty\}$ with finite dimensional stalks and vanishing global cohomology must necessarily have non-trivial monodromy around $\{0,\infty\}$.}
\end{lemma}

\begin{proof}
Let $i\colon \PP^1 - U \hookrightarrow \PP^1$ be the inclusion of the closed complement. Then we have a distinguished triangle
\[ 
j_!V \to j_*V \to i_*i^*j_*V \to j_!V[1].
\]
This triangle is an instance of the localization sequence applied to $j_\ast V$, where we note that $j^! j_\ast V\cong V$ since $j$ is an open immersion. The first map, viewed as $j_!j^\ast j_\ast V\to j_\ast V$ is the canonical map that we are concerned with, corresponding to the identity map on $j^\ast j_\ast V$ under adjunction.  Consequently, the canonical map $j_!V \to j_*V$ is an isomorphism if and only if $i^*j_*V = 0$. 

Now we want to reduce to the case $\textbf{T}$. For that, we note that restriction commutes with all the six functors by \ref{CompRES} so that we get a commutative diagram
\[
\xymatrix{
\mathbb{D}^+_N(\mathbb{G}_{\op{m}}) \ar[r]^{i^\ast j_\ast} \ar[d]_{\op{Res}_N^T} & \mathbb{D}^+_N(\{0,\infty\}) \ar[d]^{\op{Res}_N^T} \\
\mathbb{D}^+_T(\mathbb{G}_{\op{m}}) \ar[r]_{i^\ast j_\ast} & \mathbb{D}^+_T(\{0,\infty\}).
}
\]
Now recall from Remark~\ref{forget} that for a morphism of algebraic groups $G_1 \to G_2$ and a $G_2$-variety $X$, the restriction functor $\mathbb{D}^+_{G_2}(X) \to \mathbb{D}^+_{G_1}(X)$ is conservative. 
Therefore, $i^\ast j_\ast V=0$ (where we use the six functors for $N$-equivariant motives) if and only if $i^\ast j_\ast V=0$ (where we use the functors for $T$-equivariant motives). The same is true for $\op{fin}_\ast V$, by an analogous argument. Consequently, by restricting to the identity component, we may assume $H$ is connected and that we are in case \textbf{T}. 

Suppose we can find a one-parameter subgroup $\mathbb{G}_{\op{m}} \to H$ such that the composition $\mathbb{G}_{\op{m}} \to H\to \op{PGL}_2$ is non-trivial. By the argument above, we can check vanishing of the relevant motives after restriction to this one-parameter subgroup. Having done that, we can apply Springer's Homotopy Lemma \ref{lemma:springer} to $\mathbb{P}^1\setminus\{0\}$ resp. $\mathbb{P}^1\setminus\{\infty\}$ with the natural $\mathbb{G}_{\op{m}}$-contractions to see that $\fin_*V = 0$ if and only if $i^*j_*V=0$.

To complete the proof, we simply need to produce such a one-parameter subgroup. As the image of $H\to \op{PGL}_2$ is a maximal torus, there must be a semisimple element in $H$ outside the kernel of $H\to \op{PGL}_2$. Consequently, $H$ contains a maximal torus $T\subset H$ which is not completely contained in this kernel.
Hence, the induced map $T \to \op{PGL}_2$ yields a non-trivial character of $T$. Any one-parameter subgroup of $T$ not contained in the kernel of
this character gets the job done.
\end{proof}

\begin{proposition}
\label{prop:symmtilt}
Let $\mathbb{D}$ be a derivator satisfying the conditions of Theorem~\ref{thm:wtcond}, and assume the situation of \ref{symmsit}. Then for the variety with action $G\looparrowright G/K$, all cuspidals are clean. 
\end{proposition}

\begin{proof}
Let $j\colon Z\hookrightarrow G/K$ be a $B$-orbit, and let $M\in \DMT_B(Z)_{{\op{wt}}=0}$ be a cuspidal motive, i.e., $\op{Ind}_B^{P_s}(j_\ast M)=0$ for each simple reflection such that $Z$ is a proper open subvariety of $P_s\cdot Z$. We need to show that the canonical map $j_! M\to j_\ast M$ is an isomorphism. By the localization sequence, it suffices to show that for each $B$-orbit $i:Y\hookrightarrow G/K$ the restriction  $i^\ast j_! M\to i^\ast j_\ast M$ of the canonical map is an isomorphism. Note that we can actually assume $Y\subset \overline{Z}\setminus Z$, where $\overline{Z}$ denotes the closure of $Z$ in $G/K$. Since $i^\ast j_!M=0$ by base change, it suffices to show $i^\ast j_\ast M=0$ for each $B$-orbit $Y\subset \overline{Z}\setminus Z$.

Assume that $s$ is a simple reflection such that $Z$ is a proper open subvariety of $P_s\cdot Z$. In this case, $Z$ is the unique open $B$-orbit in $P_s\cdot Z$, which was denoted by $s\star Z=Z$ in \ref{ss:p1mod}; moreover, $P_s\cdot Z$ also contains an orbit other than $Z$. Using \ref{RTZl} and the assumption that $M$ is cuspidal, we know that $Z$ must be of type \textbf{T} or \textbf{N} relative to $s$, cf. the notation in \ref{ss:p1mod}. Now let $I$ be the set consisting of simple reflections $s$ such that $Z$ is of type \textbf{T} or \textbf{N} relative to $s$, and let $P_I\subseteq G$ be the standard parabolic corresponding to $I$. Then, according to \cite[\S7.2.1]{Mars-Springer}, $P_I\cdot Z = \overline{Z}$.
 
Now if $Y$ is an orbit in $\overline{Z}\setminus Z$ of codimension $1$, then  we may find a simple reflection $s\in I$ such that $s\star Y =Z$, i.e.,  $Z$ is the unique open $B$-orbit in ${P_sY}$. Otherwise $\overline{Y}$ would already be stable under $P_I$ which would contradict the previously deduced $P_I\cdot Z=\overline{Z}$. Using the commutative diagram of Lemma~\ref{lemma:reducetop1}, we are reduced to a question on $\mathbb{P}^1$ with $Z$ corresponding to the open orbit and $Y$ corresponding to a closed orbit. By Lemma~\ref{lemma:cleanlemma}, we can deduce $j_\ast M|_Y=0$ if we can show $\op{fin}_\ast M=0$. By \ref{RTZl}, $M$ must have non-trivial monodromy around $Y$, and by a variant of Lemma~\ref{lemma:gmlocalsystem} we find $\op{fin}_\ast M=0$.


Now assume that there are $B$-orbits $Y\subset \overline{Z}\setminus Z$ such that $j_*M|_Y\neq 0$. We choose one such $Y$ of minimal codimension and derive a contradiction. As before, there must exist a simple reflection $s\in I$ such that $Y$ is not open in $P_s\cdot Y$, otherwise $\overline{Y}$ would be stable under $P_I$. Since we are in the cases \textbf{T} or \textbf{N}, we know $P_s\cdot Y$ decomposes into an open $B$-orbit, the closed $B$-orbit $Y$ and maybe another closed $B$-orbit $A$, whose inclusions we denote $y$ and $a$.\footnote{Here open and closed are considered relatively in $P_s\cdot Y$, not absolutely in $G/K$.} By the assumed minimality of $Y$, $j_*M$ restricts to zero on the open orbit in $P_s\cdot Y$. From the localization sequence, we find 
\[
j_*M|_{P_sY}\cong y_*(j_*M|_{Y})\oplus  a_*(j_*M|_{A}).
\]
Since $M$ is cuspidal, we have $\ind^{P_s}_B(j_*M) =0$ by definition. From 
\[
0=\ind^{P_s}_B(j_*M)|_{P_s\cdot Y}\cong \ind^{P_s}_B(j_*M|_{P_s\cdot Y})
\]
we then deduce $\ind^{P_s}_B( y_*(j_*M|_{Y}))=0$. We now want to apply the local description of the induction $\op{Ind}_B^{P_s}$, cf. the lower left diagram of \ref{lemma:reducetop1}, to show $j_*M|_{Y}=0$. We repeat the corresponding diagram, where $H\subset P_s$ denotes the isotropy group of a point in $P_s\cdot Y$: 
\[ 
\xymatrix{
  \mathbb{D}^+_B(P_s \cdot Y) \ar[r]^-{\ind^{P_s}_B} & \mathbb{D}^+_{P_s}(P_s \cdot Y)  \\
  \mathbb{D}^+_H(\PP^1) \ar[r]^-{\fin_*}\ar[u]_-{\approx} & \mathbb{D}^+_H(\pt)\ar[u]_-{\approx} 
}
\]
Using this diagram, $y_\ast (j_\ast M|_Y)\in\mathbb{D}^+_B(P_s\cdot Y)$ corresponds to a motive $M'\in\mathbb{D}^+_H(\mathbb{P}^1)$ supported on one or two points. The vanishing of $\op{Ind}_B^{P_s}(y_\ast(j_\ast M|_Y))$ corresponds to $\op{fin}_{\mathbb{P}^1,\ast}M'=0$. Now, we get another commutative diagram from \ref{CompRES} whose vertical arrows are conservative functors by \ref{forget}:
\[
\xymatrix{
  \mathbb{D}^+_H(\PP^1) \ar[r]^-{\fin_*}\ar[d] & \mathbb{D}^+_H(\pt)\ar[d] \\
  \mathbb{D}^+(\mathbb{P}^1)\ar[r]_-{\fin_\ast} & \mathbb{D}^+(\pt).
}
\]
Now $\op{fin}_\ast M'\cong \op{fin}_\ast \op{Res}_H^1 M'=0$. For $p:\pt\to\mathbb{P}^1$ the inclusion of a point, the composition $\op{fin}_\ast\circ p_\ast:\mathbb{D}^+(\pt)\to\mathbb{D}^+(\pt)$ is isomorphic to the identity. Therefore, a motive supported on finitely many points has trivial image under $\op{fin}_{\mathbb{P}^1,\ast}$ if it is trivial to start with. Using conservativity of $\op{Res}_H^1$, we deduce from $\op{fin}_\ast M'=0$ that $M'=0$. Using the earlier local description of $\op{Ind}_B^{P_s}$, this means that $y_\ast(j_\ast M|_Y)=0$. Since $y$ is an immersion, $0=y^\ast y_\ast (j_\ast M|_Y)\cong j_\ast M|_Y$ and we arrive at a contradiction and the proof is done.
\end{proof}

\begin{Bemerkung}
We would like to point out, more clearly than in \cite{Mars-Springer}, the geometric reason why the Bott--Samelson formalism works in the case of symmetric varieties. The closure of the orbit under the parabolic generated by the simple reflections which give the directions where we have nontrivial monodromy equals the union of the $K$-orbits to which the extension of the cuspidal is trivial. This is also the reason why we are restricted to symmetric varieties and cannot prove that cuspidals are clean for arbitrary spherical varieties: we need the reference to \cite[\S 7.2.1]{Mars-Springer} which only holds for symmetric varieties.
\end{Bemerkung}

\begin{theorem}
\label{cor:symmetricWT} 
Let $\mathbb{D}$ be a derivator satisfying the conditions of Theorem~\ref{thm:wtcond}, and assume the situation of \ref{symmsit}. Then for every parabolic $P\subset G$ containing $B$, the variety with action $P\looparrowright G/K$ is equivariantly Whitney--Tate, i.e., we have 
\[
\DMT_{P\times K}^\ast(G) =\DMT_{P\times K}^!(G)= \DMT_{P\times K}(G).
\]
\end{theorem}

\begin{proof}
By Proposition~\ref{prop:gkhecke}, the variety with action is Hecke-connected, and by Proposition~\ref{prop:symmtilt} we have that cuspidals are clean. Therefore, we can apply Theorem~\ref{thm:wtcond} to deduce the claim. The formulation using $\DMT_{P\times K}(G)$ instead of $\DMT_P(G/K)$ uses again the quotient equivalence. 
\end{proof}

\begin{proposition}
\label{prop:symmpurity}
In the situation of Theorem~\ref{cor:symmetricWT}, Bott--Samelson motives in $\DMT_{P\times K}(G)$ are pointwise pure. There is a weight structure on the category $\DMT_{P\times K}(G)$ whose heart is exactly given by the subcategory of Bott--Samelson motives: 
\[
\DMT_{P\times K}^{\op{bs}}(G)= \DMT_{P\times K}(G)_{\op{wt}=0}.
\]
Consequently, in the situation of Theorem~\ref{cor:symmetricWT}, the tilting functor induces an equivalence 
\[
\op{tilt}: \op{Hot}^{\op{b}}(\DMT_{P\times K}(G)_{\op{wt}=0}) = \op{Hot}^{\op{b}}(\DMT^{\op{bs}}_{P\times K}(G))  \stackrel{\approx}{\longrightarrow} \DMT_{P\times K}(G).
\]
\end{proposition}

\begin{proof}
By \cite[\S6.4]{Mars-Springer}, each point of each $B$-orbit of $G/K$ admits a contracting slice. So Lemma \ref{lemma:contractingslices} yields pointwise purity of Bott--Samelson motives. The remaining statements then follow from Corollary~\ref{puBS}, Theorem~\ref{cor:symmetricWT} and Theorem~\ref{thm:tilting}. 
\end{proof}

\begin{remark}
As in the case of flag varietes, cf. Proposition~\ref{prop:bsmt}, the push-forward of a Bott--Samelson motive under $\mathbb{D}^+_{P\times K}(G)\to \mathbb{D}^+_{P\times K}(\pt)$ is a mixed Tate motive. However, because there are no good combinatorial models available in the symmetric case, we won't really need this assertion. 
\end{remark} 

\begin{remark}
Recall the resolutions of orbit closures from \ref{symmBSresolution}. These are proper morphisms $X\to G/K$ which factor through surjections $\pi:X\twoheadrightarrow \overline{Y}$ for $Y$ a $B$-orbit. In the multiplicity one case, these are resolutions of singularities. Motives of the form $\pi_\ast\const{X}$ are particular examples of Bott--Samelson motives. 
\end{remark}

\subsection{Wonderful compactifications} 

As our last example, we will now discuss the Bott--Samelson formalism in the situation of wonderful compactifications, cf.~Example~\ref{ex:wonderful}. 

\begin{Bemerkung}
Assume the situation of Example~\ref{ex:wonderful}, i.e.,  $G$ is a connected semi-simple group of adjoint type, $B\subseteq G$ is a Borel subgroup and $X$ is the wonderful compactification of $G$. Recall from \cite[Lemma 1.3]{SpCompact} that $X$ is a $G\times G$-variety on which $B\times B$ acts with finitely many orbits. Information on the structure of orbits, Bruhat order etc. can be obtained from \cite{SpCompact}. Some of the results in  \cite{Mars-Springer} and \cite{richardson:springer:survey} apply to general spherical varieties and hence in particular to wonderful compactifications.
\end{Bemerkung}

\begin{theorem}
\label{thm:wonderfulWT}
Let $\mathbb{D}$ be a derivator satisfying the conditions of Theorem~\ref{thm:wtcond} and assume the situation of Example~\ref{ex:wonderful}. Then for every parabolic $P\times Q\subset G\times G$, the variety with action $P\times Q\looparrowright X$ is equivariantly Whitney--Tate. 
\end{theorem}

\begin{proof}
It follows from \cite[Lemma 2.1]{SpCompact} that $X$ is Hecke-connected. Further, \cite[Lemma 1.4]{SpCompact} implies that the assumptions of Lemma \ref{lemma:easypurity} are satisfied. This implies that all cuspidals are clean. Then we can apply Theorem~\ref{thm:wtcond} to get the claim. 
\end{proof}

\begin{proposition}
\label{prop:wonderfulpurity}
In the situation of Theorem~\ref{thm:wonderfulWT}, Bott--Samelson motives in $\DMT_{P\times Q}(X)$ are pointwise pure. There is a weight structure on the category $\DMT_{P\times Q}(X)$ whose heart is exactly given by the subcategory of Bott--Samelson motives:
\[
\DMT_{P\times Q}^{\op{bs}}(X)=\DMT_{P\times Q}(X)_{\op{wt}=0}.
\]
Consequently, in the situation of Theorem~\ref{thm:wonderfulWT}, the tilting functor induces an equivalence
\[
\op{tilt}: \op{Hot}^{\op{b}}(\DMT_{P\times Q}(X)_{\op{wt}=0}) = \op{Hot}^{\op{b}}(\DMT^{\op{bs}}_{P\times Q}(X))  \stackrel{\approx}{\longrightarrow} \DMT_{P\times Q}(X).
\]
\end{proposition}

\begin{proof}
By \cite[Proposition 1.6]{SpCompact}, points of the $B\times B$-orbits of $X$ admit contracting slices. So Lemma~\ref{lemma:contractingslices} yields pointwise purity of Bott--Samelson motives. The remaining statements then follow from Corollary~\ref{puBS}, Theorem~\ref{thm:wonderfulWT} and Theorem~\ref{thm:tilting}. 
\end{proof}

\section{Complements on convolution}
\label{sec:convolution2}

In this section, we revisit the convolution of equivariant (mixed Tate) motives discussed in Section~\ref{sec:convolution} and prove some compatibilities of convolution with six functors and the tilting statements. Essentially, the proofs of all the results below require a working formalism of Bott--Samelson motives.

\subsection{Compatibility with six functors}

For the compatibility of convolution with a tilting from motives to Soergel bimodules, we will need an exchange isomorphism between ordinary push-forward and exterior product. The exchange morphism of~\ref{extprodshriek} fails to be an isomorphism in general; this can already be seen on the level of sheaves. We shortly discuss a special case sufficient for our purposes.\footnote{This result will be used to compare two categorifications, via motives and via bimodules. Since we don't have combinatorial models available in the symmetric cases, we don't need to worry about the analogous statements in the case of symmetric varieties for now.} 

\begin{proposition}
\label{prop:extprodstar}
Assume $\mathbb{D}$ is one of the derivators $\mathbf{DA}^{\et}$ or $\op{MDer}(-;\mathbb{C})$. Let $G$ be a connected split reductive group and $P,Q,R\subseteq G$ be three standard parabolic subgroups. Let $M\in\mathbb{D}^{\op{bs}}_{P\times Q}(G)$ and $N\in\mathbb{D}^{\op{bs}}_{Q\times R}(G)$ be Bott--Samelson motives. Then the exchange morphism
\[
\op{fin}_\ast (M\boxtimes N)\stackrel{\cong}{\longrightarrow} (\op{fin}_\ast M)\boxtimes (\op{fin}_\ast N)
\]
discussed in~\ref{extprodshriek} is an isomorphism. 
\end{proposition}

\begin{proof}
By Proposition~\ref{prop:bsmt}, the push-forward maps Bott--Samelson motives on $G$ to equivariant mixed Tate motives on the point (for the appropriate group). To check that we have an isomorphism in $\DMT_{P\times R}(\pt)$ we use the assumption on the derivator. In the cases assumed, we have the $\ell$-adic and Hodge realization functors, respectively, and these are conservative on mixed Tate motives, cf. Proposition~\ref{prop:conservative}.   Therefore, it suffices to check that the exchange morphism induces an isomorphism on realizations. Since realization is compatible with everything appearing in the exchange morphism, i.e., the push-forward functors $\op{fin}_\ast$ and the exterior product $\boxtimes$, the exchange morphism for motives is mapped to the corresponding exchange morphism for the realizations. Then it suffices to prove that the exchange morphism in the realization is an isomorphism:
\[
(\op{fin}\times \op{fin})_\ast (\op{Real} M\boxtimes \op{Real} N)\stackrel{\cong}{\longrightarrow} \left(\op{fin}_\ast \circ \op{Real} M\right)\boxtimes \left(\op{fin}_\ast\circ \op{Real} N\right).
\]
Since $M$ and $N$ are Bott--Samelson motives, the realizations $\op{Real}(M)$ and $\op{Real}(N)$ will be compact objects in the respective equivariant derived categories they live in. In that case, the fact that the exchange morphism is an isomorphism is known for the equivariant derived categories, cf. \cite{KS} or \cite[Section 4.11]{soergel:6f}. 
\end{proof}

\begin{remark}
The compatibility above also follows from the compatibility of Verdier duality and exterior product, cf.~\ref{extverdier}, via reduction to~\ref{extprodshriek}: 
\begin{eqnarray*}
(f\times g)_\ast (M\boxtimes N) &\cong&D\circ (f\times g)_!\circ D (M\boxtimes N) \\ &\cong&
D\circ (f\times g)_!\left( DM\boxtimes DN\right) \\
&\cong& D\left(f_! DM \boxtimes g_! DN\right)\\
&\cong& D\circ f_! \circ D (M) \boxtimes D\circ g_!\circ D (N)\\
&\cong& f_\ast M\boxtimes  g_\ast N
\end{eqnarray*}
Here we need that $M$ and $N$ are strongly dualizable and constructible, that source and target of $f$ and $g$ are smooth. There are alternative arguments to establish exchange properties for $f_\ast$ and $\boxtimes$, but these would also use smooth base change or some fibered version of it, cf. \cite[Section 4.11]{soergel:6f}.
\end{remark}

\begin{remark}
In the end, all we want to state is that the equivariant cohomology functor is weakly monoidal. To do this, it would also suffice to first pass to realization and then use the known monoidality statement for the equivariant derived categories. 
\end{remark}

\begin{proposition}
\label{prop:boo3}
Assume $\mathbb{D}$ is one of the derivators $\mathbf{DA}^{\et}$ or $\op{MDer}(-;\mathbb{C})$, and let $G$ be a connected split reductive group. Then the composition of equivariant push-forward with the motivic equivariant cohomology functor $\mathbb{H}_G^{\op{mot}}$ from Theorem~\ref{thm:tiltpoint} is weaky monoidal on Bott--Samelson objects, in the sense that for three standard parabolics $P,Q,R\subset G$ and Bott--Samelson motives $M\in\DMT_{P\times Q}^{\op{bs}}(G)$ and $N\in\DMT_{Q\times R}^{\op{bs}}(G)$, there is a natural isomorphism in $\DMT_{P\times R}(\pt)$:
\[
\left(\mathbb{H}_G^{\op{mot}}\circ\op{fin}_\ast(M)\right)\otimes_{\mathcal{A}_Q} \left(\mathbb{H}_G^{\op{mot}}\circ\op{fin}_\ast(N)\right) \stackrel{\cong}{\longrightarrow} \mathbb{H}_G^{\op{mot}}\circ\op{fin}_\ast(M\conv_Q N). 
\]
\end{proposition}

\begin{proof}
To prove the claim, we will prove that both functors in the composition are monoidal. 

We first show that the push-forward to the point is monoidal. Note that the monoidal structure on the categories of motives over $G$ as well as over $\pt$ is given by convolution. Consider the following diagram:
\[
\xymatrix{
\mathbb{D}_{P \times Q}^{\op{c}} (G) \times \mathbb{D}_{Q\times R}^{\op{c}} (G) \ar[rr]^{\op{fin}_* \times \op{fin}_*} \ar[d]_-{\boxtimes} && \ar@{=>}[dll]_-{\sim} \mathbb{D}^+_{P\times Q} (\op{pt}) \times \mathbb{D}^+_{Q \times R} (\op{pt}) \ar[d]^-{\boxtimes}\\ 
\mathbb{D}^+_{P\times Q \times Q \times R} (G \times G) \ar[d]_-{\op{Res}} \ar[rr]^{\op{fin}_*} && \ar@{=>}[dll]_-{\sim} \mathbb{D}^+_{P \times Q \times Q \times R} (\op{pt}) \ar[d]^-{\op{Res}} \\
\mathbb{D}^+_{P\times \Delta Q \times R} (G \times G) \ar[d]^-{\wr\wr}_-{(p,q)_*} \ar[rr]^{\op{fin}_*} && \ar@{=>}[dll]_-{\sim} \mathbb{D}^+_{P \times \Delta  Q \times R} (\op{pt}) \ar[d]^-{\op{Ind}_*=(p,\op{id})_*}\\
\mathbb{D}^+_{P\times R} (G \times_{Q} G) \ar[d]_-{\op{mult}_!} \ar[rr]^{\op{fin}_*}  &&\ar@{=>}[dll]_-{\sim} \mathbb{D}^+_{P \times  R} (\op{pt})\ar@{=}[d] \\ 
\mathbb{D}^+_{P\times  R} (G) \ar[rr]^-{\op{fin}_*} && \mathbb{D}^+_{P \times R} (\op{pt})
}
\]
Here, the vertical compositions on both sides provide are the compositions defining the convolution bifunctors, and the horizontal arrows are the appropriate ordinary push-forward functors associated to $G\to\pt$ (for varying group actions). Furthermore, $\Delta Q\subset Q\times Q$ denotes the diagonal inclusion, $p: P \times \Delta  Q \times R\ra P \times  R$ the projection and $q:G \times G\ra  G \times_{/Q} G$ the quotient map. 

We discuss the commutativity of the various squares involved in the diagram. The commutativity of the first square is essentially the compatibility of ordinary push-forward with the exterior product, cf.~Proposition~\ref{prop:extprodstar}. The commutativity of the second square is the commutativity of restriction functors with the other six functors, cf.~\ref{CompRES}. The commutativity of the third square is the 2-functoriality for push-forward plus the fact that the ordinary push-forward is a quasi-inverse of the generalized quotient equivalence, cf.~Proposition~\ref{prop:quotientequiv}. Finally, the commutativity of the last square is simply the 2-functoriality of the equivariant ordinary push-forward functors, noting that the parabolicity of $Q$ means that the action $\op{mult}:G \times_{/Q} X\ra X$ is proper and therefore we have  $\op{mult}_!=\op{mult}_*$.  

The second part of the proof now is the compatibility of the motivic equivariant cohomology with the multiplicative structures. Again, let $G$ be a connected split reductive group and $P,Q,R\subset G$ be standard parabolic subgroups. Consider the following commutative diagram, where the left vertical consists of full triangulated subcategories of the right vertical in the previous diagram. Note that we write $\DMT_H(\op{pt})=\langle\mathcal T_{H}\rangle_\Delta$ to save space. 

\[
\xymatrix{
  \langle\mathcal T_{P\times Q}\rangle_\Delta \times \langle\mathcal T_{Q\times R}\rangle_\Delta\ar[d]_{\boxtimes}\ar[r]^-{\simeq} &\op{Der}^{\op{b}}(\mathcal A_P\otimes\mathcal A_Q\op{-Modfg^\DZ})\times \op{Der}^{\op{b}}(\mathcal A_Q\otimes\mathcal A_R\op{-Modfg^\DZ}) \ar[d]^\boxtimes \ar@{=>}[dl]_-{\sim}\\
  \langle\mathcal T_{P \times Q \times Q \times R}\rangle_\Delta \ar[d]_{\op{Res}} \ar[r]^-{\simeq} & \op{Der}^{\op{b}}(\mathcal A_P\otimes\mathcal A_Q\otimes\mathcal A_Q\otimes \mathcal A_R \op{-Modfg^\DZ}) \ar[d]^{\mathcal A_{\Delta Q}\otimes_{(\mathcal A_Q\otimes\mathcal A_Q)}^{\op{L}}(-)} \ar@{=>}[dl]_-{\sim}\\
  \langle\mathcal T_{P \times \Delta Q  \times R}\rangle_\Delta\ar[d]_{\op{Ind}_*}\ar[r]^-{\simeq} &\op{Der}^{\op{b}}(\mathcal A_P\otimes\mathcal A_{\Delta Q}\otimes \mathcal A_R \op{-Modfg^\DZ}) \ar[d]^{\op{res}} \ar@{=>}[dl]_-{\sim}\\
  \langle\mathcal T_{P \times  R}^\infty\rangle_\Delta \ar[r]_{\simeq}  &\op{Der}^{\op{b}}(\mathcal A_P\otimes\mathcal  A_R \op{-Modcg^\DZ})
}
\]
Here, all the horizontal functors are motivic equivariant cohomology functors, which are equivalences by Theorem~\ref{thm:tiltpoint}, Corollary~\ref{cor:tiltpoint} and the remark in~\ref{sts}. On the left vertical, we have the composition of functors defining convolution for equivariant motives on the point. On the right, we have the exterior tensor product, the (derived) extension of scalars $\mathcal A_{\Delta Q}\otimes_{(\mathcal A_Q\otimes\mathcal A_Q)}^{\op{L}}(-)$ and the restriction forgetting the $\mathcal A_{\Delta Q}$-part of the module structure using \ref{tiltingv:countable}. In particular, the composition on the right-hand side can be identified with the derived tensor product $(-)\otimes_{\mathcal A_Q}^{\op{L}}(-)$ over $\mathcal A_Q$. 

The commutativity of the first square, i.e., the compatibility of the equivariant cohomology with the exterior products, is a consequence of Proposition~\ref{prop:tilttensor} resp. Proposition~\ref{prop:exterior}. The commutativity of the second square is a consequence of Proposition~\ref{tilting}, and the commutativity of the third square is a consequence of Proposition~\ref{tiltingv:countable}. 

From the commutativity of the diagrams (resp. from the explicit natural iso-transformations), we get the claimed weak monoidality. We refrain from checking that our natural isomorphisms are compatible with the associators on both sides. 
\end{proof}

\begin{Bemerkung}  
In the situation above, the morphism $\op{mult}:G\times_{/Q} G\to G$ above is projective since it is the composition of the isomorphism $G\qtimes{Q}G \sira G/Q \times G$ followed by the projection $\op{pr}_2:G/Q \times G \to G$ which has projective fiber $G/Q$. In this case, convolution is additive on weights and its relation to Verdier duality is given by the formula
\[ 
\DD(M\conv_Q N) \cong (\DD M \conv_Q \DD N)(-\dim Q) [-2\dim Q].
\]
\end{Bemerkung}

\subsection{Compatibility with tilting}

The next statement checks compatibility of the tilting results derived from Theorem~\ref{thm:tilting} (especially Propositions \ref{tiFL}, \ref{prop:symmpurity} and \ref{prop:wonderfulpurity}) with the convolution functors. This will allow to provide the multiplicative statements about categorifications of the Schur algebroid resp. its modules in the symmetric and wonderful cases. 

\begin{proposition}
\label{prop:convtilt}
Let $G$ be a split connected reductive group, let $H\subseteq G$ be a subgroup. Denote $X=G/H$ and assume that $P,Q\subset G$ are parabolic subgroups such that the conditions of Theorem~\ref{thm:tilting} are satisfied for $P\looparrowright X$ and $Q\looparrowright X$. 

\begin{enumerate}
\item The convolution functors $-\star_Q-:\mathbb{D}^+_{P\times Q}(G)\times\mathbb{D}_{Q\times H}^+(G)\to \mathbb{D}^+_{P\times H}(G)$ are compatible with equivariant mixed Tate motives, i.e., they restrict to functors 
\[
-\star_Q-:\DMT_{P\times Q}(G)\times\DMT_{Q\times H}(G)\to \DMT_{P\times H}(G)
\]
\item
Denoting by $\mathcal{T}_{?}(G)$ the subcategory of Bott--Samelson motives in the category $\DMT_?(G)$ of equivariant mixed Tate motives, the convolution bifunctors are compatible with the tilting equivalence $\op{Hot}^{\op{b}}(\mathcal{T}_?)\stackrel{\simeq}{\longrightarrow} \langle\mathcal{T}_?\rangle_\Delta$ of Theorem~\ref{thm:tilting} in the following sense: for any two Bott--Samelson motives $M\in\DMT^{\op{bs}}_{P\times Q}(G)$ and $N\in\DMT^{\op{bs}}_{Q\times H}(G)$ there is a natural equivalence
in $\DMT_{P\times H}(G)$:
\[
\op{tilt}(M\conv_Q N)\stackrel{\cong}{\longrightarrow} \op{tilt}(M)\conv_Q \op{tilt}(N).
\]
\end{enumerate}
\end{proposition}

\begin{proof}
(1) We trace through the definition of convolution, cf. Definition~\ref{def:convolution}. First, if $M\in\DMT_{P\times Q}(G)$ and $N\in\DMT_{Q\times H}(G)$ are equivariant mixed Tate motives, then we have  
\[
M\boxtimes N=\op{pr}_1^\ast M\otimes \op{pr}_2^\ast N\in \DMT_{P\times Q\times Q\times H}^\ast(G\times G)
\]
because the restriction functors along the smooth projections $\op{pr}_i$ preserve $\ast$- and $!$-pointwise mixed Tate motives by Proposition~\ref{prop:pullbackdmt} and the monoidal structure $\otimes$ restricts to $\ast$-pointwise mixed Tate motives by Proposition~\ref{prop:monodmt}. The restriction functor $\op{diag}^\ast:\mathbb{D}^+_{P\times Q\times Q\times H}(G\times G)\to \mathbb{D}_{P\times Q\times H}^+(G\times G)$ preserves equivariant mixed Tate motives by Proposition~\ref{prop:dmtres}. The quotient equivalence $\mathbb{D}^+_{P\times Q\times H}(G\times G)\simeq \mathbb{D}^+_{P\times H}(G\times_{/Q}G)$ preserves equivariant mixed Tate motives by Proposition~\ref{prop:quotientdmt}. Finally, the compatibility with $\op{mult}_!:\mathbb{D}_{P\times H}^+(G\times_{/Q} G)\to\mathbb{D}^+_{P\times H}(G)$ follows as in the proof of Proposition~\ref{Las} (which established a compatibility of ordinary pushforward along a homogeneous space bundles for $!$-equivariant mixed Tate motives), combined with the fact that Verdier duality exchanges $\ast$- and $!$-pointwise motives, cf. Proposition~\ref{prop:dmtverdier}. As a consequence, we see that $\ast$-pointwise equivariant mixed Tate motives are preserved by convolution. By assumption, all the relevant categories $\DMT_{P\times Q}(G)$, $\DMT_{Q\times H}(G)$ and $\DMT_{P\times H}(G)$ satisfy the equivariant Whitney--Tate condition, hence compatibility of convolution with $\ast$-pointwise mixed Tate motives implies that the convolution  functor restricts to the categories of equivariant mixed Tate motives as claimed. 

(2) The claim will follow from the natural iso-transformations filling the following diagram
\[
\xymatrix{
\op{Hot}^{\op{b}}(\mathcal{T}_{P\times Q}(G)) \times \op{Hot}^{\op{b}}(\mathcal{T}_{Q\times R}(G)) \ar[rr]^-{\op{tilt}\times\op{tilt}} \ar[d]_{\boxtimes} && \ar@{=>}[dll]_-{\sim} \langle\mathcal{T}_{P\times Q}(G)\rangle_\Delta\times \langle\mathcal{T}_{Q\times R}(G)\rangle_\Delta \ar[d]^-{\boxtimes}\\
\op{Hot}^{\op{b}}(\mathcal{T}_{P\times Q\times Q\times R}(G\times G)) \ar[d]_-{\op{Res}}  \ar[rr]^-{\op{tilt}} &&
\ar@{=>}[dll]_-{\sim}\langle\mathcal{T}_{P\times Q \times Q \times R}(G\times G)\rangle_\Delta\ar[d]^-{\op{Res}}\\
\op{Hot}^{\op{b}}(\mathcal{T}_{P\times Q\times R}(G\times G)) \ar[d]_-{\op{quot equiv}} \ar[rr]^-{\op{tilt}} && \ar@{=>}[dll]_-{\sim} \langle\mathcal{T}_{P\times Q\times R}(G\times G)\rangle_\Delta \ar[d]^-{\op{quot equiv}}\\
\op{Hot}^{\op{b}}(\mathcal{T}_{P\times R}(G\times_QG)) \ar[d]_-{\op{mult}_!} \ar[rr]^-{\op{tilt}} && \ar@{=>}[dll]_-{\sim} \langle\mathcal{T}_{P\times R}(G\times_Q G)\rangle_\Delta \ar[d]^-{\op{mult}_!} \\
\op{Hot}^{\op{b}}(\mathcal{T}_{P\times R}(G)) \ar[rr]^-{\op{tilt}} &&  \langle\mathcal{T}_{P\times R}(G)\rangle_\Delta
}
\]

The vertical functors are all left adjoints or given by the monoidal structure. In particular, we can apply Theorem~\ref{thm:funtilt} and Theorem~\ref{thm:tiltmonoid} to obtain the relevant isotransformations. For the first square, $M\boxtimes N=\op{pr}_1^\ast M\otimes\op{pr}_2^\ast N$, i.e., it is given as combination of the left adjoint functor $\op{pr}_i^\ast$ and the tensor product. The restriction functor in the second square and the proper pushforward $\op{mult}_!$ in the last square are also left adjoints and consequently compatible with the tilting by Theorem~\ref{thm:funtilt}. Finally, the quotient equivalence used is the inverse of a functor induced from the left-adjoint restriction, cf. Proposition~\ref{prop:quotientequiv}, and therefore compatible with the tilting by Theorem~\ref{thm:funtilt}. Consequently, we get the claimed compatibility of tilting with convolution.
\end{proof}

\begin{Bemerkung}
We refrain from checking that our natural isomorphisms are compatible with the associators on both sides. 
\end{Bemerkung}

\begin{Bemerkung}
As a consequence, the above statement that convolution preserves equivariant mixed Tate motives implies, via Propositions~\ref{prop:convres} and \ref{prop:convind}, that (in the situation for symmetric varieties) the $K$-equivariant motives on flag varieties are stable under pushforward and pullback for projection maps $G/P\to G/Q$ for parabolic subgroups $P\subset Q$. This could also be established more directly. The non-trivial assertion is the compatibility with push-forward, which can be reduced to the case of $\mathbb{P}^1$-fibrations. In these cases, the specific knowledge about the orbit structure, cf. \ref{ss:p1mod}, can be used to show that the pushforward preserves equivariant mixed Tate motives. 
\end{Bemerkung}

\section{Categorification of the Schur algebroid and knot theory}
\label{sec:parabolic}

In this section, we will now discuss the categorification of the Schur algebroid. There are two possibilities which we will both develop, one via the realization functors to the equivariant derived categories and one via hypercohomology to the combinatorial models given by Soergel bimodules. Both categorifications of the Schur algebroid are compatible. We also discuss the application of the bimodule categorification to Khovanov's link homology.

Throughout the section, we fix a field $k$ and a connected reductive group $G$ over $k$ with two standard parabolic subgroups $P,Q\subset G$. We assume that $\mathbb{D}$ is a derivator satisfying the conditions~\ref{derivator:new} as well as the grading condition~\ref{conditions:grading} and the weight condition~\ref{conditions:weight}. 

\subsection{Setup and notation}
\label{setup}

Before we come to the proofs of the main categorification results, we need to introduce notation for some functors that will appear throughout. Basically, the whole section revolves around the following commutative diagrams relating equivariant motives, equivariant derived categories and suitable categories of bimodules which will be explained in the subsequent paragraphs. These diagrams compare the categories of equivariant mixed Tate motives to previously considered categorifications, using perverse sheaves or singular Soergel bimodules.

\begin{Bemerkung}
The comparison of our motivic categorification of the universal Schur algebroid with the previous categorifications in terms of (complexes of) Soergel bimodules is based on the following commutative diagram: 
\[
\xymatrix{
\DMT_{P\times Q}(G)_{\op{wt}=0} \ar[d]_{\approx} \ar[r] & \op{Hot}^{\op{b}}(\DMT_{P\times Q}(G)_{\op{wt}=0}) \ar[d]_{\approx} \ar[r]_>>>>>>>{\approx}^>>>>>>{\op{tilt}}& \DMT_{P\times Q}(G) \ar[ld]^{\mathbb{T}} \\
\mathcal{A}_P\op{-SMod-}\mathcal{A}_Q \ar[r] & \op{Hot}^{\op{b}}(\mathcal{A}_P\op{-SMod-}\mathcal{A}_Q) 
}
\]

In the top row, we have the categories of equivariant mixed Tate motives, the gadgets we spent the last hundred pages to define, cf. in particular Theorem~\ref{thm:parabolicWT}. On the right-hand side we have the category of $P\times Q$-equivariant motives on $G$, on the left-hand side the heart of the weight structure given by equivariant motives which are pure of weight zero. In the bottom row, we have the category of singular Soergel bimodules and its homotopy category. These provide well-known categorifications of the Schur algebroid, with the slight drawback that they are not ``geometric'' enough, which complicates some computations. The left-hand vertical arrow is essentially given by equivariant motivic cohomology, sending the constant motive $_P\const{G}_Q$ to the bimodule  $\mathcal{A}_P\otimes_{\mathcal{A}_G}\mathcal{A}_Q$, and the middle vertical arrow is given by termwise extension to complexes; both vertical arrows are equivalences of categories. 

The right-hand triangle commutes by construction: the functor $\mathbb{T}$ is defined as composition of an inverse of the tilting, mapping an equivariant mixed Tate motive $M\in \DMT_{P\times Q}(G)$ to its weight complex, and then taking cohomology degree-wise, i.e., replacing the pure weight $0$ motives in the complex by the appropriate $\mathcal{A}_P$-$\mathcal{A}_Q$-bimodule. Moreover, the $P\times Q$-equivariant cohomology of a motive $M$ can be recovered via $\mathbb{T}$ by 
\[
\mathbb{H}^\bullet_{P\times Q}(M)\cong \DMT^\bullet[_P\const{G}_Q,M]\cong \op{Hot}^{\op{b}}[\mathcal{A}_P\otimes_{\mathcal{A}_G}\mathcal{A}_Q,\mathbb{T}(M)].
\]

Finally, the categorification of the universal Schur algebroid can then be obtained in two ways: either taking the split Grothendieck group for the left-hand categories $\DMT_{P\times Q}(G)_{\op{wt}=0}$ and $\mathcal{A}_P\op{-SMod-}\mathcal{A}_Q$,  respectively, or taking the ordinary Grothendieck group for the categories in the middle  $\DMT_{P\times Q}(G)$ and $\op{Hot}^{\op{b}}(\mathcal{A}_P\op{-SMod-}\mathcal{A}_Q)$, respectively.
\end{Bemerkung}

\begin{Bemerkung}
On the other hand, the categorification of the Schur algebroid via equivariant derived categories also has a motivic version. This is best explained using the following commutative diagram
\[
\xymatrix{
\DMT_{P\times Q}(G)_{\op{wt}=0} \ar[d] \ar[r]^\approx & \op{Der}^{\op{b},\op{ss},\op{ev}}_{P\times Q}(G) \ar[d] \\
\DMT_{P\times Q}(G) \ar[r]_{\op{Real}} & \op{Der}^{\op{b}}_{P\times Q}(G)
}
\]
On the left-hand side, we again have our categories of equivariant motives, the left-hand vertical arrow being the inclusion of the heart. Alternatively, via the identification $\DMT_{P\times Q}(G)\cong\op{Hot}^{\op{b}}(\DMT_{P\times Q}(G)_{\op{wt}=0})$, it is the inclusion as complexes concentrated in degree $0$. On the right-hand side, we have the equivariant derived category and its subcategory of perversely semisimple complexes concentrated in even degrees. The two horizontal arrows are both induced by the appropriate realization functor $\op{Real}:\DMT_{P\times Q}(G)\to\op{Der}^{\op{b}}_{P\times Q}(G;\Lambda)$ from equivariant mixed Tate motives to equivariant sheaves -- depending on the situation, this is the $\ell$-adic realization or the Hodge realization. Note that the realization functor induces an equivalence from the Bott--Samelson motives onto the perversely semisimple complexes. 

If we are working over a finite fields, with $\mathbb{Q}_\ell$-coefficients, then Grothendieck's function--sheaf correspondence provides a categorification of the universal Schur algebroid via the categories $\op{Der}^{\op{b},\op{ss},\op{ev}}_{P\times Q}(G)$. Combined with the above diagram, we get a categorification of the universal Schur algebroid, as split Grothendieck group of one of the categories in the top row. The ordinary Grothendieck group of the lower left corner is also identified with the appropriate part of the Schur algebroid, while the lower right corner only retains the information at $q=1$ because it doesn't detect the underlying ``mixed geometry''. 
\end{Bemerkung}

\begin{Bemerkung}
The compatibility between these two categorifications is now expressed in the following commutative diagram:
\[
\xymatrix{
\DMT_{P\times Q}(G)_{\op{wt}=0} \ar[rr] \ar[rd] && \op{Der}^{\op{b},\op{ss},\op{ev}}_{P\times Q}(G) \ar[ld] \\
&\mathcal{A}_P\op{-SMod-}\mathcal{A}_Q 
}
\]
The horizontal arrow is induced from the appropriate realization functor, and the right-hand arrow takes a perversely semisimple complex to its equivariant cohomology, viewed as $\mathcal{A}_P$-$\mathcal{A}_Q$-bimodule. Similarly, the left-hand arrow is basically given as the appropriate equivariant motivic cohomology, making the diagram commute.
\end{Bemerkung}

\begin{Bemerkung}
With all the notation and the diagram set up, what we really want to prove in this section are the following two points: 
\begin{enumerate}
\item The tilting, relating equivariant mixed Tate motives and complexes of singular Soergel bimodules is an equivalence of tensor-triangulated categories. Under this equivalence, the Bott--Samelson motives correspond exactly to the category of singular Soergel bimodules, included as heart of the homotopy t-structure on the category of complexes. 
\item The realization functor $\DMT_{P\times Q}(G)\to \op{Der}^{\op{b}}_{P\times Q}(G)$ is a degrading functor in the sense of \cite{BGSo}. 
\end{enumerate}
The above commutative diagrams then also provide an explicit compatibility relation between these two categorifications. 
\end{Bemerkung}

\subsection{Bott--Samelson motives and Soergel bimodules}

The following provides a reinterpretation of the construction of Bott--Samelson motives, using convolution rather than induction and restriction. Proofs concerning the category of Bott--Samelson motives can then be reduced to generating motives using the convolution formalism. 

\begin{Bemerkung}[{\bf Bott--Samelson motives via convolution}] 
\label{BSoo}  

Using the description of restriction and induction functors in terms of convolution, cf.~\ref{irc}, we get an alternative recursive construction of the collection of Bott--Samelson motives in Definition~\ref{BSmot} as follows. The collection of the categories $(\op{DMT}_{P\times Q}^{\op{bs}}(G))_{P,Q\supset B}$ can be described as the smallest collection of subcategories of the various $\mathbb{D}^+_{P\times Q}(G)$ (with $P,Q$ running through the standard parabolics) such that: 
\begin{enumerate}
\item For an inclusion of standard parabolics $Q\subset P$, we have the  morphisms $j_1:((P\times Q)\looparrowright P)\to((P\times Q)\looparrowright G)$ and $j_2:((Q\times P)\looparrowright P)\to((Q\times P)\looparrowright G)$ given by inclusion of the subgroup $P\subset G$. Denoting by $\leftidx{_P}{\underline{P}}{_Q}\in\mathbb{D}^+_{P\times Q}(P)$ and 
$\leftidx{_Q}{\underline{P}}{_P}\in\mathbb{D}^+_{Q\times P}(P)$ the constant equivariant motives, we require
\[
j_{1\ast}\left(\leftidx{_P}{\underline{P}}{_Q}\right)\in\DMT^{\op{bs}}_{P\times Q}(G) \quad\textrm{ and }\quad
j_{2\ast}\left(\leftidx{_Q}{\underline{P}}{_P}\right)\in\DMT^{\op{bs}}_{Q\times P}(G).
\]
\item the categories $\DMT_{P\times Q}^{\op{bs}}(G)$ are stable under $M\mapsto M(n)[2n]$;
\item the categories $\DMT_{P\times Q}^{\op{bs}}(G)$ are stable under taking direct summands;
\item the categories $\DMT_{P\times Q}^{\op{bs}}(G)$ are extension stable; and
\item the collection of categories $\left(\DMT_{P\times Q}^{\op{bs}}(G)\right)_{P,Q\supset B}$ is stable under the convolution bifunctors
\[
-\conv_Q-:\mathbb{D}^+_{P\times Q}(G)\times \mathbb{D}^+_{Q\times R}(G)\ra \mathbb{D}^+_{P\times R}(G)
\]
\end{enumerate}
The equivalence of the above definition with Definition~\ref{BSmot} will then follow from the discussion in the two points below:
\begin{itemize}
\item We discuss the differences in point (1) of the definition. 

Let $j:Z\to G/Q$ be a $P$-orbit and let $M$ be a clean motive on $Z$, i.e., a pure $P$-equivariant Tate motive such that $j_!M\to j_\ast M$ is an isomorphism. Actually, using the constant mixed Tate motives in point (1) above we can build all $P$-equivariant pure Tate motives on the orbit, showing that the new definition contains the clean motives from Definition~\ref{BSmot}.

Conversely,  the motives in point (1) above are Bott--Samelson motives: the cuspidal on the point is clean, and the constant motives on the other orbits are obtained by induction and restriction, using \ref{irc}. 
\item
We discuss the differences in point (5) of the definition. By \ref{irc}, convolution with $\leftidx{_Q}{\underline{P}}{_P}$ and $\leftidx{_P}{\underline{P}}{_Q}$ provides restriction and induction. The stability of the motives in Definition~\ref{BSmot} under convolution follows from an inductive proof, using Propositions~\ref{prop:convres} and \ref{prop:convind} to express the convolution in terms of induction  and restriction. 
\end{itemize}
\end{Bemerkung}

\begin{proposition}
\label{prop:boo4}
Let $G$ be a connected reductive group. Fix a Borel subgroup $B\subset G$ and let $Q\subset P$ be an inclusion of standard parabolic subgroups of $G$. 
\begin{enumerate}
\item The realization functor $\op{Real}:\DMT_{P\times Q}(G)\to\op{Der}^{\op{b}}_{P\times Q}(G)$ maps the Bott--Samelson motive $j_\ast \leftidx{_P}{\underline{P}}{_Q}$ to the extension of the corresponding constant sheaf $\leftidx{_P}{\underline{P}}{_Q}$. A similar statement holds for $\leftidx{_Q}{\underline{P}}{_P}$. 
\item 
We have the following quasi-isomorphisms of bimodules
\[
\mathbb{H}(\leftidx{_P}{\underline{P}}{_Q})\simeq \mathcal{A}_Q, \quad 
\mathbb{H}(\leftidx{_Q}{\underline{P}}{_P})\simeq \mathcal{A}_Q
\]
in $\mathcal{A}_P{\op{-Modfg_\Lambda^{\mathbb{Z}}-}}\mathcal{A}_Q$ and $\mathcal{A}_Q{\op{-Modfg_\Lambda^{\mathbb{Z}}-}}\mathcal{A}_P$, respectively. Here we denote by $\mathbb{H}$ the appropriate equivariant cohomology functors, and $\mathcal{A}_Q$ is equipped with the appropriate natural bimodule structure. 
\end{enumerate}
\end{proposition}

\begin{proof}
(1) Recall that $j_\ast \left(\leftidx{_P}{\underline{P}}{_Q}\right)$ takes a constant motive on a $P$-orbit of $G/Q$ and extends it to all of $G/Q$. The realization functors are compatible with the six-functor formalism, so the claim reduces to the fact that realization maps the constant mixed Tate motive $\underline{X}$ to the corresponding constant sheaf $\underline{X}$. The claim is then clear.

(2) 
follows from (1) and the computation of equivariant cohomology of the push-forward of the constant sheaf.
\end{proof}

\begin{remark}
In Example~\ref{ex:BSflag}, we explained how Bott--Samelson motives can be seen as push-forwards of constant motives from Bott--Samelson resolutions. The compatibility of realization and the six-functor formalism implies that the Bott--Samelson motive $M=p_\ast \underline{X}$ for the Bott--Samelson resolution $p:X\to G/Q$ realizes to the pushforward $p_\ast\underline{X}$ of the constant sheaf from the very same resolution. 
\end{remark}

\begin{remark}
Here is another way to see (2) in Proposition~\ref{prop:boo4}. We can consider the push-forward $\op{fin}_\ast \left(\leftidx{_P}{\underline{P}}{_Q}\right)$ along the structure morphism $((P\times Q)\looparrowright P)\to ((P\times Q)\looparrowright \pt)$. By Proposition~\ref{prop:bsmt}, this is a $(P\times Q)$-equivariant mixed Tate motive on the point. By \ref{PBPF}, we have an isomorphism
\[
\op{fin}_\ast (\leftidx{_P}{\underline{P}}{_Q})\cong \op{Ind}_Q^{P\times Q}\left(\underline{\pt}\right)
\]
in $\mathbb D^+_{P\times Q}(\op{pt})$, where the induction is along the diagonal inclusion $\Delta:Q\hookrightarrow P\times Q$. The cohomology functor in Theorem~\ref{thm:tiltpoint} maps the constant motive $\underline{\pt}\in\mathbb{D}^+_Q(\pt)$ to $\mathcal{A}_Q\in\mathcal{A}_Q\op{-Modfg_\Lambda^{\mathbb{Z}}}$, with the natural module structure. The compatibility of tilting and integration, cf.~\ref{tiltingv} Item~\ref{TInd}, implies that $\op{Ind}_Q^{P\times Q}\left(\underline{\pt}\right)$ is mapped to $\mathcal{A}_Q\in \mathcal{A}_P{\op{-Modfg_\Lambda^{\mathbb{Z}}-}}\mathcal{A}_Q$. In particular $\op{fin}_\ast( \leftidx{_P}{\underline{P}}{_Q})$ belongs to the heart of the tilting t-structure from \ref{sts}. 
\end{remark}

\begin{Bemerkung}[{\bf Tilting Bruhat cells}]
\label{TBC}
Here is another computation of equivariant cohomology which we will need later. Let $G\supset B\supset T$ be a connected reductive group with chosen Borel subgroup $B$ and maximal torus $T$. For an element $x\in \op{N}_G(T)$ (representing an element of the Weyl group $\op{N}_G(T)/T$), we can consider the corresponding Bruhat cell $((B\times B)\looparrowright BxB)$ in $((B\times B)\looparrowright G)$. Then the equivariant cohomology of the constant $B\times B$-equivariant Tate motive $\underline{BxB}$ can be computed as 
\[
\mathbb{H}(\underline{BxB})\cong \mathcal{A}_B1_x\in \mathcal{A}_B\op{-Modfg^\Lambda_\mathbb{Z}-}\mathcal{A}_B,
\]
where $1_x$ satisfies $a\cdot 1_x=1_x\cdot x(a)$ for the natural action of $x$ on $\mathcal{A}_B$ via the isomorphism $\mathcal{A}_T\cong\mathcal{A}_B$ induced from the inclusion $T\subset B$. The argument is the same as in Proposition~\ref{prop:boo4}. 
\end{Bemerkung}

\begin{corollary}
\label{cor:boo4}
The equivariant motivic cohomology functor induces a functor 
\[
\mathbb{H}:\DMT^{\op{bs}}_{P\times Q}(G)\to \mathcal{A}_P{\op{-Modfg_\Lambda^{\mathbb{Z}}-}}\mathcal{A}_Q
\]
whose essential image is the category $\mathcal{A}_P\op{-SMod-}\mathcal{A}_Q$ of even singular Soergel bimodules.
\end{corollary}

\begin{proof}
We first prove the factorization. Note that it follows directly from Proposition~\ref{prop:boo4} that the generating Bott--Samelson motives $j_\ast \left(\leftidx{_P}{\underline{P}}{_Q}\right)$ land in the category of bimodules (as opposed to the derived category of complexes). 

The first point is in proving that the extension stability in (4) remains in the category of modules and doesn't pass to the category of complexes. This follows from the fact that the Bott--Samelson motives are the heart of a weight structure on $\DMT_{P\times Q}(G)$. Therefore, for any two Bott--Samelson motives $M_1$ and $M_3$, a distinguished triangle $M_1\to M_2\to M_3\to M_1[1]$ must have the boundary map $M_3\to M_1[1]$ equal to the zero-map because $M_3$ is of weight $\leq 0$ and $M_1[1]$ is of weight $\geq 1$. The motive $M_2$ is then an extension of the motives $M_1$ and $M_3$, and the same will be true for the realization. 

The second point is proving that the convolution functors in point (5) of~\ref{BSoo} remain in the category of modules. To see this, we note that, for an inclusion of parabolic subgroups $Q\subset P$, the cohomology algebra $\mathcal{A}_Q$ is a free $\mathcal{A}_P$-module of rank $\dim\op{H}^\bullet(P/Q)$ via the induced homomorphism $\mathcal{A}_P\to\mathcal{A}_Q$. In particular, the derived tensor product is simply an ordinary tensor products, and therefore the convolutions of point (5) of~\ref{BSoo} do not leave the category of bimodules.


 Stability under (2) and (3)  in~\ref{BSoo} doesn't destroy the factorization, and the claim is proved.

It remains to prove the claim on the essential image. By Proposition~\ref{prop:boo3} and Proposition~\ref{prop:boo4}, the essential image in $\mathcal{A}_P{\op{-Modfg_\Lambda^{\mathbb{Z}}-}}\mathcal{A}_Q$ consists of those bimodules which occur as direct summands in iterated tensor products of the form \label{tifieq}
\[
\mathcal A_P\otimes_{\mathcal A_{R(1)}}\mathcal A_{S(1)}\otimes_{\mathcal A_{R(2)}}\mathcal A_{S(2)}\ldots\otimes_{\mathcal A_{R(n)}} \mathcal A_Q
\]
for chains $P\supset R(1) \subset S(1)\supset R(2)\subset S(2) \ldots\supset R(n)\subset Q$ of parabolics above $B$. Hence we recover exactly the ``singular Soergel bimodules'' investigated by Williamson in \cite{GWi}.
\end{proof}

\subsection{Full faithfulness}
The previous reductions have shown that the equivariant cohomology functor restricts to a functor from Bott--Samelson motives to singular Soergel bimodules. The heart of the present section is now the statement that this functor is indeed fully faithful, i.e., that singular Soergel bimodules provide combinatorial models for Bott--Samelson motives. 

\begin{proposition}
\label{prop:ff}
The equivariant cohomology functor 
\[
\mathbb{H}:\DMT_{P\times Q}^{\op{bs}}(G)\to \mathcal{A}_P\op{-SMod^\DZ-}\mathcal{A}_Q 
\]
is fully faithful.
\end{proposition}

\begin{proof}
Recall that we have implicitly fixed a connected reductive group $G$ with a Borel subgroup $B\subset G$. We need to show, for any pair $P,Q$ of standard parabolics and any Bott--Samelson motives $M,N\in\DMT_{P\times Q}^{\op{bs}}(G)$, that the functor $\mathbb{H}$ induces isomorphisms
\[
\mathbb{D}_{P\times Q}^+(M,N)\stackrel{\cong}{\longrightarrow} \op{Mod}_{\mathcal{A}_P-\mathcal{A}_Q}^{\mathbb{Z}}(\mathbb{H}(M),\mathbb{H}(N)).
\]
Here and in the following, we write an upper $\DZ$ on the right to denote homogeneous homomorphisms of degree zero. 

The proof of the statement will now proceed by first reducing this general statements (using the compatibility with the monoidal structure) to a very special case, which is then established via a localization argument.

To avoid overloading notation, we will mostly omit the push-forward in the generating Bott--Samelson motives $j_\ast\left(\leftidx{_P}{\underline{P}}{_Q}\right)$. 

(1) We reduce the claim to the special case where $M=j_\ast\left(\leftidx{_P}{\underline{P}}{_Q}\right)$, using an induction over the description of Bott--Samelson motives from~\ref{BSoo}. The compatibility of realization with the full six-functor formalism provides a commutative diagram
\[
\xymatrix{
\mathbb{D}^+_{P\times Q}(M(n)[2n],N)\ar[d]_\cong \ar[r] & \op{Mod}_{\mathcal{A}_P-\mathcal{A}_Q}(\mathbb{H}(M)\langle n\rangle,\mathbb{H}(N)) \ar[d]^\cong \\
\mathbb{D}^+_{P\times Q}(M,N(-n)[-2n]) \ar[r] & \op{Mod}_{\mathcal{A}_P-\mathcal{A}_Q}(\mathbb{H}(M),\mathbb{H}(N)\langle -n\rangle)
}
\]
hence full faithfulness is compatible with the appropriate twists and shifts. The standard 5-lemma argument shows that full faithfulness is compatible with extensions and direct summands.

It remains to deal with the convolution. From the discussion in~\ref{BSoo} and the comparison to the definition of Bott--Samelson motives in Definition~\ref{BSmot}, it suffices to deal with convolutions corresponding to induction and restriction functors. Assume we have an inclusion $Q\subset P$ of standard parabolics and an additional standard parabolic $R$, and assume $M=\op{Res}_P^Q M'$ for $M'\in \DMT_{P\times R}^{\op{bs}}(G)$. By \ref{irc}, the adjunction $\op{Res}_P^Q\dashv \op{Ind}_Q^P$ can be rewritten as an adjunction \[
\leftidx{_Q}{\underline{P}}{_P}\conv_P(-)\dashv \leftidx{_P}{\underline{P}}{_Q}\conv_Q(-).
\]
The monoidality result of Proposition~\ref{prop:boo3} then implies that we get a commutative diagram
\[
\xymatrix{
\mathbb{D}^+_{P\times R}(M,\leftidx{_P}{\underline{P}}{_Q}\conv N) \ar[d]^{\wr} \ar[r] & \op{Mod}^\DZ_{\mathcal A_P-\mathcal A_R}(\mathbb{H}(M), \op{Res}_{\mathcal A_Q}^{\mathcal A_P}\mathbb{H}(N)) \ar[d]^{\wr} \\
\mathbb{D}^+_{Q\times R}(\leftidx{_Q}{\underline{P}}{_P}\conv M,N) \ar[r] & \op{Mod}^\DZ_{\mathcal A_Q-\mathcal A_R}(\mathcal A_Q\otimes_{\mathcal A_P}\mathbb{H}(M), \mathbb{H}(N)) 
}
\]

Similarly, we can deal with the induction functors. Therefore, assume that $M=\op{Ind}_Q^P M'$ for $M'\in \DMT_{Q\times R}^{\op{bs}}(G)$. To use adjunction, notice that Proposition~\ref{Lasp} provides an isomorphism $\op{Ind}_! M\stackrel{\cong}{\longrightarrow} \op{Ind}_Q^P M(d)[2d]$ with $\dim P/Q=d$. In particular, it suffices to deal with the case $M=\op{Ind}_! M'$. Rewriting the adjunction $\op{Ind}_!\dashv \op{Res}_P^Q$ as the adjunction 
\[
\leftidx{_P}{\underline{P}}{_Q}(d)[2d]\conv_Q(-)\dashv \leftidx{_Q}{\underline{P}}{_P}\conv_P(-)
\]
and using monoidality  provides us with a commutative diagram
\[
\xymatrix{
\mathbb{D}^+_{P\times R}(\leftidx{_P}{\underline{P}}{_Q}\conv M(d)[2d],N) \ar[d]^{\wr} \ar[r] & \op{Mod}^\DZ_{\mathcal A_P-\mathcal A_R}(\op{Res}_{\mathcal A_Q}^{\mathcal A_P}\mathbb{H}(M)\langle d\rangle, \mathbb{H}(N)) \ar[d]^{\wr}  \\ \mathbb{D}^+_{Q\times R}(M,\leftidx{_Q}{\underline{P}}{_P}\conv N) \ar[r] & \op{Mod}^\DZ_{\mathcal A_Q-\mathcal A_R}(\mathbb{H}(M), \mathcal A_Q\otimes_{\mathcal A_P} \mathbb{H}(N)) 
}
\]
Note that there is a slight twist in the right vertical morphism. We get a commutative diagram if we compare the left-hand adjunction with the adjunction \[
\op{Res}_{\mathcal A_Q}^{\mathcal A_P}(-)\dashv \op{Hom}_{\mathcal A_P}(\mathcal A_Q,-)
\]
on the bimodule side. Since $\mathcal{A}_Q$ is a free $\mathcal{A}_P$-module of finite rank, we get a natural isomorphism $\op{Hom}_{\mathcal{A}_P}(\mathcal{A}_Q, \mathcal{A}_P)\otimes_{\mathcal{A}_P}F\sira \op{Hom}_{\mathcal{A}_P}(\mathcal{A}_Q, F)$ for any bimodule $F$. On the other hand, we have an isomorphism $\op{Hom}_{\mathcal{A}_P}(\mathcal{A}_Q,\mathcal{A}_P) \cong \mathcal{A}_Q\langle d\rangle$ as graded $\mathcal{A}_P$-modules (which basically is the Verdier duality isomorphism coming from the fact that $P/Q$ is projective). This produces the right-hand vertical isomorphism.

Combining all the above, we reduce the full faithfulness statement to the special case where $M$ is one of the generating motives in point (1) of~\ref{BSoo}, i.e., to isomorphisms 
\[
{\mathbb D}^+_{Q\times P}(\leftidx{_Q}{\underline{P}}{_P},N) \sira \op{Mod}^\DZ_{\mathcal A_Q-\mathcal A_P}( \mathcal A_Q , \mathbb{H}(N))
\]

(2) We now reduce further to the case where the parabolics $P$ and $Q$ are both equal to the chosen Borel subgroup $B$. 

Consider an inclusion $Q\subseteq P$ of standard parabolics. From the projective bundle formula, we get an isomorphism 
\[
\op{Ind}_Q^P\circ\op{Res}_P^Q (M) \cong \op{H}(P/Q)\otimes M
\]
where $\op{H}(P/Q)$ is the cohomology of the projective variety $P/Q$, viewed as constant equivariant motive in $\DMT_{P\times Q}(G)$. Similarly, on the module side, we have 
\[
\op{Res}_{\mathcal{A}_Q}^{\mathcal{A}_P}(\mathcal{A}_Q\otimes_{\mathcal{A}_P} F)\cong \op{H}(P/Q)\otimes F,
\]
and these two identifications are compatible via the equivariant cohomology functor $\mathbb{H}$. 

Consequently, by induction-restriction from $Q$ via the Borel $B$, we get a commutative diagram
\[
\xymatrix{
\mathbb{D}^+_{Q\times P}\left(\op{Ind}_B^Q \op{Res}_Q^B\left(\leftidx{_Q}{\underline{P}}{_P}\right),N\right) \ar[r] \ar[d]_\cong & \op{Mod}^\DZ_{\mathcal{A}_Q-\mathcal{A}_P}\left( \op{Res}_{\mathcal{A}_B}^{\mathcal{A}_Q}\left(\mathcal{A}_B\otimes_{\mathcal{A}_Q} \mathcal{A}_Q\right), \mathbb{H}(N)\right) \ar[d]^\cong \\
\mathbb{D}^+_{Q\times P}\left(\leftidx{_Q}{\underline{P}}{_P},N\right)^{\oplus \dim\op{H}(Q/B)} \ar[r] & \op{Mod}^\DZ_{\mathcal{A}_Q-\mathcal{A}_P}\left( \mathcal{A}_Q, \mathbb{H}(N)\right)^{\oplus\dim\op{H}(Q/B)}
}
\]
where the lower horizontal arrow is the direct sum of the $\dim\op{H}(Q/B)$ copies of the natural morphism $\mathbb{D}^+_{Q\times P}\left(\leftidx{_Q}{\underline{P}}{_P},N\right)\rightarrow \op{Mod}^\DZ_{\mathcal{A}_Q-\mathcal{A}_P}\left( \mathcal{A}_Q, \mathbb{H}(N)\right)$ induced from $\mathbb{H}$. To prove that the latter morphism is an isomorphism, it suffices to prove that the horizontal morphisms in the above diagram are isomorphisms (since a map of abelian groups is an isomorphism if a sum  of some copies of it is). For this, it suffices to show that the map 
\[
\mathbb{D}^+_{B\times P}\left(\op{Res}_Q^B \left(\leftidx{_Q}{\underline{P}}{_P}\right),N\right)\rightarrow  \op{Mod}^\DZ_{\mathcal{A}_B-\mathcal{A}_P}\left(\mathcal{A}_B\otimes_{\mathcal{A}_Q} \mathcal{A}_Q, \mathbb{H}(N)\right)
\]
is an isomorphism. Note here that $\op{Res}_Q^B \left(\leftidx{_Q}{\underline{P}}{_P}\right)\cong \leftidx{_B}{\underline{P}}{_P}$ and $\mathcal{A}_B\otimes_{\mathcal{A}_Q}\mathcal{A}_Q\cong\mathcal{A}_B$. Now applying the same argument for the $P$-action on the right, we are reduced to showing that the equivariant cohomology induces an isomorphism 
\[
\mathbb{D}^+_{B\times B}\left( \leftidx{_B}{\underline{B}}{_B},N\right)\rightarrow  \op{Mod}^\DZ_{\mathcal{A}_B-\mathcal{A}_B}\left(\mathcal{A}_B, \mathbb{H}(N)\right)
\]
for every Bott--Samelson motive $N\in\DMT_{B\times B}^{\op{bs}}(G)$. 


(3) We now prove for any Bott--Samelson motive $N$, the equivariant cohomology $\mathbb{H}$ induces an isomorphism
\[
\mathbb{D}_{B\times B}^+(\leftidx{_B}{\underline{B}}{_B}, N)\stackrel{\cong}{\longrightarrow} \op{Mod}_{\mathcal{A}_B-\mathcal{A}_B}^{\mathbb{Z}}(\mathcal{A}_B,\mathbb{H}(N)).
\]
Denote by $i:B\hookrightarrow G$ the closed  embedding of the Borel subgroup,  by $j:U=G\setminus B\hookrightarrow G$ its open complement and consider the associated localization sequence 
\[
i_!i^! N\to N\to j_\ast j^\ast N\to i_!i^!N[1]
\]
for an arbitrary Bott--Samelson motive $N\in\DMT_{B\times B}^{\op{bs}}(G)$. The motive $i_!i^!N$ is then also a Bott--Samelson motive. By an induction over the dimension of the orbits, using pointwise purity of $N$ and localization sequences, we get that $j_\ast j^\ast N$ is also a Bott--Samelson motive. As a consequence, applying equivariant cohomology, we get an exact sequence of $\mathcal{A}_B$-$\mathcal{A}_B$-bimodules
\[
0\to\mathbb{H}(i_!i^! N)\to\mathbb{H}(N)\to\mathbb{H}(j_\ast j^\ast N)\to 0.
\]

Since $i^!N$ is a pure Tate motive on $B$, $\mathbb{H}(i_!i^!N)$ will be a direct sum of copies of $\mathcal{A}_B$, possibly twisted. On the other hand, by the tilting for Bruhat-cells in~\ref{TBC}, $\mathbb{H}(j_\ast j^\ast N)$ will be an iterated extension of copies of $\mathcal{A}_B 1_x\langle n\rangle$ for elements $x\in W$ different from $e$. As a consequence, there will be no nonzero $\mathcal{A}_B$-$\mathcal{A}_B$-bimodule homomorphisms $\mathcal{A}_B\to \mathbb{H}(j_\ast j^\ast N)$. Now consider the commutative diagram
\[
\xymatrix{
\mathbb{D}^+_{B\times B}\left(i_\ast\left(\leftidx{_B}{\underline{B}}{_B}\right), i_!i^!N\right) \ar[r] \ar[d] & \op{Mod}^\DZ_{\mathcal{A}_B-\mathcal{A}_B}(\mathcal{A}_B, \mathbb{H}(i_!i^!N)) \ar[d]\\
\mathbb{D}^+_{B\times B}\left(i_\ast\left(\leftidx{_B}{\underline{B}}{_B}\right),N\right) \ar[r] & \op{Mod}^\DZ_{\mathcal{A}_B-\mathcal{A}_B}(\mathcal{A}_B, \mathbb{H}(N))
}
\]
where the horizontal morphisms are induced from $\mathbb{H}$ and the vertical morphisms are induced from the maps $i_!i^!N\to N$ in the localization sequence above. 

The top horizontal morphism is an isomorphism, since it can be rewritten (by adjunction) to a morphism between constant mixed Tate motives on $B$, where we have conservativity (and hence full faithfulness) by assumption. The left vertical morphism is an isomorphism since 
\[
\mathbb{D}^+_{B\times B}\left(i_\ast \left(\leftidx{_B}{\underline{B}}{_B}\right),j_\ast j^\ast N\right)\cong \mathbb{D}^+_{B\times B}\left(j^\ast i_!\left(\leftidx{_B}{\underline{B}}{_B}\right), j^\ast N\right)=0
\]
The right vertical morphism is an isomorphism: it is automatically injective because $\mathbb{H}(i_!i^!N)\to \mathbb{H}(N)$ is, and it is surjective because we noted above that there can be no nontrivial bimodule homomorphisms 
\[
\mathcal{A}_B\to \mathbb{H}(j_\ast j^\ast N)\cong \op{coker}\left(\mathbb{H}(i_!i^!N)\to\mathbb{H}(N)\right). 
\]
From these statements, we deduce that the lower horizontal morphism in the above commutative diagram is also an  isomorphism. But this was what remained to prove. 
\end{proof}

The following now completes the proof of Theorem~\ref{thm:motivebimod} in the introduction. 

\begin{corollary}
\label{cor:pf142}
We have a zig-zag of equivalences of triangulated categories  
\[
\mathbb{T}: \DMT_{P\times Q}(G) \stackrel{\approx}{\longleftarrow} \op{Hot}^{\op{b}}\left(\DMT_{P\times Q}^{\op{bs}}(G)\right) \stackrel{\approx}{\longrightarrow}  \op{Hot}^{\op{b}}\left(\mathcal{A}_P\op{-SMod^\DZ-}\mathcal{A}_Q\right)
\]
The restriction to Bott--Samelson motives coincides with equivariant motivic cohomology. The equivalence is compatible with the monoidal structures, i.e., for any equivariant mixed Tate motives $M\in \DMT_{P\times Q}(G)$ and $N\in \DMT_{Q\times R}(G)$, there are natural isomorphisms in $\op{Hot}^{\op{b}}\left(\mathcal{A}_P\op{-SMod^\DZ-}\mathcal{A}_R\right)$:
\[
\mathbb{T}(M\conv_Q N)\stackrel{\cong}{\longrightarrow} \mathbb{T}(M)\otimes_{\mathcal{A}_Q}\mathbb{T}(N).
\]
\end{corollary}

\begin{proof}
The first equivalence is the tilting functor, cf. Corollary~\ref{tiFL}. For the second equivalence, we note that Corollary~\ref{cor:boo4} and Proposition~\ref{prop:ff} imply that equivariant cohomology induces an equivalence 
\[
\mathbb{H}:\DMT_{P\times Q}^{\op{bs}}(G)\stackrel{\simeq}{\longrightarrow} \mathcal{A}_P\op{-SMod^{\mathbb{Z}}-}\mathcal{A}_Q.
\]
The second equivalence is then the induced equivalence of categories of complexes. 

It remains to prove that the equivalences are compatible with the monoidal structures. The compatibility of convolution with the tilting equivalence is proven in Proposition~\ref{prop:convtilt}. For the second equivalence between homotopy categories, this follows from the naturality of the isomorphims of the additive categories in Proposition~\ref{prop:boo3}. 
\end{proof}

\subsection{Grading on equivariant derived category}
Now we deal with another categorification of the Schur algebroid: we will show that the equivariant mixed Tate motives provide a graded version of the well-known equivariant derived categories, with the appropriate realization functor as degrading functor. This will, in particular, prove Theorems~\ref{thm:motivesparabolic} and \ref{thm:functionsheaf} from the introduction. 

We begin with the proof of Theorem~\ref{thm:motivesparabolic} which establishes that equivariant mixed Tate motives provide a graded version of the corresponding equivariant derived category.

\begin{theorem}
\label{thm:gradedparabolic}
\begin{enumerate}
\item Let $k=\mathbb{F}_q$ be a finite field and assume that we take  $\mathbb{D}=\mathbf{DA}^{\et}(-;\mathbb{Q})$ as the homotopical stable algebraic derivator underlying the construction of equivariant motives. Let $G$ be a split connected reductive group over $k$, and let $P,Q\subseteq G$ be two parabolic subgroups. Then the $\ell$-adic realization functor
\[
\op{Real}_\ell:\DMT_{P\times Q}(G)_{\op{wt}=0}\hookrightarrow \DMT_{P\times Q}(G)\to\op{Der}^{\op{b},\op{ss},\op{ev}}_{P\times Q}(G;\mathbb{Q}_\ell) 
\]
on the category of Bott--Samelson motives is fully faithful, with essential image consisting of intersection complexes concentrated in even degrees. 
\item Let $k=\mathbb{C}$ and consider $\mathbb{D}=\op{MDer}(-;\mathbb{C})$ as homotopical stable algebraic derivator underlying the construction of equivariant motives. Let $G$ be a connected reductive group over $k$, and let $P,Q\subseteq G$ be two parabolic subgroups. Then the Hodge realization functor
\[
\op{Real}_{\op{H}}:\DMT_{P\times Q}(G)_{\op{wt}=0}\hookrightarrow \DMT_{P\times Q}(G)\to\op{Der}^{\op{b},\op{ss},\op{ev}}_{P\times Q}(G;\mathbb{C}) 
\]
on the category of Bott--Samelson motives is fully faithful, with essential image consisting of intersection complexes concentrated in even degrees. 
\item Motivic lifts of the standard and costandard objects in the equivariant derived categories above are given by $i_\ast\left(\const{PwQ}\right)$ and $i_!\left(\const{PwQ}\right)$, respectively, where $w\in W$ is an element of the Weyl group, $PwQ$ the corresponding $P\times Q$-orbit and $i:(P\times Q\looparrowright PwQ)\to (P\times Q\looparrowright G)$. 
\item The functors in points (1) and (2) are compatible with convolution and Verdier duality.
\item The equivariant mixed Tate motives provide a grading on the equivariant derived categories in the sense of~\ref{bgs43}, with the realization functors in (1) and (2) as degrading functors. 
\end{enumerate}
\end{theorem}

\begin{proof}
(3) follows directly from Proposition~\ref{prop:boo4} and the description of the standard and costandard objects in the equivariant derived categories as push-forwards of constant sheaves.

The proofs for (1) and (2) are similar, so we only need to prove (1). Let $M$ be the constant motive on the unique closed $P$-orbit  $j:Y\hookrightarrow G/Q$.  Then $j_\ast M\cong j_!M$ by definition, hence its realization  is an intersection cohomology sheaf and therefore perverse. Now the other steps in the construction of Bott--Samelson motives, cf. Definitions~\ref{def:BSclosure} and \ref{BSmot}, preserve the required perversity. Therefore, Bott--Samelson motives will map to direct sums of intersection cohomology complexes and by Theorem~\ref{thm:parabolicWT} and Proposition~\ref{tiFL} this implies that Bott--Samelson motives are sent to direct sums of even shifts of intersection complexes. Note that the shifts we obtain are only even, this is the realization of the operation $M\mapsto M(1)[2]$. We get every intersection complex as a direct summand because (3) provides the motivic lifts standard and costandard objects. 

(4) follows directly from the fact that realization functors are compatible with the full six-functor formalism, cf. Section~\ref{sec:realization}, as well as quotient and induction equivalences. 

(5) Realization is obviously an exact functor. Faithfulness essentially follows from the conservativity of realization functors on equivariant mixed Tate motives, cf. Proposition~\ref{prop:conservative}. Note that by (3) we have preferred lifts for irreducible objects from $\op{Der}^{\op{b}}_{P\times Q}(G)$: the irreducible objects are the intersection complexes for orbits, and these have explicit lifts as Bott--Samelson motives. So it remains to show that for any two equivariant mixed Tate motives $M,N\in\DMT_{P\times Q}(G)$ we have isomorphisms
\[
\bigoplus_{i\in\mathbb{Z}} \DMT_{P\times Q}\left(M,N(i)\right) \stackrel{\cong}{\longrightarrow} \op{Der}^{\op{b}}_{P\times Q}\left(\op{Real}  M, \op{Real} N\right)
\]
We can write any Bott--Samelson motive as iterated extension of standard and costandard objects, and a standard devissage argument reduces the claim to morphisms between standard and costandard objects. In this case, the claim follows from the full faithfulness on constant mixed Tate motives over orbits, cf.~Proposition~\ref{prop:conservative}. 
\end{proof}

Now we prove Theorem~\ref{thm:functionsheaf} which provides another categorification of the Schur algebroid, via the equivariant derived categories. Special cases provide categorifications of the Hecke algebra $(P=Q=B)$ and the parabolic Hecke modules of \cite{Deopar} $(Q=B)$. 

\begin{theorem}
\label{thm:schur}
Assume one of the situations in points (1) or (2) of Theorem~\ref{thm:gradedparabolic}, let $G$ be a split connected reductive group, and let $P,Q$ be two parabolic subgroups. 
\begin{enumerate}
\item Then the composition
\[
\op{K}_0\left(\DMT_{P\times Q}(G)\right)\stackrel{\op{Real}}{\longrightarrow} \op{K}_0^{\oplus}\left(\op{Der}^{\op{b},\op{ss},\op{ev}}_{P\times Q}(G)\right)\to {}^P\mathcal{H}^Q
\]
of realization with the function--sheaf correspondence of Grothendieck induces an isomorphism, providing a categorification of the Schur algebroid via equivariant mixed Tate motives. 
\item The tilting functor $\op{Hot}^{\op{b}}(\DMT_{P\times Q}(G)_{\op{wt}=0})\to \DMT_{P\times Q}(G)$ induces an isomorphism of Grothendieck groups. Combined with (1), this provides a categorification of the Schur algebroid via Bott--Samelson motives. 
\item The isomorphisms above are compatible with convolution. The involution induced by Verdier duality on the motives side is mapped to the Kazhdan--Lusztig involution on the Schur algebroid side.
\end{enumerate}
\end{theorem}

\begin{proof}
For (1), we can use the classical categorification of ${}^P\mathcal{H}^Q$ in terms of even shifts of intersection complexes together with Theorem~\ref{thm:gradedparabolic} (1). 

By Corollary~\ref{tiFL}, the tilting functor is an equivalence and thus (2) is clear.

(3) Compatibility with the convolution follows from the fact that realization is compatible with the full six-functor formalism. The claim on the Kazhdan--Lusztig involution follows from compatibility of realization with Verdier duality and known computations in the equivariant derived category. 
\end{proof}

\begin{Bemerkung}
Note that the realization functors map the Bott--Samelson motives $p_{w\ast}\const{X_w}$ for the Bott--Samelson resolution $p_w:X_w\to G/B$ of the Schubert cell for the word $w\in W$ to the corresponding push-forwards of the constant sheaf from $X_w$. The realization of $p_{w\ast}\const{X_w}$ then decomposes as a direct sum of IC-sheaves. In particular, the images of indecomposable Bott--Samelson motives in the Hecke algebra ${}^B\mathcal{H}^B$ will be given by the element $C_w$ of the Kazhdan--Lusztig basis. The standard motives (given by pushforwards of constant sheaves on Bruhat cells) will be mapped to the standard basis of the Hecke algebra. 

On the other hand, we can also map equivariant mixed Tate motives to the homotopy category of complexes of Soergel bimodules. Under this tilting functor, the indecomposable Bott--Samelson motives are mapped to Soergel bimodules, and the standard motives are mapped to Rouquier complexes. 
\end{Bemerkung}

We finally discuss the compatibility of the two categorifications of the Schur algebroid. From the diagrams in Section~\ref{setup}, we obtain the following diagram, related to the categorification via equivariant derived categories: 
\[
\xymatrix{
\op{K}_0^{\oplus}(\DMT_{P\times Q}(G)_{\op{wt}=0}) \ar[d]^{\op{tilt}}_{\approx}  \ar[r]^\approx & \op{K}_0^{\oplus}\left(\op{Der}^{\op{b},\op{ss},\op{ev}}_{P\times Q}(G)\right) \ar[d] \ar[rr] && \mathcal{H} \ar[d]^{\textrm{evaluation at } q}\\
\op{K}_0(\DMT_{P\times Q}(G)) \ar[r]_{\op{Real}_\ell} & \op{K}_0(\op{Der}^{\op{b}}_{P\times Q}(\pt;\mathbb{Q}_\ell)) \ar[rr]_>>>>>>>>>{\textrm{function-sheaf}} && \mathcal{H}_q
}
\]
Note that the function--sheaf correspondence requires that we are working over a finite field (to have Frobenius traces available). The right upper horizontal arrow is the categorification of the universal Schur algebroid considered in previous works, going via perversely semisimple complexes. This commutative diagram now expresses that we have a motivic categorification of the universal Schur algebroid which is compatible with the previously known one. 

Similarly, we can consider a commutative diagram of isomorphisms expressing the compatibility of the motivic categorification with the one using singular Soergel bimodules: 
\[
\xymatrix{
\op{K}_0^{\oplus}(\DMT_{P\times Q}(G)_{\op{wt}=0}) \ar[d]_{\cong} \ar[r] &  \op{K}_0(\DMT_{P\times Q}(G))  \ar[d]^{\mathbb{T}}_{\cong} \\
\op{K}_0^{\oplus}(\mathcal{A}_P\op{-SMod-}\mathcal{A}_Q) \ar[r] & \op{K}_0(\op{Hot}^{\op{b}}(\mathcal{A}_P\op{-SMod-}\mathcal{A}_Q))
}
\]

For the compatibility of the two approaches, we can now combine all the previously considered diagrams into one: 
\[
\xymatrix{
\op{K}_0(\DMT_{P\times Q}(G))\ar[rr]^{\op{Real}} && 
\op{K}_0^{\oplus}\left(\op{Der}^{\op{b},\op{ss},\op{ev}}_{P\times Q}(G)\right) 
 \ar[dd]^{\textrm{function-sheaf}} \\
\op{K}_0(\DMT_{P\times Q}(G)_{\op{wt}=0})\ar[d]_{\mathbb{H}} \ar[u]^{\op{tilt}}
\\
\op{K}_0(\mathcal{A}_P{\op{-SMod-}}\mathcal{A}_Q)  \ar[rr]_{\sim} &&
{}^P\mathcal{H}^Q
}
\]

The top horizontal map is the realization functor, the right-hand vertical map is the previously known categorification via perversely semi-simple complexes, and the corresponding composition is the categorification of the Schur algebroid discussed in Theorem~\ref{thm:schur}. The bottom horizontal map is the known categorification of the Schur algebroid via singular Soergel bimodules, the left vertical maps are the identification of equivariant motives and complexes of singular Soergel bimodules via equivariant cohomology and tilting equivalence. Their composition is the categorification of the Schur algebroid following from Corollary~\ref{cor:pf142}. 

Commutativity of the square follows from the diagrams in Section~\ref{setup}: the upper composition takes the realization of a motive, computes its equivariant cohomology and produces the function encoding the traces of Frobenius; the lower composition also simply takes hypercohomology and encodes its bimodule structure. The commutativity is then a known computation with bimodules and equivariant sheaves. 

Now it remains to compare the categorification of the universal Schur algebroid with the categorification of the Schur algebroid over a fixed finite field $\mathbb{F}_q$. For the universal Schur algebroid, the $\mathbb{Z}[v,v^{-1}]$-module structure is given by (a square root of) the Tate twist $(-)\otimes\mathbb{Q}(1)$. For the Schur algebroid over a fixed finite field $\mathbb{F}_q$, the specialization of the $\mathbb{Z}[v,v^{-1}]$-module structure from the universal one is encoded in the Frobenius-action of the Weil sheaves over $\overline{\mathbb{F}_q}$. The relation between the two is given by the fact that for mixed Tate motives, the eigenvalue of the Frobenius action encodes exactly the weight.

\begin{Bemerkung}
The relevance of this categorification $\op{K}_0(\DMT_{B\times B}(G))\to{}^B\mathcal{H}^B$ for representation theory is the following. Denoting by $i:BwB\hookrightarrow G$ the inclusion of a double coset $BwB$ into $G$, the classes of the standard objects $i_\ast(BwB)$ in $\op{K}_0(\DMT_{B\times B}(G))$ map to the \emph{standard basis} of ${}^B\mathcal{H}^B$ and the classes of the costandard objects $i_!(BwB)$ map to the dual basis of the standard basis. On the other hand, the intermediate extensions, which are motivic lifts of the intersection complexes, map to the \emph{Kazhdan--Lusztig basis} of the Hecke algebra ${}^B\mathcal{H}^B$. Actually, the coefficients have to be extended from $\mathbb{Z}[q,q^{-1}]$ by adding a square root $v=\sqrt{q}$ to accomodate intersection complexes concentrated in odd degrees. Then the change-of-basis matrix over $\mathbb{Z}[v,v^{-1}]$ from the standard basis to the Kazhdan--Lusztig basis describes exactly the multiplicities of indecomposables in induced modules; this is encoded in the Kazhdan--Lusztig polynomials (resp. their values at $1$). Now the fact that $\op{K}_0(\DMT_{B\times B}(G))$ has an additional integer grading allows to extract the coefficients of the Kazhdan--Lusztig polynomials. On the level of the categorification, the coefficients of Kazhdan--Lusztig polynomials can then be related to Ext-groups in the categories $\DMT_{B\times B}(G)$ of equivariant mixed Tate motives. References for further information on the relevant representation theoretic aspects include \cite{KL-C,KL-S,Sp,soergel:slides}.
\end{Bemerkung}

\subsection{It's ... knot motives}
\label{sec:knot}

Now, as an application of the categorification of the Schur algebroid using equivariant mixed Tate motives, we can provide a proof of Corollary~\ref{cor:link} in the introduction. We first compute the images of standard and costandard equivariant Tate motives. 

\begin{proposition}
\label{prop:rouquier}
Let $G$ be a split connected reductive group, let $B\subset G$ be a Borel subgroup. Fix a set of simple reflections $S$ of the Weyl group $W$, let $s\in S$ be a simple reflection, and denote by $i:((B\times B)\looparrowright BsB)\to ((B\times B)\looparrowright G)$ be the inclusion of the corresponding Bruhat cell in the group. Then the functor
\[
\mathbb{T}:\DMT_{B\times B}(G)\to \op{Hot}^{\op{b}}\left(\mathcal{A}_B\op{-SMod^{\mathbb{Z}}-}\mathcal{A}_B\right)
\]
from Corollary~\ref{cor:pf142} maps the motives $i_!\left(\underline{BsB}\right)$ and $i_\ast\left(\underline{BsB}\right)$ to the so-called Rouquier complexes
\[
\left[\mathcal{A}_B\otimes_{\mathcal{A}_B^s}\mathcal{A}_B\to\mathcal{A}_B\right], \quad\textrm{ and }\quad
\left[\mathcal{A}_B\langle 2\rangle \to \mathcal{A}_B\otimes_{\mathcal{A}_B^s}\mathcal{A}_B\right]
\]
In both cases, the complexes are graded such that the term $\mathcal{A}_B\otimes_{\mathcal{A}_B^s}\mathcal{A}_B$ sits in degree zero. 
\end{proposition}

\begin{proof}
It suffices to consider the pushforward to $P$, where $P=P_s$ is the minimal standard parabolic corresponding to the simple reflection $s$. Denote by $i:B\hra P$ the closed embedding of the Borel subgroup and by $j:BsB\hra P$ the open embedding of its complement. In the associated localization triangle
\[
j_!j^!\left(\underline{P}\right)\ra \underline{P}\ra i_*i^*\left(\underline{P}\right)\ra j_!j^!\left(\underline{P}\right)[1]
\] 
we can replace $j^!\left(\underline{P}\right)\cong \underline{BsB}$ and $i^\ast \left(\underline{P}\right)\cong \underline{B}$. The localization triangle then implies that we have an isomorphism of motives
\[
j_!\underline{BsB}\cong\op{tilt}\left(\left[\underline{P}\to i_\ast\underline{B}\right]\right)
\]
where on the right we have the tilting functor applied to a complex in the category $\op{Hot}^{\op{b}}\left(\DMT_{B\times B}(G)\right)$, where the motive $\underline{P}$ sits in degree $0$. Then we can rewrite $\underline{P}$ in terms of convolution of generating Bott--Samelson motives as
\[
\underline{P}\cong\leftidx{_B}{\underline{P}}_B =\leftidx{_B}{\underline{P}}_P\conv_P \leftidx{_P}{\underline{P}}_B.
\]
Applying the equivariant cohomology to this, we see that $\underline{P}$ maps to $\mathcal{A}_B\otimes_{\mathcal{A}_B^s}\mathcal{A}_B$, and we obtain
\[
\mathbb{T}\left(\left[\underline{P}\to i_\ast\underline{B}\right]\right)\cong 
\left[\mathcal{A}_B\otimes_{\mathcal{A}_B^s}\mathcal{A}_B\to\mathcal{A}_B\right]
\] 
One checks that the differential of the complex on the right is simply the multiplication map, and therefore the complex on the right is the first Rouquier complex from \cite{RouC}. 

The second statement follows similarly, by considering the appropriate localization triangle 
\[
i_!i^!\left(\underline{P}\right)\ra \underline{P}\ra j_*j^*\left(\underline{P}\right)\ra i_!i^!\left(\underline{P}\right)[1]
\]
where we can replace $j^\ast\left(\underline{P}\right)\cong \underline{BsB}$ and $i^!\left(\underline{P}\right)\cong \underline{B}(1)[2]$. We get an isomorphism of motives 
\[
j_\ast\underline{BsB}\cong\op{tilt}\left(\left[i_\ast\underline{B}(1)[2] \to \underline{P}\right]\right)
\]
where now $\underline{P}$ sits in degree $0$. Applying equivariant cohomology, we get 
\[
\mathbb{T}\left(\left[i_\ast\underline{B}(1)[2]\to \underline{P}\right]\right)\cong 
\left[\mathcal{A}_B\langle 2\rangle \to \mathcal{A}_B\otimes_{\mathcal{A}_B^s}\mathcal{A}_B\right]
\] 
The differential is the (unique up to scalar) non-zero morphism, and we obtain precisely the second Rouquier complex from \cite{RouC}. 
\end{proof}

\begin{Bemerkung}
As a consequence of the above, we obtain an assignment from braids on $n$ strands to equivariant mixed Tate motives. Denote by $s\in W$ the simple reflection in the Weyl group $\op{S}_n$ corresponding to the transposition $(i,i+1)$, and denote by $j:BsB\hookrightarrow \op{GL}_n$ the inclusion of the corresponding double coset. An overcrossing of the strands $i$ and $i+1$ is mapped to the costandard motive $j_!\const{BsB}\in \DMT_{B\times B}(\op{GL}_n)$, and an undercrossing of the strands $i$ and $i+1$ is mapped to the standard motive $j_\ast\const{BsB}(1)[2]\in\DMT_{B\times B}(\op{GL}_n)$. Then a braid, given by composition $s_{i_1}^{\pm 1}\circ\cdots s_{i_m}^{\pm 1}$ of such crossings is mapped to the tensor product $j_{i_1,!/\ast}\const{Bs_{i_1}B}\otimes \cdots\otimes j_{i_m,!/\ast}\const{Bs_{i_m}B}$ in $\DMT_{B\times B}(\op{GL}_n)$. 

By Proposition~\ref{prop:rouquier} and Corollary~\ref{cor:pf142}, the realization/hypercohomology of such a motive associated to a braid in $\op{Hot}^{\op{b}}\left(\mathcal{A}_B\op{-SMod^{\mathbb{Z}}-}\mathcal{A}_B\right)$ is the tensor product of the appropriate Rouquier complexes.  These complexes of graded bimodules associated to a braid are those used by Khovanov \cite{KhoHH} to obtain triply graded invariants of said braid closed to a knot  by applying termwise Hochschild homology and then taking cohomology. In geometric terms this has been reformulated and extended by \cite{WW}.
\end{Bemerkung}

At this point, we have an assignment from braids to $B\times B$-equivariant mixed Tate motives on $\op{GL}_n$, or alternatively, an assignment from the generators $\sigma_i^{\pm 1}$ of the braid group (corresponding to the appropriate simple reflections $s_i\in \op{S}_n$) to such motives. Showing that this assignment factors through the braid group requires checking the braid relations. For the original construction of Khovanov \cite{KhoHH} with Rouquier complexes, the proof of the braid relations requires computing the tensor products of Rouquier complexes. The motivic lift of Khovanov's construction described above now provides a more geometric proof of the braid relations, reducing everything to isomorphisms of the relevant Bruhat cells: 

\begin{Bemerkung}
Let $G$ be a split connected reductive group with Weyl group $W$, and denote by $l:W\to\mathbb{N}$ the length function. Given elements $x,y\in W$ of the Weyl group with $l(xy)=l(x)+l(y)$, it is well known that the multiplication induces an isomorphism of Bruhat cells
\[
BxB\times_B ByB\stackrel{\cong}{\longrightarrow} BxyB.
\]
Consequently, we get isomorphisms for convolutions of constant equivariant mixed Tate motives
\[
j_!\left(\underline{BxB}\right)\conv_B j_!\left(\underline{ByB}\right)\stackrel{\cong}{\longrightarrow} j_!\left(\underline{BxyB}\right),\,\textrm{ and }\,
j_\ast\left(\underline{BxB}\right)\conv_B j_\ast\left(\underline{ByB}\right)\stackrel{\cong}{\longrightarrow} j_\ast\left(\underline{BxyB}\right)
\]
Note that the maps $j$ in the above are all different: whenever applied to the constant motive $\underline{BuB}$, the appropriate map $j$ is the inclusion of the corresponding Bruhat cell $j_u:(B\times B\looparrowright BuB)\hookrightarrow (B\times B\looparrowright G)$. 

As a consequence of the above isomorphisms, we get the braid relations for the standard/costandard equivariant Tate motives: for any two simple reflections $s,t\in W$ with $sts=tst$, we find
\[
j_!\left(\underline{BsB}\right)\conv_B j_!\left(\underline{BtB}\right)\conv_B j_!\left(\underline{BsB}\right)\cong j_!\left(\underline{BtB}\right)\conv_B j_!\left(\underline{BsB}\right)\conv_B j_!\left(\underline{BtB}\right),
\]
\[
j_\ast\left(\underline{BsB}\right)\conv_B j_\ast\left(\underline{BtB}\right)\conv_B j_\ast\left(\underline{BsB}\right)\cong j_\ast\left(\underline{BtB}\right)\conv_B j_\ast\left(\underline{BsB}\right)\conv_B j_\ast\left(\underline{BtB}\right). 
\]

In addition, one easily checks   
\[
j_!\left(\underline{BsB}\right)\conv_B  j_*\left(\underline{BsB}\right)\cong i_\ast\left(\underline{B}(-1)[-2]\right)\cong j_*\left(\underline{BsB}\right) \conv_B  j_!\left(\underline{BsB}\right)
\] 
either by calculation with Rouquier complexes or by (possibly more transparent) geometric arguments. As a consequence, we get a homomorphism of the braid group to the  monoid of isomorphism classes of $(B \times B)$-equivariant mixed Tate motives on $G$, with monoid multiplication given by convolution. Instead of the motivic category $\DMT_{B\times B}(G)$, we can equivalently use the homotopy category $\op{Hot}^{\op{b}}(\mathcal{A}_B\op{-SMod^\DZ-}\mathcal{A}_B)$ of Soergel bimodules and obtain an action of the braid group on the latter category.
\end{Bemerkung}

Now we only need to translate the termwise Hochschild homology construction required for closing the braid to a link to get a motivic lift of Khovanov's triply graded knot homology, cf. \cite{WW}.

\section{Hecke modules for symmetric varieties}
\label{sec:symmetric}

In this section, we show how equivariant mixed Tate motives provide a different model for the Hecke-module of Mars--Springer \cite{Mars-Springer}, cf. also \cite[Section 3.7]{springer:schubert}. More precisely, we show that the $\ell$-adic realization induces an isomorphism on the relevant Grothendieck groups. This in particular provides a grading on the Hecke module associated to a symmetric variety. We also discuss the relevance of this motivic realization for the study of Koszul duality patterns in the representation theory of real Lie groups. 

\subsection{Setup and notation}

\begin{Bemerkung}
\label{symmsetup}
Throughout the section, we fix an algebraically closed field $k$ of characteristic $\neq 2$ and a connected reductive group $G$, with an involution $\theta:G\to G$ and a $\theta$-stable Borel subgroup $B\subset G$. Set $K=G^\theta$, and let $P$ be a standard parabolic. We assume that $\mathbb{D}$ is derivator satisfying the conditions \ref{derivator:new} as well as the grading condition \ref{conditions:grading} and the weight condition \ref{conditions:weight}. 
\end{Bemerkung}

\begin{Bemerkung}
\label{symmdiag}
Assume either that $\mathbb{D}=\mathbf{DA}^{\et}(-;\Lambda)$ or $\mathbb{D}=\op{MDer}(-;\mathbb{C})$, and denote by $\op{Real}$ the appropriate realization functor. Consider the following diagram of Grothendieck groups\footnote{Except for the right-hand arrow, all the maps in the diagram naturally arise from functors on the categorical level. However, the right-hand map is given by alternating sum of the equivariant cohomology groups of perverse sheaves. It could be made categorical using a bit of $\mathscr{D}$-module machinery, but we stick to simply using the map $\mathcal{H}$ considered in \cite{Mars-Springer}.}:
\[
\xymatrix{
\op{K}_0(\op{Hot}^{\op{b}}(\DMT_{P\times K}(G)_{\op{wt}=0})) \ar[rr]^{\op{Real}} \ar[d]_{\op{tilt}} && \op{K}_0(\op{Hot}^{\op{b}}(\op{Per}^{\op{ss}}_{P\times K}(G;\Lambda))) \ar[d]^{\mathcal{H}} \\
\op{K}_0(\DMT_{P\times K}(G)) \ar[rr]_{\op{Real}} && \op{K}_0(\op{Der}^{\op{b}}_{P\times K}(G;\Lambda)). 
}
\]

On the left-hand side, we have the categories of motives discussed in Section~\ref{sec:tiltingapp}. The category on the upper right-hand side is a category of $P\times K$-equivariant perverse sheaves on the group $G$, as was considered in \cite{Mars-Springer}. The left-hand vertical arrow is the tilting functor of Proposition~\ref{prop:symmpurity}, and the lower horizontal is given by realization functors on equivariant categories, cf. Proposition~\ref{prop:realization}, restricted to equivariant mixed Tate motives.

The top horizontal functor and the right-hand vertical functor are a bit more complicated to set up. For the top horizontal functor, we will prove that the realization of a pure weight $0$ motive is an intersection complex, cf. Theorem~\ref{thm:gradedsymm}. However, it is not necessarily a perverse sheaf, as it may be concentrated in the wrong degrees, but a suitable shift of it will be a semisimple perverse sheaf. This is why the target of the realization functor in the top horizontal is the category of complexes of perverse sheaves. The right-hand vertical morphism is then induced from taking cohomology of perverse sheaves, cf. \cite{Mars-Springer}. 

The commutativity of the above diagram follows from the fact that the realization functors commute with the full six functor formalism. Essentially both paths through the diagram compute the equivariant cohomology of a $P\times K$-equivariant motive on $G$ which is pure of weight $0$. In a sense, the above diagram provides a nice motivic categorification of the cohomology map $\mathcal{H}$ of \cite{Mars-Springer}. 
\end{Bemerkung}

\begin{Bemerkung}
The goal of the next section, cf. Theorem~\ref{thm:comparisonMS}, is to show that the horizontal arrows, induced from the realization functors, are isomorphisms of the respective Grothendieck groups. The lower horizontal arrow makes equivariant mixed Tate motives a graded version of the equivariant derived category, cf. Theorem~\ref{thm:gradedsymm}. Moreover, the collection of categories $\DMT_{P\times K}(G)$ where $P\leq G$ runs through the standard parabolics in $G$ is compatible with the convolution by objects from the categories $\DMT_{Q\times P}(G)$ categorifying the Schur algebroid. In abuse of language, the categories $\DMT_{P\times K}(G)$ form a module for the motivic categorification of the Schur algebroid, cf. Proposition~\ref{prop:schurmodulesymm}. A similar statement is true for the categories $\op{Hot}^{\op{b}}(\DMT_{P\times K}(G)_{\op{wt}=0})$ and the corresponding motivic realization of the Schur algebroid, and the left vertical tilting map in the above diagram is compatible with these module structures.
\end{Bemerkung}

\begin{remark}
\label{rem:fq}
As discussed in \ref{descent}, equivariant mixed Tate motives are very combinatorial; in particular, if we have a finite field $k=\mathbb{F}_q$, we get an induced equivalence of mixed Tate motives over $\mathbb{F}_q$ and over $\overline{\mathbb{F}_q}$. The same is then true for the equivariant mixed Tate motives in our cases. In particular, it suffices to work over the algebraic closure $\overline{\mathbb{F}_q}$ to establish the equivariant Whitney--Tate condition, and then the same thing will be true over the finite field. Over the finite field, the $\ell$-adic realization functors can be equipped with natural Frobenius structures. Moreover, the weight arising from eigenvalues of Frobenius will agree with the motivic weight via the $\ell$-adic realization functor, and this is how we do the comparison to \cite{Mars-Springer}. 
\end{remark}

\subsection{Categorification of the Mars--Springer module}

We first establish that the motivic categories $\DMT_{P\times K}(G)$ provide a graded version of the equivariant derived categories of symmetric varieties. 

\begin{theorem}
\label{thm:gradedsymm}
Assume the situation in \ref{symmsetup}
\begin{enumerate}
\item Let $k=\mathbb{F}_q$ be a finite field of odd characteristic, and let $\mathbb{D}=\mathbf{DA}^{\et}(-;\mathbb{Q}_\ell)$ be the homotopical stable algebraic derivator underlying the construction of equivariant motives. Then composition with the $\ell$-adic realization functor
\[
\DMT_{P\times K}(G)_{\op{wt}=0} \hookrightarrow \DMT_{P\times K}(G) \xrightarrow{\op{Real}_\ell} \op{Der}^{\op{b}}_{P\times K}(G;\mathbb{Q}_\ell)
\] 
is fully faithful on the category of Bott--Samelson motives. The essential image consists of intersection complexes concentrated in even degrees. 
\item Let $k=\mathbb{C}$ and consider  $\mathbb{D}=\op{MDer}(-;\mathbb{C})$ as homotopical stable algebraic derivator underlying the construction of equivariant motives. Then the composition with the Hodge realization functor
\[
\DMT_{P\times K}(G)_{\op{wt}=0} \hookrightarrow \DMT_{P\times K}(G) \xrightarrow{\op{Real}_{\op{H}}} \op{Der}^{\op{b}}_{P\times K}(G;\mathbb{C})
\] 
is fully faithful on the category of Bott--Samelson motives. The essential image consists of intersection complexes concentrated in even degrees. 
\item Motivic lifts of the standard and costandard objects in the equivariant derived categories above are given by $i_\ast(\mathcal{L})$ and $i_!(\mathcal{L})$, respectively, where $i:PxK\hookrightarrow G$ is the inclusion of a $P\times K$-double coset in $G$ and $\mathcal{L}$ is a  $K$-equivariant motivic local system on the corresponding $K$-orbit of $G/P$. 
\item The functors in points (1) and (2) are compatible with convolution and Verdier duality. 
\item Adjoining a root of the Tate twist, the equivariant mixed Tate motives provide a grading on the constructible equivariant derived categories in the sense of \ref{bgs43}, with the realization functors in (1) and (2) as degrading functors.
\end{enumerate}
\end{theorem}

\begin{proof}
The arguments are similar to the ones in Theorem~\ref{thm:gradedparabolic}. 

We first note that the realization functors in statements (1) and (2) are compatible with the six functors, cf. Proposition~\ref{prop:realization}. Compatibility with Verdier duality is then clear, and compatibility with convolution is proved in Proposition~\ref{prop:schurmodulesymm}. Thus (4) follows if we have (1) and (2).

To prove (3), we first note that the standard and costandard objects in the equivariant derived category are obtained exactly as $i_\ast(\mathcal{L})$ and $i_!(\mathcal{L})$ where $i:PxK\hookrightarrow G$ is the inclusion of a double coset and $\mathcal{L}$ is a $K$-equivariant local system on the corresponding $K$-orbit in $G/P$, cf. e.g. the discussion in \cite[Section 4]{So-L}. Since the realization functors are compatible with the six functors, it suffices to show that local systems lift to the motivic categories $\DMT_K(Y)$ where $Y=K/K'$ denotes the $K$-orbit in $G/P$. Using the induction  equivalence of Proposition~\ref{cor:indequiv}, we have $\DMT_K(Y)\approx\DMT_{K'}(\op{pt})$ and therefore the $K$-equivariant local systems on $Y$ correspond exactly to the motivic representations of the finite group of components of the stabilizer group $K'$. Now, by the induction equivalence for the ordinary equivariant derived categories, a local system (in the topological sense) on $Y$ corresponds similarly to a representation of the component group of $K'$. For such a representation, we then have obvious lifts to the motivic categories, and this establishes (3). 

Now we prove (1), (2) is proved similarly. Let $M$ be a clean motive on a $B$-orbit $j:V\to G/K$. Then $j_\ast M\cong j_!M$ by definition, hence its realization  is an intersection cohomology sheaf and therefore perverse. Now the other steps in the construction of Bott--Samelson motives, cf. Definitions~\ref{def:BSclosure} and \ref{BSmot}, preserve the required perversity. Therefore, Bott--Samelson motives will map to direct sums of intersection cohomology complexes and by Theorem~\ref{cor:symmetricWT} and Proposition~\ref{prop:symmpurity} this implies that Bott--Samelson motives are sent to direct sums of even shifts of intersection complexes. Note that the shifts we obtain are only even, this is the realization of the operation $M\mapsto M(1)[2]$. We get every intersection complex as a direct summand because we have explicit lifts of the relevant local systems on orbits from (3). 

(5) Realization is an exact functor, and faithfulness follows from conservativity of realization on equivariant mixed Tate motives, cf. Proposition~\ref{prop:conservative}. By (3) we have lifts for the irreducible objects in $\op{Der}^{\op{b}}_{P\times K}(G)$ given by Bott--Samelson motives. As in the proof of Theorem~\ref{thm:gradedparabolic}, for two equivariant mixed Tate motives $M,N\in \DMT_{P\times K}(G)$ we can compute 
\[
\bigoplus_{i\in\mathbb{Z}}\DMT_{P\times K}(M,N(i))\xrightarrow{\cong} \op{Der}^{\op{b}}_{P\times K}(\op{Real} M,\op{Real} N)
\]
by writing the motives as iterated extensions of standard and costandard objects from (3) and reducing the claim to morphisms between those with a standard devissage argument. This reduces us to the claim for morphisms between motivic local systems on $K$-orbits of $G/P$. Again, this reduces to a full faithfulness claim for motives (equivariant for the stabilizer of the $K$-orbit) on the point where we can employ conservativity, cf. Proposition~\ref{prop:conservative}. 

Finally, adjoining a root of the Tate twist on the motivic side is necessary to get all shifts of intersection complexes, not just the even ones. 
\end{proof}

Now we want to show that the motivic categories $\DMT_{P\times K}(G)$ provide a categorification of the module over the Schur algebroid considered in \cite{Mars-Springer}. Actually, only the Hecke module $\DMT_{B\times K}(G)$ was considered in loc.cit., but once the idea of Schur algebroid is formulated explicitly, the results in \cite{Mars-Springer} can be extended to this setting. We first formulate explicitly the statements about the convolution functors, this implies that the Grothendieck groups of the motivic categories $\DMT_{P\times K}(G)$ will form a module for the Schur algebroid. 

\begin{proposition}
\label{prop:schurmodulesymm}
Assume the situation of \ref{symmsetup}. 
\begin{enumerate}
\item We have, for every pair of standard parabolic subgroups $P,Q\subset G$, convolution functors 
\[
-\star_Q-:\DMT_{P\times Q}(G)\times \DMT_{Q\times K}(G)\to \DMT_{P\times K}(G).
\]
\item 
These functors restrict to the hearts of the relevant weight structures:  
\[
-\star_Q-:\DMT_{P\times Q}(G)_{\op{wt}=0}\times \DMT_{Q\times K}(G)_{\op{wt}=0}\to \DMT_{P\times K}(G)_{\op{wt}=0}.
\]
\item The above convolutions are compatible with the tilting functor of Proposition~\ref{prop:symmpurity}:
\[
\op{tilt}: \op{Hot}^{\op{b}}(\DMT_{P\times Q}(G)_{\op{wt}=0})\to \DMT_{P\times Q}(G)
\]
\item On the level of Grothendieck groups, these can be identified, via the isomorphisms from Theorem~\ref{thm:comparisonMS}, with the module structure for the Schur algebroid developed in \cite{Mars-Springer}.
\end{enumerate}
\end{proposition}

\begin{proof}
(1) The first statement is essentially the content of Proposition~\ref{prop:convtilt}, together with the associativity of convolution, cf. Proposition~\ref{prop:convassoc}. 

(2) Following the argument of Proposition~\ref{prop:convtilt}, we can actually show that the convolution functors $-\star_Q-$ can be restricted to the hearts of the relevant weight structures, cf. Section~\ref{sec:weights}. First, the exterior product $M\boxtimes N=\op{pr}_1^\ast M\otimes \op{pr}_2^\ast N$ of pure weight $0$ motives is pure of weight $0$: the projections $\op{pr}_i$ are smooth and therefore the restriction functors are weight-exact, cf. Proposition~\ref{prop:wtres}, and the tensor product of pure weight $0$ motives is again pure of weight $0$, cf. Proposition~\ref{prop:wttensor}. The restriction along the diagonal $P\times Q\times K\hookrightarrow P\times Q\times Q\times K$ is weight-exact by Proposition~\ref{prop:wtres}. Then the weight structures are constructed in such a way that they are compatible with the quotient equivalence of Proposition~\ref{prop:quotientequiv}, and finally the remaining proper pushforward $\op{mult}_!$ is also weight-exact by Propositions~\ref{prop:wtres} and \ref{prop:wtpush} since the relevant multiplication map $G\times_{/Q}G\to G$ is a $G/Q$-fiber bundle and hence projective. 

(3) Compatibility with the tilting follows from the general results Theorem~\ref{thm:funtilt} and \ref{thm:tiltmonoid}, noting that all the functors in the definition of convolution, cf. Definition~\ref{def:convolution}, are left adjoints or monoidal structures. 

(4) The final statement concerning the categorification of the module structure from \cite{Mars-Springer}, it is clear from the definition in Section 3.2 of loc.cit. that the module structure is given by convolution (resp. the induced map on Grothendieck groups). The claim then follows from the fact that the realization functors in the diagram in \ref{symmdiag} are compatible with the full six-functor formalism, cf. Theorem~\ref{thm:gradedsymm}.
\end{proof}

Now we can formulate the compatibility of the above module structure with the one discussed in \cite{Mars-Springer} via the realization functors: 

\begin{theorem}
\label{thm:comparisonMS}
On the level of Grothendieck groups, the $\ell$-adic case of the diagram in \ref{symmdiag} induces the following commutative diagram of  Hecke-modules
\[
\xymatrix{
\op{K}_0^{\oplus}\left(\DMT_{B\times K}(G)_{\op{wt}=0}\right) \ar[rr]^{\op{Real}_\ell} \ar[d]_{\cong}^{\op{tilt}} & &
\op{K}_0^{\oplus}(\mathcal{A}_{G/K})\ar[d]^\cong \\
\op{K}_0(\DMT_{B\times K}(G))\ar[rr]_{\mathbb{H}^\bullet(\op{Real}_\ell)} && \op{K}_0(\mathcal{C}_{G/K}),
}
\]
where $\op{K}_0(\mathcal{A}_{G/K})$ and $\op{K}_0(\mathcal{C}_{G/K})$ are the Hecke-modules considered in \cite{Mars-Springer}. The horizontal maps become isomorphisms after extending scalars to $\mathbb{Z}[C]$ on the left-hand side.
\end{theorem}

\begin{proof}
The Hecke module structures on the left-hand side have been established in Proposition~\ref{prop:schurmodulesymm}, and the module structures on the right-hand side are the ones from \cite{Mars-Springer}. The left vertical isomorphism $\op{tilt}$ is a consequence of Proposition~\ref{prop:symmpurity}, and the right vertical isomorphism is established in \cite{Mars-Springer}. 

We want to show that the horizontal maps induced from $\ell$-adic realization functors are isomorphisms, after extending scalars to the ring $\mathbb{Z}[C]$ considered in \cite{Mars-Springer}. By Theorem~\ref{thm:gradedsymm} (1), the real realization $\op{Real}_\ell$  maps Bott--Samelson motives to direct sums of intersection complexes. This implies surjectivity: from the definition in \cite[Section 3.1]{Mars-Springer}, the group $\op{K}_0(\mathcal{A}_Y)$ is generated by the classes of intermediate extensions $j_{\ast!}(\mathcal{L}_Z[\dim Z])$ for $\mathcal{L}_Z$ a local system on the $K$-orbit $Z$ in $G/B$. After extension of scalars on the left-hand side (to ensure existence of odd-degree shifts) these intermediate extensions lift to Bott--Samelson motives by Theorem~\ref{thm:gradedsymm} (1). 

To prove injectivity, we note that the motivic Hecke modules on the left-hand side of the diagram are free $\mathbb{Z}[q,q^{-1}]$-modules of finite rank. A basis is given by the irreducible motivic local systems on the $K$-orbits of $G/B$. As in the proof of (5) of Theorem~\ref{thm:gradedsymm}, the conservativity on equivariant mixed Tate motives, cf. Proposition~\ref{prop:conservative}, implies that we have a bijection between the motivic local systems and the $\ell$-adic local systems, induced by the realization functor. In particular, the realization functor $\op{Real}_\ell:\op{K}_0\left(\DMT_{B\times K}(G)_{\op{wt}=0}\right)\to \op{K}_0(\mathcal{A}_{G/K})$ maps a basis of the motivic Hecke module to a basis of the Hecke module of Mars--Springer. This implies the required injectivity. 
\end{proof}

\subsection{Application to representation theory} 

We discuss some of the representation-theoretic consequences of the previous results. The first point is that the tilting results in Section~\ref{sec:tiltingapp} are statements about the formality of the equivariant derived categories. In particular, we are able to establish the conjecture of Soergel and Lunts for symmetric varieties, cf. \cite{So-L} and \cite{Lu-tor}. A different approach to formality of equivariant derived categories has been explored recently in (yet unpublished) work of Brion and Joshua \cite{brion:joshua}. 

\begin{Bemerkung}
\label{SLconj}
  To formulate the result,  we need to set up some notation, cf. \cite[0.1]{Lu-tor}. Let $B\looparrowright G/K$ be a symmetric variety with action by a Borel subgroup $B\subseteq G$. Let $\mathcal{L}_1,\dots,\mathcal{L}_m$ be the collection of isomorphism classes of simple equivariant perverse sheaves on $G/K$, and put $\mathcal{L}=\bigoplus\mathcal{L}_i$. Consider the Ext-algebra $A=\op{Ext}^\bullet_{\op{Der}_B(G/K)}(\mathcal{L},\mathcal{L})^{\op{op}}$ and the projective $A$-modules $Q_i=A e_i$ where $e_i:\mathcal{L}\to\mathcal{L}_i$ is the projection. Then we have a dg-algebra $\mathcal{A}=(A,d=0)$ and denote by $\op{Der}_{\mathcal{A}}^f$ the subcategory of the derived category $\op{Der}_{\mathcal{A}}$ of dg-modules over $\mathcal{A}$ generated by the dg-$\mathcal{A}$-modules $Q_i$. 
\end{Bemerkung}

\begin{theorem}[Soergel--Lunts conjecture]
\label{thm:slsymm}
\index{Soergel--Lunts conjecture}
In the situation of \ref{SLconj}, there is a natural equivalence of triangulated categories
\[
\op{Der}^{\op{b}}_B(G/K)\approx\op{Der}_{\mathcal{A}}^f.
\]
\end{theorem}

\begin{proof}
Recall from \cite[0.3]{Lu-tor} that, parallel to \ref{SLconj}, the equivariant derived category $\op{Der}^{\op{b}}_B(G/K)$ is generated by finitely many objects $F_1,\dots,F_n$ and we can consider the dg-algebra $\mathcal{B}=\op{End}^\bullet_{\op{Der}_B(G/K)}(\bigoplus F_i,\bigoplus F_i)^{\op{op}}$. There are also dg-modules $P_i=\mathcal{B}e_i$ for the projection $e_i:F\to F_i$. Denoting by $\op{Der}^f_{\mathcal{B}}$ the subcategory of the derived category of dg-$\mathcal{B}$-modules generated by $P_i$, \cite[Proposition 0.3.1]{Lu-tor} establishes a natural equivalence of triangulated categories
\[
\op{Der}_B(G/K)\approx \op{Der}^f_{\mathcal{B}}.
\]
The claim follows if we can prove that the natural inclusion $\mathcal{A}\hookrightarrow\mathcal{B}$ is a quasi-isomorphism. 

Now we can apply Theorem~\ref{thm:gradedsymm} (5) which states that $\DMT_{B\times K}(G)$ is a graded version of the equivariant derived category $\op{Der}^{\op{b}}_{B\times K}(G)$. This allows to compute
\[
\op{End}_{\op{Der}_{B\times K}(G)}\left(F,F\right)\cong \bigoplus_{j\in\mathbb{Z}}\DMT_{B\times K}(\tilde{F},\tilde{F}(j))
\]
where $\tilde{F}$ is a lift of the generator $F=\bigoplus F_i$ to $\DMT_{B\times K}(G)$. On the right-hand side, only one of the summands for $j\in\mathbb{Z}$ can be non-trivial for any indecomposable object, due to weight reasons. In particular, the endomorphism-dg-algebra $\op{End}_{\op{Der}_{B\times K}(G)}\left(F,F\right)$ is formal. But this implies that $\mathcal{A}\hookrightarrow\mathcal{B}$ is a quasi-isomorphism which proves the formality conjecture. 
\end{proof}

As a second application to representation theory, the results about equivariant mixed Tate motives for symmetric varieties allow to prove Conjectures 4.2.2 and 4.2.3 of \cite{So-L}. 

\begin{Bemerkung}
We recall the relevant setup. The goal is to study the representations of a real Lie group $G_{\mathbb{R}}$ with maximal compact subgroup $K_{\mathbb{R}}$. However, the geometric approach to representation theory of $G_{\mathbb{R}}$ considers the complexified group $G=G_{\mathbb{R}}\times_{\mathbb{R}}\mathbb{C}$ with maximal subgroup $K=K_{\mathbb{R}}\times_{\mathbb{R}}\mathbb{C}$. By the Matsuki correspondence, the orbits of $G_{\mathbb{R}}$ on the complex flag variety $G/B$ are in bijective correspondence with the $K$-orbits on the flag variety $G/B$. The relation between representations of the real Lie group (actually, Harish-Chandra-modules with trivial central character) and perverse sheaves on $K$-orbits on $G/B$ is discussed in \cite{BB} and more generally (for the derived categories and arbitrary characters) in \cite{BeLuK}. 

The geometric side of the conjectural Koszul duality for representations of real Lie groups from \cite{So-L} considers the action of the reductive complex group on suitable varieties $X(\chi)$. In a more recent formulation, cf. \cite{soergel:slides}, the relevant varieties with action are of the form $B\looparrowright \mathcal{Z}^1_\gamma(\op{Gal}(\mathbb{C}/\mathbb{R}),G)$ where a Borel subgroup of $G$ acts on the space of 1-cocycles for the Galois group, twisted by a holomorphic involution $\gamma$ of $G$.\footnote{More precisely, we should probably be talking about the Langlands dual groups here; but as we are not discussing Koszul duality issues, this doesn't really matter.} Since the stabilizer in $G$ of such a 1-cocycle is a symmetric subgroup, the variety with action above breaks up into a disjoint union of varieties $B\looparrowright G/K'$ where $K'$ is a symmetric subgroup (these correspond to complexifications of maximal compacts for inner forms of $G_{\mathbb{R}}$). In the original formulation of \cite{So-L}, the varieties considered were $G\looparrowright G\times_BG/K'$ so that we have equivalences of equivariant derived categories $\op{Der}_G(G\times_{/B}G/K')\approx\op{Der}_B(G/K')$. In particular, the symmetric variety case discussed in this section is exactly the relevant case for the Koszul duality question in the representation theory of real Lie groups. 
\end{Bemerkung}

\begin{Bemerkung}
Let now $B\looparrowright G/K$ (or $G\looparrowright G\times_{/B}G/K$) be a symmetric variety as considered above. As in \ref{SLconj}, we define the geometric extension algebra $\op{Ext}_G(X)=\op{Ext}^\bullet_G(\mathcal{L},\mathcal{L})$, where $\mathcal{L}$ is the direct sum of the finitely many simple $G$-equivariant perverse sheaves on $X$. 
\end{Bemerkung}

\begin{theorem}[geometric part of Soergel's conjectures]
\label{thm:abvconjecture}
Adjoining a square root of the Tate twist, the category $\DMT_G(X)$ satisfies the requirements for $\mathcal{D}_g$ in \cite[Conjecture 4.2.2 and 4.2.3]{So-L}. More precisely:
\begin{enumerate}
\item The functor $\DMT_G(X)\to\op{Der}^{\op{b}}_G(X)$ is a degrading functor in the sense of \cite{BGSo}. 
\item Motivic lifts of the standard and costandard objects in $\op{Der}^{\op{b}}_G(X)$ are given by $i_\ast(\mathcal{L})$ and $i_!(\mathcal{L})$, respectively, where $i:PxK\hookrightarrow G$ is the inclusion of a $P\times K$-double coset in $G$ and $\mathcal{L}$ is a  $K$-equivariant local system on the corresponding $K$-orbit of $G/P$. These lifts are unique up to Tate twist. 
\item There is a fully faithful embedding $\DMT_G(X)\hookrightarrow \op{Der}^-(\op{Ext}_G(X)\op{-mod})$, related to the ``geometric t-structure'' on the derived category of the geometric extension algebra.
\end{enumerate}
\end{theorem}

\begin{proof}
Actually the statements (1) and (2) have already been established in Theorem~\ref{thm:gradedsymm}. It remains to note that Conjecture 4.2.2 of \cite{So-L} is really the existence of a graded version of the equivariant derived category $\op{Der}_B(G/K)$ together with the existence of lifts of the standard objects. Conjecture 4.2.3 of \cite{So-L} is the existence of lifts of the costandard objects. The uniqueness required in the conjecture follows from the uniqueness of the lifts of standard and costandard objects up to Tate twists together with the fact that the nontriviality of the Hom implies that there is no choice of Tate twist. 

For (3), the fully faithful embedding arises from 
\[
\DMT_G(X)_{\op{wt}=0}\hookrightarrow \op{Der}^-(\op{Ext}_G(X)\op{-mod})
\]
which maps a weight $0$ motive to the cohomology of its realization (similar to the constructions in Sections~\ref{sec:motcohom} and \ref{sec:tilting}), viewed as a module over the geometric extension algebra. Alternatively, we can map a motive into the motivic lift of the direct sum of the simple objects $\mathcal{L}_x$. This embeds $\DMT_G(X)_{\op{wt}=0}$ as a tilting subcategory, and we can apply Theorem~\ref{thm:derivatortilting} to extend this embedding to a fully faithful embedding of $\op{Hot}^{\op{b}}(\DMT_G(X)_{\op{wt}=0})$. The latter is, by Proposition~\ref{prop:symmpurity}, equivalent to $\DMT_G(X)$ and this proves the claim. 
\end{proof}

\begin{Bemerkung}
There is one remaining claim in the geometric part of Soergel's conjectures \cite{So-L}: the essential image of the embedding 
\[
\DMT_G(X)_{\op{wt}=0}\hookrightarrow \op{Der}^-(\op{Ext}_G(X)\op{-mod})
\]
should be identified with the heart of the geometric t-structure on the derived category of the geometric extension algebra. 
\end{Bemerkung}

\begin{Bemerkung}
We don't discuss explicitly how the Koszul duality conjectures should model Vogan's character duality for real Lie groups. Further information on combinatorics in the representation theory of real Lie groups and Kazhdan--Lusztig--Vogan polynomials can be found in \cite{lusztig:vogan,Vo4,ABV,So-L}, cf. also \cite[3.10]{springer:schubert}. 

Another direction, relating motivic theories with representation theory of real Lie groups is opened up by work of Casian and Stanton \cite{casian:stanton}. Their work relates the integral cohomology of real flag varieties with the geometric approach to representation theory of real Lie groups via Beilinson--Bernstein localization of Harish-Chandra-modules \cite{BB}. The motivic analogue of the integral cohomology of real flag varieties is given by the Chow--Witt rings of flag varieties (over $\mathbb{R}$ or more general fields) which have been studied and partially computed recently. More work needs to be done to understand those connections properly.
\end{Bemerkung}


\section{Hecke modules for wonderful compactifications}
\label{sec:wonderful}

In this section, we provide a model for the Hecke module associated in \cite{SpCompact} to a wonderful compactification, which we will call the \emph{wonderful Hecke module}. The model uses equivariant mixed Tate motives, and we show that the $\ell$-adic realization induces an isomorphism from the category of $B\times B$-equivariant mixed Tate motives on the wonderful compactification to Springer's Hecke module. 

For the comparison with Springer's Hecke module of \cite{SpCompact}, Remark~\ref{rem:fq} applies again: Springer's arguments work over finite fields while our motivic categories usually work over algebraically closed fields. But the descent statements for equivariant mixed Tate motives imply that our descriptions of motivic categories work equally well over finite fields and thus can be compared to Springer's constructions via the $\ell$-adic realization functors. 

\begin{Bemerkung}
\label{wonderfulsetup}
Let $G$ be an adjoint semisimple group and let $G\times G\looparrowright X$ be a wonderful compactification of $G$. Assume that $\mathbb{D}=\mathbf{DA}^{\et}(-;\Lambda)$ or $\mathbb{D}=\op{MDer}(-;\mathbb{C})$ and denote by $\op{Real}$ the appropriate realization functor. For fixed standard parabolic subgroups $P,Q\subset G$, consider the following diagram of Grothendieck groups:
\[
\xymatrix{
\op{K}_0(\op{Hot}^{\op{b}}(\DMT_{P\times Q}(X)_{\op{wt}=0})) \ar[rr]^{\op{Real}} \ar[d]_{\op{tilt}} && \op{K}_0(\op{Hot}^{\op{b}}(\op{Per}^{\op{ss}}_{P\times Q}(X;\Lambda))) \ar[d]^{\mathcal{H}} \\
\op{K}_0(\DMT_{P\times Q}(X)) \ar[rr]_{\op{Real}} && \op{K}_0(\op{Der}^{\op{b}}_{P\times Q}(X;\Lambda)). 
}
\]
On the left-hand side, we have the motivic categories of Section~\ref{sec:tiltingapp}; on the right-hand side, we have the categories of equivariant perverse sheaves on $X$ and the equivariant derived category considered in \cite{SpCompact}. The left vertical functor is the tilting functor of Proposition~\ref{prop:wonderfulpurity} and the horizontal functors are the realization functors of Proposition~\ref{prop:realization}. The other functors are set up in the same way as in the case of symmetric varieties, cf. \ref{symmdiag}. 
\end{Bemerkung}

\begin{theorem}
\label{thm:gradedwonderful}
Assume the situation in \ref{wonderfulsetup}
\begin{enumerate}
\item Let $k=\mathbb{F}_q$ be a finite field of odd characteristic, and let $\mathbb{D}=\mathbf{DA}^{\et}(-;\mathbb{Q}_\ell)$ be the homotopical stable algebraic derivator underlying the construction of equivariant motives. Let $G$ be an adjoint semisimple group, let $G\times G\looparrowright X$ be a wonderful compactification, and let $P\times Q$ be a parabolic subgroup of $G\times G$. Then the $\ell$-adic realization functor
\[
\op{Real}_\ell:\DMT_{P\times Q}(X)\to \op{Der}_{P\times Q}(X;\mathbb{Q}_\ell)
\] 
is fully faithful on the heart of the weight structure, i.e., the category of Bott--Samelson motives. The essential image consists of intersection complexes concentrated in even degrees. 
\item Let $k=\mathbb{C}$ and consider  $\mathbb{D}=\op{MDer}(-;\mathbb{C})$ as homotopical stable algebraic derivator underlying the construction of equivariant motives. Let $G$ be an adjoint semisimple group, let $G\times G\looparrowright X$ be a wonderful compactification, and let $P\times Q$ be a parabolic subgroup of $G\times G$. Then the Hodge realization functor
\[
\op{Real}_{\op{H}}:\DMT_{P\times Q}(X)\to \op{Der}_{P\times Q}(X;\mathbb{C})
\] 
is fully faithful on the heart of the weight structure, i.e., the category of Bott--Samelson motives. The essential image consists of intersection complexes concentrated in even degrees. 
\item Motivic lifts of the standard and costandard objects in the equivariant derived categories above are given by $i_\ast(\mathcal{L})$ and $i_!(\mathcal{L})$, respectively, where $i:PxQ\hookrightarrow X$ is the inclusion of a $P\times Q$-orbit in $X$ and $\mathcal{L}$ is a $P\times Q$-equivariant local system (cf. \cite[Section 5]{SpCompact}). 
\item The functors in points (1) and (2) are compatible with convolution and Verdier duality. 
\item Adjoining a root of the Tate twist, the equivariant mixed Tate motives provide a grading on the constructible equivariant derived categories in the sense of \ref{bgs43}, with the realization functors in (1) and (2) as degrading functors. 
\end{enumerate}
\end{theorem}

\begin{proof}
The arguments are the same as in the proof of Theorem~\ref{thm:gradedsymm}, with the occasional replacement of the tilting and purity results for the symmetric case by Theorem~\ref{thm:wonderfulWT} and Proposition~\ref{prop:wonderfulpurity}. 
\end{proof}

\begin{proposition}
\label{prop:schurmodulewonderful}
Assume the situation of \ref{wonderfulsetup}
\begin{enumerate}
\item We have, for standard parabolic subgroups $P_1,P_2,Q_1,Q_2\subset G$, convolution functors
\[
-\star_Q-:\DMT_{(P_1\times P_2)\times (Q_1\times Q_2)}(G\times G)\times \DMT_{P_2\times Q_2}(X)\to \DMT_{P_1\times Q_1}(X). 
\]
\item These functors restrict to the hearts of the relevant weight structures.
\item The above convolution functors are compatible with the tilting functor of Proposition~\ref{prop:wonderfulpurity}:
\[
\op{tilt}:\op{Hot}^{\op{b}}(\DMT_{P\times Q}(X)_{\op{wt}=0})\to \DMT_{P\times Q}(X).
\]
\item On the level of Grothendieck groups, the case $P=Q=B$ can be identified, via the isomorphisms of Theorem~\ref{thm:comparisonSpCompact}, with the Hecke module structure from \cite{SpCompact}. 
\end{enumerate}
\end{proposition}

\begin{proof}
This is proved the same way as Proposition~\ref{prop:schurmodulesymm}, but using Theorem~\ref{thm:gradedwonderful} for the identification of the Hecke module structure with the ones of \cite{SpCompact}. 
\end{proof}

\begin{theorem}
\label{thm:comparisonSpCompact}
On the level of Grothendieck groups, the $\ell$-adic realization functor induces the following commutative diagram  of Hecke-modules
\[
\xymatrix{
\op{K}_0^{\oplus}\left(\DMT_{B\times B}(X)_{\op{wt}=0}\right) \ar[rr]^{\op{Real}_\ell} \ar[d]_{\cong}^{\op{tilt}} & &
\op{K}_0^{\oplus}(\mathcal{A}_{X})\ar[d]^\cong \\
\op{K}_0(\DMT_{B\times B}(X))\ar[rr]_{\mathbb{H}^\bullet(\op{Real}_\ell)} && \op{K}_0(\mathcal{C}_{X}),
}
\]
where $\op{K}_0(\mathcal{A}_{X})$ and $\op{K}_0(\mathcal{C}_{X})$ are the Hecke-modules considered in \cite{SpCompact}. The horizontal maps become isomorphisms after extending scalars to $\mathbb{Z}[C]$ on the left-hand side, cf. \cite[Section 3]{SpCompact}. 
\end{theorem}

\begin{proof}
Again, the proof is closely parallel to the one of Theorem~\ref{thm:comparisonMS}, using Theorem~\ref{thm:gradedwonderful} and Proposition~\ref{prop:schurmodulewonderful}. 
\end{proof}

\begin{theorem}[Soergel--Lunts conjecture]
There is a natural equivalence of categories 
\[
\op{Der}^{\op{b}}_{P\times Q}(X)\approx\op{Der}^f_{\mathcal{A}},
\]
where the category $\op{Der}^f_{\mathcal{A}}$ is the one discussed in \cite{Lu-tor}, cf. also \ref{SLconj}, obtained from the simple equivariant perverse sheaves on the wonderful compactification $G\times G\looparrowright X$. 
\end{theorem}

\begin{proof}
Similar to the proof of Theorem~\ref{thm:slsymm}, using the graded versions from Theorem~\ref{thm:gradedwonderful}. 
\end{proof}

\begin{Bemerkung}
Another approach to the formality of the equivariant derived category for symmetric varieties has been explored by Brion and Joshua \cite{brion:joshua}. 
\end{Bemerkung}

\appendix

\chapter{Motives of classifying and homogeneous spaces}
\renewcommand{\thesection}{A.\arabic{section}}

In this appendix, we provide a recollection on the structure of motives of classifying spaces of reductive groups as well as motives of homogeneous spaces. Most important for us will be that these are mixed Tate motives, which is relevant for having induction functors on categories of equivariant mixed Tate motives. We would like to point out at the very beginning that we work, as always, with rational coefficients. In fact, some of the assertions about motives being mixed Tate are wrong integrally or with finite coefficients. 

\section{Bar constructions and their motives} 
\label{sec:bar}

\begin{definition}
\label{def:twosided}
\index{bar construction}
\index{classifying space}
Let $k$ be a field, and let $G$ be an algebraic group over $k$. Let $(X\looparrowleft G)$ and $(G\looparrowright Y)$ be varieties with right and left $G$-actions, respectively. Then the \emph{two-sided bar construction} is the following simplicial variety 
\[
{\op{B}}(X,G,Y)_n:=X\times G^n\times Y.
\]
\end{definition}

\begin{example}
The universal $G$-bundle can be written as the natural morphism 
\[
{\op{E}}G={\op{B}}(\op{pt},G,G)\to {\op{B}}(\op{pt},G,\op{pt})={\op{B}}G
\]
induced from the morphism $\op{fin}_G:G\to\op{pt}$. The Borel construction for a variety $G\looparrowright X$, cf. Definition~\ref{def:borel}, is also a special case ${\op{E}}G\times_{/G}X={\op{B}}(\op{pt},G,X)$. 

More specifically, if $G$ is an algebraic group with a closed subgroup $H\leq G$, the bar construction ${\op{B}}(G,H,\op{pt})$ for the variety with action $G\looparrowleft H$ is a model for the homogeneous space $G/H$. The claim that the natural projection ${\op{B}}(G,H,\op{pt})\to G/H$ is a weak equivalence in the $\mathbb{A}^1$-homotopy category (alternatively, induces an isomorphism of motives) is established in the proof of Proposition~\ref{prop:freeaction}. 
\end{example}

\begin{lemma}
\label{lem:bgtate}
Let $k$ be a field and let $\mathbb{D}$ be a homotopical stable algebraic derivator satisfying the conditions of~\ref{derivator:new}. Let $G$ be an algebraic group over $k$ and let $(X\looparrowleft G)$ and $(G\looparrowright Y)$ be varieties with right and left $G$-actions, respectively. Assume that $\op{M}(G), \op{M}(X),\op{M}(Y)\in \DMT(\op{pt})$, i.e., all varieties involved have mixed Tate motives. Then we have  $\op{M}({\op{B}}(X,G,Y))\in \op{Ind-}\DMT(\op{pt})$, i.e., the two-sided bar construction  is an Ind-mixed Tate motive. 
\end{lemma}

\begin{proof}
By Definition~\ref{mtder}, the category $\DMT(\op{pt})$ is closed under triangles and tensor products. Passing to the Ind-objects implies that we additionally acquire closure under countable colimits. In particular, countable colimits exist in the category $\op{Ind-}\DMT(\op{pt})$ of Ind-mixed Tate motives. 

Now the functor $\op{M}$ associating to a smooth variety $X$ its motive $\op{M}(X)\in\mathbb{D}(\op{pt})$ can be extended to smooth simplicial varieties $X_\bullet$ as follows: we apply the functor $\op{M}$ termwise to get a simplicial motive $\mathbf{\Delta}^{\op{op}}\to\mathbb{D}(\op{pt}):[n]\mapsto\op{M}(X_n)$, and then the realization (alternatively, the homotopy colimit) is the motive $\op{M}(X_\bullet)$ of the simplicial object. In particular, the motive of the simplicial object is a (countable) homotopy colimit of a diagram whose individual terms are of the form $\op{M}(X_n)$. As discussed in the first paragraph, the category of Ind-mixed Tate motives is closed under countable homotopy colimits, hence we see that $\op{M}(X_\bullet)\in \op{Ind-}\DMT(\op{pt})$ if $\op{M}(X_n)\in\DMT(\op{pt})$ for all $n\in\mathbb{N}$.  

Under the assumptions on the motives of $G$, $X$ and $Y$, we know that the individual terms $X\times G^n\times Y$, $n\in\mathbb{N}$ in the simplicial variety ${\op{B}}(X,G,Y)$ are all mixed Tate because products of varieties correspond to tensor products of their motives. By the previous discussion, $\op{M}({\op{B}}(X,G,Y))\in\op{Ind-}\DMT(\op{pt})$ which proves our claim. 
\end{proof}

\begin{Bemerkung}
The motive of a linear group $G$ over a perfect field $k$ is well-known to be mixed Tate. Note that for any linear group $G$, the natural projection $G\to G/{\op{R}_{\op{u}}}G$ to the reductive group $G/{\op{R}_{\op{u}}}G$ induces an isomorphism of motives. Moreover, if $\pi_0(G)=G/G^0$ is a direct sum of $k$-points, then $G$ is a disjoint union of copies of $G^0$, hence the motive $\op{M}(G)$ is a direct sum of copies of $\op{M}(G^0)$. Now, for a split reductive group $G$ over a perfect field $k$, the motive $\op{M}(G)$ with rational coefficients was computed in \cite[Theorem 6.1]{biglari}. From his computations, it follows in particular that the motives of reductive groups are mixed Tate. Note that while Biglari's computations are done in Voevodsky's category of motives, similar formulas can be obtained by the same methods for any homotopical stable algebraic derivator $\mathbb{D}$ satisfying the conditions of~\ref{derivator:new}. 
\end{Bemerkung}

\begin{remark}
It should be pointed out that rational coefficients are really necessary here, cf. \cite{totaro:motivebg}. There are finite groups where the motive (with integral coefficients) of the classifying space fails to be mixed Tate. This can be detected e.g. by noting that the corresponding approximations of the classifying space $\mathbb{A}^n\sslash G$ fail to be stably rational and have non-trivial unramified invariants. Of course, with rational coefficients, the motive of the classifying space of a finite group is trivial. 
\end{remark}

\begin{corollary}
\label{prop:tatehomogeneous}
Let $k$ be a perfect field and let $G$ be a split reductive group over $k$. Let $\mathbb{D}$ be a homotopical stable algebraic derivator satisfying the conditions of~\ref{derivator:new}. 
\begin{enumerate}
\item The motive $\op{M}({\op{B}}G)$ is Ind-mixed Tate.
\item For a closed subgroup $H\leq G$, the motive $\op{M}(G/H)\cong \op{M}({\op{B}}(G,H,\op{pt}))$ is mixed Tate.
\end{enumerate}
\end{corollary}

\begin{proof}
(1) follows from Lemma~\ref{lem:bgtate}, using the fact that the motive $\op{M}(G)$ is mixed Tate. (2) also follows from Lemma~\ref{lem:bgtate}. Both $\op{M}(G)$ and $\op{M}(H)$ are mixed Tate by Biglari's result. The isomorphism of motives follows from Proposition~\ref{prop:freeaction}. It remains to note that $\op{M}(G/H)$ is a geometric motive because $G/H$ is a smooth scheme, so that $\op{M}(G/H)$ is actually a mixed Tate motive (not a general Ind-object).
\end{proof}

\begin{Bemerkung}
Having a good model ${\op{B}}(G,H,\op{pt})$ for the homogeneous space $G/H$, one could also set up a motivic version of the Eilenberg--Moore spectral sequence, by considering suitable filtrations of the two-sided bar construction. This would allow to compute motives of homogeneous spaces (with rational coefficients) from the knowledge of the induced morphisms $\op{M}(H)\to\op{M}(G)$. An Eilenberg--Moore spectral sequence for motivic cohomology with rational coefficients was discussed in \cite{krishna:em}. 
\end{Bemerkung} 

\begin{proposition}
\label{prop:tatelocalsystem}
Let $G\looparrowright X$ be a homogeneous variety with $G$-action. If $M\in \DMT_G(X)$, then $\op{fin}_!\op{For}(M) \in \DMT(\pt)$.
\end{proposition}

\begin{proof}
By assumption $X=G/H$, and the motive $\op{M}(G/H)$ is mixed Tate by Corollary~\ref{prop:tatehomogeneous}. In particular $\op{fin}_\sharp\underline{X}\in \DMT(\op{pt})$. Relative purity implies $\op{fin}_!\underline{X}\in \DMT(\op{pt})$, and then also $\op{fin}_\ast\underline{X}\in\DMT(\op{pt})$ by Verdier duality. If $H$ is connected, any equivariant mixed Tate motive $M\in \DMT_G(X)$ is an extension of constant motives $\const{X}$, possibly twisted and shifted. In particular, for such a motive we will have $\op{fin}_!\op{For}(M)\in \DMT(\pt)$. If $H$ is not connected, then $p:G/H^0\to G/H$ is a finite \'etale map. By \cite[Lemme 2.1.165]{ayoub:thesis1}, the composition of canonical transformations $\op{id}\Rightarrow p_\ast p^\ast\Leftarrow p_!p^!\Rightarrow\op{id}$ is invertible, given by multiplication with the degree of $p$. In particular, $M$ appears as a direct summand of $\op{Ind}_{H^0}^H\op{Res}_H^{H^0} M\cong p_\ast p^\ast M$. But from the connected case, $\op{fin}_! p_\ast p^\ast M$ is the pushforward of a mixed Tate motive on $G/H^0$; so this is a mixed Tate motive by the connected case discussed above. Since $\op{fin}_!M$ is a direct summand of $\op{fin}_! p_\ast p^\ast M$, we have proved the claim. 
\end{proof}

\section{Motives of classifying spaces}
\label{sec:motivebg}

In this section, we will provide motivic versions of well-known basic computations related to groups and their classifying spaces, describing the equivariant cohomology rings as well as the behaviour of equivariant cohomology under the six functor formalism. Most of this is well-known, cf. e.g.~\cite{totaro,totaro:motivebg}. This information will be relevant input for the tilting results in Sections~\ref{sec:motcohom} and \ref{sec:tilting}, allowing to describe the category of equivariant mixed Tate motives over a point in terms of (a homotopy category of) modules over the equivariant cohomology ring.

In this section we work with varieties (satisfying the standing assumptions~\ref{standing}) over a perfect field $k$ and a derivator $\mathbb{D}$ satisfying the conditions of~\ref{derivator:new}. The coefficient ring of the derivator $\mathbb{D}$ will be denoted by $\Lambda$, and assumed to be a field of characteristic zero. 

\subsection{Algebraic preliminaries} 

We shortly discuss the coinvariant algebras and twisted group rings which appear in the description of equivariant motivic cohomology rings. 

\begin{Bemerkung}
\label{bem:coinv}
\index{coinvariant algebras}
We recall the standard description of coinvariant algebras. Let $G$ be a split reductive group, with identity component $G^0$. Let $T\subset G^0$ be a split maximal torus, let $N=\op{N}_{G^0}(T)$ be the normalizer of $T$ in $G^0$, and let $W=N/T$ be the Weyl group of $G^0$. Write $\mathfrak{X}(T)$ for the group of characters of $T$ over $k$. Then $W$ acts on $\mathfrak{X}(T)$ in the evident way, and we set
\[ S^W = \mathrm{Sym}(\mathfrak{X}(T) \otimes_{\ZZ} \Lambda)^W.\]
We view $\mathrm{Sym}(\mathfrak{X}(T) \otimes_{\ZZ} \Lambda)$ (and consequently $S^W$) as a graded ring, with the generating character group in degree $1$. 
\end{Bemerkung}

\begin{example}
For $G=\op{GL}_n$, the ring $\op{Sym}\left(\mathfrak{X}(T)\otimes_{\mathbb{Z}}\Lambda\right)$ is a polynomial ring in $n$ variables. The Weyl group $W$ is the symmetric group $\op{S}_n$ and the coinvariant algebra $S^W$ is the algebra of symmetric polynomials, generated by the elementary symmetric polynomials. On the cohomological side, the elementary symmetric polynomials correspond to the Chern classes. 
\end{example}

\begin{Bemerkung}
\index{twisted group rings}
\label{twistedgroup}
Let $G$ be a split reductive group with identity component $G^0$. The finite group $\pi_0(G):=G/G^0$ of components acts on $G^0$ via conjugation. This induces an action of $\pi_0(G)$ on the coinvariant algebra $S^W$. Using this action, we can consider the \emph{twisted group ring} $S^W\rtimes\Lambda[\pi_0(G)]$, cf. \cite[Section 2.4]{So-L}: the underlying abelian group is the tensor product $S^W\otimes\Lambda[\pi_0(G)]$ of the coinvariant algebra and the group ring $\Lambda[\pi_0(G)]$, and the multiplication is given by 
\[ 
(f\otimes u)\cdot (g \otimes v) = (f\cdot g^u \otimes u\cdot v).
\]
A more detailed discussion of the twisted group ring and its relevance in representation theory of real reductive groups can be found in \cite[Section 2.4]{So-L}. 
\end{Bemerkung}

\begin{example}
If $G$ is a finite group, then $\op{fin}_G:G\to\pt$ is a finite \'etale morphism. Then $(\op{fin}_G)_\ast\op{fin}_G^\ast \const{\pt}\cong \op{Res}_G^1\op{Ind}_1^G\Lambda$ is the group ring of the finite group $G$. 
\end{example}

\subsection{Finite group torsors}

We first describe the behaviour of motives in torsors under finite groups. Note that for us, a finite group will be an algebraic group $G$ whose underlying variety is just a disjoint union of finitely many copies of $\Spec k$. 

Let $W$ be a finite group of cardinality $n$. Then  $\fin_*\const{W} \mapright{\cong} \oplus_{i=1}^n \const{\pt}$, and  the composition 
\[ 
\const{\pt} \to \fin_*\const{W} \mapright{\cong} \bigoplus_{i=1}^n \const{\pt}
\]
obtained by pre-composing with the adjunction map $\id \to \fin_{W*}\fin^*_W$ is the diagonal map. Similarly, 
\[ 
\bigoplus_{i=1}^n \const{\pt} \mapright{\cong} \fin_*\const{W} \to \const{\pt}
\]
obtained by composing with the adjunction $\fin_{W*}\fin_W^* \to \id$ 
is the evident sum map. Consequently, the composition of the adjunction maps above
\begin{equation}
\label{multW}
\const{\pt} \to \fin_*\const{W} \to \const{\pt}
\end{equation}
is multiplication with the cardinality of $W$. 

\begin{lemma}
\label{finitetorsorNew}
Let $\mathbb{D}$ be a derivator satisfying the conditions of~\ref{derivator:new}. Let $W$ be a finite group and let $\pi\colon X\to Y$ be a $W$-torsor, and let $M\in \mathbb{D}(Y)$ be a motive on $Y$. Then the composition of the adjunction maps
\[
M \to \pi_*\pi^*M \to M 
\]
is an isomorphism.
\end{lemma}

\begin{proof}
The projection formula reduces us to the case $M=\const{Y}$.
As pullback along surjective maps is conservative, it suffices to verify the assertion after pulling back along $X\to Y$. Now we have a cartesian square
\[
\xymatrix{
W\times X \ar[r]\ar[d] & X\ar[d] \\ X \ar[r]& Y
}
\]
and the result follows by applying base change and \eqref{multW}.
\end{proof}

\begin{Bemerkung}
Now we want to consider endomorphism rings of motives. Assume the situation of Lemma~\ref{finitetorsorNew}. 

Any element $w\in W$ gives rise to an automorphism of $X$ over $Y$, which we still denote by $w$. For each motive $M\in\mathbb{D}(Y)$, the identity $\pi\circ w=w$  gives rise to a canonical isomorphism $\mathrm{iso}\colon w^\ast \circ \pi^\ast M\mapright{\cong} \pi^\ast M$. Now, given a motive $M\in\mathbb{D}(Y)$, we can define a right action of $W$ on the endomorphism algebra $\End_X(\pi^*M)$ of $\pi^\ast M$ in the category $\mathbb{D}(X)$ as follows: for $w\in W$ and an endomorphism $f\in \End_X(\pi^\ast M)$, we can set 
\[ 
f^w = \mathrm{iso}\circ w^*(f) \circ \mathrm{iso}^{-1}.
\]
Using this right action, we can define an algebra $\End_X(\pi^*M)\rtimes \Lambda[W]$ (similar to the twisted group ring in \ref{twistedgroup}) whose underlying set is $\End_X(\pi^*M)\otimes \Lambda[W]$ and whose multiplication is given by
\[ 
(f\otimes u)\cdot (g \otimes v) = (f\circ g^u \otimes uv).
\]
\end{Bemerkung}

\begin{Bemerkung}
Still keeping the same situation, the elements $w\in W$ of the group $W$ also provide endomorphisms of the motive $\pi_\ast \pi^\ast M$ in the category $\mathbb{D}(Y)$. For $w\in W$, the composition
\[ 
\pi_*\pi^*M \to \pi_*w_*w^*\pi^*M\mapright{\cong}\pi_*w^*\pi^*M \mapright{\cong}\pi_*\pi^*M
\]
provides a map $[w]\in \End_Y(\pi_*\pi^*M)$. The first map is the adjunction map $\id \to w_*w^*$, and the second and third isomorphism arise from the identity $\pi\circ w=w$.
\end{Bemerkung}

\begin{theorem}
\label{endring}
Let $\mathbb{D}$ be a derivator satisfying the conditions of~\ref{derivator:new}. Let $W$ be a finite group, and let $\pi\colon X\to Y$ be a $W$-torsor. For each motive $M\in \mathbb{D}(Y)$, the map
\begin{align*} 
\End_X(\pi^*M)\rtimes \Lambda[W] &\to \End_Y(\pi_*\pi^*M), \\ 
f\otimes w &\mapsto \pi_*(f)\circ [w]
\end{align*}
is a canonical isomorphism of $\Lambda$-algebras.
\end{theorem}

\begin{proof}
Let $f\in \End_X(\pi^*M)$ and $w\in W$. To show that our map is an algebra morphism, it suffices to show
\[ [w]\circ \pi_*(f) = \pi_*(f^w) \circ [w].\]
This is immediate from the commutativity of the following diagram:
\[\xymatrix{
\pi_*\pi^*M \ar[r]\ar[d]_-{\pi_*(f)} & \pi_*w_*w^*\pi^*M\ar[r]^-{\cong}\ar[d]_-{\pi_*w_*w^*(f)} &\pi_*w^*\pi^*M \ar[r]^-{\pi_*(\mathrm{iso})}\ar[d]_-{\pi_*w^*(f)} & \pi_*\pi^*M \ar[d]_-{\pi_*(f^w)} \\
\pi_*\pi^*M \ar[r] & \pi_*w_*w^*\pi^*M\ar[r]^-{\cong} &\pi_*w^*\pi^*M \ar[r]^-{\pi_*(\mathrm{iso})} & \pi_*\pi^*M
}\]
Now consider the composition 
\[ 
\End_X(\pi^*M)\rtimes \Lambda[W] \to \End_Y(\pi_*\pi^*M) \mapright{\cong}
\mathbb{D}_X(\pi^*\pi_*\pi^*M, \pi^*M)
\]
with the adjunction isomorphism. As $\pi\colon X\to Y$ is a $W$-torsor, we have a cartesian square
\[ \xymatrix{
W\times X \ar[d]_-{p} \ar[r]^-{p} & X\ar[d]_-{\pi} \\
X\ar[r]^-{\pi} & Y
}\]
By base change
\[ 
\mathbb{D}_X(\pi^*\pi_*\pi^*M, \pi^*M) \mapright{\cong} \mathbb{D}_X(p_*p^*\pi^*M, \pi^*M). 
\]
Applying the projection formula yields
\[ 
\mathbb{D}_X(p_*p^*\pi^*M, \pi^*M) \mapright{\cong} \bigoplus_{v\in W}\End_X(\pi^*M).
\]
So to check that our map is an isomorphism, it is sufficient to verify that the map
\[ \End_X(\pi^*M)\rtimes \Lambda[W] \to \bigoplus_{v\in W}\End_X(\pi^*M), \]
obtained via the above identifications, is an isomorphism of (finitely generated free) $\End_X(\pi^\ast M)$-modules. For $f\in \End_X(\pi^*M)$ and $w\in W$, this map sends $f\otimes w$ to $\pi_\ast(f)\circ[w]$ and the projection formula sends the latter to the tuple which is $f$ on the $w$-component and $0$ elsewhere, cf. \cite[Corollary 2.1.166]{ayoub:thesis1}. This implies that we have an isomorphism, as claimed.
\end{proof}

\subsection{Variety of maximal tori}

\begin{lemma}
\label{GNacyclic}
Let $\mathbb{D}$ be a derivator satisfying the conditions~\ref{derivator:new} and the grading condition~\ref{conditions:grading}. Denote by $\Lambda$ the coefficients of $\mathbb{D}$, which we assume to be a field of characteristic zero. 

Let $G$ be a split connected reductive group. Let $T\subset G$ be a split maximal torus and $N=\op{N}_G(T)$ the normalizer of $T$ in $G$.   
Then the canonical map $\const{\pt} \to \fin_*\const{G/N}$ yields an identification
\[ 
\mathbb{D}_k(\const{\pt}, \fin_*\const{G/N}(d)[2d]) \cong
\mathbb{D}_k(\const{\pt}, \const{\pt}(d)[2d]) = 
\begin{cases}
\Lambda &\mbox{if $d=0$}, \\
0 &\mbox{otherwise}.
\end{cases}
\]
\end{lemma}

\begin{proof}
Let $W=N/T$ be the Weyl group. Then $G/T\to G/N$ is a $W$-torsor, and the projection $G/T\to G/B$ induces an isomorphism $\fin_\ast\const{G/T}\cong \fin_\ast\const{G/B}$. Then Lemma \ref{finitetorsorNew} provides an identification 
\[
\mathbb{D}_k(\const{\pt}, \fin_*\const{G/N}(d)[2d])\cong \mathbb{D}_k(\const{\pt}, \fin_*\const{G/B}(d)[2d])^W.
\]
By the standard argument using the cellularity of $G/B$, the latter can be identified with the $W$-invariants in the Chow ring $\op{CH}^\ast(G/B)$. Since the  $W$-representation on $\op{CH}^\ast(G/B)$ is known to be the regular representation, the result follows. 
\end{proof}

\begin{remark}
In the proof above one can avoid having to explicitly know the $W$-action on $\op{CH}^*(G/B)$ under the assumption that $\mathrm{char}(k)$ is coprime to  $|W|$ using the cycle class map
\[ 
\op{CH}^*(G/N) \otimes \QQ_{\ell} \mapright{\cong} \op{H}^{*}(G/N; \QQ_{\ell}),
\]
where the right hand side denotes $\ell$-adic cohomology with $\ell\neq \mathrm{char}(k)$. In the case $G/N$, the cycle class map is known to be an isomorphism, hence it suffices to show that the $\ell$-adic cohomology is trivial. As $|W|$ is coprime to $\mathrm{char}(k)$, \cite[Corollary 3.3]{illusie:zheng} yields
\[ 
\sum_i (-1)^i \op{H}^i(G/N; \QQ_{\ell}) = \frac{1}{|W|}\sum_i (-1)^i \op{H}^i(G/T;\QQ_{\ell}) = 1.
\]
Now $\op{H}^*(G/N; \QQ_{\ell})$ vanishes in odd degrees, so this gets the job done.

In characteristic zero this is even simpler, using Betti realization and multiplicativity of Euler characteristics. The primary difficulty in characteristic $p$ above, requiring the restrictions on the characteristic, is that Euler characteristics are not multiplicative e.g. for Artin--Schreier coverings.
\end{remark}

\begin{proposition}
\label{azyk}
Let $\mathbb{D}$ be a derivator satisfying the conditions~\ref{derivator:new} and the grading condition~\ref{conditions:grading}. Denote by $\Lambda$ the coefficients of $\mathbb{D}$, which we assume to be a field of characteristic zero. 

Let $G$ be a split connected reductive group. Let $T\subset G$ be a split maximal torus and $N=\op{N}_G(T)$ the normalizer of $T$ in $G$. Then the canonical map $\const{\pt} \to \fin_*\const{G/N}$ is an isomorphism.
\end{proposition}

\begin{proof}
Let $C$ be the cone of $\Lambda\to \fin_*\const{G/N}$. Lemma \ref{GNacyclic} implies
\[ 
\Hom(\const{\pt}, C(d)[2d]) = 0 
\]
for all $d$. On the other hand, $\fin_\ast\const{G/N}$ is a direct summand of $\fin_\ast(G/B)$ as can be seen by applying Lemma~\ref{finitetorsorNew} to the covering $G/T\to G/N$ and noting that the $\mathbb{A}^1$-equivalence $G/B\to G/T$ induces an isomorphism $\fin_\ast(G/B)\cong \fin_\ast(G/T)$. The standard computation using the cell structure of $G/B$ implies that 
\[
\fin_\ast\const{G/B}\cong \bigoplus_{w\in W}\const{\pt}(-l(w))[-2l(w)],
\]
where $W=N/T$ is the Weyl group and $l\colon W \to \ZZ_{\geq 0}$ is the length function. Now $\fin_\ast\const{G/N}$ being a direct summand of the latter implies that $C=0$.
\end{proof}


\begin{Bemerkung}
\label{bem:torus}
To compare the coinvariant algebra to the Chow ring of the classifying space, we first define a morphism from the symmetric algebra to the Chow ring of the classifying space of the maximal torus. Each character $\chi\in \mathfrak X(T)$ has an associated first Chern class 
\[
\op{c}_1(\chi)\in \op{H}^2_{\mathbb{D},T}(\pt; \Lambda(1)),
\]
induced in motivic cohomology from the map ${\op{B}}\chi:{\op{B}}T\to{\op{B}}\mathbb{G}_{\op{m}}$. In slightly more detail, let $L(\chi)$ be the $1$-dimensional $T$-representation over $k$ determined by $\chi$, viewed as a $T$-equivariant line bundle $f\colon L(\chi) \to \pt$. Let $i\colon \pt\hookrightarrow L(\chi)$ be the zero section. Apply $\fin_{*}$ to the counit of adjunction $i_*i^! \to \id$, and use the isomorphism $i^!\const{L(\chi)} \mapright{\cong} \Lambda(-1)[-2]$ to obtain a map $\fin_*\const{L(\chi)} \to \Lambda(1)[2]$. Then the first Chern class $\op{c}_1(\chi)$ is the map
\[ \Lambda \mapright{\cong} \fin_*\const{L(\chi)} \to \Lambda(1)[2].\]
This determines a homomorphism of graded rings
\[ 
\mathrm{Sym}(\mathfrak X(T) \otimes_{\ZZ} \Lambda) \to \bigoplus_d \op{H}^{2d}_{\mathbb{D},T}(\pt; \Lambda(d)).
\]
\end{Bemerkung}

\begin{Bemerkung}
\label{weaksplitting}
It now remains to provide a relation between the coinvariant algebra and the Chow ring of the classifying space. We have canonical isomorphisms of graded rings
\begin{equation}
\op{H}^i_{\mathbb{D},G^0}(\pt; \Lambda(j)) \mapright{\cong} \op{H}^i_{\mathbb{D},N}(\pt; \Lambda(j)) \mapright{\cong} \op{H}^i_{\mathbb{D},T}(\pt; \Lambda(j))^W,
\end{equation}
where the first map (induced from the inclusion $N\hookrightarrow G^0$) is an  isomorphism by Proposition~\ref{azyk}, and the second isomorphism is a consequence of Lemma~\ref{finitetorsorNew}. Combining \ref{bem:coinv} and \ref{bem:torus} with the above isomorphisms, we obtain a morphism of graded rings
\[ 
S^W \to \bigoplus_d \op{H}^{2d}_{\mathbb{D},G^0}(\pt; \Lambda(d)). 
\]
\end{Bemerkung}

\begin{proposition}
\label{prop:swcompconnect}
Let $\mathbb{D}$ be a derivator satisfying the conditions of~\ref{derivator:new} and the grading condition~\ref{conditions:grading}, and let $G$ be a split reductive group with identity component $G^0$. Then the canonical map
\[ 
S^W \to \bigoplus_d \op{H}^{2d}_{\mathbb{D},G^0}(\pt; \Lambda(d)) 
\]
is an isomorphism.
\end{proposition}

\begin{proof}
By \ref{weaksplitting}, this reduces to the case of a split torus. This, in turn, reduces to the case of $\GG_{\op{m}}$. In this situation, one can approximate the classifying space ${\op{B}}\mathbb{G}_{\op{m}}$ by projective spaces $\mathbb{P}^n$. Then the computation of 
\[
\op{H}^{2d}_{\mathbb{D},\mathbb{G}_{\op{m}}}(\pt;\Lambda(d))= \mathbb{D}_{\mathbb{G}_{\op{m}}}(\pt)\left(\const{\pt},\const{\pt}(d)[2d]\right) 
\]
via these approximations, cf. Section~\ref{equiv:chow}, proves the claim.
\end{proof}

\begin{Bemerkung}
Computations of the Chow rings of classifying spaces can also be found in the literature, cf. e.g. \cite{totaro}. Most of the available literature focuses on integral computations, but of course the above rational computations are well-known. 
\end{Bemerkung}

\begin{proposition}
\label{prop:swcomp}
Let $\mathbb{D}$ be a derivator satisfying the conditions of \ref{derivator:new} and the grading condition~\ref{conditions:grading}. Let $G$ be a split reductive group. Then the canonical map 
\[
S^W\rtimes\Lambda[\pi_0(G)] \to \bigoplus_d \op{H}^{2d}_{\mathbb{D},G}(\pt; \Lambda(d)) 
\]
is an isomorphism.
\end{proposition}

\begin{proof}
Let $\pi_0(G) = G/G^0$ be the group of components of $G$. The canonical isomorphism from Theorem~\ref{endring}, applied to approximations of the Borel construction produces a canonical isomorphism
\[ 
\Hom_{\mathbb{D}(\pt)}(\ind_{G_0}^G\Lambda, \ind_{G_0}^G\Lambda(j)[i]) \cong \op{H}^i_{\mathbb{D},G^0}(\pt; \Lambda(j)) \rtimes \Lambda[\pi_0(G)].
\]
The claim then follows from Proposition~\ref{prop:swcompconnect}.
\end{proof}

\begin{proposition}
\label{prop:motbg}
Let $G$ be a split reductive group. 
\begin{enumerate}
\item The motive of the classifying space ${\op{B}}G$ is pure of weight zero (for the weight structure on Ind-mixed Tate motives induced from Remark~\ref{WDM}). 
\item If $M\in \DMT_G(\pt)_{\op{wt}=0}$ then $(\op{fin}_{{\op{B}}G})_\ast M\in \DMT(\pt)_{\op{wt}=0}$. 
\end{enumerate}
\end{proposition}

\begin{proof}
(1) follows directly from the above explicit calculations.  

(2) For a finite group $H$, we have $\op{Ind}_\ast\cong \op{Ind}_!$, and hence by Proposition~\ref{prop:wtres} (5), the induction functor is weight-exact. Using the restriction-induction adjunction, we can reduce the second claim to the case of connected split reductive groups $G$. In this case, all the mixed Tate motives are extensions of constant motives $\const{{\op{B}}G}$, possibly Tate twisted and shifted. So the claim reduces to the case $M=\const{{\op{B}}G}$ which is the first assertion. 
\end{proof}

\begin{Bemerkung}
One could also compute the motives with compact support of the classifying space as in \cite[Section 8]{totaro:motivebg}. With rational coefficients, the resulting formulas would be much simpler than the ones in \cite{totaro:motivebg}; they would agree with the formulas for compactly supported cohomology of classifying spaces of the corresponding complex Lie groups. 
\end{Bemerkung}

We can also draw some consequences for equivariant motivic cohomology of homogeneous spaces. 

\begin{lemma}[{\bf Cohomology of homogeneous varieties}]
\label{fgre}  
Let $G$ be a linear algebraic group, let $\phi:H\subset G$ be the inclusion of a closed subgroup, and let $\phi^\ast:\mathcal{A}_G\ra \mathcal{A}_H$ be the induced homomorphism of Chow rings of classifying spaces. Then  $\mathcal{A}_H$ is finitely generated for this $\mathcal{A}_G$-module structure.
\end{lemma}

\begin{proof}
  Choose a maximal torus $S\subset H$, enlarge it to a maximal torus $T\subset G$, and consider the commutative square
\[
\xymatrix{
  \mathcal A_G \ar[r]\ar[d]&  \mathcal A_H \ar[d]\\
  \mathcal A_T \ar[r]& \mathcal A_S
}
\]
By Proposition~\ref{prop:swcomp}, we know that the vertical maps are finite ring extensions and each of our rings, in particular $\mathcal{A}_H$, is known to be of finite type over the coefficient field. The lemma follows.
\end{proof}

\begin{remark}
Essentially, the reason for the finiteness above is that the homotopy fiber of ${\op{B}}\phi:{\op{B}}H\to{\op{B}}G$ is $G/H$, and the motive of the latter is mixed Tate. Hence we only need finitely many generators to describe this motive in $\DMT(\pt)$. As mentioned before, in specific situations, one can use the Eilenberg--Moore spectral sequence to do explicit computations. 
\end{remark}

\chapter[Derivators and tilting]{Derivators and tilting\\ (an appendix by F. H\"ormann and M. Wendt)}
\renewcommand{\thesection}{B.\arabic{section}}

\appendix

The appendix recalls basic definitions and statements about derivators, with a particular view toward the homotopical stable algebraic derivators of \cite{ayoub:thesis1}. We provide a tilting result for derivators as well as various results related to compatibility of tilting with six functors which are used in the description of categories of equivariant motives in terms of complexes of modules over the equivariant Chow ring. 

\section{Derivators} 
\label{sec:derivators}

We recall the definition of derivator from \cite[Definition 1.10]{groth}. As a preliminary notion, a \emph{prederivator} $\mathbb{D}$ is a strict 2-functor $\mathbb{D}:\mathsf{Cat}^{\op{op}}\to\mathsf{CAT}$ where we follow the post-Russell--Whitehead habit of ignoring set-theoretic issues. We'll also consider variations of the notion that have a source different from $\mathsf{Cat}^{\op{op}}$, e.g. the category of diagrams of schemes over some base $S$. The category with one object (which only has an identity automorphism) is denoted by $\ast$. For every category $I$ and object $i\in I$, there is an induced evaluation functor $\op{ev}_i:\mathbb{D}(I)\to\mathbb{D}(\ast)$. 

\begin{definition}
\label{def:derivator}\index{derivator}
A \emph{derivator} $\mathbb{D}$ is a prederivator such that 
\begin{enumerate}
\item For two categories $I_1$ and $I_2$ the natural functor $\mathbb{D}(I_1\sqcup I_2)\to\mathbb{D}(I_1)\times\mathbb{D}(I_2)$ is an equivalence, and $\mathbb{D}(\emptyset)$ is the terminal category. 
\item A morphism $f:X\to Y$ in $\mathbb{D}(I)$ is an isomorphism if and only if $\op{ev}_i(f):\op{ev}_i(X)\to \op{ev}_i(Y)$ is an isomorphism in $\mathbb{D}(\ast)$ for every object $i\in I$. 
\item For every functor $u:I\to J$ there are homotopy left and right Kan extensions along $u$:
\[
u_!:\mathbb{D}(I)\leftrightarrows \mathbb{D}(J):u^\ast \quad\textrm{ and } \quad u^\ast:\mathbb{D}(J)\leftrightarrows\mathbb{D}(I):u_\ast.
\]
\item The Kan extensions in (3) can be computed as homotopy colimits and limits of suitable slice categories.
\end{enumerate}
\end{definition}

For the functors $u:I\to \ast$, the left homotopy Kan extension $u_!$ is a version of the homotopy colimit of $I$-diagrams, and the right homotopy Kan extension $u_\ast$ is a version of the homotopy limit. Point (4) in the above definition provides a more general relation between the Kan extensions and homotopy colimits and limits. Since we will not need it we refrain from giving a more precise statement which would require introducing a bulk of notation,  see \cite[pp.8--10]{groth} for the precise definitions.

Combining the evaluation functors $\op{ev}_i:\mathbb{D}(I)\to\mathbb{D}(\ast)$ for all objects $i\in I$ together with the natural transformations induced from morphisms $i\to j$ in $I$, we get diagram evaluation functors
\[
\op{ev}:\mathbb{D}(I)\to \op{Fun}(I,\mathbb{D}(\ast)),
\]
called \emph{underlying diagram functor} in \cite{groth}. A derivator is called \emph{strong} if the partial diagram evaluation $\mathbb{D}([1]\times J)\to\op{Fun}([1],\mathbb{D}(J))$ is full and essentially surjective for every category $J$. 

By \cite[Section 2.1]{groth}, a \emph{morphism of prederivators} $F:\mathbb{D}\to\mathbb{D}'$ is a pseudo-natural transformation between the 2-functors $\mathbb{D}$ and $\mathbb{D}'$, i.e., a collection of functors $F_J:\mathbb{D}(J)\to\mathbb{D}'(J)$ for every $J\in\mathsf{Cat}$ which is compatible with restriction functors $u^\ast$ for $u:J\to K$ via explicit natural isomorphisms. 

The following definitions are given in \cite[Definition 3.1, 4.1]{groth}.
\begin{definition}
\index{derivator:stable}
A derivator $\mathbb{D}$ is \emph{pointed} if the underlying category $\mathbb{D}(\ast)$ of $\mathbb{D}$ is pointed. A strong derivator $\mathbb{D}$ is \emph{stable} if it is pointed and objects of $\mathbb{D}(\Box)$ are cartesian if and only if they are cocartesian. 
\end{definition}

\begin{example}
Let $\mathcal{A}$ be an Grothendieck abelian category. For any diagram $I$ the category $\op{Fun}(I,\mathcal{A})$ of $I$-diagrams in $\mathcal{A}$ is again an abelian category. Every morphism of diagrams $f:I\to J$ induces a restriction functor $f^\ast:\op{Fun}(J,\mathcal{A})\to\op{Fun}(I,\mathcal{A})$ which is exact. Then we have an associated derivator $\mathbb{D}_{\mathcal{A}}$, which sends the diagram $I$  to the derived category $\mathbb{D}_{\mathcal{A}}(I):=\op{Der}(\op{Fun}(I,\mathcal{A}))$ and which sends the morphism $f:I\to J$ of diagrams to the induced functor $f^\ast:\mathbb{D}_{\mathcal{A}}(J)\to\mathbb{D}_{\mathcal{A}}(I)$. 

All the examples of derivators appearing throughout the present paper are derivators associated to localizations of suitable model categories as in \cite[Theorem 1.38]{groth}.
\end{example}

\index{derivator:monoidal}
Now we recall the notion of monoidal derivators from \cite{groth:ponto:shulman}. Prederivators form a 2-category $\mathsf{pDer}$ whose morphisms were discussed above and whose 2-cells are given by natural transformations. The category $\mathsf{pDer}$ is cartesian, with products computed pointwise in $\mathsf{CAT}$. A \emph{monoidal prederivator} is a strict 2-functor $\mathbb{D}:\mathsf{Cat}^{\op{op}}\to\mathsf{MonCAT}$ from small categories to not necessarily small monoidal categories. Similarly, we have the notions of symmetric monoidal prederivators and strong resp. lax monoidal morphisms of prederivators. A \emph{monoidal derivator} is then a monoidal prederivator which is a derivator and whose product $\otimes: \mathbb{D}\times\mathbb{D} \to\mathbb{D}$ is cocontinuous in each variable. For a precise discussion of what that means, we refer to \cite[Section 3]{groth:ponto:shulman}; roughly speaking, monoidality of derivators requires compatibility of the monoidal structures with the left homotopy Kan extensions. 

Now we discuss Ayoub's notion of homotopical stable algebraic derivator which is the basis for his formulation of the six-functor formalism for motivic categories. In the algebraic setting, the source category $\mathsf{Cat}$ is replaced by a more algebraic category of diagrams of schemes over a suitable base scheme $S$. Effectively, one fixes a suitable subcategory $\mathsf{Dia}$ of $\mathsf{Cat}$ of diagram shapes and then considers the category $\mathsf{DiaSch}$ of $\mathsf{Dia}$-diagrams of quasi-projective schemes over $S$ with quasi-projective morphisms, cf. \cite[Section 2.4.1]{ayoub:thesis1}. An \emph{algebraic prederivator} is then a 2-functor $\mathbb{D}:\mathsf{DiaSch}/S\to \mathfrak{TR}$ from diagram of schemes to triangulated categories.\footnote{Note that Ayoub's definition doesn't require the 2-functor to be strict.}

Now we can recall the definition of homotopical stable algebraic derivator from \cite[Section 2.4.2]{ayoub:thesis1}. 

\begin{definition}
\label{def:hsad}\index{derivator!homotopical stable algebraic}
An algebraic prederivator  $\mathbb{D}:\mathsf{DiaSch}/S\to\mathfrak{TR}$ is called \emph{homotopical stable algebraic derivator} if the following axioms are satisfied: 
\begin{description}
\item[DerAlg 0] If $(\mathscr{F},\mathcal{I})$ is a diagram of quasi-projective $S$-schemes with $\mathcal{I}$ a discrete category, then we have an equivalence 
\[
\mathbb{D}(\mathscr{F},\mathcal{I})\to \prod_{i\in\op{Ob}(\mathcal{I})}\mathbb{D}(\mathscr{F}(i)).
\]
\item[DerAlg 1] Let $\alpha:\mathcal{J}\to\mathcal{I}$ an essentially surjective functor and a diagram $(\mathscr{F},\mathcal{I})$, the functor $\mathbb{D}(\mathscr{F},\mathcal{I})\to\mathbb{D}(\mathscr{F}\circ\alpha,\mathcal{J})$ is conservative. 
\item[DerAlg 2R] For every morphism $(f,\alpha):(\mathscr{F},\mathcal{I})\to (\mathscr{G},\mathcal{J})$, the induced functor $(f,\alpha)^\ast$ has a right adjoint $(f,\alpha)_\ast$. 
\item[DerAlg 2L] For every morphism $(f,\alpha):(\mathscr{F},\mathcal{I})\to (\mathscr{G},\mathcal{J})$ which is pointwise smooth, the induced functor $(f,\alpha)^\ast$ has a left adjoint $(f,\alpha)_\sharp$. 
\item[DerAlg 3R] For a diagram 
\[
\xymatrix{(\mathscr{G}\circ\alpha,\mathcal{J}) \ar[r]^\alpha \ar[d]_{f|_{\mathcal{J}}} & (\mathscr{G},\mathcal{I}) \ar[d]^f \\
(\mathscr{F}\circ\alpha,\mathcal{J}) \ar[r]_\alpha & (\mathscr{F},\mathcal{I}),
}
\]
the natural exchange transformation  $\alpha^\ast f_\ast\to (f|_{\mathcal{J}})_\ast\circ \alpha^\ast$ is an isotransformation.
\item[DerAlg 3L] In a diagram as above, where $f$ is cartesian and pointwise smooth, the natural exchange transformation $(f|_{\mathcal{J}})_\sharp\circ\alpha^\ast \to \alpha^\ast f_\sharp$ is an isotransformation. 
\item[DerAlg 4] For every quasi-projective $S$-scheme $X$, the  2-functor $\mathbb{D}(X,-):\mathsf{Dia}\to\mathfrak{TR}$ is a stable derivator in the sense discussed above. 
\item[DerAlg 5] The $2$-functor $\mathsf{H}=\mathbb{D}(-,\ast):\mathsf{Sch}/S\to\mathfrak{TR}$ is a homotopical stable 2-functor, i.e.,
\begin{enumerate}
\item $\mathsf{H}(\emptyset)=0$
\item For every morphism $f:X\to Y$ the natural functor $f^\ast:\mathsf{H}(Y)\to\mathsf{H}(X)$ has a right adjoint $f_\ast$, and for each immersion $i$ the counit $i^\ast i_\ast\to\op{id}$ is an isotransformation.
\item For every smooth morphism $f:X\to Y$, the functor $f^\ast$ has a left adjoint $f_\sharp$. For smooth base-change squares, the exchange transformation $f_\sharp' g'^\ast\to g^\ast f_\sharp$ is an isotransformation. 
\item If $j:U\hookrightarrow X$ is an open immersion with closed complement $i:Z\hookrightarrow X$, the pair $(j^\ast,i^\ast)$ is conservative.
\item For the canonical projection $p:X\times\mathbb{A}^1\to X$, the unit $\op{id}\to p_\ast p^\ast$ is an isotransformation.
\item For $s$ the zero section of $p:X\times\mathbb{A}^1\to X$, the endofunctor $p_\sharp s_\ast$ is an auto-equivalence of $\mathsf{H}(X)$.
\end{enumerate}
\end{description}
\end{definition}

\begin{remark}
It should be noted that there is some difference between Ayoub's axioms for derivators in \cite[Definition 2.1.43]{ayoub:thesis1} and the definitions above taken from \cite{groth}. Ayoub's axioms 1 and 2 encode axioms 1 and 2 in Definition~\ref{def:derivator} plus strongness, 3 is the existence of Kan extensions, 4 encodes the base change formulas. Axiom 5 is the stability and 6 enforces a relation between induced triangulated structure from the derivator-definition of suspension and the one given. 

Note in particular that Ayoub's definition implies that the derivators are also pointed in the sense of \cite{groth}. 
\end{remark}

Finally, we recall from \cite[Definition 2.4.48]{ayoub:thesis1} the appropriate notion of monoidal structures. A \emph{monoidal} homotopical stable algebraic derivator is a 2-functor $\mathbb{D}:\mathsf{DiaSch}/S\to\mathfrak{MonoTR}$ from diagrams of schemes to monoidal triangulated categories such that 
\begin{enumerate}
\item Forgetting the monoidal structure yields a homotopical stable algebraic derivator. 
\item For a pointwise smooth morphism $(f,\alpha):(\mathscr{G},\mathcal{J})\to(\mathscr{F},\mathcal{I})$ of schemes, the following morphisms are isomorphisms if either (i) $\alpha$ is the identity and $f$ is cartesian, or (b) the morphism of diagrams of schemes is the identity of a single scheme $X$: 
\[
(f,\alpha)_\sharp((f,\alpha)^\ast A\otimes_{\mathscr{G},\mathcal{J}}B') \to A\otimes_{\mathscr{F},\mathcal{I}}(f,\alpha)_\sharp B'
\]
\[ (f,\alpha)_\sharp( A'\otimes_{\mathscr{G},\mathcal{J}} (f,\alpha)^\ast B) \to (f,\alpha)_\sharp A'\otimes_{\mathscr{F},\mathcal{I}}B.
\]
\end{enumerate}
 
By \cite[Corollaire 2.4.51 and Proposition 2.1.152]{ayoub:thesis1}, there is a monoidal analogue of Axiom DerAlg 4: if $\mathbb{D}$ is a monoidal homotopical stable algebraic derivator, then for every quasi-projective $S$-scheme $X$, the induced 2-functor $\mathbb{D}(X,-):\mathsf{Dia}\to\mathfrak{MonoTR}$ is a monoidal stable derivator in the sense of \cite{groth:ponto:shulman}. 

\begin{convention}
\label{derivator:conditions}
\index{derivator:standing assumptions}
In all of this work, we consider homotopical stable algebraic derivators $\mathbb{D}$ over $\op{Var}/k$, i.e., we work over schemes which are quasi-projective, separated and of finite type over some field $k$. 
When we speak of \emph{motives}, we are referring to objects in some category $\mathbb{D}(X)$.

In addition to these requirements, the following assumptions are made for all the derivators $\mathbb{D}$ that appear:
\begin{enumerate}
\item The derivators $\mathbb{D}$ are $\Lambda$-linear for some ring $\Lambda\supset \mathbb{Q}$. Most of the time, fields of characteristic zero will suffice.
\item The derivators $\mathbb{D}$ are \emph{separated} in the sense of Definition 2.1.160 of \cite{ayoub:thesis1}, i.e., for any surjective morphism $f:X\to Y$ of varieties, the associated functor $f^\ast$ is conservative.
\item The derivators $\mathbb{D}$ are monoidal in the sense discussed above. For a variety $X$, the tensor unit in $\mathbb{D}(X)$ is denoted by $\const{X}$, cf.~\ref{def:motive}. We denote $\const{X}(1)[1]$ the cone of the unit map $X\to\mathbb{G}_{\op{m},X}$. 
\item The derivators $\mathbb{D}$ are equipped with a t-structure, i.e., for each variety $X\in\op{Var}/k$, the category $\mathbb{D}(X)$ admits a t-structure such that the associated six functors have the exactness properties detailed in \cite[Scholie 2.2.95]{ayoub:thesis1}. 
\item The derivators $\mathbb{D}$ are orientable in the sense that for a vector bundle $E\to X$ of rank $r$ we have $\op{M}_X(\op{Th}(E))\cong \underline{X}(r)[2r]$, where $\op{Th}(E)=E/(E\setminus X)$ is the Thom space of the vector bundle $E$ and the motive $\op{M}_X(-)$ is defined as in \ref{def:motive}.
\end{enumerate}
\end{convention}

\begin{remark}
See Section~\ref{sec:basicmotives} for the construction of various motivic homotopical stable algebraic derivators that are used in this work. The above requirements are satisfied for these. 
\end{remark}

\section{Tilting for derivators} 
\label{sec:fritz}

The purpose of this appendix is to establish lifting results for diagrams in derivators. This allows, under suitable conditions, to define a total complex or tilting functor for complexes in derived categories of motives, which is needed in the tilting statements for equivariant mixed Tate motives and their compatibilities with the six functor formalism.

\begin{remark}
In the following, we use ``isotransformation'' as a synonym  for ``natural  transformation, which is an isomorphism in the category of functors''. Quite often we draw diagrams with categories in the corners and functors as arrows and double arrows as transformations of functors filling the cells. These are meant to convey which transformations or isotransformations of compositions of functors we claim to exist or to even explicitly construct.
\end{remark}

\subsection{Dold--Kan correspondence}

Let $\mathcal{A}$ be an abelian category, and let $X\in\op{Fun}(\mathbf{\Delta}^{\op{op}},\mathcal{A})$ be a simplicial object. The \emph{normalized chain complex $\op{N}(X)$} is given by $\op{N}(U_n)=\bigcap_{i=0}^{n-1}\ker \op{d}_i^n$ with the boundary map $(-1)^n\op{d}_n^n$, cf. \cite[Tag 0194, before Lemma 14.23.4]{stacks}. This induces a functor $\op{N}:\op{Fun}(\mathbf{\Delta}^{\op{op}},\mathcal{A})\to \op{Ch}_{\geq 0}(\mathcal{A})$ which is an equivalence of categories, cf. \cite[Tag 019D, Theorem 14.24.3]{stacks}, called the \emph{Dold--Kan correspondence.} 

For an abelian category $\mathcal{A}$, consider a complex 
\[
A=\left(A_n\xrightarrow{\op{d}_{n-1}} A_{n-1}\to\cdots\xrightarrow{\op{d}_1} A_1\right)
\]
of objects in $\op{Ch}(\mathcal{A})$, i.e., a complex of complexes. Under the Dold--Kan correspondence, it corresponds to a simplicial object $A^\Delta\in\op{Fun}(\mathbf{\Delta}^{\op{op}},\op{Ch}(\mathcal{A}))$. An explicit realization of an inverse of the Dold--Kan correspondence can be found in the proof of \cite[Tag 019D, Theorem 14.24.3]{stacks}. The total complex of the double complex $A$ is equivalent to the homotopy colimit of  $A^\Delta$, viewed as a simplicial diagram in $\op{Ch}(\mathcal{A})$, using the usual model structure on chain complexes in an abelian category. This follows by applying the Dold--Kan correspondence once more, interpret the total complex of the double complex $A$ as the diagonal of a bisimplicial set and then see that the latter is equivalent to the homotopy colimit of the simplicial diagram $A^\Delta$. 

The functors in the Dold--Kan correspondence are lax monoidal functors, but not strict because the Eilenberg--Zilber map is only a homotopy inverse of the Alexander--Whitney map. This provides a compatibility of the monoidal structures induced on the derived categories. There are also properly monoidal versions of the Dold--Kan correspondence, cf. e.g. \cite{schwede:shipley}.

\subsection{Directed categories}

The key tool in the discussion of the tilting theorem are directed categories. This provides a structure encoding inductive construction of diagrams; in the context of tilting theory, the structure of directed categories allows to inductively lift diagrams from the homotopy category to the respective diagram category of the derivator, provided suitable hom-sets vanish. 

\begin{definition}
\label{def:directed}
A category $\mathcal{I}$ is called \emph{directed} if there exists a functor $\deg:\mathcal{I}\to\mathbb{N}$ such that preimages of identity maps are identity. Here the category structure on $\mathbb{N}$ is the one given by the total order $\leq$ on $\mathbb{N}$: there exists a unique morphism $n\to m$ if and only if $n\leq m$. 
\end{definition}

\begin{definition}
The key example of a directed category is the category $\mathbf{\Delta}$ with the obvious degree function. 
\end{definition}

\begin{remark}
\label{rem:decomposein}
For a directed category $\mathcal{I}$, we denote by $\mathcal{I}_{\leq n}$ the full subcategory of $\mathcal{I}$ of objects with $\deg\leq n$. 

There is a diagram of embeddings of full subcategories 
\[
\mathcal{I}_{<n}\xrightarrow{i}\mathcal{I}\xleftarrow{j} \mathcal{I}_n,
\]
where the indices refer to preimages of $[0,n-1]$ and $\{n\}$ under $\deg$. In particular, $\mathcal{I}_n$ is a discrete category. By \cite[4.3]{groth}, there is a distinguished triangle
\[
j_! j^\ast Y\to Y\to i_\ast i^\ast Y\to j_!j^\ast Y[1] 
\]
for any $Y\in\mathbb{D}(\mathcal{I})$. For $X,Y\in\mathcal{D}$, this triangle induces a long exact sequence of abelian groups
\[
\cdots\to\mathcal{D}(X,j_!j^\ast Y[k])\to \mathcal{D}(X,Y[k])\to \mathcal{D}(X,i_\ast i^\ast Y[k])\to \mathcal{D}(X,j_!j^\ast Y[k+1])\to \cdots
\]
\end{remark}

\begin{definition}
Let $\mathcal{M}$ be a category having all small limits and colimits. The inclusion $\mathcal{I}_{\leq n}$ induces a restriction functor $\op{Fun}(\mathcal{I},\mathcal{M})\to \op{Fun}(\mathcal{I}_{\leq n},\mathcal{M})$. The left adjoint $\op{sk}_n$ of the restriction functor is called the $n$-skeleton, and the right adjoint $\op{cosk}_n$ of the restriction functor is called the $n$-cosekeleton. 
\end{definition}

\begin{example}
Following \cite[Tag 017Z, Definition 14.12.1]{stacks}, denote by $\mathbf{\Delta}_{\leq n}$ the full subcategory of the simplicial category $\mathbf{\Delta}$ with objects $[0], [1], \dots,[n]$. An \emph{$n$-truncated simplicial object} in a category $\mathcal{C}$ is  a functor $\mathbf{\Delta}^{\op{op}}_{\leq n}\to\mathcal{C}$. For any simplicial object $X\in\op{Fun}(\mathbf{\Delta}^{\op{op}},\mathcal{C})$, the above notions of skeleton and coskeleton reduce to the classical notions. In particular, the \emph{$n$-skeleton functor} is given as left adjoint of the restriction along the inclusion $\mathbf{\Delta}_{\leq n}\hookrightarrow\mathbf{\Delta}$. 
\end{example} 

There are natural transformations $\tau_n:\op{sk}_n\to\op{cosk}_n$. A diagram $X:\mathcal{I}_{\leq n-1}\to\mathcal{M}$ together with a family of factorizations $\op{sk}_n(X)^c\xrightarrow{i^c} X^c\xrightarrow{p^c}\op{cosk}_n(X)^c$ of the natural transformations $\tau_{n-1}^c:\op{sk}_{n-1}(X)^c\to \op{cosk}_{n-1}(X)^c$ for each object $c\in\mathcal{I}$ with $\deg(c)=n$ uniquely determines a diagram $X:\mathcal{I}_{n}\to\mathcal{M}$ whose restriction to degree $n-1$ coincides with the given diagram, cf. \cite[Lemma 3.10]{riehl:verity}. The objects $\op{sk}_{n-1}(\op{res}_{n-1}(X))^c$ and $\op{cosk}_{n-1}(\op{res}_{n-1}(X))^c$ with $n=\deg(c)$ are called the $n$-th \emph{latching} and \emph{matching objects} of the diagram $X$, respectively. The associated counit $\op{sk}_{n-1}(\op{res}_{n-1}(X))^c\to X^c$ and unit $X^c\to\op{cosk}_{n-1}(\op{res}_{n-1}(X))^c$ are called \emph{latching} and \emph{matching maps} of $X$, respectively. 



In the setting of derivators, we can replace the latching objects by diagrams over a suitably defined category: 

\begin{definition}
\label{def:latching}
For a diagram $I$ of length $n$, we define the \emph{latching category} $L_n$ (a directed category of length $n-1$) as follows: the  objects are morphisms $\alpha:X\to Y$ of $I$ with $\deg(Y)=n, \deg(X)<n$, and the morphisms are given by the evident commutative diagrams. Mapping morphisms to source and target yields functors $p_n:L_n\to I_n$ and $\iota_n:L_n\to I_{<n}$. 
\end{definition}

\begin{lemma}
\label{lem:latching}
Let $I$ be a directed category of length $n$, let $\mathbb{D}$ be a derivator and consider an object $X\in\mathbb{D}(I)$. Then there is an exact triangle in $\mathbb{D}(I_n)$ of the form
\[
p_{n,!} \iota_n^\ast X\to j^\ast X\to j^?X\to p_{n,!}\iota_n^\ast X[1]
\]
where $j^\ast X\to j^?X$ is the adjoint map of $j_!X\to j_\ast X$ arising from the composition $j_!X \xleftarrow{\cong} j^!j^\ast j_\ast X\rightarrow j_\ast X$. 
\end{lemma}

\begin{proof}
We recall from \cite[Section 3.1]{groth} the description of the functor $j^?X$. Note that the functor $j:I_n\to I$ is a cosieve, cf. \cite[Definition 1.28]{groth}, and $i:I_{<n}\to I$ is the complementary sieve. From \cite[Proof of Lemmma 3.7]{groth}, the cylinder diagram for the sieve $i:I_{<n}\to I$ is given as the pushout
\[
\xymatrix{
I_{<n}\ar[r]^i \ar[d]_{\iota_0} & I \ar[d]^s\\
I_{<n}\times [1] \ar[r] & \op{cyl}
}
\]
where $\iota_0:I_{<n}\to I_{<n}\times [1]$ is the inclusion at $0$ and $s:I\to\op{cyl}$ denotes the corresponding inclusion. Applying the universal property of the pushout to the two functors $\op{id}:I\to I$ and $I_{<n}\times[1]\xrightarrow{\op{pr}}I_{<n}\xrightarrow{i}I$ induces a projection functor denoted by $q:\op{cyl}\to I$. With this notation, the functor $j^?$ is given by the composition 
\[
j^?:\mathbb{D}(I)\xrightarrow{s_\ast} \mathbb{D}(\op{cyl},I_{<n}) \xrightarrow{q_!} \mathbb{D}(I,I_{<n}) \xrightarrow{j^\ast} \mathbb{D}(I_n),
\]
cf. \cite[end of Section 3.1]{groth}. 

Consider the commutative diagram
\[
\xymatrix{
\op{cyl}\times_{/I}I_n\ar[r]^\alpha \ar[d]_p & \op{cyl} \ar[d]^q \\
I_n \ar[r]_j & I
}
\]
where $\op{cyl}\times_{/I}I_n$ denotes the appropriate comma category. By an argument as in \cite[proof of Theorem 1.31]{groth}, the diagram is $\mathbb{D}$-exact. In particular, 
\[
j^?=j^\ast \circ q_!\circ s_\ast=p_!\circ \alpha^\ast\circ s_\ast. 
\]
We can also write the comma category above as a pushout (obtained as the mapping cylinder pushout above)
\[
\xymatrix{
I_{<n}\times_{/I}I_n \ar[r] \ar[d] & I\times_{/I}I_n \ar[d] \\
I_{<n}\times_{/I} I_n\times [1] \ar[r] & \op{cyl}\times_{/I} I_n,
}
\]
where $I_{<n}\times_{/I}I_n$ is the latching category of Definition~\ref{def:latching}.

Denoting by  $\ulcorner$ the category $(0,1)\leftarrow(0,0)\rightarrow(1,0)$, we have a functor $F:\op{cyl}\times_{/I}I_n\to\ulcorner\times I_n$ obtained via the universal property from the following two functors: the functor 
\[
I_{<n}\times_{/I} I_n\times[1]\to\ulcorner\times I_n:(\alpha:x\to y,\epsilon) \mapsto (0,\epsilon)\times y
\]
and the functor 
\[
I\times_{/I}I_n\to\ulcorner\times I_n: (\alpha:x\to y) \mapsto \left\{\begin{array}{ll}(0,0)\times y & \nu(x)<n\\ (1,0)\times y & \nu(x)=n. \end{array}\right.
\]
Note that the two definitions are compatible because they send a morphism $\alpha:x\to y$ of $I_{<n}\times_{/I}I_n$ to $(0,0)\times y$. We also have the projection functor $P:\ulcorner\times I_n\to I_n$, and it is clear that the composition $P\circ F=p$. Consequently, $j^\ast=P_!F_!\alpha^\ast s_\ast$. 

We now claim that the underlying diagram of $F_!\alpha^\ast s_\ast X$ is the diagram
\[
\xymatrix{
p_{n,!}\iota_n^\ast X\ar[r] \ar[d] & j^\ast X\\ 0
}
\]
where the functors $p_n$ and $\iota_n$ are the ones defined in Definition~\ref{def:latching}. This will imply the claim about the exact triangle. The corner of the underlying diagram of $F_!\alpha^\ast s_\ast X$ has the $I_n$-indexed homotopy colimit over $\op{cyl}\times_{/I}I_n\times_{/\ulcorner}(0,0)=I_{<n}\times_{/I}I_n=L_n$, and it is easy to check that $\iota_n^\ast=\alpha^\ast s_\ast$. The entry $(1,0)$ of the underlying diagram is the $I_n$-indexed homotopy colimit over $\op{cyl}\times_{/I}I_n\times_{/\ulcorner}(1,0)=I\times_{/I}I_n$. This is simply $j^\ast$ because the only maps that appear are the $\op{id}_y$, cf. the case distinction above. The entry $(0,1)$ of the underlying diagram is the $I_n$-indexed homotopy colimit over $\op{cyl}\times_{/I}I_n\times_{/\ulcorner}(1,0)=(1)\times\op{cyl}\times_{/I}I_n$. But $s_\ast X$ is, by definition, identically zero on this part. 
\end{proof}

\subsection{The tilting theorem}
Fix an abelian category $\mathcal{A}$. Recall that the Dold--Kan correspondence allowed to write the total complex for a double complex in $\mathcal{A}$ as a homotopy colimit over a simplicial diagram in $\op{Ch}(\mathcal{A})$. Defining total complexes for complexes in the derived category $\op{Der}^{(\op{b})}(\mathcal{A})$ is significantly more complicated. In derivator language, there is a zig-zag of functors
\[
\op{Fun}(\mathbf{\Delta}^{\op{op}},\mathbb{D}_{\mathcal{A}}(\ast))\xleftarrow{\op{ev}} \mathbb{D}_{\mathcal{A}}(\mathbf{\Delta}^{\op{op}})\xrightarrow{\op{hocolim}} \mathbb{D}_{\mathcal{A}}(\ast)
\]
where the first arrow is the evaluation (or underlying diagram) functor, sending an object in $\mathbb{D}(\mathbf{\Delta}^{\op{op}})$, represented by a complex of simplicial diagrams, to a simplicial diagram of complexes by evaluating at every point of the diagram. The second arrow takes the homotopy colimit of the simplicial object (corresponding to the double complex). However, the evaluation functor $\mathbb{D}_{\mathcal{A}}(\mathbf{\Delta}^{\op{op}})\to \op{Fun}(\mathbf{\Delta}^{\op{op}},\mathbb{D}_{\mathcal{A}}(\ast))$ is in general neither fully faithful nor essentially surjective. The problem in defining a ``total complex'' for a complex in the derived category is to find a lift of the complex along the evaluation functor, and to show that the resulting homotopy colimit is independent of such choices. 

The purpose of the tilting result for derivators is now to identify conditions on a subcategory of $\mathbb{D}(\ast)$ such that the evaluation functor, when restricted to diagrams with values in this subcategory, is an equivalence. This will imply the existence of lifts, well-defined up to isomorphism, which then allows to define a total complex or tilting functor.

\begin{definition}
\label{def:tiltsubcat}
\index{tilting subcategory}
Let $\mathcal{T}$ be a triangulated category. We call a full additive subcategory $\mathcal{X}\subset \mathcal{T}$ a \emph{tilting subcategory} if for all natural numbers $k\geq 1$ and all objects $A,B\in\mathcal{X}$, we have 
\[
\mathcal{X}(A[k],B)=0.
\]
\end{definition}


\begin{theorem}
\label{thm:fritz}
Let $\mathbb{D}$ be a stable derivator, let $I$ be a finite directed category and let $\mathcal{X}\subset \mathbb{D}(\ast)$ be a tilting subcategory. Then the diagram evaluation $\op{ev}:\mathbb{D}(I)\to \op{Fun}(I,\mathcal{D}(\ast))$ restricts to an equivalence 
\[
\op{ev}:\mathbb{D}(I)|_{\mathcal{X}}\to\op{Fun}(I,\mathcal{X}),
\]
where $\mathbb{D}(I)|_{\mathcal{X}}$ denotes the preimage of $\op{Fun}(I,\mathcal{X})$ under the evaluation functor $\op{ev}$. 
\end{theorem}


\begin{proof}
The result will follow from Propositions~\ref{prop:tiltff} and \ref{prop:tiltes}.
\end{proof}

\begin{proposition}
\label{prop:tiltff}
In the situation of Theorem~\ref{thm:fritz}, denote $\mathcal{D}=\mathbb{D}(I)|_{\mathcal{X}}$. For any two objects  $X,Y\in\mathcal{D}$ and any $k>0$, the evaluation functor induces isomorphisms 
\[
\mathcal{D}(X,Y)\cong\mathcal{X}(\op{ev}(X),\op{ev}(Y))\; \textrm{ and }\;
\mathcal{D}(X[k],Y)=0.
\]
In particular, the restriction $\op{ev}:\mathbb{D}(I)|_{\mathcal{X}}\to\op{Fun}(I,\mathcal{X})$ of the evaluation functor is fully faithful. 
\end{proposition}

\begin{proof}
It suffices to prove the statement for the categories $I_{\leq n}$, since a directed colimit of isomorphisms will give an isomorphism (and the evaluation functor commutes with colimits). Then the proof is by induction on the length of $I$. 

The base case is given by index categories $I$ of length $1$, i.e., finite discrete categories. In this case, the evaluation functor is an equivalence, by the derivator axioms. 

For the inductive step, consider the category $I_{\leq n}$, replace $X$ and $Y$ by the appropriate restrictions, and consider the exact sequence of Remark~\ref{rem:decomposein}. By our inductive assumption, we have 
\[
\mathcal{D}(X,i_\ast i^\ast Y[k])\cong \mathbb{D}_{<n}(i^\ast X, i^\ast Y[k])\cong \left\{\begin{array}{ll} 
\mathcal{X}(\op{ev}(X),\op{ev}(Y)) & k=0\\
0 & k<0 \end{array}\right.
\]
Via the latching category and Lemma~\ref{lem:latching}, there is an adjunction $j^?\dashv j_!$ and hence an isomorphism
\[
\mathcal{D}(X,j_!j^\ast Y[k])\cong \mathbb{D}_n(j^?X,j^\ast Y[k]).
\]
Using the distinguished triangle of Lemma~\ref{lem:latching} produces a long exact sequence 
\begin{eqnarray*}
\cdots\to\mathbb{D}_n(p_{n,!}\iota_n^\ast X,j^\ast Y[k-1])&\to& \mathbb{D}_n(j^?X,j^\ast Y[k])\to \\
\to \mathbb{D}_n(j^\ast X, j^\ast Y[k])&\to& \mathbb{D}_n(p_{n,!}\iota_n^\ast X,j^\ast Y[k])\to \cdots
\end{eqnarray*}
where the functors $p_n$ and $\iota_n$ are the ones defined in Definition~\ref{def:latching}. For the middle arrow, we note that Lemma~\ref{lem:latching} implies the commutativity of the following diagram
\[
\xymatrix{
\mathcal{D}(X,j_!j^\ast Y[k]) \ar[r] \ar[d]_\cong & \mathcal{D}(X,Y[k]) \ar[d]^{j^\ast} \\
\mathbb{D}_n(j^?X,j^\ast Y[k]) \ar[r] & \mathbb{D}_n(j^\ast X, j^\ast Y[k]).
}
\]

We first prove the required vanishing. Putting in $k<0$ in the above exact sequence implies $\mathbb{D}_n(j^\ast X, j^\ast Y[k])=0$ because $I_n$ is  discrete, and 
\[
\mathbb{D}_n(p_{n,!}\iota_n^\ast X,j^\ast Y[l])=0
\]
for $l=k,k-1$ by the inductive assumption since $L_n$ is of length $\leq n-1$. Consequently, we have $\mathcal{D}(X,j_!j^\ast Y[k])=\mathbb{D}_n(j^?X,j^\ast Y[k])=0$. Combined with the previous exact sequence, this implies $\mathcal{D}(X,Y[k])=0$ for $k<0$. 

Now we want to prove the identification in degree $0$. We have the exact sequence
\[
\cdots\to\mathcal{D}(X,j_!j^\ast Y)\to \mathcal{D}(X,Y)\to \mathcal{X}(\op{ev}(i^\ast X),\op{ev}(i^\ast Y))\to \mathcal{D}(X,j_!j^\ast Y[1])\to \cdots
\]
Considering the case $k=0$ of the previous exact sequence, we still have the vanishing $\mathbb{D}_n(p_{n,!}\iota_n^\ast X,j^\ast Y[-1])=0$ which yields 
\[
\mathbb{D}_n(j^?X,j^\ast Y)=\ker\left(\mathbb{D}_n(j^\ast X, j^\ast Y)\to \mathbb{D}_{L_n}(\iota_n^\ast X,p_n^\ast j^\ast Y)\right). 
\]
By induction assumption, we can rewrite this as 
\[
\mathbb{D}_n(j^?X,j^\ast Y)=\ker\left(\mathcal{X}(\op{ev}(j^\ast X), \op{ev}(j^\ast Y))\to \mathcal{X}(\op{ev}(\iota_n^\ast X),\op{ev}(p_n^\ast j^\ast Y))\right)=:T_1. 
\]

Now assume $k=1$ and consider the previous exact sequence. In this case, we are interested in the kernel $T_2=\ker\left(\mathbb{D}_n(j^?X,j^\ast Y[1])\to \mathbb{D}_n(j^\ast X, j^\ast Y[1])\right)$. By the inductive assumption and the exact sequence, we can rewrite this as 
\begin{eqnarray*}
T_2&=&\op{coker}\left( \mathbb{D}_n(j^\ast X, j^\ast Y)\to \mathbb{D}_{L_n}(\iota_n^\ast X,p_n^\ast j^\ast Y)\right)\\
&=& \op{coker}\left( \mathcal{X}(\op{ev}(j^\ast X), \op{ev}(j^\ast Y))\to \mathcal{X}(\op{ev}(\iota_n^\ast X),\op{ev}(p_n^\ast j^\ast Y))\right).
\end{eqnarray*}

Putting these statements together allows to rewrite the previous sequence as 
\[
0\to T_1\to \mathcal{D}(X,Y)\to \mathcal{X}(\op{ev}(i^\ast X),\op{ev}(i^\ast Y))\to T_2.
\]
There is a ladder of exact sequences containing this and a similar exact sequence where $\mathcal{D}(X,Y)$ is replaced by  $\mathcal{X}(\op{ev}(X),\op{ev}(Y))$. By the 5-lemma, we obtain that evaluation induces an isomorphism $\mathcal{D}(X,Y)\cong\mathcal{X}(\op{ev}(X),\op{ev}(Y))$, finishing the proof of full faithfulness. 
\end{proof}

\begin{remark}
A similar result will hold for infinite directed categories if we assume that the derivator is continuous, i.e., that the category $\mathbb{D}(I)$ is equivalent to the 2-colimit of the diagram of categories $\mathbb{D}(I_{\leq n})$ given by the natural restriction functors. We'll not really need this point.
\end{remark}

\begin{proposition}
\label{prop:tiltes}
In the situation of Theorem~\ref{thm:fritz}, denote $\mathcal{D}=\mathbb{D}(I)|_{\mathcal{X}}$. Then the restriction $\op{ev}:\mathbb{D}(I)|_{\mathcal{X}}\to\op{Fun}(I,\mathcal{X})$ of the evaluation functor is essentially surjective. 
\end{proposition}

\begin{proof}
Let $Y\in\op{Fun}(I,\mathcal{X})$ and assume by induction that we have constructed preimages $A\in\mathbb{D}(I_{<n})$ and $B\in\mathbb{D}(I_n)$ of the restrictions $i^\ast Y\in\op{Fun}(I_{<n},\mathcal{X})$ and $j^\ast Y\in \op{Fun}(I_n,\mathcal{X})$. From the relevant adjunctions we obtain
\[
\mathbb{D}_{I_{\leq n}}(i_\ast A, j_!B[1])\cong \mathbb{D}_{I_n}(j^?i_\ast A, B[1]).
\]
Now we use the exact sequence established in the proof of Proposition~\ref{prop:tiltff}:
\[
0\to \mathbb{D}_{I_n}(p_{n,!}\iota_n^\ast i_\ast A, B) \to \mathbb{D}_{I_n}(j^?i_\ast A,B[1]) \to \mathbb{D}_{I_n}(j^\ast i_\ast A,B[1])=0.
\]
This in particular means 
\[
\op{Ext}_{I_{\leq n}}(i_\ast A, j_! B)\cong\op{Hom}(\op{ev}(\iota_n^\ast A),\op{ev}(p_n^\ast B)) 
\]
where the Hom on the right-hand side is taken in the category $\op{Fun}(L_n,\mathbb{D}(\ast))$, i.e., we are considering morphisms of ordinary diagrams over the latching category $L_n$. The diagram $Y\in \op{Fun}(I,\mathcal{X})$ defines an object of morphisms over the latching category, and the above isomorphism yields a corresponding coherent lift of $X$ to the degree $\leq n$ part of $I$.
\end{proof}

\section{Consequences of the tilting theorem}
\label{sec:tiltcompat}

\begin{theorem}
\label{thm:derivatortilting}
Let  $\mathbb{D}$ be a stable derivator and let $\iota:\mathcal{X}\subset \mathbb{D}(\ast)$ be a tilting subcategory. Assume that $\op{hocolim}:\mathbb{D}(\mathbf{\Delta}^{\op{op}})\to\mathbb{D}(\ast)$ maps homotopic maps of simplicial objects to the same maps in $\mathbb{D}(\ast)$. Then $\iota$ extends to a fully faithful triangulated functor 
\[
\op{Hot}^{\op{b}}(\mathcal{X})\xrightarrow{\op{tilt}}\mathbb{D}(\ast).
\]
\end{theorem}

\begin{proof}
We first note that we have the following composition:
\[
\op{Ch}^-(\mathcal{X})\approx \mathbf{\Delta}^{\op{op}}\mathcal{X} \xleftarrow{\op{ev},\approx}\mathbb{D}(\mathbf{\Delta}^{\op{op}})|_{\mathcal{X}}\subset \mathbb{D}(\mathbf{\Delta}^{\op{op}}) \xrightarrow{\op{hocolim}} \mathbb{D}(\ast). 
\]
The first equivalence is given by the Dold--Kan correspondence between complexes bounded above by $0$ and simplicial objects. The second functor is the evaluation functor which is an equivalence by Theorem~\ref{thm:fritz}. The last functor is the homotopy left Kan extension for the terminal functor $\mathbf{\Delta}^{\op{op}}\to \ast$, which is identified with the homotopy colimit of a simplicial object. In other words, the condition that $\mathcal{X}$ is a tilting subcategory allows to produce coherent lifts of simplicial diagrams, unique up to isomorphism, and take their homotopy colimit.

Now we want to pass to the homotopy category. From the Dold--Kan correspondence, we get an induced equivalence $\op{Hot}(\op{Ch}^-(\mathcal{X}))\approx \op{Hot}(\mathbf{\Delta}^{\op{op}}\mathcal{X})$ of homotopy categories. The fact that the above functor factors through $\op{Hot}^{\op{b}}(\mathcal{X})=\op{Hot}(\op{Ch}^-(\mathcal{X}))\to \mathbb{D}(\ast)$ follows from the assumptions. 

Every step in the construction is triangulated, therefore the composition is a triangulated functor. 

We still need to prove that the tilting functor is fully faithful, i.e., that for all complexes $A,B\in\op{Hot}^{\op{b}}(\mathcal{X})$ we have induced isomorphisms
\[
\op{Hot}^{\op{b}}_{\mathcal{X}}(A,B)\to \mathbb{D}(\ast)(\op{tilt}(A),\op{tilt}(B)).
\] 
This is done by an induction on the length of the complexes involved. 
By assumption, $\mathcal{X}\subset \mathbb{D}(\ast)$ is fully faithful, so that the claim is true whenever the complexes $A$ and $B$ have length 1. Now fix a complex $A$ of length $1$, and let $B$ be a complex sitting in an exact sequence $0\to B'\to B\to C\to 0$ where $B'$ has length $n-1$ and $C$ has length $1$. Then we get a ladder of long exact sequences 
\[
\xymatrix{
\cdots \ar[r] & \op{Hot}(A,B') \ar[r] \ar[d]^\cong & \op{Hot}(A,B) \ar[r] \ar[d] &\op{Hot}(A,C) \ar[r] \ar[d]^\cong & \cdots\\
\cdots \ar[r] & \mathbb{D}(\op{tilt}A,\op{tilt}B') \ar[r] & \mathbb{D}(\op{tilt} A,\op{tilt}B) \ar[r] &\mathbb{D}(\op{tilt}A,\op{tilt}C) \ar[r] & \cdots
}
\]
Note that we omitted decorations. The right isomorphism was the base case of the induction, the left isomorphism is the inductive assumption, and the five-lemma implies an isomorphism in the middle. A similar argument shows fully faithfulness for arbitrary bounded complexes in $\mathcal{X}$.
\end{proof}

\begin{remark}
Note that the assumption that $\op{hocolim}:\mathbb{D}(\mathbf{\Delta}^{\op{op}})\to\mathbb{D}(\ast)$ maps homotopic maps of simplicial objects to the same maps in $\mathbb{D}(\ast)$ is satisfied for all the derivators we consider since these arise as localizations of derived categories of sheaves. This works more generally for derivators arising from simplicial model categories. 
\end{remark}

\begin{proposition}
\label{prop:essimg}
Let  $\mathbb{D}$ be a stable derivator and let $\iota:\mathcal{X}\subset \mathbb{D}(\ast)$ be a tilting subcategory. Then the essential image of the tilting functor $\op{tilt}:\op{Hot}^{\op{b}}(\mathcal{X})\to\mathbb{D}(\ast)$ is the triangulated subcategory $\langle\mathcal{X}\rangle_\Delta$ generated by the tilting subcategory. If in addition, $\mathcal{X}$ is idempotent complete, then the essential image is the thick subcategory of $\mathbb{D}(\ast)$ generated by $\mathcal{T}$. 
\end{proposition}

\begin{proof}
First, we note that an induction on the length of the complexes implies that the tilting functor lands inside the triangulated subcategory $\langle\mathcal{T}\rangle_\Delta$. The base case is clear since a complex of length $1$ is sent to the respective object in the tilting subcategory. For the induction step, let $C$ be a complex of length $n$. We can decompose $C$ via a triangle $C'\to C\to A\to C'[1]$, where $C'$ has length $n-1$ and $A$ is a complex of length $1$. Since the tilting functor is triangulated, and $C',A\in\langle\mathcal{T}\rangle_\Delta$, we find that $C\in\langle\mathcal{T}\rangle_\Delta$ as well. 

Now it remains to show that every object $M\in \langle\mathcal{T}\rangle_\Delta$ is in the essential image of the tilting functor. By definition, we can construct $M$ by finitely many distinguished triangles, starting with objects from $\mathcal{T}$. The objects of $\mathcal{T}$ are obviously in the essential image. Now assume that $M_1,M_2$ are in the essential and there is a triangle $M_1\to M_2\to M\to M_1[1]$. By Theorem~\ref{thm:derivatortilting}, the tilting functor is fully faithful, hence we have a unique lift of the morphism $M_1\to M_2$ to $\op{Hot}^{\op{b}}(\mathcal{T})$. This lift $\tilde{M_1}\to\tilde{M_2}$ can be completed to a triangle in $\op{Hot}^{\op{b}}(\mathcal{T})$, which maps to the triangle $M_1\to M_2\to M\to M_1[1]$ hence the cone of $\tilde{M_1}\to\tilde{M_2}$ maps to an object isomorphic to $M$. Therefore, $M$ is in the essential image. By induction on the number of triangles needed to present an object in $\langle\mathcal{T}\rangle_\Delta$, the claim is proved.

Finally, the last statement follows. Any projector which splits in $\mathbb{D}(\ast)$ must also split in the homotopy category of complexes by assumption. 
\end{proof}

The next result provides a general condition on the compatibility of the tilting functor of Theorem~\ref{thm:derivatortilting} with morphisms of derivators. This will be applied in Section~\ref{sec:tilting} to prove compatibility of tilting with restriction functors. 

\begin{theorem}
\label{thm:funtilt}
Let $F:\mathbb{D}_1\to\mathbb{D}_2$ be a morphism of derivators which preserves homotopy left Kan extensions. Assume that there are tilting subcategories $\mathcal{X}_i\subset \mathbb{D}_i(\ast)$ such that $F$ factors through a functor $F:\mathcal{X}_1\to\mathcal{X}_2$. Then there is a diagram, commutative up to isotransformation
\[
 \xymatrix{
\op{Hot}^{\op{b}}(\mathcal{X}_1)\ar[d]_F \ar[r] & \mathbb{D}_1(\ast)\ar[d]^{F} \ar@{=>}[dl]_{\sim}\\
\op{Hot}^{\op{b}}(\mathcal{X}_2) \ar[r] 
& \mathbb{D}_2(\ast). 
}
\]
The horizontal functors are tilting functors which exist by Theorem~\ref{thm:derivatortilting}, and the left vertical functor is the functor induced by applying $F$ levelwise.
\end{theorem}

\begin{proof}
We need to check that every step of the construction in the proof of Theorem~\ref{thm:derivatortilting} is compatible with the functor $F$, i.e., yields isotransformations filling the respective commutative squares. 

First, we note that the Dold--Kan correspondence between complexes in $\mathcal{X}$ and simplicial objects in $\mathcal{X}$ is functorial. In particular, we even get an isotransformation without passing to homotopy categories. Moreover, under the Dold--Kan correspondence, the homotopies for complexes correspond exactly to homotopies for the simplicial objects, yields a diagram 
\[
\xymatrix{
\op{Hot}(\op{Ch}^-(\mathcal{X}_1))\ar[d]_F \ar[r] & \op{Hot}(\mathbf{\Delta}^{\op{op}}(\mathcal{X}_1)) \ar[d]^{F} \ar@{=>}[dl]_{\sim}\\
\op{Hot}(\op{Ch}^-(\mathcal{X}_2)) \ar[r] 
& \op{Hot}(\mathbf{\Delta}^{\op{op}}(\mathcal{X}_2)). 
}
\]

Now we consider the evaluation functors $\op{ev}:\mathbb{D}(\mathbf{\Delta}^{\op{op}})\to \op{Hot}(\mathbf{\Delta}^{\op{op}})$. We required that $F:\mathbb{D}_1\to\mathbb{D}_2$ is a morphism of derivators. This implies, in particular, that for every functor of diagrams $u:I\to J$, there are explicitly given isotransformations $u^\ast\circ F_J\to F_I\circ u^\ast$. Now the evaluation functors are given exactly by applying the evaluation to every object in the diagram. In particular, the isotransformations encoding that $F$ is a morphism of derivators assemble into an isotransformation filling the commutative diagram
\[
\xymatrix{
\mathbb{D}_1(\mathbf{\Delta}^{\op{op}})\ar[d]_F \ar[r]^{\op{ev}} & \op{Hot}(\mathbf{\Delta}^{\op{op}}(\mathcal{X}_1)) \ar[d]^{F} \ar@{=>}[dl]_{\sim}\\
\mathbb{D}_2(\mathbf{\Delta}^{\op{op}}) \ar[r]_{\op{ev}} 
& \op{Hot}(\mathbf{\Delta}^{\op{op}}(\mathcal{X}_2)). 
}
\]

Finally, we need to consider the homotopy colimit step. Since $F$ is a morphism of derivators, we have natural transformations $\gamma_{u,!}^F:u_!\circ F_I\to F_J\circ u_!$ for every functor $u:I\to J$. Requiring that $F$ preserves homotopy left Kan extensions is exactly the requirement that that these natural transformations are isomorphisms for every functor $u:I\to J$. In particular, we get the required isotransformation
\[
\xymatrix{
\mathbb{D}_1(\mathbf{\Delta}^{\op{op}})\ar[d]_F \ar[r]^{u_!} & \mathbb{D}_1(\ast) \ar[d]^{F} \ar@{=>}[dl]_{\sim}\\
\mathbb{D}_2(\mathbf{\Delta}^{\op{op}}) \ar[r]_{u_!} 
& \mathbb{D}_2(\ast).
}
\]
where $u:\mathbf{\Delta}^{\op{op}}\to\ast$ is the terminal morphism. 

Combining the three isotransformation described above proves the claim.
\end{proof}

\begin{theorem}
\label{thm:tiltmonoid}
Let $\mathbb{D}$ be a monoidal stable derivator, and let $\mathcal{X}\subset \mathbb{D}(\ast)$ be a tilting $\otimes$-subcategory. Then the tilting functor $\op{Hot}^{\op{b}}(\mathcal{X})\to \mathbb{D}(\ast)$ is monoidal, where the monoidal structure on $\op{Hot}^{\op{b}}(\mathcal{X})$ is the natural one induced from the monoidal structure on $\mathcal{X}$. 
\end{theorem}

\begin{proof}
We follow the structure of argument of Theorem~\ref{thm:funtilt}. In the first step, the Dold--Kan correspondence induces lax monoidal functors on the categories of complexes. If we use the monoidal Dold--Kan correspondence instead, the functor will be strongly monoidal. For the second step, since a prederivator can be seen as a lift of $\mathsf{Cat}\to\mathsf{CAT}$ to monoidal categories with strong monoidal functors, the functors $u^\ast:\mathbb{D}(K)\to\mathbb{D}(J)$ for any functor $u:J\to K$ are strong monoidal functors. This implies, in particular, that the diagram evaluation functors $\mathbb{D}(\mathbf{\Delta}^{\op{op}})\to\op{Fun}(\mathbf{\Delta}^{\op{op}},\mathbb{D}(\ast))$ are strong monoidal functors. Now the inverse of a strong monoidal equivalence is also strong monoidal.\footnote{This is exercise 1.4.6 in Turaev--Virelizier: Monoidal categories and topological field theory; and is also discussed as MathStackExchange question 183285.}

The final step is now to consider the homotopy colimit. Since we assume that $\mathbb{D}$ is a monoidal derivator in the sense of \cite{groth:ponto:shulman}, $\otimes$ is cocontinuous in each variable, i.e., homotopy colimits in each variable commute with $\otimes$. Consider the following diagram:
\[
\xymatrix{
\mathbb{D}(\mathbf{\Delta}^{\op{op}})\times\mathbb{D}(\mathbf{\Delta}^{\op{op}}) \ar[d]_{\otimes} \ar[r]^{(u_!,\op{id})} \ar@/_5pc/[dd]_{\otimes}& \mathbb{D}(\ast)\times\mathbb{D}(\mathbf{\Delta}^{\op{op}}) \ar[d]^{\otimes} \ar@{=>}[dl]_{\sim} \ar[r]^{(\op{id},u_!)} & \mathbb{D}(\ast)\times\mathbb{D}(\ast) \ar[d]^{\otimes} \ar@{=>}[dl]_{\sim}\\
\mathbb{D}(\mathbf{\Delta}^{\op{op}}\times\mathbf{\Delta}^{\op{op}}) \ar[r]_{u_!} \ar[d]_{\Delta^\ast}
& \mathbb{D}(\ast\times\mathbf{\Delta}^{\op{op}}) \ar[r]_{u_!} & \mathbb{D}(\ast) \\
\mathbb{D}(\mathbf{\Delta}^{\op{op}}) \ar[rru]_{u_!}
}
\]
The isotransformations written in the squares are direct consequences of the fact that we have a monoidal derivator. There is an isotransformation in the left that identifies the internal tensor product as composition of the external tensor product with the restriction along the diagonal, this is a consequence of \cite[Theorem 3.11]{groth:ponto:shulman}. Finally, there is also an isotransformation in the lower triangle, which is a version of the  basic theorem of bisimplicial sets: the realization of a bisimplicial set by first taking horizontal and then vertical realization is equivalent to taking the realization of the diagonal. Taking all these statements together, we find that the outer diagram can be filled by an isotransformation. This implies the claim that the tilting functor is in fact monoidal. 
\end{proof}


\backmatter 

\bibliographystyle{amsalpha}
\bibliography{pub}
\printindex

\end{document}